\documentclass[12pt,reqno]{amsart}
\usepackage{amsmath,amsfonts,amssymb,amsthm,amsxtra,amscd}
\usepackage[all]{xy}
\usepackage{mathrsfs} 
\usepackage{graphicx} 
\numberwithin{equation}{section}

\theoremstyle{plain} 
\newtheorem{thm}{Theorem}[section] 
\newtheorem{prop}[thm]{Proposition} 
\newtheorem{lemma}[thm]{Lemma} 
\newtheorem{cor}[thm]{Corollary}

\theoremstyle{definition} 
\newtheorem{defin}{Definition}[section] 
 
\newtheorem{remark}[thm]{Remark}

\newcommand{\C}{{\mathbb C}}
 
\newcommand{\p}{{\mathbb P}}

\newcommand{\z}{{\mathbb Z}} 
\newcommand{\pj}{{{\mathbb P}^1}}
\newcommand{\pii}{{{\mathbb P}^2}}
\newcommand{\piii}{{{\mathbb P}^3}}

\newcommand{\pv}{{{\mathbb P}^5}}  

\newcommand{\scc}{\mathscr{C}} 
\newcommand{\sce}{\mathscr{E}}
\newcommand{\scf}{\mathscr{F}} 
\newcommand{\scg}{\mathscr{G}}
\newcommand{\sco}{\mathscr{O}} 
\newcommand{\sch}{\mathscr{H}}
\newcommand{\sci}{\mathscr{I}} 

\newcommand{\sck}{\mathscr{K}}
\newcommand{\scl}{\mathscr{L}}
\newcommand{\scm}{\mathscr{M}}
\newcommand{\scn}{\mathscr{N}} 
\newcommand{\scq}{\mathscr{Q}} 
\newcommand{\sct}{\mathscr{T}}

\newcommand{\tH}{\text{H}} 
\newcommand{\h}{\text{h}}

\newcommand{\izo}{\overset{\sim}{\rightarrow}} 
\newcommand{\Izo}{\overset{\sim}{\longrightarrow}} 
\newcommand{\ra}{\rightarrow} 
\newcommand{\lra}{\longrightarrow} 
\newcommand{\xra}{\xrightarrow}  
\newcommand{\vb}{\, \vert \, } 
\newcommand{\prim}{{\, \prime}} 
\newcommand{\secund}{{\prime \prime}}
\newcommand{\Ker}{\text{Ker}\, }
\newcommand{\Cok}{\text{Coker}\, }

\newcommand{\e}{\varepsilon}

\begin{document}

\title[Vector bundles with $c_1 = 5$ on ${\mathbb P}^3$]{Globally generated 
vector bundles with $c_1 = 5$ on ${\mathbb P}^3$} 

\author[Anghel]{Cristian~Anghel} 
\address{Institute of Mathematics ``Simion Stoilow'' of the Romanian Academy, 
         P.O. Box 1-764,   
\newline          
         RO--014700, Bucharest, Romania} 
\email{Cristian.Anghel@imar.ro}   

\author[Coand\u{a}]{Iustin~Coand\u{a}} 
\address{Institute of Mathematics ``Simion Stoilow'' of the Romanian Academy, 
         P.O. Box 1-764, 
\newline 
         RO--014700, Bucharest, Romania} 
\email{Iustin.Coanda@imar.ro} 

\author[Manolache]{Nicolae~Manolache} 
\address{Institute of Mathematics ``Simion Stoilow'' of the Romanian Academy, 
         P.O. Box 1-764, 
\newline 
         RO--014700, Bucharest, Romania} 
\email{Nicolae.Manolache@imar.ro} 



\begin{abstract} 
We provide a classification of globally generated vector bundles with 
$c_1 = 5$ on the projective 3-space. The classification is complete (except 
for one case) but not as detailed as the corresponding classification in the 
case $c_1 = 4$ from our paper [Memoirs A.M.S., Vol.~253, No.~1209 (2018), also 
arXiv:1305.3464]. We determine, at least, the pairs of integers $(a , b)$ for 
which there exist globally generated vector bundles on the projective 3-space 
with Chern classes $c_1 = 5$, $c_2 = a$, $c_3 = b$ (except for the case 
$(12 , 0)$ and the complementary case $(13 , 5)$ which remain undecided), we 
describe the Horrocks monads of these vector bundles and we organise them into 
several families with irreducible bases. 
We use some of the results from our paper [arXiv:1502.05553] (for which we 
give, however, a direct selfcontained proof in one of the appendices of the 
present paper) to reduce the problem to the classification of stable rank 3 
vector bundles $F$ with $c_1(F) = -1$, $2 \leq c_2(F) \leq 4$, having the 
property that $F(2)$ is globally generated. We use, then, the spectrum of such 
a bundle to get the necessary cohomological information. Some of the 
constructions appearing in the present paper are used (and reproduced, for the 
reader's convenience) in another paper of ours [arXiv:1711.06060] in which we 
provide an alternative to Chang and Ran's proof of the unirationality of 
the moduli spaces of curves of degree at most 13 from 
[Invent. Math. 76 (1984), 41--54].      
\end{abstract}

\maketitle 
\tableofcontents

\section*{Introduction} 

In this paper, which is a sequel to our work \cite{acm1}--\cite{acm3}, we 
classify the globally generated vector bundles with $c_1 = 5$ on $\piii$ 
(the analogous but much simpler classification on $\pii$ can be found in 
\cite[Sect.~3]{acm1}).   
If $E$ is such a bundle and if $c_2$, $c_3$ are the other two Chern classes of 
$E$ then, as a consequence of the theorem of Riemann-Roch (recalled in 
Remark~\ref{R:chern}), $c_3 \equiv c_1c_2 \pmod{2}$ hence, in our case, $c_3 
\equiv c_2 \pmod{2}$. We shall work, most of the time, under some 
\emph{additional assumptions} (that appear, already, in the paper of Sierra 
and Ugaglia \cite{su}). Firstly, we can assume that $\tH^i(E^\vee) = 0$, 
$i= 0,\, 1$, (where $E^\vee$ is the dual of $E$). Indeed, if $E$ is a globally 
generated vector bundle on $\piii$ then $\tH^0(E^\vee) = 0$ if and only if 
$E$ has no trivial direct summand and, in this case, considering the 
universal extension $0 \ra s\sco_\piii \ra \widetilde{E} \ra E \ra 0$ with 
$s := \h^1(E^\vee)$, $\widetilde{E}$ satisfies $\tH^i(\widetilde{E}^\vee) = 0$, 
$i = 0,\, 1$. 

Secondly, we can assume that $c_2 \leq c_1^2/2$ (hence, in our case, that 
$c_2 \leq 12$). Indeed, if $E$ is a globally generated vector bundle on 
$\piii$ then the dual $P(E)$ of the kernel of the evaluation morphism 
$\tH^0(E) \otimes_k \sco_\piii \ra E$ has ``complementary'' Chern classes 
$c_1(P(E)) = c_1$, $c_2(P(E)) = c_1^2 - c_2$, 
$c_3(P(E)) = c_3 + c_1(c_1^2 - 2c_2)$ and if, moreover, $\tH^i(E^\vee) = 0$, 
$i = 0,\, 1$, then $P(P(E)) \simeq E$.  

There is, also, a third additional assumption, specific to the case $c_1 = 5$.
If $c_1 = 5$, $c_2 \leq 12$, $\tH^i(E^\vee) = 0$, $i= 0,\, 1$, and 
$\tH^0(E(-2)) \neq 0$ then, by \cite[Prop.~2.4]{acm1}, \cite[Prop.~2.10]{acm1} 
and \cite[Thm.~0.1]{acm2}, either $E \simeq \sco_\piii(a) \oplus E_1$, where 
$a$ is an integer with $2 \leq a \leq 5$ and $E_1$ is a globally generated 
vector bundle with $c_1(E_1) = 5 - a$ and $\tH^i(E_1^\vee) = 0$, $i = 0,\, 1$, 
or $E \simeq G(3)$, where $G$ is a stable rank 2 vector bundle with 
$c_1(G) = -1$ and $c_2(G) = 2$ (we provide, for the reader's convenience, a 
different proof of this fact in Appendix~\ref{A:h0e(-2)neq0}). Consequently, 
we shall also assume that $\tH^0(E(-2)) = 0$.

The idea that makes this classification possible is the following one$\, :$ 
asume that $E$ has rank $r \geq 3$. Then $r - 3$ general global sections of 
$E$ define an exact sequence$\, :$ 
\[
0 \lra (r - 3)\sco_\piii \lra E \lra E^\prim \lra 0\, , 
\]
where $E^\prim$ is a rank 3 vector bundle, with the same Chern classes as $E$. 
It is quite easy to reduce the classification of globally generated vector 
bundles $E$ of rank $r \geq 3$, with $c_1 = 5$ and satisfying the above 
additional assumptions to the case where $E^\prim$ is \emph{stable} (see 
Lemma~\ref{L:eprimunstable}). The advantage of this reduction is that one can 
use, now, the cohomological information furnished by the \emph{spectrum} of 
a stable rank 3 vector bundle (see Okonek and Spindler \cite{oks2}, 
\cite{oks3}, and Coand\u{a} \cite{coa}$\, ;$ their results are recalled, with 
complete, partially new proofs, in Appendix~\ref{A:spectrum}). 
For technical reasons one works, actually,    
with the ``normalized'' vector bundle $F := E^\prim(-2)$ which has $c_1(F) = 
-1$. If $c_2 \leq 12$ then $c_2(F) \leq 4$. Since, for a stable rank 3 
vector bundle $F$ with $c_1(F) = -1$, one has $c_2(F) \geq 1$ and if 
$c_2(F) = 1$ then $F \simeq \Omega_\piii(1)$, our classification problem 
reduces to the following one$\, :$ determine the stable rank 3 vector bundles 
$F$ on $\piii$ with $c_1(F) = -1$ and $2 \leq c_2(F) \leq 4$, such that $F(2)$ 
is globally generated. Note that the stable rank 3 vector bundles $F$ with 
$c_1(F) = -1$ and $c_2(F) = 2$ were studied by Okonek and Spindler \cite{oks} 
but there is no similar study for $c_2(F) \geq 3$. 

The result we obtain is the following$\, :$ 

\begin{thm}\label{T:main} 
Let $E$ be an indecomposable globally generated vector bundle on $\piii$, 
of rank at least $2$, with Chern classes $c_1 = 5$, $c_2 \leq 12$ and $c_3$, 
and such that ${\fam0 H}^i(E^\vee) = 0$, $i = 0,\, 1$. Then one of the 
following holds$\, :$ 
\begin{enumerate} 
\item[(i)] $c_2 = 8$, $c_3 = 0$ and $E \simeq G(3)$ where $G$ is a rank $2$ 
vector bundle with $c_1(G) = -1$, $c_2(G) = 2$, and ${\fam0 H}^0(G) = 0$$\, ;$ 
\item[(ii)] $c_2 = 9$, $c_3 = 5$ and $E \simeq \Omega_\piii(3)$$\, ;$ 
\item[(iii)] $c_2 = 10$, $c_3 = 6$ and $E$ can be realized as a non-trivial 
extension$\, :$ 
\[
0 \lra \sco_\piii(1) \lra E \lra G(2) \lra 0\, , 
\] 
where $G$ is a $2$-instanton$\, ;$ 
\item[(iv)] $c_2 = 10$, $c_3 = 4$ and $E(-2)$ is the kernel of an epimorphism 
$\sco_\piii(1) \oplus 3\sco_\piii \ra \sco_\piii(2)$$\, ;$ 
\item[(v)] $c_2 = 10$, $c_3 = 0$ and $E \simeq G(3)$ where $G$  is a general 
rank $2$ vector bundle with $c_1(G) = -1$, $c_2(G) = 4$, and 
${\fam0 H}^0(G(1)) = 0$$\, ;$ 
\item[(vi)] $c_2 = 11$, $c_3 = 9$ and $E(-2)$ is the cohomology of a 
(not necessarily minimal) monad of the form$\, :$ 
\[
0 \lra \sco_\piii(-1) \lra 5\sco_\piii \oplus 2\sco_\piii(-1) \lra 
2\sco_\piii(1) \lra 0\, ;
\]
\item[(vii)] $c_2 = 11$, $c_3 = 7$ and $E$ can be realized as a non-trivial 
extension$\, :$ 
\[
0 \lra G(2) \lra E \lra \sco_\piii(1) \lra 0\, , 
\]
where $G$ is a $3$-instanton with ${\fam0 h}^0(G(1)) \leq 1$$\, ;$ 
\item[(viii)] $c_2 = 11$, $c_3 = 7$ and $E(-2)$ is the cohomology of a general 
monad of the form$\, :$ 
\[
0 \lra 2\sco_\piii(-1) \lra 8\sco_\piii \lra 3\sco_\piii(1) \lra 0\, ; 
\]
\item[(ix)] $c_2 = 11$, $c_3 = 5$ and $E(-2)$ is the kernel of an arbitrary  
epimorphism ${\fam0 T}_\piii(-1) \oplus \sco_\piii \ra \sco_\piii(2)$$\, ;$ 
\item[(x)] $c_2 = 12$, $c_3 = 14$ and $E$ has a resolution of the form$\, :$ 
\[
0 \lra \sco_\piii(-1) \xra{\left(\begin{smallmatrix} u\\ v 
\end{smallmatrix}\right)} 
E_1 \oplus 4\sco_\piii \lra E \lra 0\, , 
\]
with $v$ defined by $x_0, \ldots , x_3$ and with $E_1$ defined by an exact 
sequence$\, :$ 
\[
0 \lra \sco_\piii(-1) \oplus \sco_\piii \lra E_1^\vee \lra \sci_X \lra 0\, , 
\]
where $X$ is either the union of two disjoint lines or its degeneration, a 
double line on a nonsingular quadric surface$\, ;$ 
\item[(xi)] $c_2 = 12$, $c_3 = 14$ and $E(-2)$ is the kernel of a general 
epimorphism $2\sco_\piii \oplus 6\sco_\piii(-1) \ra \sco_\piii(1) \oplus 
\sco_\piii$. In the typical situation, $E(-1)$ is the kernel of the evaluation 
morphism $6\sco_\piii \ra \sco_H(2)$ of $\sco_H(2)$, where $H$ is a plane in 
$\piii$$\, ;$ 
\item[(xii)] $c_2 = 12$, $c_3 = 12$ and $E$ has a resolution$\, :$ 
\[
0 \lra \sco_\piii(-1) \xra{\left(\begin{smallmatrix} u\\ v 
\end{smallmatrix}\right)} 
G(2) \oplus 4\sco_\piii \lra E \lra 0\, , 
\] 
where $G$ is a $3$-instanton with ${\fam0 h}^0(G(1)) \leq 1$ and $v$ is 
defined be $x_0, \ldots , x_3$$\, ;$ 
\item[(xiii)] $c_2 = 12$, $c_3 = 12$ and $E(-2)$ is the cohomology of a 
(not necessarily minimal) monad of the form$\, :$ 
\[
0 \lra \sco_\piii(-1) \lra 4\sco_\piii \oplus 4\sco_\piii(-1) \lra 
2\sco_\piii(1) \lra 0\, ; 
\] 
\item[(xiv)] $c_2 = 12$, $c_3 = 10$ and $E(-2)$ is the cohomology of a general 
(not necessarily minimal) monad of the form$\, :$ 
\[
0 \lra 2\sco_\piii(-1) \lra 7\sco_\piii \oplus 2\sco_\piii(-1) \lra 
3\sco_\piii(1) \lra 0\, ; 
\] 
\item[(xv)] $c_2 = 12$, $c_3 = 8$ and $E$ can be realized as a non-trivial 
extension$\, :$ 
\[
0 \lra G(2) \lra E \lra \sco_\piii(1) \lra 0\, , 
\]
where $G$ is a $4$-instanton with ${\fam0 h}^0(G(1)) \leq 1$$\, ;$
\item[(xvi)] $c_2 = 12$, $c_3 = 8$ and $E(-2)$ is the cohomology of general 
monad of the form$\, :$ 
\[
0 \lra 3\sco_\piii(-1) \lra 10\sco_\piii \lra 4\sco_\piii(1) \lra 0\, ; 
\]
\item[(xvii)] $c_2 = 12$, $c_3 = 6$ and $E(-2)$ is the cohomology of a general 
monad of the form$\, :$ 
\[
0 \lra 2\sco_\piii(-1) \lra 7\sco_\piii \lra \sco_\piii(2) \oplus \sco_\piii(1) 
\lra 0\, . 
\]
\end{enumerate} 
\end{thm} 

The theorem follows from Prop.~\ref{P:c2geq9}, 
Prop.~\ref{P:eprimunstablec2=10}, Prop.~\ref{P:eprimstablec2=10}, 
Prop.~\ref{P:eprimunstablec2=11}, Prop.~\ref{P:h1e(-3)=0c2=11}, 
Prop.~\ref{P:h2e(-3)neq0c2=11}, Prop.~\ref{P:h2e(-3)=0h1e(-3)neq0c2=11}, 
Prop.~\ref{P:h0e(-1)geq2c2=12}, Prop.~\ref{P:eprimunstablec2=12}, 
Prop.~\ref{P:h1e(-3)=0c2=12}, Prop.~\ref{P:h2e(-3)neq0c2=12} and 
Prop.~\ref{P:h2e(-3)=0h1e(-3)neq0c2=12} from the next sections. We list, 
actually, in those propositions all the globally generated vector bundles $E$ 
on $\piii$ with $c_1 = 5$, $c_2 \leq 12$, $\tH^i(E^\vee) = 0$, $i = 0,\, 1$, 
and $\tH^0(E(-2)) = 0$, indecomposable or not. We have to say that 
our proof of the above theorem is incomplete in the sense that we were 
unable to show that there exists no globally generated rank 2 vector bundle 
on $\piii$ with $c_1 = 5$, $c_2 = 12$ (however, as Ph.~Ellia communicated us, 
this is one of the results of a forthcoming paper by Ellia, Gruson and Skiti 
on stable rank 2 vector bundles on $\piii$ with $c_1 = -1$ and "minimal" 
spectrum $(0, \ldots , 0, -1 , \ldots , -1)$).  
The above mentioned propositions also show that all the indecomposable 
globally generated vector bundles $E$ on $\piii$ with $c_1 = 5$, $c_2 \leq 
12$, satisfy $\tH^i(E) = 0$ for $i \geq 1$ (a condition that ensures the 
openess of global generation in flat families). 

Our classification in the case $c_1 = 5$ is less precise than the 
classification in the case $c_1 = 4$ from \cite{acm1} since we were not able 
to give an explict meaning of the term ``general'', which is used several times 
in the statement of the above theorem. We were, also, not able to show that 
the monads occuring in item (xvi) of the theorem can be put toghether into a 
family with irreducible base (although this is probably true). 

It is hard to comment on such a lenghty proof. 
Its most difficult cases are those for which $c_2 = 12$, $c_3 \leq 12$, and 
the rank 3 vector bundle $E^\prim$ associated to $E$ is stable (or $E$ has 
rank 2). They occupy more than half of the proof (see 
Prop.~\ref{P:h2e(-3)=0h1e(-3)neq0c2=12}). In these cases, using the 
cohomological information furnished by the spectrum of the ``normalized'' 
rank 3 vector bundle $F := E^\prim(-2)$, we are able to describe the Horrocks 
monad of $E$, under the assumption that $E$ is globally generated. Then we 
have to decide whether there exist monads of the previously determined shape 
whose cohomology sheaves are globally generated. 

For example, if $c_2 = 12$ and $c_3 = 6$, we show, firstly, that if $E$ is 
globally generated then $E(-2)$ is the cohomology sheaf of a monad of the 
form$\, :$ 
\[
0 \lra 2\sco_\piii(-1) \lra 7\sco_\piii \lra \sco_\piii(2) \oplus \sco_\piii(1) 
\lra 0\, ,  
\]     
and we prove, next, that there really exist globally generated vector bundles 
$E$ such that $E(-2)$ is the cohomology sheaf of such a monad (see 
Construction 6.3 and Construction 6.4 in the proof of 
Prop.~\ref{P:h2e(-3)=0h1e(-3)neq0c2=12}). 

As another example, if $c_2 = 12$ and $c_3 = 4$ we prove, assuming $E$ 
globally generated, that $E(-2)$ is the cohomology sheaf of a monad of the 
form$\, :$ 
\[
0 \lra \sco_\piii(-1) \lra 2\sco_\piii(1) \oplus 4\sco_\piii \lra 
2\sco_\piii(2) \lra 0\, . 
\]
Then we show that if $F$ is the cohomology sheaf of such a monad then $F(2)$ 
cannot be globally generated (see Case 8 in the proof of 
Prop.~\ref{P:h2e(-3)=0h1e(-3)neq0c2=12}). In particular, we show that if 
$F$ is the kernel of an epimorphism $2\sco_\piii(1) \oplus \text{T}_\piii(-1) 
\ra 2\sco_\piii(2)$ then $F(2)$ is not globally generated. We do not have a 
general method for proving such statements. We use specific arguments in each 
case. This explains, in part, the length of the paper.  

In the ``missing case'' $c_2 = 12$, $c_3 = 0$ (and $\text{rk}\, E = 2$) one 
can show that, in case $E$ is globally generated, $E(-3)$ is the cohomology 
sheaf of an anti-selfdual monad of the form$\, :$ 
\[
0 \lra 3\sco_\piii(-2) \overset{\beta}{\lra} 4\sco_\piii \oplus 
4\sco_\piii(-1) \overset{\alpha}{\lra} 3\sco_\piii(1) \lra 0\, , 
\] 
with the property that the degeneracy locus of the component $\alpha_1 \colon 
4\sco_\piii \ra 3\sco_\piii(1)$ of $\alpha$ has codimension 2 in $\piii$. We 
were not able to show that if $G$ is the cohomology sheaf of such a monad 
then $G(3)$ cannot be globally generated. Notice that $\chi(G(3)) = 3$ hence 
for a general monad of the above shape one should have $\tH^1(G(3)) = 0$ while 
if $G(3)$ is globally generated then necessarily $\h^0(G(3)) \geq 5$, i.e., 
$\h^1(G(3)) \geq 2$ which shows that the monads producing globally generated 
vector bundles are ``special'' among the above monads (if they exist at all). 

Our initial motivation for writing this paper was (and still is) to get a 
starting point for the classification of globally generated vector bundles 
with $c_1 = 5$ on $\p^n$, $n \geq 4$, which is likely to be accomplished in a 
reasonable number of pages. We found, meanwhile, that some of the 
constructions appearing in this paper can be used to get an alternative to 
Chang and Ran's arguments \cite{cr} showing the unirationality of the moduli 
spaces of curves of genus $g \leq 13$. This is the subject of our recent paper 
\cite{acm4}.  

As for the organization of the paper, the main sections, which are devoted to 
the analysis of the various cases of our classification problem, are followed 
by a number of appendices containig complementary or auxiliary results. 

Appendix~\ref{A:h0e(-2)neq0} contains a short proof of the 
case $\tH^0(E(-2)) \neq 0$ of the classification of globally generated vector 
bundles with $c_1 = 5$ on $\piii$. We included this proof here to keep the 
paper selfcontained, that is, to avoid any reference to our lenghty 
paper \cite{acm2}. 

In Appendix~\ref{A:spectrum} we give complete, partially new proofs of the 
properties of the spectrum of a stable rank 3 vector bundle on $\piii$, 
following the approach of Hartshorne \cite[Sect.~7]{ha} and 
\cite[Prop.~5.1]{ha2}. 

The Appendices~\ref{A:(3;-1,2,-)}--\ref{A:instantons} contain auxiliary 
results about some special classes of vector bundles on $\piii$ related to 
our classification. 

In Appendix~\ref{A:serre} we include a proof of a general version of 
Serre's method of extensions (which is used in several constructions of 
globally generated vector bundles throughout the paper) while 
Appendix~\ref{A:miscellaneous} gathers a number of miscellaneous auxiliary 
results.       

\vskip2mm 

\noindent
{\bf Notation.}\quad (i) We denote by $S = k[x_0, \ldots ,x_n]$ the projective 
coordinate ring of the projective $n$-space $\p^n$ over an algebraically 
closed field $k$ of characteristic 0. 

(ii) If $\scf$ is a coherent sheaf on $\p^n$ and $i \in \z$, we denote by 
$\tH^i_\ast(\scf)$ the graded $S$-module $\bigoplus_{l \in \z}\tH^i(\scf(l))$. 

(iii) If $X$ is a closed subscheme of $\p^n$, we denote by 
$\sci_X \subset \sco_\p$ its ideal sheaf. If $Y$ is a closed subscheme of 
$X$, we denote by $\sci_{Y,X} \subset \sco_X$ the ideal sheaf defining $Y$ as 
a closed subscheme of $X$. In other words, $\sci_{Y,X} = \sci_Y/\sci_X$. 

(iv) If $\scf$ is a coherent sheaf on $\p^n$ and $X\subset \p^n$ a closed 
subscheme, we put $\scf_X := \scf \otimes_{\sco_\p} \sco_X$ and $\scf \vb X := 
i^\ast\scf$, where $i : X \ra \p^n$ is the inclusion.              

%\vskip1cm 

\section{Preliminaries}\label{S:prelim} 

Let $E$ be a globally generated vector bundle of rank $r$ on $\piii$, with 
$c_1 = 5$. 
Then $r-1$ \emph{general} global sections of $E$ define an exact sequence$\, :$ 
\begin{equation}\label{E:oeiy} 
0 \lra (r-1)\sco_\piii \lra E \lra \sci_Y(5) \lra 0 
\end{equation}
with $Y$ a nonsingular (but not necessarily connected) curve, of degree $c_2$ 
and with $\chi(\sco_Y) = -\frac{1}{2}(c_3 + c_2)$ (see Remark~\ref{R:chern} 
below). Dualizing \eqref{E:oeiy} one gets an exact sequence$\, :$ 
\begin{equation}\label{E:oeiydual} 
0 \lra \sco_\piii(-5) \lra E^\vee \lra (r-1)\sco_\piii \overset{\delta}{\lra} 
\omega_Y(-1) \lra 0\, . 
\end{equation}   

\begin{lemma}\label{L:yconnected}  
Let $E$ be a globally generated vector bundle on $\piii$ with $c_1 = 5$. If 
$c_2 \leq 13$ then the nonsingular curve $Y$ occuring in \eqref{E:oeiy} is 
connected. 
\end{lemma}

\begin{proof}
We list, firstly, the nonsigular connected curves $C \subset \piii$, of degree 
$d \leq 7$, such that $\omega_C(-1)$ is globally generated. The last condition 
implies that $d \geq 4$ and that $2g-2-d = \text{deg}\, \omega_C(-1) \geq 0$. 
The plane curves of degree $4 \leq d \leq 7$ obviously occur in our list. 
If $C$ is not a plane curve then the Castelnuovo bound (see 
\cite[IV,~Thm.~6.4]{hag}) implies that $g \leq 1$ for $d = 4$, 
$g \leq 2$ for $d = 5$, $g \leq 4$ for $d = 6$, and $g \leq 6$ for $d = 7$. 
Recalling the relation $2g - 2 \geq d$, one sees that the only possible cases 
are $d = 6$, $g = 4$ and $d = 7$, $g \in \{5,\, 6\}$. The case $d = 7$, 
$g = 5$ cannot occur because, in that case, $\omega_C(-1)$ would be a globally 
generated line bundle of degree 1 on $C$ and this would contradict the fact 
that $C$ has positive genus. In the remaining two cases, $g$ takes the 
largest value allowed by the Castelnuovo bound, hence $C$ is contained in a 
quadric surface. If $d = 6$, $g = 4$ then $C$ is a complete intersection of 
type $(2,3)$. If $d = 7$, $g = 6$ then either $C$ is contained in a quadric 
cone or it is a divisor of type $(3,4)$ on a nonsingular quadric surface. 
In both cases, $C$ is directly linked to a line by a complete intersection of 
type $(2,4)$. In particular, $C$ is arithmetically Cohen-Macaulay.        

We show, next, that none of the connected components of $Y$ is a plane curve. 
\emph{Indeed}, if $C$ were such a component then, as we saw above, it must 
have degree $d \geq 4$. Then, any line $L$ contained in the plane $H$ that 
contains $C$ and such that $\text{length}(L \cap (Y \setminus C)) \geq 2$ 
would be a secant to Y, of order at least $d + 2 \geq 6$. Since 
$\text{deg}(Y \setminus C) \geq 4$, such lines $L$ certainly would exist 
and this would \emph{contradict} the fact that $\sci_Y(5)$ is globally 
generated.

Assume, now, that $Y$ is not connected. In this case one must have 
$Y = C \cup C^\prim$ with $C$, $C^\prim$ nonsingular connected curves, $C$ of 
degree 6 and genus 4, and $C^\prim$ either of degree 6 and genus 4 or of degree 
7 and genus 6. In both cases, $C$ and $C^\prim$ are arithmetically 
Cohen-Macaulay and $I(C)_2 = kf$, $I(C^\prim )_2 = kf^\prim$, with $f,\, f^\prim$ 
quadratic forms in four indeterminates. Then $I(Y) = I(C)I(C^\prim)$ (see, for 
example, \cite[Lemma~B.1]{acm2}) hence $I(Y)_5 = fI(C^\prim)_3 + 
f^\prim I(C)_3$. One deduces that the subscheme of $\piii$ defined by the forms 
in $I(Y)_5 = \tH^0(\sci_Y(5))$ contains the complete intersection 
$\{f = f^\prim = 0\}$ which \emph{contradicts} the fact that $\sci_Y(5)$ is 
globally generated. It thus remains that $Y$ is connected. 
\end{proof}

We want to show, now, that most part of the classification of globally 
generated vector bundles $E$ on $\piii$ with $c_1 = 5$ and 
such that $\tH^i(E^\vee) = 0$, $i = 0,\, 1$, and $\tH^0(E(-2)) = 0$ reduces 
to the classification of globally generated \emph{stable} rank 3 vector bundles 
with the same Chern classes. \emph{Indeed}, if such a bundle $E$ has rank 
$r \geq 3$ then $r-3$ general global sections of $E$ define an exact 
sequence$\, :$ 
\begin{equation}\label{E:oeeprim} 
0 \lra (r-3)\sco_\piii \lra E \lra E^\prim \lra 0 
\end{equation}  
where $E^\prim$ is a globally generated rank 3 vector bundle with the same 
Chern classes as $E$. Consider the \emph{normalized} rank 3 vector bundle 
$F := E^\prim(-2)$. It has Chern classes$\, :$  
\begin{equation}\label{E:chernf}
c_1(F) = -1,\   c_2(F) = c_2 - 8,\   c_3(F) = c_3 - 2c_2 + 12\, . 
\end{equation} 
Since the condition $\tH^i(E^\vee) = 0$, $i = 0,\, 1$, is equivalent, by Serre 
duality, to the condition $\tH^i(E(-4)) = 0$, $i = 2,\, 3$, one deduces, from 
the exact sequence \eqref{E:oeeprim} that$\, :$ 
\begin{equation}\label{E:r} 
r = 3 + \h^2(E^\prim(-4)) = 3 + \h^2(F(-2))\, . 
\end{equation}
Notice that $E^\prim$ is stable if and only if $F$ is stable which is 
equivalent to $\tH^0(F) = 0$ and $\tH^0(F^\vee(-1)) = 0$. Since, from the exact 
sequence \eqref{E:oeeprim}, one has $\tH^0(E^\prim(-2)) = 0$ it follows 
that $E^\prim$ is stable if and only if $\tH^0(E^{\prim \vee}(1)) = 0$. 

\begin{lemma}\label{L:eprimunstable} 
Let $E$ be a globally generated vector bundle of rank $r \geq 3$ on $\piii$, 
with $c_1 = 5$ and such that ${\fam0 H}^i(E^\vee) = 0$, $i = 0,\, 1$, and 
${\fam0 H}^0(E(-2)) = 0$. Let $E^\prim$ be the rank $3$ vector bundle 
associated to $E$ in the exact sequence \eqref{E:oeeprim}. If $E^\prim$ is not 
stable then one of the following holds$\, :$ 
\begin{enumerate} 
\item[(i)] $c_3 = c_2 - 4$ and $E^\prim$ can be realized as an 
extension$\, :$ 
\[
0 \lra F^\prim(2) \lra E^\prim \lra \sco_\piii(1) \lra 0
\]
where $F^\prim$ is a stable rank $2$ vector bundle with  
$c_1(F^\prim) = 0$, $c_2(F^\prim) = c_2 - 8$$\, ;$ 
\item[(ii)] $c_3 = c_2$ and $E^\prim$ can be realized as an extension$\, :$ 
\[
0 \lra F^\prim(2) \lra E^\prim \lra \sci_L(1) \lra 0
\]
where $F^\prim$ is a stable rank $2$ vector bundle with  
$c_1(F^\prim) = 0$, $c_2(F^\prim) = c_2 - 9$ and $L$ is a line. 
\end{enumerate} 
\end{lemma} 

\begin{proof} 
As we noticed before the statement of the lemma, $E^\prim$ is not stable if and 
only if $\tH^0(E^{\prim \vee}(1)) \neq 0$. In this case, a non-zero section $s$ 
of $\tH^0(E^{\prim \vee}(1))$ defines a non-zero morphism $\phi : E^\prim \ra 
\sco_\piii(1)$. Since $\tH^0(E^{\prim \vee}) = 0$ (dualize the exact sequence 
\eqref{E:oeeprim}), the image of $\phi$ is of the form $\sci_Z(1)$, where $Z$ 
(= the zero scheme of $s$) is a closed subscheme of $\piii$, of codimension 
$\geq 2$. $E^\prim$ globally generated implies $\sci_Z(1)$ globally generated, 
hence $Z$ must be the empty set, a simple point or a line. But 
$c_3(E^{\prim \vee}(1)) = -c_3 + c_2 -4 \equiv 0 \pmod{2}$ (because $c_3 \equiv 
c_2 \pmod{2}$). One deduces that $Z$ cannot be a simple point. The lemma 
follows immediately.  
\end{proof} 

\begin{remark}\label{R:eprimunstable} 
It follows immediately from formula \eqref{E:r} that$\, :$  

(i) If $E^\prim$ is as in Lemma~\ref{L:eprimunstable}(i) then 
$r = 3 + \h^2(F^\prim(-2))$$\, ;$ 

(ii) If $E^\prim$ is as in Lemma~\ref{L:eprimunstable}(ii) then 
$r = 5 + \h^2(F^\prim(-2))$. 
\end{remark}

\begin{remark}\label{R:eprimunstable2}
(i) Assume that $E^\prim$ is as in Lemma~\ref{L:eprimunstable}(i). Since 
$E^\prim$ is globally generated, the map $\tH^0(E^\prim) \ra \tH^0(\sco_\piii(1))$ 
must be surjective. Applying the Snake Lemma to the diagram$\, :$ 
\[
\begin{CD}
0 @>>> \Omega_\piii(1) @>{u}>> 4\sco_\piii @>>> \sco_\piii(1) 
@>>> 0\\
@. @V{\phi}VV @VVV @\vert\\ 
0 @>>> F^\prim (2) @>>> E^\prim @>>> \sco_\piii(1) @>>> 0  
\end{CD} 
\]
one gets an exact sequence$\, :$ 
\[
0 \lra \Omega_\piii(1) \xra{\left(\begin{smallmatrix} \phi\\ u 
\end{smallmatrix}\right)}  
F^\prim (2) \oplus 4\sco_\piii \lra E^\prim \lra 0\, . 
\]
The condition $E^\prim$ globally generated is equivalent to the fact that 
the morphism$\, :$ 
\[
(\text{ev}\, ,\, \phi) \colon (\tH^0(F^\prim (2)) \otimes_k \sco_\piii) \oplus 
\Omega_\piii(1) \lra F^\prim (2) 
\] 
is an epimorphism. In particular, $F^\prim(3)$ must be globally generated. 

Now, let $\xi \in \tH^1(F^\prim(1))$ be an element generating the image of 
$\tH^1(\phi(-1)) : \tH^1(\Omega_\piii) \ra \tH^1(F^\prim (1))$. Modulo the 
isomorphism $\text{Ext}^1(\sco_\piii(1), F^\prim (2)) \simeq \tH^1(F^\prim (1))$, 
$\xi$ defines the extension from Lemma~\ref{L:eprimunstable}(i). Moreover, 
$h\xi = 0$ in $\tH^1(F^\prim (2))$, $\forall \, h \in \tH^0(\sco_\piii(1))$. 
Finally, if $\xi$ generates $\tH^1(F^\prim (1))$ and $\tH^2(F^\prim ) = 0$ then 
$E^\prim$ is 0-regular. 

Notice that if $\tH^2(F^\prim(-2)) = 0$, i.e., if $F^\prim$ is a mathematical 
instanton bundle, then, by Remark~\ref{R:eprimunstable}(i), $E$ has rank 3 
hence $E = E^\prim$. 

\vskip2mm 

(ii) Assume that $E^\prim$ is as in Lemma~\ref{L:eprimunstable}(ii). Since 
$E^\prim$ is globally generated, the map $\tH^0(E^\prim) \ra \tH^0(\sci_L(1))$ 
must be surjective. Applying the Snake Lemma to the diagram$\, :$ 
\[
\begin{CD}
0 @>>> \sco_\piii(-1) @>>> 2\sco_\piii @>>> \sci_L(1) @>>> 0\\
@. @VVV @VVV @\vert\\
0 @>>> F^\prim(2) @>>> E^\prim @>>> \sci_L(1) @>>> 0
\end{CD}
\]  
one deduces an exact sequence$\, :$ 
\[
0 \lra \sco_\piii(-1) \lra F^\prim(2) \oplus 2\sco_\piii \lra E^\prim \lra 0\, . 
\]
Assume, now, that $\tH^2(F^\prim(-2)) = 0$, i.e., that $F^\prim$ is a 
mathematical instanton bundle. In this case, by 
Remark~\ref{R:eprimunstable}(ii), $E$ has rank 5. 
It follows that 
$\text{Ext}^1(F^\prim(2), \sco_\piii) \simeq \tH^1(F^{\prim \vee}(-2)) 
\simeq \tH^2(F^\prim(-2))^\vee = 0$. One gets a commutative diagram$\, :$ 
\[
\begin{CD}
0 @>>> \sco_\piii(-1) @>>> F^\prim(2) \oplus 2\sco_\piii @>>> E^\prim @>>> 0\\
@. @VVV @VVV @\vert\\
0 @>>> 2\sco_\piii @>>> E @>>> E^\prim @>>> 0
\end{CD}
\]
from which one deduces an exact sequence$\, :$ 
\[
0 \lra \sco_\piii(-1) \xra{\left(\begin{smallmatrix} u\\ v 
\end{smallmatrix}\right)} 
F^\prim(2) \oplus 4\sco_\piii \lra E \lra 0\, . 
\]
Since $\tH^i(E^\vee) = 0$, $i = 0,\, 1$, it follows that $\tH^0(v^\vee)$ is an 
isomorphism. If, moreover, the multiplication map 
$\tH^0(F^\prim(2)) \otimes \tH^0(\sco_\piii(1)) \ra \tH^0(F^\prim(3))$ is 
surjective one can assume, up to an automorphism of 
$F^\prim(2) \oplus 4\sco_\piii$, that $u = 0$. Consequently, under the above 
additional assumptions, one has $E \simeq F^\prim(2) \oplus \text{T}_\piii(-1)$. 
\end{remark} 

\begin{prop}\label{P:c2geq9} 
Let $E$ be a globally generated vector bundle on $\piii$ with $c_1 = 5$ and 
such that ${\fam0 H}^i(E^\vee) = 0$, $i = 0,\, 1$, and ${\fam0 H}^0(E(-2)) = 
0$. Then $c_2 \geq 9$ and if $c_2 = 9$ then $c_3 = 5$ and one of the following 
holds$\, :$ 
\begin{enumerate}
\item[(i)] $E \simeq \Omega_\piii(3)$$\, ;$ 
\item[(ii)] $E \simeq \sco_\piii(1) \oplus N(2)$, where $N$ is a 
nullcorrelation bundle. 
\end{enumerate}
\end{prop} 

\begin{proof}
Assume, firstly, that $E$ has rank 2. In this case, $E = G(3)$ where $G$ is a 
rank 2 vector bundle with $c_1(G) = -1$ and $\tH^0(G(1)) = 0$. In particular, 
$G$ is stable (i.e., $\tH^0(G) = 0$). It follows that $c_2(G) \geq 2$ (use 
\cite[Cor.~3.3]{ha} and the fact that $c_2(G) \equiv 0 \pmod{2}$). But, as 
shown by Hartshorne and Sols \cite{hs} and by Manolache \cite{ma}, if 
$c_2(G) = 2$ then $\tH^0(G(1)) \neq 0$. It remains that $c_2(G) \geq 4$ hence 
$c_2 = c_2(G) + 3c_1(G) + 3^2 \geq 10$. 

Assume, now, that $E$ has rank $\geq 3$. Let $E^\prim$ be the rank 3 vector 
bundle associated to $E$ in the exact sequence \eqref{E:oeeprim} and $F = 
E^\prim(-2)$. According to Lemma~\ref{L:eprimunstable}, one has to consider 
three cases$\, :$ 

\vskip2mm 

\noindent
{\bf Case 1.}\quad $E^\prim$ \emph{as in Lemma}~\ref{L:eprimunstable}(i). 

\vskip2mm 

\noindent
In this case, $c_2 = c_2(F^\prim) + 8$. Since $F^\prim$ is stable it follows 
that $c_2(F^\prim) \geq 1$ hence $c_2 \geq 9$. Moreover, if $c_2 = 9$, i.e., 
if $c_2(F^\prim) = 1$, then $F^\prim$ is isomorphic to a nullcorrelation bundle 
$N$. Since $\tH^1(N(1)) = 0$ it follows that $E^\prim \simeq \sco_\piii(1) 
\oplus N(2)$. Since $\tH^1(N(-2)) = 0$, formula \eqref{E:r} implies that 
$E$ has rank 3, hence $E = E^\prim$. 

\vskip2mm 

\noindent
{\bf Case 2.}\quad $E^\prim$ \emph{as in Lemma}~\ref{L:eprimunstable}(ii). 

\vskip2mm 

\noindent
In this case, $c_2 = c_2(F^\prim) + 9 \geq 10$. 

\vskip2mm 

\noindent
{\bf Case 3.}\quad $E^\prim$ \emph{stable}. 

\vskip2mm

\noindent
In this case, $F$ is stable with $c_1(F) = -1$ hence $c_2(F) \geq 1$ (see 
Schneider \cite{sch}) hence $c_2 = c_2(F) + 8 \geq 9$. Moreover, if $c_2 = 9$, 
i.e., if $c_2(F) = 1$, then $F \simeq \Omega_\piii(1)$. Since 
$\tH^2(\Omega_\piii(-1)) = 0$, formula \eqref{E:r} shows that $E$ has rank 3, 
hence $E = E^\prim \simeq \Omega_\piii(3)$.       
\end{proof} 

\begin{remark}\label{R:spectrumf} 
Let us recall, from Okonek and Spindler \cite{oks2}, \cite{oks3}, and from 
\cite{coa}, the definition and main properties of the spectrum of a stable 
rank 3 vector bundle. Let $F$ be a stable rank 3 vector bundle on $\piii$, 
with $c_1(F) = -1$, $c_2(F) = c$. The \emph{spectrum} of $F$ is a sequence 
$k_F = (k_i)_{1 \leq i \leq c}$ of integers $k_1 \geq k_2 \geq \cdots \geq k_c$ 
with the following properties$\, :$ 
\begin{enumerate} 
\item[(i)] $\h^1(F(l)) = \h^0(\bigoplus_{i=1}^c\sco_\pj(k_i + l + 1))$ for 
$l \leq -1$$\, ;$ 
\item[(ii)] $\h^2(F(l)) = \h^1(\bigoplus_{i=1}^c\sco_\pj(k_i + l + 1))$ for 
$l \geq -2$$\, ;$ 
\item[(iii)] $-2\sum k_i = c_3(F) + c\  (= c_3(F(1)))$$\, ;$ 
\item[(iv)] If $k \geq 0$ occurs in the spectrum then $0,\, 1, \ldots ,\, k$ 
occur too$\, ;$ 
\item[(v)] If $k \leq -1$ occurs in the spectrum then $-1,\, -2, \ldots ,\, k$ 
occur too$\, ;$ 
\item[(vi)] If $0$ does not occur in the spectrum then $-1$ occurs at least  
twice$\, ;$ 
\item[(vii)] If $-1 \geq k_{i-1} > k_i > k_{i+1}$ for some $i$ with $2 \leq i 
\leq c-1$ then $k_{i+1} > k_{i+2} > \cdots > k_c$ and $F$ has an \emph{unstable 
plane} $H$ of order $-k_c$, that is, $\tH^0(F_H^\vee(k_c)) \neq 0$ and 
$\tH^0(F_H^\vee(k_c-1)) = 0$.  
\end{enumerate}   

Notice that if $F = E^\prim (-2)$ , where $E^\prime$ is the rank 3 vector bundle 
associated to a globally generated vector bundle $E$ with $c_1 = 5$ in the 
exact sequence \eqref{E:oeeprim}, then relation (iii) above becomes$\, :$ 
\[
-2{\textstyle \sum} k_i = c_3 - c_2 + 4\, . 
\]
We give, for the reader's convenience, complete (partially new) arguments 
for the above properties of the spectrum in Appendix~\ref{A:spectrum}. 
Another reason for including that appendix in the paper is that we use the 
main technical point of Hartshorne's approach to the spectrum in the proof of 
Lemma~\ref{L:h2e(-3)=0} below.    
\end{remark}

\begin{lemma}\label{L:spectrumf} 
Let $F$ be a stable rank $3$ vector bundle on $\piii$ with $c_1(F) = -1$, 
$c_2(F) = c$ and let $k_F = (k_i)_{1 \leq i \leq c}$ be its spectrum. Assume 
that $2 \leq c \leq 4$ and that $F(2)$ is globally generated. Then 
$1 \geq k_1 \geq \cdots \geq k_c \geq -2$. 
\end{lemma}

\begin{proof} 
Let $E := F(2)$. One has $c_1(E) = 5$, $c_2(E) = c + 8 \leq 12$. By 
Lemma~\ref{L:yconnected}, $E$ can be realized as an extension$\, :$ 
\[
0 \lra 2\sco_\piii \lra E \lra \sci_Y(5) \lra 0
\] 
with $Y$ a nonsingular \emph{connected} curve. It follows that $\tH^1(F(-3)) 
= \tH^1(E(-5)) \simeq \tH^1(\sci_Y) = 0$ and this implies that $k_1 \leq 1$. 

On the other hand, it follows from the properties of the spectrum that if 
$c = 2$ then $k_2 \geq -1$ and if $c = 3$ then $k_3 \geq -2$. If $c = 4$  then 
the only spectra for which $k_4 \leq -3$ are $(0,-1,-2,-3)$ and 
$(-1,-1,-2,-3)$. For both of these spectra, $F$ has an unstable plane $H$ of 
order 3. A non-zero element of $\tH^0(F_H^\vee(-3))$ defines an epimorphism 
$F_H \ra \sci_{Z,H}(-3) \ra 0$, where $Z = \emptyset$ or it is a 0-dimensional 
subscheme of $H$. But this \emph{contradicts} the fact that $F(2)$ is 
globally generated.  
\end{proof}

\begin{remark}\label{R:chern} 
We record, here, a number of formulae concerning Chern classes that we shall 
need in the sequel. 

(a) If $\scf$ is a coherent sheaf of rank $r$ on $\pii$ with Chern classes 
$c_1,\, c_2$, then the \emph{Riemann-Roch formula} says that$\, :$ 
\[
\chi(\scf(l)) = \chi\left(((r-1)\sco_\pii \oplus \sco_\pii(c_1))(l)\right) 
- c_2\, , \  \forall \, l \in \z\, .  
\] 
It follows, in particular, that$\, :$ 
\begin{equation}\label{E:c2onp2} 
c_2 = \chi(\sco_\pii(c_1 - 1)) - \chi(\scf(-1))\, . 
\end{equation}
Now, let $\Gamma$ be a closed subscheme of $\pii$, of dimension $\leq 0$. The 
Hilbert polynomial $t \mapsto \chi(\sco_\Gamma(t))$ is a constant that we 
denote by $\text{deg}\, \Gamma$. Then, for $t \in \z$$\, :$ 
\[
c_1(\sci_\Gamma(t)) = t\, ,\  c_2(\sci_\Gamma(t)) = \text{deg}\, \Gamma \, . 
\] 
\emph{Indeed}, the first relation follows by restricting the exact 
sequence$\, :$ 
\[
0 \lra \sci_\Gamma(t) \lra \sco_\pii(t) \lra \sco_\Gamma(t) \lra 0
\] 
to a line $L \subset \pii$ not intersecting $\Gamma$, and the second one by 
applying \eqref{E:c2onp2}. 

\vskip2mm 

(b) If $\scf$ is a coherent sheaf of rank $r$ on $\piii$ with Chern classes 
$c_1,\, c_2\, , c_3$, then the \emph{Riemann-Roch formula} says that$\, :$ 
\[
\chi(\scf(l)) = \chi\left(((r-1)\sco_\piii \oplus \sco_\piii(c_1))(l)\right) 
- (l + 2)c_2 + \frac{1}{2}(c_3 - c_1c_2)\, , \  \forall \, l \in \z\, .  
\]  
It follows, in particular, that$\, :$ 
\begin{equation}\label{E:c3onp3} 
\frac{1}{2}(c_3 - c_1c_2) = \chi(\scf(-2)) - \chi(\sco_\piii(c_1 - 2))\, . 
\end{equation}
Now, let $Z$ be a closed subscheme of $\piii$ of dimension $\leq 1$. The 
Hilbert polynomial of $\sco_Z$ has the form $\chi(\sco_Z(t)) = dt + 
\chi(\sco_Z)$, for some non-negative integer $d$, which we denote by 
$\text{deg}\, Z$ (hence $\text{deg}\, Z = 0$ if and only if $\dim Z \leq 0$). 
Then, for $t \in \z$$\, :$ 
\[
c_1(\sci_Z(t)) = t\, ,\  c_2(\sci_Z(t)) = \text{deg}\, Z\, ,\  
c_3(\sci_Z(t)) = (4 - t)\text{deg}\, Z - 2\chi(\sco_Z)\, . 
\]  
\emph{Indeed}, the first two relations follow by restricting $\sci_Z(t)$ to a 
general plane, and the third relation from \eqref{E:c3onp3}.  
%can be deduced by computing 
%$\chi(\sci_Z(t-2))$ in two different ways$\, :$ firstly, by using the exact 
%sequence$\, :$ 
%\[
%0 \lra \sci_Z(t-2) \lra \sco_\piii(t-2) \lra \sco_Z(t-2) \lra 0
%\]
%and, then, by applying the Riemann-Roch formula, with $l = -2$, to $\scf := 
%\sci_Z(t)$. 

\vskip2mm 

(c) Using the notation from (b), one has an exact sequence$\, :$ 
\[
0 \lra \sct \lra \sco_Z \lra \sco_{Z_{\text{CM}}} \lra 0 
\]
where $Z_{\text{CM}}$ is a closed subscheme of $\piii$, locally Cohen-Macaulay 
of pure dimension 1 (or empty), and with $\dim \text{Supp}\, \sct \leq 0$. 
Of course, $\text{deg}\, Z_{\text{CM}} = \text{deg}\, Z$ and $\chi(\sco_Z) = 
\chi(\sco_{Z_{\text{CM}}}) + \text{length}\, \sct$. 

Assume, now, that one has an exact sequence$\, :$ 
\[
0 \lra \scf \lra \sce \lra \sci_Z(t) \lra 0 
\] 
with $\sce$ locally free of rank 3, $\scf$ reflexive of rank 2, and $Z$ as 
above. Then $c_3(\scf) = \text{length}\, \sct$ and $\, :$ 
\begin{gather*} 
c_1(\sce) = c_1(\scf) + t\, ,\  c_2(\sce) = c_2(\scf) + c_1(\scf)t + 
\text{deg}\, Z_{\text{CM}}\, ,\\ 
c_3(\sce) = -c_3(\scf) + c_2(\scf)t + (c_1(\scf) - t + 4)\text{deg}\, 
Z_{\text{CM}} - 2\chi(\sco_{Z_{\text{CM}}})\, . 
\end{gather*}
\emph{Indeed}, it follows from \cite[Prop.~2.6]{ha} that $c_3(\scf) = 
\text{length}\, \sce xt^1(\scf , \omega_\piii)$. But  
\[
\sce xt^1(\scf , \omega_\piii) \simeq \sce xt^2(\sci_Z(t) , \omega_\piii) 
\simeq \sce xt^3(\sco_Z(t) , \omega_\piii) \simeq 
\sce xt^3(\sct , \omega_\piii)
\]
and $\text{length}\, \sce xt^3(\sct , \omega_\piii) = \text{length}\, \sct$. 
The formulae about the Chern classes of $\sce$ follow, now, from (b). 
\end{remark}

\begin{remark}\label{R:monadsandliaison} 
Let us recall a method of constructing monads of vector bundles on $\piii$ 
using liaison techniques. Let $E$ be a vector bundle on $\piii$ appearing as 
an extension$\, :$ 
\[
0 \lra A \lra E \lra \sci_Y(t) \lra 0 
\]
where $A$ is a direct sum of line bundles and $Y$ is a curve (= locally 
Cohen-Macaulay closed subscheme of $\piii$, of pure dimension 1). Assume that 
one knows a monad$\, :$ 
\[
B^\bullet \  :\  0 \lra B^{-1} \overset{d^{-1}}{\lra} B^0 
\overset{d^0}{\lra} B^1 \lra 0 
\] 
of $\sci_Y(t)$ (this means that the $B^i$'s are direct sums of line bundles, 
$\sch^i(B^\bullet) = 0$ for $i \neq 0$ and $\sch^0(B^\bullet) \simeq \sci_Y(t)$). 
The above extension defines a morphism $B^\bullet \ra A[1]$ in the derived 
category of coherent sheaves on $\piii$. This morphism can be represented by 
a morphism of sheaves $\phi : B^{-1} \ra A$. Then$\, :$ 
\[
0 \lra B^{-1} \xra{\left(\begin{smallmatrix} d^{-1}\\ \phi 
\end{smallmatrix}\right)} 
B^0 \oplus A \xra{(d^0,\, 0)} B^1 \lra 0 
\]
is a monad of $E$. 

\vskip2mm 

Now, assume that $Y$ can be directly linked to another curve $Y^\prime$ by a 
complete intersection defined by two homogeneous polynomials $f,\, g$, of 
degrees $a$ and $b$, respectively. Assume that one knows a resolution$\, :$ 
\[
0 \lra A_2 \overset{\delta_2}{\lra} A_1 \overset{\delta_1}{\lra}  
A_0 \lra \sci_{Y^\prime} \lra 0 
\]
of $\sci_{Y^\prime}$, with $A_0,\, A_1,\, A_2$ direct sums of line bundles. The 
morphism $\sco_\piii(-a) \oplus \sco_\piii(-b) \ra \sci_{Y^\prime}$ defined by 
$(f,\, g)$ can be lifted to a morphism $\psi : \sco_\piii(-a) \oplus 
\sco_\piii(-b) \ra A_0$. Then, by a result of D. Ferrand (see Peskine and 
Szpiro \cite[Prop.~2.5]{ps}), 
\[
0 \lra A_0^\vee \xra{\left(\begin{smallmatrix} \delta_1^\vee\\ \psi^\vee 
\end{smallmatrix}\right)} 
A_1^\vee \oplus \sco_\piii(a) \oplus \sco_\piii(b) 
\xra{(\delta_2^\vee ,\, 0)} A_2^\vee \lra 0 
\]
is a monad of $\sci_Y(a+b)$. 

\vskip2mm 

Notice that if $f$ belongs to a minimal system of generators of the 
homogeneous ideal $I(Y^\prime) \subset S$ then the component $\psi_1 : 
\sco_\piii(-a) \ra A_0$ of $\psi$ maps $\sco_\piii(-a)$ isomorphically onto a 
direct summand of $A_0$ (i.e., $A_0 = A_0^\prime \oplus \sco_\piii(-a)$) and, in 
this case, one gets a simplified monad of $\sci_Y(a+b)$$\, :$ 
\[
0 \lra A_0^{\prime \vee} \xra{\left(\begin{smallmatrix} \delta_1^{\prime \vee}\\ 
\psi_2^{\prime \vee} \end{smallmatrix}\right)} A_1^\vee \oplus \sco_\piii(b) 
\xra{(\delta_2^\vee ,\, 0)} A_2^\vee \lra 0\, .  
\]
\end{remark}

\begin{lemma}\label{L:edual(1)} 
Let $E$ be a globally generated vector bundle on $\piii$ with $c_1 = 5$, 
and such that ${\fam0 H}^i(E^\vee) = 0$, $i = 0,\, 1$, and ${\fam0 H}^0(E(-2)) 
= 0$. 

\emph{(a)} If ${\fam0 H}^1(E(-3)) = 0$ then $E^\vee$ is $1$-regular. 

\emph{(b)} If ${\fam0 H}^0(E^\vee (1)) = 0$ and if one considers the 
universal extension$\, :$ 
\[
0 \lra E(-3) \lra E_3 \lra {\fam0 H}^1(E(-3))\otimes_k \sco_\piii \lra 0  
\]   
then $E_3$ is $1$-regular. Moreover, ${\fam0 H}^0(P(E_3(1))(-2)) = 0$. 
\end{lemma}

\begin{proof} 
(a) One has $\tH^1(E^\vee ) = 0$, $\tH^2(E^\vee (-1)) \simeq \tH^1(E(-3))^\vee 
= 0$, and $\tH^3(E^\vee (-2)) \simeq \tH^0(E(-2))^\vee = 0$. 

(b) By construction, $\tH^1(E_3) = 0$. Then $\tH^2(E_3(-1)) \simeq 
\tH^2(E(-4)) \simeq \tH^1(E^\vee )^\vee = 0$, and $\tH^3(E_3(-2)) \simeq 
\tH^3(E(-5)) \simeq \tH^0(E^\vee (1))^\vee = 0$. Finally, $\tH^0(P(E_3(1))(-2)) 
\simeq \tH^1(E_3^\vee (-3)) \simeq \tH^2(E_3(-1))^\vee = 0$. 
\end{proof}

\begin{lemma}\label{L:h2e(-3)=1} 
Let $E$ be a globally generated vector bundle on $\piii$ with $c_1 = 5$, 
$c_2 \leq 13$, such that ${\fam0 H}^i(E^\vee) = 0$, $i = 0,\, 1$, and 
${\fam0 H}^0(E(-2)) = 0$. If ${\fam0 h}^2(E(-3)) = 1$ then there exist  
exact sequences$\, :$ 
\begin{gather*} 
0 \lra \sco_\piii(-1) \xra{\left(\begin{smallmatrix} u\\ v 
\end{smallmatrix}\right)} E_1 \oplus 4\sco_\piii \lra E \lra 0\, ,\\ 
0 \lra (r-5)\sco_\piii \lra E_1 \lra \scf_1(2) \lra 0\, ,  
\end{gather*}
with $v \colon \sco_\piii(-1) \ra 4\sco_\piii$ defined by four linearly 
independent linear forms and with $\scf_1$ a stable rank $2$ reflexive sheaf 
with $c_1(\scf_1) = 0$, $c_2(\scf_1) = c_2 - 9$, $c_3(\scf_1) = c_3 - c_2$.  
\end{lemma}

\begin{proof} 
Recall the exact sequence \eqref{E:oeiy} and the fact that, by 
Lemma~\ref{L:yconnected}, the curve $Y$ appearing there is connected. 
One has $\h^2(E(-3)) = \h^2(\sci_Y(2)) = \h^1(\sco_Y(2)) = 
\h^0(\omega_Y(-2))$ hence $\h^0(\omega_Y(-2)) = 1$.   
Let $\sigma$ be a nonzero global section of $\omega_Y(-2)$. $x_0\sigma , \ldots 
, x_3\sigma$ are linearly independent elements of $\tH^0(\omega_Y(-1))$. 
Consider $\sigma_4, \ldots , \sigma_{r-2} \in \tH^0(\omega_Y(-1))$ such that 
$x_0\sigma, \ldots , \sigma_{r-2}$ is a $k$-basis of $\tH^0(\omega_Y(-1))$ 
(recall that $\h^0(\omega_Y(-1)) = r-1$). Then $\sigma, \sigma_4, \ldots , 
\sigma_{r-2}$ define an extension$\, :$  
\[
0 \lra \sco_\piii(-1) \oplus (r-5)\sco_\piii \lra E_1 \lra \sci_Y(5) \lra 0 
\]   
with $E_1$ a locally free sheaf. Dualizing this extension, one gets an exact 
sequence$\, :$ 
\[
0 \lra \sco_\piii(-5) \lra E_1^\vee \lra \sco_\piii(1) \oplus (r-5)\sco_\piii 
\xra{\displaystyle \delta_1} \omega_Y(-1) \lra 0 
\]
with $\delta_1$ defined by $\sigma, \sigma_4, \ldots , \sigma_{r-2}$. One 
deduces that $\tH^i(E_1^\vee) = 0$, $i = 0,\, 1$. $\sigma$ alone defines an 
extension$\, :$   
\[
0 \lra \sco_\piii(-3) \lra \scf_1 \lra \sci_Y(3) \lra 0 
\]
with $\scf_1$ a rank 2 reflexive sheaf with Chern classes $c_1(\scf_1) = 0$, 
$c_2(\scf_1) = \text{deg}\, Y - 9 = c_2 - 9$, $c_3(\scf_1) = 
\text{deg}(\omega_Y(-2)) = c_3 - c_2$. Since $\tH^0(\sci_Y(3)) = 0$, $\scf_1$ 
is stable. By construction, one has an exact sequence$\, :$ 
\[
0 \lra (r-5)\sco_\piii \lra E_1 \lra \scf_1(2) \lra 0\, . 
\]
Since $\tH^1(E_1^\vee) = 0$ one gets a commutative diagram$\, :$ 
\[
\begin{CD} 
0 @>>> \sco_\piii(-1) \oplus (r-5)\sco_\piii @>>> E_1 @>>> \sci_Y(5) @>>> 0\\
@. @VVV @VVV @\vert\\
0 @>>> (r-1)\sco_\piii @>>> E @>>> \sci_Y(5) @>>> 0
\end{CD}
\]
from which one deduces an exact sequence$\, :$ 
\[
0 \lra \sco_\piii(-1) \oplus (r-5)\sco_\piii \lra E_1 \oplus (r-1)\sco_\piii 
\lra E \lra 0\, . 
\]
Dualizing this exact sequence and taking into account that $\tH^i(E^\vee) = 0$ 
and $\tH^i(E_1^\vee) = 0$, $i = 0,\, 1$, one gets the first exact sequence from 
the statement. 

One can deduce, using this exact sequence, that $E_1$ has rank $r - 3$ and 
Chern classes $c_1(E_1) = 4$, $c_2(E_1) = c_2 - 5$, $c_3(E_1) = c_3 - c_2$.  
Moreover, denoting by $\scq$ the cokernel of the 
evaluation morphism $\tH^0(E_1)\otimes_k\sco_\piii \ra E_1$, the condition $E$ 
globally generated is equivalent to the fact that the composite morphism 
$\sco_\piii(-1) \overset{u}{\lra} E_1 \ra \scq$ is an epimorphism. In 
particular, $E_1(1)$ must be globally generated. 
\end{proof}

\begin{remark}\label{R:c1=5onp2} 
We recall, here, from \cite{acm1}, the part of the classification of globally 
generated vector bundles with $c_1 = 5$ on $\pii$ that we shall need in the 
sequel. Let $E^\prim$ be a globally generated vector bundle on $\pii$, with 
$c_1 = 5$ and $10 \leq c_2 \leq 13$. By \cite[Lemma~1.2]{acm1}, one has 
$E^\prim \simeq G^\prim \oplus t\sco_\pii$, with $G^\prim$ defined by an exact 
sequence$\, :$ 
\[
0 \lra s\sco_\pii \lra F^\prim \lra G^\prim \lra 0\, , 
\]
where $F^\prim$ is a globally generated vector bundle, with the same Chern 
classes as $E^\prim$, and such that $\tH^i(F^{\prim \vee}) = 0$, $i = 0,\, 1$. 
Moreover, $s = \h^1(E^{\prim \vee})$ and $t = \h^0(E^{\prim \vee})$. 

Now, according to the proof of \cite[Prop.~3.6]{acm1}, one has$\, :$ 
\begin{enumerate} 
\item[(i)] If $c_2 = 10$ then $F^\prim$ is either $\sco_\pii(3) \oplus 
P(\sco_\pii(2))$ or $\sco_\pii(2) \oplus 2\sco_\pii(1) \oplus 
\text{T}_\pii(-1)$$\, ;$ 
\item[(ii)] If $c_2 = 11$ then $F^\prim$ is either $\sco_\pii(2) \oplus 
\sco_\pii(1) \oplus 2\text{T}_\pii(-1)$ or $4\sco_\pii(1) \oplus 
\text{T}_\pii(-1)$$\, ;$ 
\item[(iii)] If $c_2 = 12$ then $F^\prim$ is either $\sco_\pii(2) \oplus 
3\text{T}_\pii(-1)$ or $\sco_\pii(2) \oplus \sco_\pii(1) \oplus P(\sco_\pii(2))$ 
or $3\sco_\pii(1) \oplus 2\text{T}_\pii(-1)$$\, ;$ 
\item[(iv)] If $c_2 = 13$ then $F^\prim$ is either $\sco_\pii(2) \oplus 
\text{T}_\pii(-1) \oplus P(\sco_\pii(2))$ or $3\sco_\pii(1) \oplus 
P(\sco_\pii(2))$ or $2\sco_\pii(1) \oplus 3\text{T}_\pii(-1)$.   
\end{enumerate}

Notice that, actually, (iv) follows from (iii). 
\end{remark} 

\begin{remark}\label{R:muh1e(-4)} 
As a consequence of Remark~\ref{R:c1=5onp2}, if $E$ is a globally generated 
vector bundle on $\piii$ with $c_1 = 5$ and $11 \leq c_2 \leq 13$ then 
$\tH^0(E_H(-3)) = 0$, for every plane $H \subset \piii$. It follows that, for 
every nonzero linear form $h \in \tH^0(\sco_\piii(1))$, the multiplication by 
$h \colon \tH^1(E(-4)) \ra \tH^1(E(-3))$ is injective. Applying the Bilinear 
Map Lemma \cite[Lemma~5.1]{ha} to the multiplication map $\mu \colon 
\tH^1(E(-4)) \otimes \tH^0(\sco_\piii(1)) \ra \tH^1(E(-3))$ one deduces that 
the rank of $\mu$ is at least $\h^1(E(-4)) + 3$ (assuming that $\tH^1(E(-4)) 
\neq 0$).  
\end{remark}

\begin{lemma}\label{L:h2e(-3)=0}  
Let $E$ be a globally generated vector bundle on $\piii$ of rank $r \geq 3$, 
with $c_1 = 5$, $10 \leq c_2 \leq 13$, such that ${\fam0 H}^i(E^\vee) = 0$, 
$i = 0,\, 1$, and ${\fam0 H}^0(E(-2)) = 0$. Assume that the rank $3$ vector 
bundle $E^\prim$ associated to $E$ in the exact sequence \eqref{E:oeeprim} is 
stable and that ${\fam0 H}^2(E(-3)) = 0$. 
Put $s := {\fam0 h}^1(E(-3)) - {\fam0 h}^1(E(-4))$. Then$\, :$ 

\emph{(a)} ${\fam0 H}^2(E(l)) = 0$, $\forall \, l \geq -4$, hence 
${\fam0 H}^1(E^\vee (l)) = 0$, $\forall \, l \leq 0$$\, ;$ 

\emph{(b)} The graded $S$-module ${\fam0 H}^1_\ast(E^\vee)$ is generated in 
degrees $\leq 2$ and if ${\fam0 H}^1(E(-4)) = 0$ then it is generated by 
${\fam0 H}^1(E^\vee(1))$$\, ;$ 

\emph{(c)} ${\fam0 H}^0(E_H^\vee) = 0$, for any plane $H \subset \piii$,  
and ${\fam0 h}^1(E_H^\vee) = s$, for the general (resp., any) plane $H \subset 
\piii$ if $c_2 = 10$ (resp., $11 \leq c_2 \leq 13$)$\, ;$  

\emph{(d)} ${\fam0 H}^0(E^\vee(1)) \izo {\fam0 H}^0(E_H^\vee(1))$ and 
${\fam0 h}^1(E_H^\vee(1)) = {\fam0 h}^1(E^\vee(1)) + {\fam0 h}^2(E^\vee)$, 
for any plane $H \subset \piii$$\, ;$ 

\emph{(e)} If $s \leq 1$ then ${\fam0 H}^1_\ast(E^\vee) = 0$ hence  
${\fam0 H}^2_\ast(E) = 0$$\, ;$ 

\emph{(f)} If $11 \leq c_2 \leq 13$ then   
${\fam0 h}^1(E_H^\vee(l)) \leq {\fam0 max}({\fam0 h}^1(E_H^\vee(l-1)) - 1 , 0)$, 
$\forall \, l \geq 1$, for any plane $H \subset \piii$$\, ;$   

\emph{(g)} If $11 \leq c_2 \leq 13$ and $s = 2$ then either 
${\fam0 H}^1_\ast(E^\vee) \simeq {\underline k}(-1)$ or 
${\fam0 H}^1_\ast(E^\vee) = 0$$\, ;$ 

\emph{(h)} If $11 \leq c_2 \leq 13$, $s = 3$, and ${\fam0 H}^1_\ast(E^\vee)$ is 
not generated by ${\fam0 H}^1(E^\vee(1))$ then either ${\fam0 H}^1_\ast(E^\vee) 
\simeq {\underline k}(-1) \oplus {\underline k}(-2)$ or 
${\fam0 H}^1_\ast(E^\vee) \simeq {\underline k}(-2)$.  
\end{lemma}

\begin{proof}
Consider the normalized rank 3 vector bundle $F := E^\prim (-2)$. It has 
Chern classes $c_1(F) = -1$, $c_2(F) = c_2 - 8$. Since $F$ is stable, 
the restriction theorem of Schneider \cite{sch} (see, also, Ein et al. 
\cite[Thm.~3.4]{ehv}) implies that $F_H$ is stable, for the general plane 
$H \subset \piii$. Since $F_H$ has rank 3, this means that $\tH^0(F_H) = 0$ 
and $\tH^0(F_H^\vee (-1)) = 0$. In particular, $\tH^0(E_H(-2)) = 0$, for the 
general plane $H \subset \piii$. Notice that, by Remark~\ref{R:c1=5onp2}, 
one has, for any plane $H \subset \piii$, $\tH^0(E_H(-4)) = 0$ if $c_2 = 10$, 
and $\tH^0(E_H(-3)) = 0$ if $11 \leq c_2 \leq 13$.   

(a) One has $\tH^2(E(-3)) = 0$ and $\tH^3(E(-4)) \simeq \tH^0(E^\vee)^\vee = 0$ 
hence (see, for example, \cite[Lemma~1.21(a)]{acm1}) $\tH^2(E(l)) = 0$, 
$\forall \, l \geq -3$. Moreover, $\tH^2(E(-4)) \simeq \tH^1(E^\vee)^\vee = 0$. 

(b) One has $\tH^2(E^\vee(1)) \simeq \tH^1(E(-5))^\vee \simeq \tH^1(F(-3))^\vee = 
0$ by Lemma~\ref{L:spectrumf} and $\tH^3(E^\vee) \simeq \tH^0(E(-4))^\vee = 0$. 
The first assertion follows, now, from a slight generalization of the 
Castelnuovo-Mumford lemma (see, for example, \cite[Lemma~1.21(b)]{acm1}). If  
$\tH^1(E(-4)) = 0$ then $\tH^2(E^\vee) \simeq \tH^1(E(-4))^\vee = 0$ 
and $\tH^3(E^\vee(-1)) \simeq \tH^0(E(-3))^\vee = 0$ and one applies, again,  
the above mentioned result. 

(c) One uses the fact that $\tH^0(E^\vee) = 0$, that $\tH^1(E^\vee(-1)) \simeq 
\tH^2(E(-3))^\vee = 0$, that $\tH^1(E^\vee) = 0$, and that $\tH^2(E_H^\vee) 
\simeq \tH^0(E_H(-3))^\vee = 0$, for the general plane $H \subset \piii$ if 
$c_2 = 10$ and for any plane if $11 \leq c_2 \leq 13$ (plus Serre duality).  

(d) One uses the fact that $\tH^i(E^\vee) = 0$, $i = 0,\, 1$, that 
$\tH^2(E_H^\vee(1)) \simeq \tH^0(E_H(-4))^\vee = 0$, for any plane $H \subset 
\piii$, and that $\tH^2(E^\vee(1)) \simeq \tH^1(E(-5))^\vee = 0$, by 
Lemma~\ref{L:yconnected}.  

(e) Let $H \subset \piii$ be a general plane. According to 
\cite[Lemma~1.2]{acm1} and to (c), one has an exact sequence$\, :$ 
\[
0 \lra s\sco_H \lra K \lra E_H \lra 0\, , 
\]  
where $K$ is a globally generated vector bundle on $H \simeq \pii$ with 
$\tH^i(K^\vee) = 0$, $i = 0,\, 1$. Since $\tH^2(K^\vee(-1)) \simeq 
\tH^0(K(-2))^\vee = 0$, $K^\vee$ is 1-regular. In particular, $K^\vee(1)$ is 
globally generated. If $\e : t\sco_H \ra \sco_H(1)$ is an epimorphism then 
$\tH^0(\e(l))$ is surjective, $\forall \, l \geq 0$. One deduces that 
$\tH^1(E_H^\vee(l)) = 0$, $\forall \, l \geq 1$. Since $\tH^1(E^\vee) = 0$, 
one deduces easily that $\tH^1(E^\vee(l)) = 0$, $\forall l \geq 1$. 
Together with (a) this implies that $\tH^1_\ast(E^\vee ) = 0$.   

(f) We treat, firstly, the case $l = 1$. If $H \subset \piii$ is an arbitrary 
plane then $\h^1(E_H^\vee) = s$ (by (c)) and $\h^1(E_H^\vee(1)) = \h^1(E^\vee(1)) 
+ \h^2(E^\vee)$ (by (d)). It follows that, in order to prove that 
$\h^1(E_H^\vee(1)) \leq \text{max}(\h^1(E_H^\vee) - 1 , 0)$, one can assume that 
$H$ is a general plane. We shall, actually, assume that $F_H$ is stable. 
By Serre duality on $H$, one has $\h^1(E_H^\vee) = \h^1(E_H(-3))$ and 
$\h^1(E_H^\vee(1)) \simeq \h^1(E_H(-4))$. Using the exact sequence$\, :$ 
\[
0 \lra (r - 3)\sco_H \lra E_H \lra F_H(2) \lra 0\, ,  
\]  
and applying Prop.~\ref{P:nsubseth1astf}(a) from Appendix~\ref{A:spectrum} 
to $F_H$ one gets that $\h^1(E_H(-4)) = 0$ or $\h^1(E_H(-4)) < 
\h^1(E_H(-3))$.     

Assume, now, that $l \geq 2$ and that $H \subset \piii$ is an arbitrary 
plane. One has $\tH^0(F_H(-1)) = 0$ (because $\tH^0(E_H(-3)) = 0$ by  
Remark~\ref{R:c1=5onp2}) and $\tH^0(F_H^\vee(-2)) = 0$ (because 
$\tH^0(E_H^\vee) = 0$ by (c)). Remark~\ref{R:conditiononf}(ii) implies, now, 
that $F_H(-1)$ satisfies the condition from the hypothesis of 
Prop.~\ref{P:nsubseth1astf}. One has $\h^1(E_H^\vee(l)) = \h^1(E_H(-l-3))$ and 
$\h^1(E_H^\vee(l-1)) = \h^1(E_H(-l-2))$. Applying Prop.~\ref{P:nsubseth1astf} 
to $F_H(-1)$ one gets that $\h^1(E_H(-l-3)) = 0$ or $\h^1(E_H(-l-3)) < 
\h^1(E_H(-l-2))$. 

(g) It follows, from (f) (and (c)), that, for every plane $H \subset \piii$, 
one has $\h^1(E_H^\vee(1)) \leq 1$ and $\h^1(E_H^\vee(l)) = 0$, $\forall \, l 
\geq 2$. The second relation in (d) implies that $\h^1(E^\vee(1)) \leq 1$. 
Since $\tH^1(E_H^\vee(2)) = 0$, for every plane $H \subset \piii$, it follows 
that the multiplication map $h \colon \tH^1(E^\vee(1)) \ra \tH^1(E^\vee(2))$ is 
surjective, $\forall \, 0 \neq h \in \tH^0(\sco_\piii(1))$. Applying the 
Bilinear Map Lemma \cite[Lemma~5.1]{ha}, one deduces that $\tH^1(E^\vee(2)) 
= 0$ and this implies, now, that $\tH^1(E^\vee(l)) = 0$, $\forall \, 
l \geq 3$.      

(h) It follows, from (f) (and (c)), that, for every plane $H \subset \piii$, 
one has $\h^1(E_H^\vee(1)) \leq 2$, $\h^1(E_H^\vee(2)) \leq 
\text{max}(\h^1(E_H^\vee(1)) - 1 , 0)$ and $\h^1(E_H^\vee(l)) = 0$, $\forall \, 
l \geq 3$. Since $\tH^1_\ast(E^\vee)$ is not generated by $\tH^1(E^\vee(1))$, 
one deduces, from (b), that $\h^2(E^\vee) = \h^1(E(-4)) \geq 1$. The second 
relation in (d) implies, now, that $\h^1(E^\vee(1)) = \h^1(E_H^\vee(1)) - 
\h^2(E^\vee) \leq 1$. 

If $H \subset \piii$ is a plane of equation $h = 0$ then one has an exact 
sequence$\, :$ 
\[
\tH^1(E^\vee(1)) \overset{h}{\lra} \tH^1(E^\vee(2)) \lra \tH^1(E_H^\vee(2)) 
\lra 0 
\]  
(because $\h^2(E^\vee(1)) = \h^1(E(-5)) = 0$ by Lemma~\ref{L:yconnected}). 
One deduces that $\h^1(E^\vee(2)) \leq 2$. One cannot, actually, have 
$\h^1(E^\vee(2)) = 2$ because, in that case, one would have $\h^1(E^\vee(1)) 
= 1$ and the multiplication $h \colon \tH^1(E^\vee(1)) \ra \tH^1(E^\vee(2))$ by 
any non-zero linear form $h$ would be injective and this would imply that 
$\h^1(E^\vee(2)) \geq 4$. It remains that $\h^1(E^\vee(2)) \leq 1$. Since 
$\tH^1_\ast(E^\vee)$ is not generated by $\tH^1(E^\vee(1))$, it follows that 
$\h^1(E^\vee(2)) = 1$ and, $\forall \, 0 \neq h \in \tH^0(\sco_\piii(1))$, the 
multiplication map $h \colon \tH^1(E^\vee(1)) \ra \tH^1(E^\vee(2))$ is the 
zero map. One must, also, have $\h^1(E_H^\vee(2)) = 1$ and $\h^1(E_H^\vee(1)) = 
2$, $\forall \, H \subset \piii$, hence $\tH^1_\ast(E^\vee) \simeq 
{\underline k}(-1) \oplus {\underline k}(-2)$ if $\h^1(E(-4)) = 1$ and 
$\tH^1_\ast(E^\vee) \simeq {\underline k}(-2)$ if $\h^1(E(-4)) = 2$.   
\end{proof} 

\begin{remark}\label{R:h2e(-3)=0}  
(i) If $E$ is a vector bundle on $\piii$ with $\tH^0(E^\vee) = 0$ and 
$\tH^2(E(-3)) = 0$ then the graded $S$-module $\tH^1_\ast(E)$ is generated in 
degrees $\leq -2$. 

(ii) If $E$ is a vector bundle on $\piii$ with $\tH^0(E^\vee(1)) = 0$ and 
$\tH^1(E^\vee) = 0$ then the graded $S$-module $\tH^1_\ast(E)$ is generated in 
degrees $\leq -3$. 

\vskip2mm 

\noindent 
\emph{Indeed}, since $\h^0(E^\vee) = \h^3(E(-4))$, $\h^1(E^\vee) = \h^2(E(-4))$ 
and $\h^0(E^\vee(1)) = \h^3(E(-5))$, both assertions follow from the 
Castelnuovo-Mumford lemma (in the slightly more general form stated in 
\cite[Lemma~1.21]{acm1}).   
\end{remark}

\begin{remark}\label{R:beilinson} 
If $E$ is a vector bundle on $\piii$ with $\tH^i(E^\vee) = 0$, $i = 0,\, 1$, 
$\tH^0(E(-2)) = 0$ and $\tH^2(E(-3)) = 0$ then the Beilinson monad of $E(-1)$ 
has the following shape$\, :$ 
\[
\tH^1(E(-4))\otimes \Omega^3_\p(3) \ra \tH^1(E(-3))\otimes \Omega^2_\p(2) \ra 
\begin{matrix} \tH^1(E(-2))\otimes \Omega^1_\p(1)\\ \oplus\\ 
\tH^0(E(-1))\otimes \sco_\p \end{matrix} \ra \tH^1(E(-1))\otimes \sco_\p 
\]
the term of cohomological degree 0 being the direct sum (for information about 
Beilinson monads see, for example, \cite[Thm.~1.23~and~Remark~1.25]{acm1}). 

We shall use this monad to get lower bounds, in concrete situations, for the 
rank of the multiplication map $\mu \colon \tH^1(E(-3)) \otimes S_1 \ra 
\tH^1(E(-2))$, where $S_1 := \tH^0(\sco_\piii(1))$. More precisely, according to 
a result of Eisenbud, Fl\o ystad and Schreyer \cite{efs}, the component 
$\tH^1(E(-3)) \otimes_k \Omega_\piii^2(2) \ra \tH^1(E(-2)) \otimes_k 
\Omega_\piii^1(1)$ of the differential in the middle of the Beilinson monad 
can be identified with the composite map$\, :$ 
\begin{gather*}
\tH^1(E(-3)) \otimes_k \Omega_\piii^2(2) \lra \tH^1(E(-3)) \otimes_k 
\Omega_\piii^1(1) \otimes_k \Omega_\piii^1(1) \lra\\  
\lra \tH^1(E(-3)) \otimes_k S_1 \otimes_k \Omega_\piii^1(1) 
\overset{\mu}{\lra} \tH^1(E(-2)) \otimes_k \Omega_\piii^1(1)\, . 
\end{gather*} 
One deduces that if $c$ is the corank of $\mu$ then one has an exact 
sequence$\, :$ 
\[
0 \lra E(-1) \lra c\Omega_\p^1(1) \oplus Q \lra \tH^1(E(-1)) \otimes \sco_\p 
\lra 0\, , 
\]
where $Q$ is a vector bundle admitting a resolution of the form$\, :$ 
\[
0 \lra 
\tH^1(E(-4))\otimes \Omega^3_\p(3) \lra \tH^1(E(-3)) \otimes \Omega^2_\p(2) 
\lra \begin{matrix} \text{Im}\, \mu \otimes \Omega^1_\p(1)\\ \oplus\\ 
\tH^0(E(-1))\otimes \sco_\p \end{matrix} \lra Q \lra 0\, . 
\]
$Q$ is 1-regular and its Chern classes can be calculated, for example, 
applying relation \eqref{E:c2onp2} from Remark~\ref{R:chern} to $Q_H$, where 
$H \subset \piii$ is a plane, and relation \eqref{E:c3onp3} to Q. If the 
rank of $\mu$ is too small then it will turn out that such a bundle $Q$ 
cannot exist.  
\end{remark}

\section{The case $c_2 = 10$}\label{S:c2=10} 

We denote, in this section, by $E$ a globally generated vector bundle on 
$\piii$, with $c_1 = 5$, $c_2 = 10$, and such that $\tH^i(E^\vee) = 0$, 
$i = 0,\, 1$, and $\tH^0(E(-2)) = 0$. If $E$ has rank $r \geq 3$, we denote 
by $E^\prim$ the rank 3 vector bundle associated to $E$ in the exact sequence 
\eqref{E:oeeprim} and by $F$ the normalized bundle $E^\prim (-2)$. 

\begin{prop}\label{P:eprimunstablec2=10} 
If $E$ has rank $\geq 3$ and $E^\prim$ is not stable then one of the following 
holds$\, :$ 
\begin{enumerate}
\item[(i)] $c_3 = 6$ and $E \simeq \sco_\piii(1) \oplus F^\prim (2)$, where 
$F^\prim$ is a $2$-instanton$\, ;$ 
\item[(ii)] $c_3 = 10$ and $E \simeq {\fam0 T}_\piii(-1) \oplus N(2)$, where 
$N$ is a nullcorrelation bundle. 
\end{enumerate}
\end{prop} 

\begin{proof} 
Lemma~\ref{L:eprimunstable} implies that either $c_3 = 6$ and $E^\prim$ can be 
realized as an extension$\, :$ 
\[
0 \lra F^\prim(2) \lra E^\prim \lra \sco_\piii(1) \lra 0 
\]
where $F^\prim$ is a stable rank 2 vector bundle with $c_1(F^\prim) = 0$, 
$c_2(F^\prim) = 2$, or $c_3 = 10$ and $E^\prim$ can be realized as an 
extension$\, :$ 
\[
0 \lra F^\prim(2) \lra E^\prim \lra \sci_L(1) \lra 0 
\]
where $F^\prim$ is a stable rank 2 vector bundle with $c_1(F^\prim) = 0$, 
$c_2(F^\prim) = 1$ and $L$ is a line. 

\vskip2mm 

In the former case, $F^\prim$ is a 2-instanton hence, by 
Remark~\ref{R:eprimunstable}(i), $E$ has rank 3 (because $\tH^2(F^\prim(-2)) = 
0$) hence $E = E^\prim$. Since $\tH^1(F^\prim(1)) = 0$ (see, for example, 
\cite[Remark~4.7]{acm1}) the above extension splits hence $E \simeq 
\sco_\piii(1) \oplus F^\prim(2)$. 

\vskip2mm 

In the latter case, $F^\prim$ is a nullcorrelation bundle $N$ (that is, an 
1-instanton). Since $N$ is 1-regular, the multiplication map  
$\tH^0(N(2)) \otimes \tH^0(\sco_\piii(1)) \ra \tH^0(N(3))$ is surjective. 
It follows, now, from Remark~\ref{R:eprimunstable2}(ii), that $E \simeq N(2) 
\oplus \text{T}_\piii(-1)$. 
\end{proof}

\begin{prop}\label{P:eprimstablec2=10} 
Let $E$ be a rank $r$ vector bundle on $\piii$ as at the beginning of this 
section. Assume that either $r \geq 3$ and the rank $3$ vector bundle $E^\prim$ 
associated to $E$ in the exact sequence \eqref{E:oeeprim} is stable or that 
$r = 2$. Then one of the following holds$\, :$ 
\begin{enumerate} 
\item[(i)] $c_3 = 10$ and $E \simeq 5\sco_\piii(1)$$\, ;$ 
\item[(ii)] $c_3 = 8$ and $E \simeq \sco_\piii(1) \oplus E_0$ where, up to a 
linear change of coordinates, $E_0$ is the kernel of the epimorphism 
\[
(x_0,\, x_1,\, x_2,\, x_3^2) : 3\sco_\piii(2) \oplus \sco_\piii(1) \lra 
\sco_\piii(3)\, ; 
\]
\item[(iii)] $c_3 = 6$ and $E$ can be realized as a nontrivial extension$\, :$ 
\[
0 \lra \sco_\piii(1) \lra E \lra G(2) \lra 0\, ,
\] 
where $G$ is a $2$-instanton$\, ;$ 
\item[(iv)] $c_3 = 4$ and one has an exact sequence$\, :$ 
\[
0 \lra E \lra \sco_\piii(3) \oplus 3\sco_\piii(2) \lra \sco_\piii(4) \lra 0\, ; 
\] 
\item[(v)] $c_3 = 0$ and $E \simeq G(3)$, where $G$ is a general rank $2$ 
vector bundle with $c_1(G) = -1$, $c_2(G) = 4$ and such that ${\fam0 H}^0(G(1)) 
= 0$. 
\end{enumerate} 
\end{prop} 

\begin{proof} 
We treat, firstly, the case $r \geq 3$. Let $F = E^\prim (-2)$ be the normalized 
rank 3 vector bundle associated to $E^\prim$. It has Chern classes $c_1(F) = 
-1$, $c_2(F) = 2$, $c_3(F) = c_3 - 8$ (see \eqref{E:chernf}). 
If $k_F = (k_1, k_2)$ is the spectrum of $F$ then $c_3(F) = -2\sum k_i - 2$ 
(see Remark~\ref{R:spectrumf}(iii)). Moreover, $r = 3 + \h^2(F(-2))$ (see 
\eqref{E:r}). According to Remark~\ref{R:spectrumf}, the possible spectra of 
$F$ are $(-1,-1)$, $(0,-1)$, $(0,0)$ and $(1,0)$.    

\vskip2mm 

\noindent
{\bf Case 1.}\quad $F$ \emph{has spectrum} $(-1,-1)$. 

\vskip2mm 

\noindent 
In this case, $r = 5$, $c_3(F) = 2$ and $c_3 = 10$. 
Using the spectrum, one gets that $\tH^1(E(-3)) \simeq \tH^1(F(-1)) = 0$.  
Lemma~\ref{L:edual(1)} implies, now, that $E^\vee$ is 1-regular. 
In particular, $E^\vee (1)$ is globally generated. Since $c_1(E^\vee (1)) = 0$ 
it follows that $E^\vee (1) \simeq 5\sco_\piii$ hence $E \simeq 5\sco_\piii(1)$.  

\vskip2mm 

\noindent 
{\bf Case 2.}\quad $F$ \emph{has spectrum} $(0,-1)$. 

\vskip2mm 

\noindent
In this case, $r = 4$, $c_3(F) = 0$ and $c_3 = 8$. 
Using the spectrum, one gets that $\h^1(E(l)) = \h^1(F(l + 2)) = 0$ for 
$l \leq -4$, $\h^1(E(-3)) = 1$ and $\tH^2(E(-3)) \simeq \tH^2(F(-1)) = 0$.  
In particular, $s := \h^1(E(-3)) - \h^1(E(-4)) = 1$. 
Lemma~\ref{L:h2e(-3)=0}(e) implies, now, that $\tH^2_\ast(E) = 0$.   
Applying Riemann-Roch to $F$ one gets that $\h^1(E(-2)) = \h^1(F) = 1$. 
Moreover, by the proof of Lemma~\ref{L:(3;-1,2,0)}, $\h^1(E(l)) = 
\h^1(F(l + 2)) = 0$, for $l \geq -1$. 
Since, for a general plane $H \subset \piii$, of equation $h = 0$, one has 
$\tH^0(E_H(-2)) \simeq \tH^0(F_H) = 0$, the multiplication by 
$h \colon \tH^1(E(-3)) \ra \tH^1(E(-2))$ is non-zero. One deduces that there 
exists a $k$-basis of $h_0, \ldots , h_3$ of $S_1$ such that$\, :$ 
\[
\tH^1_\ast(E) \simeq (S/(h_0,\, h_1,\, h_2,\, h_3^2))(3)\, . 
\] 
Denoting by $E_0$ the kernel of the epimorphism$\, :$ 
\[
(h_0,\, h_1,\, h_2,\, h_3^2) : 3\sco_\piii(2) \oplus \sco_\piii(1) \lra 
\sco_\piii(3)
\] 
one gets that $E \simeq A \oplus E_0$, where $A$ is a direct sum of line 
bundles (by Horrocks theory). Using the fact that $\text{rk}\, E = 4$ and 
$c_1(E) = 5$ one deduces that $A = \sco_\piii(1)$. 

\vskip2mm 

\noindent
{\bf Case 3.}\quad $F$ \emph{has spectrum} $(0,0)$. 

\vskip2mm 

\noindent
In this case, $r = 3$, $c_3(F) = -2$ and $c_3 = 6$. Since $r = 3$ it follows 
that $E = F(2)$. Using Lemma~\ref{L:(3;-1,2,-2)}, one gets that $E$ is as in 
item (iii) of the statement. 

\vskip2mm 

\noindent 
{\bf Case 4.}\quad $F$ \emph{has spectrum} $(1,0)$. 

\vskip2mm 

\noindent
In this case, $r = 3$, $c_3(F) = -4$ and $c_3 = 4$. Since $r = 3$ it follows 
that $E = F(2)$. Using Lemma~\ref{L:(3;-1,2,-4)}, one gets that $E$ is as in 
item (iv) of the statement. 

\vskip2mm 

\noindent
{\bf Case 5.}\quad $E$ \emph{has rank} $2$. 

\vskip2mm 

\noindent 
In this case, $G := E(-3)$ is a rank 2 vector bundle with $c_1(G) = -1$, 
$c_2(G) = 4$ and such that $\tH^0(G(1)) = 0$ hence $E$ is as in item (v) of 
the statement. 

\vskip2mm 

\noindent 
{\bf Construction 5.0.}\quad According to the results of B\u{a}nic\u{a} and 
Manolache \cite{bm} a general rank 2 vector bundle $G$ on $\piii$ 
with $c_1(G) = -1$, $c_2(G) = 4$ and $\tH^0(G(1)) = 0$ can be realized as an 
extension$\, :$   
\[
0 \lra \sco_\piii(-2) \lra G \lra \sci_X(1) \lra 0\, , 
\]
where $X$ is a double structure on a twisted cubic curve $C \subset \piii$ 
(see, also, Prop.~\ref{P:doubleskewcubic} from Appendix~\ref{A:(2;-1,4,0)}). 
We want to show that, for any such bundle, $G(3)$ is globally generated.  

\emph{Indeed}, one must have $\omega_X \simeq \sco_X(-1)$. 
By the results of Ferrand \cite{f}, if $X^\prim$ is a double structure on a 
nonsingular curve $C^\prim$ in a nonsingular threefold $P^\prim$ and 
if $L$ is a line bundle on $P^\prim$ then $\omega_{X^\prim} \simeq L \vb X^\prim$ 
if and only if $\sci_{C^\prim}/\sci_{X^\prim} \simeq \omega_{C^\prim} \otimes 
(L^{-1} \vb C^\prim)$.  
It follows that, in our case, the ideal sheaf of $X$ is defined by an exact 
sequence$\, :$ 
\[
0 \lra \sci_X \lra \sci_C \lra \omega_C(1) \lra 0\, . 
\] 
Now, $C$ is the image of an embedding $\nu : \pj \ra \piii$ such that 
$\nu^\ast \sco_\piii(1) \simeq \sco_\pj(3)$. Choose a basis $t_0,\, t_1$ of 
$\tH^0(\sco_\pj(1))$ and let $h_0, \ldots , h_3$ be the basis of 
$\tH^0(\sco_\piii(1))$ for which $\nu^\ast(h_i) = t_0^{3-i}t_1^i$, 
$i = 0, \ldots , 3$. Applying $\nu^\ast$ to the exact sequence$\, :$ 
\[
0 \lra 2\sco_\piii(-3) \xra{\left(\begin{smallmatrix} h_0 & h_1\\ h_1 & h_2\\ 
h_2 & h_3 \end{smallmatrix}\right)} 3\sco_\piii(-2) \lra \sci_C 
\lra 0\, , 
\]
and taking into account the matrix relation$\, :$ 
\[
\begin{pmatrix} t_0^3 & t_0^2t_1\\ t_0^2t_1 & t_0t_1^2\\ t_0^2t_1 & t_1^3 
\end{pmatrix} = 
\begin{pmatrix} t_0^2\\ t_0t_1\\ t_1^2 \end{pmatrix} (t_0\, ,\, t_1)\, , 
\]
one gets an exact sequence$\, :$ 
\[
0 \lra \sco_\pj(-8) \lra 3\sco_\pj(-6) \lra \nu^\ast(\sci_C/\sci_C^2) \lra 0 
\] 
from which one derives that that $\nu^\ast(\sci_C/\sci_C^2) \simeq 
2\sco_\pj(-5)$. Moreover, using the exact sequence$\, :$ 
\[
0 \lra \nu_\ast\sco_\pj(-2) \lra 3\sco_C \lra (\sci_C/\sci_C^2)(2) \lra 0\, ,  
\] 
one deduces that $\tH^0(3\sco_C(1)) \ra \tH^0((\sci_C/\sci_C^2)(3))$ is 
surjective. This implies that the map $\tH^0(\sci_C(3)) \ra 
\tH^0((\sci_C/\sci_C^2)(3))$ is surjective, hence $\tH^1(\sci_C^2(3)) = 0$. It 
follows that $\sci_C^2$ is 4-regular. Now, using the exact sequence$\, :$ 
\[
0 \lra \sci_X/\sci_C^2 \lra \sci_C/\sci_C^2 \lra \omega_C(1) \lra 0\, , 
\]
one deduces that $\nu^\ast(\sci_X/\sci_C^2) \simeq \sco_\pj(-11)$. Since 
$\sci_C^2(4)$ is globally generated and $\tH^1(\sci_C^2(4)) = 0$, it follows 
that $\sci_X(4)$ is globally generated hence $G(3)$ is globally generated. 

\vskip2mm 

\noindent 
{\bf Construction 5.1.}\quad A special type of rank 2 vector bundles $G$ on 
$\piii$ with $c_1(G) = -1$, $c_2(G) = 4$ and $\tH^0(G(1)) = 0$ are those 
bundles that can be realized as extensions$\, :$ 
\[
0 \lra \sco_\piii(-2) \lra G \lra \sci_X(1) \lra 0\, , 
\] 
where $X$ is the union of three mutually disjoint nonsingular conics 
$C_0,\, C_1,\, C_2$. Let $H_i \subset \piii$ be the plane containing $C_i$, 
$i = 0,\, 1,\, 2$. If $H_0 \cap H_1 \cap C_2 \neq \emptyset$ then $G(3)$ is 
not globally generated because $H_0 \cap H_1$ is a 5-secant of $X$. 

We want to show that, on the other hand, if $H_0 \cap H_1 \cap H_2$ consists 
of only one point $P$ which does not belong to any of the conics 
$C_0,\, C_1,\, C_2$ then $G(3)$ is globally generated (which is, of course, 
equivalent to the fact that $\sci_X(4)$ is globally generated).  

\emph{Indeed}, put $Y := C_0 \cup C_1$, $L := H_0 \cap H_1$ and $\Gamma := 
H_2 \cap Y$. $\Gamma$ is a complete intersection of type $(2,2)$ in $H_2 
\simeq \pii$. Let $h_2 = 0$ be an equation of $H_2$ ($h_2 \in 
\tH^0(\sco_\piii(1))$) and let $h_2 = 0$, $q_2 = 0$ be equations of $C_2$ 
($q_2 \in \tH^0(\sco_\piii(2))$). Since $X = Y \cup C_2$, one has an exact 
sequence$\, :$ 
\[
0 \lra \sci_Y(-1) \overset{h_2}{\lra} \sci_X \lra q_2\sci_{\Gamma , H_2}(-2) 
\lra 0\, . 
\]  

\noindent
{\bf Claim.}\quad \emph{The cokernel of the evaluation morphism} 
$\e^\prim \colon \tH^0(\sci_Y(3))\otimes_k\sco_\piii \ra \sci_Y(3)$ 
\emph{is isomorphic to} $\sci_{L \cap Y , L}(3)$. 

\vskip2mm 

\noindent 
\emph{Indeed}, the map $\tH^0(\sco_Y(1)) \ra \tH^0(\sco_{L \cap Y}(1))$ is 
surjective ($L \cap Y$ is the disjoint union of $L \cap C_0$ and $L \cap 
C_1$). Using the exact sequences$\, :$ 
\[
0 \lra \tH^0(\sco_{Y \cup L}(l)) \lra \tH^0(\sco_Y(l)) \oplus \tH^0(\sco_L(l)) 
\lra \tH^0(\sco_{L \cap Y}(l)) \lra 0\, ,\  l = 1,\, 2\, ,   
\]  
one gets, on one hand, that $\tH^1(\sco_{Y \cup L}(1)) = 0$ and, on the other 
hand, that $\h^0(\sco_{Y \cup L}(2)) = 9$. Since $\h^0(\sci_{Y \cup L}(2)) = 
\h^0(\sci_Y(2)) = 1$, it follows that $\tH^1(\sci_{Y \cup L}(2)) = 0$ hence 
$\sci_{Y \cup L}$ is 3-regular. In particular, $\sci_{Y \cup L}(3)$ is globally 
generated. Since $\tH^0(\sci_Y(3)) = \tH^0(\sci_{Y \cup L}(3))$ (because $L$ is 
a 4-secant of $Y$) one deduces that the cokernel of the evaluation morphism 
of $\sci_Y(3)$ is $(\sci_Y/\sci_{Y \cup L})(3) \simeq \sci_{L \cap Y , L}(3)$ and 
the claim is proven.   

\vskip2mm 

Now, let $\sck$ be the kernel of the evaluation epimorphism $\e^\secund \colon 
q_2\tH^0(\sci_{\Gamma , H_2}(2))\otimes_k\sco_\piii \ra q_2\sci_{\Gamma , H_2}(2)$. 
($\sck$ is, actually, a stable rank 2 reflexive sheaf on $\piii$ with 
$c_1(\sck) = -1$, $c_2(\sck) = 2$, $c_3(\sck) = 4$). Since $\tH^1(\sci_Y(3)) 
= 0$ one gets a commutative diagram$\, :$
\[
\SelectTips{cm}{12}\xymatrix{
0\ar[r] & \tH^0(\sci_Y(3)) \otimes \sco_\piii\ar[r]^-{h_2}\ar[d]^{\e^\prim}  & 
\tH^0(\sci_X(4)) \otimes \sco_\piii\ar[r]\ar[d]^{\e} &  
q_2\tH^0(\sci_{\Gamma , H_2}(2)) \otimes \sco_\piii\ar[r]\ar[d]^{\e^\secund} & 0\\ 
0\ar[r] & \sci_Y(3)\ar[r]^-{h_2} & \sci_X(4)\ar[r] & 
q_2\sci_{\Gamma , H_2}(2)\ar[r] & 0} 
\]
In order to show that $\sci_X(4)$ is globally generated one must show that the 
connecting morphism $\delta : \sck \ra \sci_{L \cap Y , L}(3)$, obtained by 
applying the Snake Lemma to this diagram, is an epimorphism. 

Since $\sci_{L \cap Y , L}(4) \simeq \sco_L$ it suffices, actually, to show that 
there exists a global section of $\sck(1)$ whose image by $\delta(1)$ is a 
non-zero global section of $\sci_{L \cap Y , L}(4)$. 
But $\tH^0(\sck(1)) = q_2\tH^0(\sci_{\Gamma , H_2}(2)) \otimes h_2$. Let $g$ be an 
element of $\tH^0(\sci_{\Gamma , H_2}(2))$. $q_2g$ can be lifted to 
${\widetilde g} \in \tH^0(\sci_X(4))$. $\e(1)$ maps ${\widetilde g} \otimes 
h_2$ to ${\widetilde g}h_2 \in \tH^0(\sci_X(5))$. But ${\widetilde g}$ 
can be also considered as an element of $\tH^0(\sci_Y(4))$. It is clear, now, 
that$\, :$ 
\[
\delta(1)(q_2g \otimes h_2) = {\widetilde g} \vb L \in 
\tH^0(\sci_{L \cap Y , L}(4))\, . 
\]    

Now, since $\Gamma$ is a complete intersection of type (2,2) in $H_2$ and since 
$H_2 \cap L$ does not belong to $\Gamma$, there exists $g \in 
\tH^0(\sci_{\Gamma , H_2}(2))$ not vanishing in $H_2 \cap L$. On the other hand, 
$q_2$ does not vanish in $H_2 \cap L$ because $H_2 \cap L$ does not belong to 
$C_2$. It follows that $\widetilde g$ does not vanish in $H_2 \cap L$ hence 
${\widetilde g} \vb L$ is a nonzero global section of $\sci_{L \cap Y , L}(4)$. 
This concludes the proof of the global generation of $\sci_X(4)$ and, 
consequently, that of $G(3)$.  
\end{proof}

\section{The case $c_2 = 11$}\label{S:c2=11} 

We denote, in this section, by $E$ a globally generated vector bundle on 
$\piii$, with $c_1 = 5$, $c_2 = 11$, and such that $\tH^i(E^\vee) = 0$, 
$i = 0,\, 1$, and $\tH^0(E(-2)) = 0$. Since $c_3 \equiv c_2 \pmod{2}$ it 
follows that $E$ must have rank $r \geq 3$. We denote by $E^\prim$ the rank 3 
vector bundle associated to $E$ in the exact sequence \eqref{E:oeeprim} and by 
$F$ the normalized bundle $E^\prim (-2)$.

\begin{lemma}\label{L:h1e(l)} 
${\fam0 h}^1(E(l)) \leq {\fam0 max}(0,\, {\fam0 h}^1(E(l-1)) - 3)$, 
$\forall \, l \geq -1$. 
\end{lemma}

\begin{proof} 
Let $H \subset \piii$ be an arbitrary plane, of equation $h = 0$. 
$E_H$ is a globally generated vector bundle on $H \simeq \pii$, with 
$c_1(E_H) = 5$ and $c_2(E_H) = 11$. It follows, from Remark~\ref{R:c1=5onp2},   
that $E_H$ is 0-regular. 
In particular, $\tH^1(E_H(l)) = 0$, $\forall \, l \geq -1$. One deduces that 
multiplication by $h \colon \tH^1(E(l-1)) \ra \tH^1(E(l))$ is surjective, 
$\forall \, l \geq -1$. Applying the Bilinear Map Lemma 
\cite[Lemma~5.1]{ha} to $\tH^1(E(l))^\vee \otimes \tH^0(\sco_\piii(1)) \ra 
\tH^1(E(l-1))^\vee$ one gets the desired inequality. 
\end{proof}

\begin{prop}\label{P:eprimunstablec2=11} 
If $E^\prim$ is not stable then one of the following holds$\, :$ 
\begin{enumerate} 
\item[(i)] $c_3 = 7$ and $E$ can be realized as an extension$\, :$ 
\[
0 \lra F^\prim (2) \lra E \lra \sco_\piii(1) \lra 0 
\]
where $F^\prim$ is a $3$-instanton with ${\fam0 h}^0(F^\prim (1)) \leq 1$$\, ;$ 
\item[(ii)] $c_3 = 11$ and $E \simeq {\fam0 T}_\piii(-1) \oplus F^\prim (2)$, 
where $F^\prim$ is a $2$-instanton. 
\end{enumerate}
\end{prop}

\begin{proof} 
Lemma~\ref{L:eprimunstable} implies that either $c_3 = 7$ and $E^\prim$ can be 
realized as an extension$\, :$ 
\[
0 \lra F^\prim(2) \lra E^\prim \lra \sco_\piii(1) \lra 0 
\]
where $F^\prim$ is a stable rank 2 vector bundle with $c_1(F^\prim) = 0$, 
$c_2(F^\prim) = 3$, or $c_3 = 11$ and $E^\prim$ can be realized as an 
extension$\, :$ 
\[
0 \lra F^\prim(2) \lra E^\prim \lra \sci_L(1) \lra 0 
\]
where $F^\prim$ is a stable rank 2 vector bundle with $c_1(F^\prim) = 0$, 
$c_2(F^\prim) = 2$ (hence a 2-instanton) and $L$ is a line. 

\vskip2mm 

In the former case, $F^\prim$ has, \emph{a priori}, two possible spectra$\, :$ 
$(1,0,-1)$ and $(0,0,0)$. Using Riemann-Roch, one gets that $\h^1(F^\prim ) = 
4$ and $\h^0(F^\prim (1)) - \h^1(F^\prim (1)) = -1$. Lemma~\ref{L:h1e(l)} 
implies that $\h^1(E^\prim (-1)) \leq 1$ hence $\h^1(F^\prim (1)) \leq 2$ 
hence $\h^0(F^\prim (1)) \leq 1$. But if $F^\prim$ has spectrum $(1,0,-1)$ then, 
by \cite[Lemma~9.15]{ha}, $\h^0(F^\prim(1)) = 2$. It remains that $F^\prim$ has 
spectrum $(0,0,0)$ hence it is a 3-instanton. Moreover, it must satisfy 
$\h^0(F^\prim (1)) \leq 1$. By Remark~\ref{R:eprimunstable}, one has $r = 3$ 
(because $\tH^2(F^\prim(-2)) = 0$) hence $E = E^\prim$. 

\vskip2mm 

In the latter case, it follows, from Remark~\ref{R:eprimunstable2}(ii),   
that $E \simeq F^\prim (2) \oplus \text{T}_\piii(-1)$.  
\end{proof} 

\begin{remark}\label{R:eprimunstablec2=11} 
Let $F^\prim$ be a 3-instanton with $\h^0(F^\prim (1)) \leq 1$. We want to 
characterize the extensions$\, :$ 
\begin{equation}\label{E:fprimeoc2=11}
0 \lra F^\prim (2) \lra E \lra \sco_\piii(1) \lra 0
\end{equation}
with $E$ globally generated. Such an extension is defined by an element $\xi 
\in \tH^1(F^\prim (1))$. Since $\tH^1(F^\prim (2)) = 0$, $\xi$ is annihilated by 
$S_1$ as an element of the graded $S$-module $\tH^1_\ast(F^\prim)$. It follows 
that there exists a morphism $\phi^\prime : \Omega_\piii \ra F^\prim (1)$ such 
that the image of $\tH^1(\phi^\prime)$ is $k\xi$. If 
$\phi^\secund$ is another such morphism then $\phi^\secund - \phi^\prime$ 
factorizes as $\Omega_\piii \ra 4\sco_\piii(-1) \ra F^\prim (1)$. 

According to Remark~\ref{R:eprimunstable2}(i), $E$ is globally generated if and 
only if the morphism 
\[
(\tH^0(F^\prim (2))\otimes_k \sco_\piii) \oplus \Omega_\piii(1) 
\xra{(\text{ev},\, \phi^\prime (1))} F^\prime (2)
\]  
is an epimorphism. We use, now, a stratification of the moduli space 
of 3-instantons, due to Gruson and Skiti \cite{gs} and recalled in 
Remark~\ref{R:gs}. 

\vskip2mm 

(i) If $\tH^0(F^\prim (1)) = 0$ and $F^\prim$ has no jumping line of order 3 then 
$F^\prim (2)$ is globally generated hence, for any extension 
\eqref{E:fprimeoc2=11}, $E$ is globally generated. Recall that, in this case, 
$\h^1(F^\prim (1)) = 1$. 

\vskip2mm 
 
(ii) If $\tH^0(F^\prim (1)) = 0$ and $F^\prim$ has a jumping line $L$ of order 3 
then the cokernel of the evaluation morphism $\text{ev} : 
\tH^0(F^\prim (2))\otimes_k \sco_\piii \ra F^\prim (2)$ is $\sco_L(-1)$. Moreover, 
$\tH^1(F^\prim (1)) \ra \tH^1(\sco_L(-2))$ is an isomorphism. Since 
$\Omega_\piii \vb L \simeq 2\sco_L(-1) \oplus \sco_L(-2)$ it follows that 
the composite morphism$\, :$ 
\[
\Omega_\piii(1) \xra{\phi^\prime (1)} F^\prim (2) \lra \sco_L(-1) 
\] 
is an epimorphism if and only if the induced map $\tH^1(\Omega_\piii) \ra 
\tH^1(\sco_L(-2))$ is an isomorphism. 

Consequently, in this case, the bundle $E$ from the extension 
\eqref{E:fprimeoc2=11} is globally generated if and only if the extension is 
nontrivial. 

\vskip2mm 

(iii) If $\h^0(F^\prim (1)) = 1$ then $F^\prim$ has two jumping lines $L$ and 
$L^\prime$ of order 3 and the cokernel of the evaluation morphism of 
$F^\prim (2)$ is $\sco_{L \cup L^\prime}(-1)$. Moreover, the map $\tH^1(F^\prim (1)) 
\ra \tH^1(\sco_L(-2)) \oplus \tH^1(\sco_{L^\prime}(-2))$ is an isomorphism. 

In this case, the bundle $E$ defined by the extension \eqref{E:fprimeoc2=11} 
is globally generated if and only if the element $\xi$ of $\tH^1(F^\prim (1))$ 
defining the extension sits in none of the kernels of the maps 
$\tH^1(F^\prim (1)) \ra \tH^1(\sco_L(-2))$ and $\tH^1(F^\prim (1)) \ra 
\tH^1(\sco_{L^\prime}(-2))$. 
\end{remark} 

\begin{prop}\label{P:h1e(-3)=0c2=11} 
Let $E$ and $E^\prim$ be as at the beginning of this section. If $E^\prim$ is 
stable and ${\fam0 H}^1(E(-3)) = 0$ then one of the following holds$\, :$ 
\begin{enumerate} 
\item[(i)] $c_3 = 15$ and $E \simeq 4\sco_\piii(1) \oplus 
{\fam0 T}_\piii(-1)$$\, ;$ 
\item[(ii)] $c_3 = 13$ and $E \simeq 3\sco_\piii(1) \oplus \Omega_\piii(2)$. 
\end{enumerate}
\end{prop}

\begin{proof} 
Let $F = E^\prim (-2)$ be the normalized rank 3 vector bundle associated to 
$E^\prim$ and let $k_F = (k_1,\, k_2,\, k_3)$ be the spectrum of $F$. One has 
$c_1(F) = -1$, $c_2(F) = 3$, $c_3(F) = c_3 - 10$ (see \eqref{E:chernf}) and 
$c_3(F) = -2\sum k_i - 3$ (see Remark~\ref{R:spectrumf}(iii)). Moreover,  
one has, from relation \eqref{E:r}, $r = 3 + \h^2(F(-2))$.  
Since $\tH^1(E(-3)) \simeq \tH^1(F(-1))$, the hypothesis $\tH^1(E(-3)) = 0$ is 
equivalent to $k_1 \leq -1$. Taking into account Lemma~\ref{L:spectrumf}, 
it follows that, under our hypotheses, the only possible spectra are 
$(-1,-1,-2)$ and $(-1,-1,-1)$. 

\vskip2mm 

\noindent 
{\bf Case 1.}\quad $F$ \emph{has spectrum} $(-1,-1,-2)$. 

\vskip2mm 

\noindent 
In this case, $r = 7$, $c_3(F) = 5$ and $c_3 = 15$. By 
Lemma~\ref{L:edual(1)}(a), $E^\vee (1)$ is globally generated. The Chern 
classes of $E^\vee (1)$ are $c_1(E^\vee (1)) = 2$, $c_2(E^\vee (1)) = 2$, 
$c_3(E^\vee (1)) = 0$ (recall \cite[Lemma~2.1]{ha}). The globally generated 
vector bundles on $\piii$ with $c_1 = 2$ have been classified by Sierra and    
Ugaglia \cite{su}. Using their results (see, also, \cite[Prop.~2.3]{acm1}) 
and taking into account \cite[Lemma~1.2]{acm1}, one gets that there exist 
integers $s$ and $t$ such that $E^\vee (1) \simeq t\sco_\piii \oplus G$ where 
$G$ is a vector bundle defined by an exact sequence$\, :$ 
\[
0 \lra s\sco_\piii \lra \Omega_\piii(2) \lra G \lra 0\, ,
\]
One deduces that $E \simeq t\sco_\piii(1) \oplus G^\vee (1)$. Since $E$ has 
rank 7 and $G$ has rank $\leq 3$, it follows that $t \geq 4$ hence 
$c_1(G^\vee (1)) \leq 1$. Since $G^\vee (1)$ is globally generated, one deduces 
(see, for example, the comment after \cite[Lemma~2.1]{acm1}) that $G^\vee (1) 
\simeq \text{T}_\piii(-1)$ and $t = 4$. 

\vskip2mm 

\noindent 
{\bf Case 2.}\quad $F$ \emph{has spectrum} $(-1,-1,-1)$. 

\vskip2mm 

\noindent 
In this case, $r = 6$, $c_3(F) = 3$ and $c_3 = 13$. By 
Lemma~\ref{L:edual(1)}(a), $E^\vee (1)$ is globally generated. One has   
$c_1(E^\vee (1)) = 1$, $c_2(E^\vee (1)) = 1$, $c_3(E^\vee (1)) = 1$. It follows, 
as in Case 1, that $E^\vee (1) \simeq 3\sco_\piii \oplus \text{T}_\piii(-1)$ 
hence $E \simeq 3\sco_\piii(1) \oplus \Omega_\piii(2)$.    
\end{proof}

\begin{prop}\label{P:h2e(-3)neq0c2=11} 
Let $E$ and $E^\prim$ be as at the beginning of this section. If $E^\prim$ is 
stable and ${\fam0 H}^2(E(-3)) \neq 0$ then one of the following holds$\, :$ 
\begin{enumerate} 
\item[(i)] $c_3 = 15$ and $E \simeq {\fam0 T}_\piii(-1) \oplus 4\sco_\piii(1)$.
\item[(ii)] $c_3 = 13$ and $E \simeq {\fam0 T}_\piii(-1) \oplus E_1$ where, 
up to a linear change of coordinates, $E_1$ is the kernel of the 
epimorphism$\, :$ 
\[
(x_0,\, x_1,\, x_2,\, x_3^2) : 3\sco_\piii(2) \oplus \sco_\piii(1) \lra 
\sco_\piii(3)\, ; 
\]  
\end{enumerate}
\end{prop}

\begin{proof}
Let $F = E^\prim (-2)$ be the normalized rank 3 vector bundle associated to 
$E^\prim$ and let $k_F = (k_1,\, k_2,\, k_3)$ be the spectrum of $F$. Since 
$\tH^2(E(-3)) \simeq \tH^2(F(-1))$ the hypothesis $\tH^2(E(-3)) \neq 0$ is 
equivalent to $k_3 = -2$ (taking into account Lemma~\ref{L:spectrumf}). It 
follows that, under our hypotheses, only two spectra can occur$\, :$ 
$(-1,-1,-2)$ and $(0,-1,-2)$. If the spectrum of $F$ is $(-1,-1,-2)$ then, 
as we saw in Case 1 of the proof of Prop.~\ref{P:h1e(-3)=0c2=11}, $E$ is as 
in item (i) from the statement.   

It remains to consider the case where $F$ has spectrum $(0,-1,-2)$. 
In this case, $r = 6$, $c_3(F) = 3$ and $c_3 = 13$. Moreover, $\h^2(E(-3)) = 
\h^2(F(-1)) = 1$. According to Lemma~\ref{L:h2e(-3)=1}, one has an exact 
sequence$\, :$ 
\[
0 \lra \sco_\piii(-1) \xra{\left(\begin{smallmatrix} u\\ v 
\end{smallmatrix}\right)} E_1 \oplus 4\sco_\piii \lra E \lra 0\, , 
\] 
with $v \colon \sco_\piii(-1) \ra 4\sco_\piii$ defined by $4$ linearly 
independent linear forms, where $E_1$ is a vector bundle of rank $3$, 
with $\tH^i(E_1^\vee) = 0$, $i = 0,\, 1$, with Chern classes $c_1(E_1) = 
4$, $c_2(E_1) = 6$, $c_3(E_1) = 2$. Moreover, one has an exact sequence$\, :$ 
\[
0 \lra \sco_\piii \lra E_1 \lra \scf_1(2) \lra 0\, . 
\] 
where $\scf_1$ is a stable rank 2 reflexive sheaf with $c_1(\scf_1) = 0$, 
$c_2(\scf_1) = 2$, $c_3(\scf_1) = 2$. It follows, from 
\cite[Table~2.8.1]{ch}, that $\scf_1$ is 2-regular hence $E_1$ is 0-regular. 
In particular, $E_1$ is globally generated and the multiplication map 
$\tH^0(E_1) \otimes_k \tH^0(\sco_\piii(1)) \ra \tH^0(E_1(1))$ is surjective. One 
gets, now, from the exact sequence relating $E$ and $E_1$,  
that $E \simeq \text{T}_\piii(-1) \oplus E_1$. Moreover, by 
\cite[Prop.~4.8]{acm1}, $E_1$ is isomorphic, up to a linear change of 
coordinates, to the kernel of the epimorphism from item (ii) of the statement.  
\end{proof}

\begin{lemma}\label{L:h0f(1)} 
Let $F$ be a stable rank $3$ vector bundle on $\piii$ with $c_1(F) = -1$, 
$c_2(F) = 3$ and such that $F(2)$ is globally generated. Then$\, :$ 

\emph{(a)} ${\fam0 H}^0(F(1)) = 0$ if $c_3(F) \leq -5$$\, ;$ 

\emph{(b)} ${\fam0 h}^0(F(1)) \leq 1$ if $c_3(F) = -3$. 

\emph{(c)} $F$ is $2$-regular if $c_3(F) \geq -1$.   
\end{lemma}

\begin{proof} 
(a) We will show that $\tH^0(F(1)) \neq 0$ implies $c_3(F) \geq -3$.  
If $F(1)$ has a nonzero global section then this one defines an exact 
sequence$\, :$ 
\[
0 \lra \scg \lra F^\vee \overset{\sigma}{\lra} \sci_W(1) \lra 0\, , 
\]
where $W$ is a closed subscheme of $\piii$ with $\dim W \leq 1$ (because 
$\tH^0(F) = 0$ since $F$ is stable) and $\scg$ is a rank 2 reflexive sheaf 
with $c_1(\scg) =0$, $c_2(\scg) = 3 - \text{deg}\, W$ (see 
Remark~\ref{R:chern}(c)). Moreover, by the same remark$\, :$ 
\[
c_3(F) = c_3(\scg) - 3 - 2\, \text{deg}\, W + 2\chi(\sco_{W_{\text{CM}}})\, . 
\]

\noindent 
$\bullet$\quad 
If $\text{deg}\, W = 0$ (which means that $\dim W \leq 0$ hence 
$W_{\text{CM}} = \emptyset$) or if $\text{deg}\, W = 1$ (in which case 
$W_{\text{CM}}$ is a line $L \subset \piii$) then $c_3(F) \geq -3$ (because 
$c_3(\scg) \geq 0$). 

\vskip2mm 

\noindent 
$\bullet$\quad 
If $\text{deg}\, W = 2$ then $W_{\text{CM}}$ is either a nonsingular conic, or 
the union of two lines (disjoint or not), or a double structure on a line 
$L \subset \piii$ defined by an exact sequence$\, :$ 
\[
0 \lra \sci_{W_{\text{CM}}} \lra \sci_L \lra \sco_L(l) \lra 0\, , 
\] 
with $l \geq -1$. It follows that $\chi(\sco_{W_{\text{CM}}}) \geq 1$ and 
$\chi(\sco_{W_{\text{CM}}}) = 1$ if and only if $W_{\text{CM}}$ is a complete 
intersection of type (1,2). 

If $W_{\text{CM}}$ is a complete intersection of 
type (1,2) and if $H$ is the plane containing it then 
$\sci_{W \cap H , H} = f\sci_{\Gamma , H}(-2)$, where $f \in \tH^0(\sco_H(2))$ is an 
equation of $W_{\text{CM}}$ on $H$ and $\Gamma$ is a subscheme of $H$ of 
dimension $\leq 0$, hence $\sigma_H$ defines a nonzero morphism $F_H^\vee 
\ra \sco_H(-1)$. Consequently, $\tH^0(F_H(-1)) \neq 0$. But this 
\emph{is not possible} because $F_H(2)$ is a globally generated vector bundle 
on $H \simeq \pii$ with $c_1(F_H(2)) = 5$ and $c_2(F_H(2)) = 11$ (see  
Remark~\ref{R:c1=5onp2}). 

It remains that $W_{\text{CM}}$ is not a complete intersection of type (1,2), 
hence $\chi(\sco_{W_{\text{CM}}}) \geq 2$, hence $c_3(F) \geq -3$ (because 
$c_3(\scg) \geq 0$). 

\vskip2mm 

\noindent 
$\bullet$\quad 
Finally, if $\text{deg}\, W = 3$ then $c_1(\scg) = 0$ and $c_2(\scg) = 0$. 
Since $\tH^0(\scg (-1)) = 0$ (because $\tH^0(F^\vee (-1)) = 0$ due to the 
stability of $F$) it follows that $\scg \simeq 2\sco_\piii$. Since $\scg$ is a 
locally free sheaf one deduces that $W = W_{\text{CM}}$. 

As in the case $\text{deg}\, W = 2$ above, $W$ cannot be a complete 
intersection of type (1,3) and, even more, cannot contain, as a subscheme, a 
complete intersection of type (1,2). It follows that one of the following 
holds$\, :$ (i) $W$ is a twisted cubic curve$\, :$ (ii) $W$ is the union of 
three mutually disjoint lines$\, ;$ (iii) $W = X \cup L^\prime$, 
where $X$ is a double structure on a line $L$ such that $\sci_L/\sci_X \simeq 
\sco_L(l)$ with $l \geq 0$ and $L^\prime$ is a line not intersecting 
$L$$\, ;$ (iv) $W$ is a triple structure on a line $L$ containing a double 
structure $X$ on $L$ such that $\sci_L/\sci_X \simeq \sco_L(l)$ and 
$\sci_X/\sci_W \simeq \sco_L(2l + m)$ with $l \geq 0$ and $m \geq 0$ (see 
\cite{bf} or, for example, \cite[\S~A.5]{acm2}; the necessary results are 
recalled in Remark~\ref{R:multilines} from Appendix~\ref{A:h0e(-2)neq0}). 
In the cases (ii)--(iv), one has $\chi(\sco_W) \geq 3$ hence $c_3(F) \geq -3$. 

The case (i) cannot occur. \emph{Indeed}, assume that $W$ is a twisted cubic 
curve. Dualizing the exact sequence$\, :$ 
\[
0 \lra 2\sco_\piii \lra F^\vee \lra \sci_W(1) \lra 0\, , 
\]  
one gets an exact sequence$\, :$ 
\[
0 \lra \sco_\piii(-1) \lra F \lra 2\sco_\piii 
\overset{\delta}{\lra} \omega_W(3) \lra 0\, . 
\]
$W$ is the image of an embedding $\pj \ra \piii$ such that $\sco_W(1)$ 
corresponds to $\sco_\pj(3)$. It follows that the kernel of $\delta_W : 
2\sco_W \ra \omega_W(3)$ corresponds to $\sco_\pj(-7)$. Taking into account 
the exact sequence $F_W \ra 2\sco_W \ra \omega_W(3) \ra 0$, this 
\emph{contradicts} the fact that $F_W(2)$ is globally generated. 

\vskip2mm 

(b) Assume that $c_3(F) = -3$. Lemma~\ref{L:spectrumf} implies that 
$\tH^2(F(l)) = 0$ for $l \geq 0$. One deduces, from Riemann-Roch, that 
$\h^1(F) = -\chi(F) = 4$ and that $\h^0(F(1)) - \h^1(F(1)) = \chi(F(1)) = 0$. 
Lemma~\ref{L:h1e(l)} implies, now, that $\h^1(F(1)) \leq 1$ hence 
$\h^0(F(1)) \leq 1$. 

\vskip2mm 

(c) With the arguments from the proof of (b), one has $\h^1(F) = -\chi(F) \leq 
3$ hence, by Lemma~\ref{L:h1e(l)}, $\h^1(F(1)) = 0$.  
\end{proof}

\begin{lemma}\label{L:h0fdual} 
If $F$ is a stable rank $3$ vector bundle on $\piii$ with $c_1(F) = -1$ then 
${\fam0 h}^0(F^\vee) \leq 1 + {\fam0 h}^0(F(1))$. 
\end{lemma}

\begin{proof} 
A nonzero global section of $F^\vee$ defines an exact sequence$\, :$ 
\[
0 \lra \scf \lra F \lra \sci_Z \lra 0\, , 
\]
where $Z$ is a closed subscheme of $\piii$ with $\dim Z \leq 1$ and $\scf$ 
is a rank 2 reflexive sheaf with $c_1(\scf) = -1$. It follows that $\scf^\vee 
\simeq \scf(1)$. Dualizing the above exact sequence one thus gets an exact 
sequence $0 \ra \sco_\piii \ra F^\vee \ra \scf(1)$. 
Since $\h^0(\scf(1)) \leq \h^0(F(1))$, the inequality from the statement 
follows. 
\end{proof}

\begin{prop}\label{P:h2e(-3)=0h1e(-3)neq0c2=11} 
Let $E$ and $E^\prim$ be as at the beginning of this section. Assume that 
$E^\prim$ is stable, that ${\fam0 H}^2(E(-3)) = 0$, and that 
${\fam0 H}^1(E(-3)) \neq 0$. Then one of the following holds$\, :$ 
\begin{enumerate} 
\item[(i)] $c_3 = 11$ and $E \simeq \sco_\piii(1) \oplus E_0$ where, up to a 
linear change of coordinates, $E_0$ is the kernel of the epimorphism$\, :$ 
\[
(x_0,\, x_1,\, x_2^2,\, x_2x_3,\, x_3^2) \colon 2\sco_\piii(2) \oplus 
3\sco_\piii(1) \lra \sco_\piii(3)\, ; 
\]  
\item[(ii)] $c_3 = 9$ and $E(-2)$ is the cohomology of a monad$\, :$
\[
0 \lra \sco_\piii(-1) \overset{\beta}{\lra} 5\sco_\piii \oplus 2\sco_\piii(-1)  
\overset{\alpha}{\lra} 2\sco_\piii(1) \lra 0\, ,  
\] 
with ${\fam0 H}^0(\alpha(1)) \colon {\fam0 H}^0(5\sco_\piii(1) \oplus 
2\sco_\piii) \ra {\fam0 H}^0(2\sco_\piii(2))$ surjective. Up to a linear change 
of coordinates in $\piii$, there are only finitely many such monads, which 
are listed in the proof$\, ;$ 
\item[(iii)] $c_3 = 7$ and $E(-2)$ is the cohomology of a general monad$\, :$ 
\[
0 \lra 2\sco_\piii(-1) \lra 8\sco_\piii \lra 3\sco_\piii(1) \lra 0\, ; 
\] 
\item[(iv)] $c_3 = 5$ and $E(-2)$ is the kernel of an arbitrary  
epimorphism$\, :$ 
\[
{\fam0 T}_\piii(-1) \oplus \sco_\piii \lra \sco_\piii(2)\, . 
\] 
\end{enumerate}
\end{prop} 

\begin{proof} 
Let $F = E^\prim (-2)$ be the normalized rank 3 vector bundle associated to 
$E^\prim$ and let $k_F = (k_1,\, k_2,\, k_3)$ be the spectrum of $F$. The 
condition $\tH^2(E(-3)) = 0$ is equivalent to $k_3 \geq -1$ and the condition 
$\tH^1(E(-3)) \neq 0$ is equivalent $k_1 \geq 0$. Taking into account 
Lemma~\ref{L:spectrumf}, the possible spectra of $F$ are $(0,-1,-1)$, 
$(0,0,-1)$, $(1,0,-1)$, $(0,0,0)$, $(1,0,0)$ and $(1,1,0)$. 

The spectrum $(1,0,-1)$ cannot, actually, occur in our context. \emph{Indeed},  
using the spectrum one would get that $\h^1(E(-4)) = \h^1(F(-2)) = 1$ and 
$\h^1(E(-3)) = \h^1(F(-1)) = 3$. But this \emph{would contradict} 
Remark~\ref{R:muh1e(-4)}.   

We analyse, now, case by case, the remaining spectra. Recall the formulae 
from the beginning of the proof of Prop.~\ref{P:h1e(-3)=0c2=11}. 

\vskip2mm 

\noindent 
{\bf Case 1.}\quad $F$ \emph{has spectrum} $(0,-1,-1)$. 

\vskip2mm 

\noindent 
In this case, $r = 5$, $c_3(F) = 1$ and $c_3 = 11$. Using the spectrum, one 
gets that $\h^1(E(l)) = \h^1(F(l+2)) = 0$ for $l \leq -4$ and 
$\h^1(E(-3)) = 1$. In particular, $s := \h^1(E(-3)) - \h^1(E(-4)) = 1$ hence, 
by Lemma~\ref{L:h2e(-3)=0}(e), one must have $\tH^2_\ast(E) = 0$. On the other 
hand, by Riemann-Roch, $\h^1(E(-2)) = \h^1(F) = 2$. Lemma~\ref{L:h1e(l)} 
implies that $\tH^1(E(l)) = 0$ for $l \geq -1$. One deduces, from  
Riemann-Roch, that $\h^0(E(-1)) = \h^0(F(1)) = 2$.  

We want to show, now, that the multiplication map $\mu \colon \tH^1(E(-3))  
\otimes_k \tH^0(\sco_\piii(1)) \ra \tH^1(E(-2))$ is surjective. 
\emph{Indeed}, if $\mu$ is not surjective then, by Remark~\ref{R:beilinson}, 
there would exist an exact sequence$\, :$ 
\[
0 \lra \Omega_\piii^2(2) \lra \Omega_\piii^1(1) \oplus 2\sco_\piii \lra 
Q \lra 0\, , 
\]
with $Q$ locally free. $Q$ would have rank 2 and Chern classes $c_1(Q) = 1$, 
$c_2(Q) = 1$, $c_3(Q) = -1$ which \emph{is not possible}. It thus remains 
that $\mu$ is surjective.  

If one considers the universal extension$\, :$ 
\[
0 \lra E(-3) \lra E_3 \overset{\phi}{\lra} \sco_\piii \lra 0\, . 
\]
then $\tH^1_\ast(E_3) = 0$ and $\tH^2_\ast(E_3) \simeq \tH^2_\ast(E) = 0$ hence 
$E_3$ is a direct sum of line bundles. $E_3$ has rank 6,  
$\h^0(E_3) = 0$, $\h^0(E_3(1)) = 2$ and $\h^0(E_3^\vee(-3)) = 0$. The only 
possibility is $E_3 \simeq 2\sco_\piii(-1) \oplus 4\sco_\piii(-2)$. The 
component $\phi_1 \colon 2\sco_\piii(-1) \ra \sco_\piii$ of $\phi$ is defined by 
two linearly independent linear forms $h_0$ and $h_1$. Complete $h_0,\, h_1$ 
to a $k$-basis $h_0, \ldots , h_3$ of $\tH^0(\sco_\piii(1))$. Then, up to 
an automorphism of $E_3$, one can assume that the component $\phi_2 \colon 
4\sco_\piii(-2) \ra \sco_\piii$ of $\phi$ is defined by $h_2^2,\, h_2h_3,\, 
h_3^2,\, 0$. It follows that $E$ is as in item (i) of the statement.   

\vskip2mm 

\noindent 
{\bf Case 2.}\quad $F$ \emph{has spectrum} $(0,0,-1)$. 

\vskip2mm 

\noindent 
In this case, $r = 4$, $c_3(F) = -1$ and $c_3 = 9$. 
Using the spectrum, one gets that 
$\h^1(E(l)) = \h^1(F(l+2)) = 0$ for $l \leq -4$ and $\h^1(E(-3)) = 2$. 
In particular, $s := \h^1(E(-3)) - \h^1(E(-4)) = 2$. 
Using Riemann-Roch, $\h^1(E(-2)) = \h^1(F) = 3$. Lemma~\ref{L:h1e(l)} implies 
that $\tH^1(E(l)) = 0$ for $l \geq -1$ and one deduces, now, from  
Riemann-Roch, that $\h^0(E(-1)) = \h^0(F(1)) = 1$. 

We assert, now, that the multiplication map $\mu \colon \tH^1(E(-3)) \otimes 
\tH^0(\sco_\piii(1)) \ra \tH^1(E(-2))$ is surjective. 
\emph{Indeed}, if it is not then, by Remark~\ref{R:beilinson}, one would 
have an exact sequence$\, :$ 
\[
0 \lra 2\Omega_\piii^2(2) \lra 2\Omega_\piii^1(1) \oplus \sco_\piii \lra 
Q \lra 0\, ,  
\]
with $Q$ locally free. $Q$ would have rank 1 but this is, clearly, 
\emph{not possible}. It remains that $\mu$ is surjective hence the graded 
$S$-module $\tH^1_\ast(E)$ is generated by $\tH^1(E(-3))$. 

On the other hand, since $\tH^1(E(-4)) = 0$, Lemma~\ref{L:h2e(-3)=0}(b) 
implies that the graded $S$-module $\tH^1_\ast(E^\vee)$ is generated by 
$\tH^1(E^\vee(1))$. By Lemma~\ref{L:h2e(-3)=0}(d),(f), $\h^1(E^\vee(1)) \leq 1$. 
One deduces, now, that $E(-2)$ has a Horrocks monad of the form$\, :$ 
\[
0 \lra \sco_\piii(-1) \overset{\beta}{\lra} B \overset{\alpha}{\lra} 
2\sco_\piii(1) \lra 0\, ,  
\]
where $B$ is a direct sum of line bundles (see Barth and Hulek \cite{bh} 
for information about Horrocks monads). $B$ has rank 7, $\h^0(B(-1)) = 0$, 
$\h^0(B) = 5$ and $\h^0(B^\vee(-2)) = 0$ hence $B \simeq 5\sco_\piii \oplus 
2\sco_\piii(-1)$. Moreover, as we saw above, $\tH^1(E(-1)) = 0$ hence 
$\tH^0(\alpha(1))$ is surjective.   

Let $\phi \colon 5\sco_\piii \ra 2\sco_\piii(1)$ be the first component 
of $\alpha$. Since 
$\tH^0(E(-2)) = 0$, it follows that $\tH^0(\phi)$ is injective. Since  
$\tH^0(\alpha(1))$ is surjective, $\tH^0(\phi(1)) \colon 
\tH^0(5\sco_\piii(1)) \ra \tH^0(2\sco_\piii(2))$ has corank $\leq 2$. 

$\phi$ is defined by a $k$-linear map $A \colon k^5 \ra k^2 \otimes S_1$. 
\emph{We assert} that, for every nonzero $k$-linear function $k^2 \ra k$, the 
composite map $k^5 \overset{A}{\lra} k^2 \otimes S_1 \ra S_1$ has rank 
$\geq 3$. \emph{Indeed}, otherwise, $A$ would be represented (up to 
automorphisms of $k^5$ and $k^2$) by a matrix of linear forms of the form$\, :$ 
\[
\begin{pmatrix} 
h_{00} & h_{01} & h_{02} & h_{03} & h_{04}\\ 
h_{10} & h_{11} & 0 & 0 & 0 
\end{pmatrix}\, . 
\]
In this case, one could choose $h_2,\, h_3 \in S_1$ such that their images into 
$S_1/(kh_{10} + kh_{11})$ are linearly independent. Then the images of 
$(0\, ,\, h_2^2)^{\text{t}}$, $(0\, ,\, h_2h_3)^{\text{t}}$, 
$(0\, ,\, h_3^2)^{\text{t}}$ into 
$\tH^0(2\sco_\piii(2))/\text{Im}\, \tH^0(\phi(1))$ would be linearly 
independent, which would \emph{contradict} the fact that $\tH^0(\phi(1))$ 
has corank $\leq 2$. 

One can also associate to $A$ a morphism $\psi \colon 5\sco_\pj \ra 
\sco_\pj(1) \otimes_k S_1$. The above verifed assertion shows that $\psi$ 
has rank $\geq 3$ at every point of $\pj$. The morphisms $\psi$ and, 
consequently, $\phi$ can be, now, classified using the method used in 
\cite[Remark~B.1]{acm1}. One gets that, up to a linear change of coordinates 
in $\piii$, either $E \simeq \sco_\piii(1) \oplus E_0$ where $E_0$ is the 
kernel of the epimorphism$\, :$ 
\[
\begin{pmatrix} 
x_0 & x_1 & x_2 & x_3 & 0\\ 
0 & x_0 & x_1 & x_2 & x_3 
\end{pmatrix} 
\colon 5\sco_\piii(2) \lra 2\sco_\piii(3)\, , 
\] 
or $E$ is isomorphic to the kernel of the epimorphism$\, :$ 
\[
\begin{pmatrix} 
x_0 & x_1 & x_2 & x_3 & 0 & 0\\ 
0 & 0 & x_1 & x_2 & x_3 & x_0^2 
\end{pmatrix} 
\colon 5\sco_\piii(2) \oplus \sco_\piii(1) \lra 2\sco_\piii(3)\, , 
\]   
or $E(-2)$ is the cohomology of a minimal monad having the shape from the 
statement, with$\, :$ 
\begin{gather*} 
\alpha = 
\begin{pmatrix} 
x_0 & 0 & x_2 & x_3 & 0 & x_1^2 & 0\\ 
0 & x_1 & 0 & x_2 & x_3 & 0 & x_0^2 
\end{pmatrix}\, ,\,  
\beta = 
(-x_2\, ,\, -x_3\, ,\, x_0\, ,\, 0\, ,\, x_1\, ,\, 0\, ,\, 0)^{\text{t}} \, , 
\end{gather*}
or $E(-2)$ is the cohomology of a minimal monad having the shape from the 
statement, with$\, :$ 
\begin{gather*} 
\alpha = 
\begin{pmatrix} 
x_0 & x_1 & x_2 & x_3 & 0 & x_0^2 & 0\\ 
x_1 & 0 & 0 & x_2 & x_3 & 0 & x_0^2 
\end{pmatrix}\, ,\,  
\beta = 
(-x_2\, ,\, -x_3\, ,\, x_0\, ,\, x_1\, ,\, 0\, ,\, 0\, ,\, 0)^{\text{t}} \, .
\end{gather*} 

\vskip2mm 

\noindent 
{\bf Case 3.}\quad $F$ \emph{has spectrum} $(0,0,0)$. 

\vskip2mm 

\noindent 
In this case, $r = 3$ (hence $E = F(2)$), $c_3(F) = -3$ and $c_3 = 7$. 
Using the spectrum, one gets $\h^1(E(l)) = \h^1(F(l + 2)) = 0$ for $l \leq -4$ 
and $\h^1(E(-3)) = 3$. By Riemann-Roch, $\h^1(E(-2)) = 4$.  
Moreover, $\tH^0(E^\vee(1)) = \tH^0(F^\vee(-1)) = 0$, 
because $F$ is stable. 
Remark~\ref{R:h2e(-3)=0}(ii) implies that the graded $S$-module $\tH^1_\ast(E)$ 
is generated by $\tH^1(E(-3))$. On the other hand, since $\tH^1(E(-4)) = 0$, 
Lemma~\ref{L:h2e(-3)=0}(b) implies that the graded $S$-module 
$\tH^1_\ast(E^\vee)$ is generated by $\tH^1(E^\vee(1))$. Since $\tH^0(E^\vee(1)) 
= 0$, it follows that $\h^1(E^\vee(1)) = -\chi(E^\vee(1)) = -\chi(F^\vee(-1)) 
= 2$. One deduces that $E(-2) = F$ has a Horrocks monad of the form$\, :$ 
\[
0 \lra 2\sco_\piii(-1) \lra B \lra 3\sco_\piii(1) \lra 0\, ,  
\]
where $B$ is a direct sum of line bundles. $B$ has rank 8, $\tH^0(B(-1)) = 0$ 
and $\h^0(B) = \h^0(3\sco_\piii(1)) - \h^1(E(-2)) = 8$. It follows that 
$B \simeq 8\sco_\piii$ hence $E(-2)$ is the cohomology of a monad of the form 
from item (iii) of the statement.   

We show, now, that for every $t \in \{0,\, 1,\, 2\}$, there exists a stable 
rank 3 vector bundle $F$ with $c_1(F) = -1$, $c_2(F) = 3$, $c_3(F) = -3$, 
$\h^0(F^\vee) = t$ such that $E := F(2)$ is globally generated. 

\vskip2mm 

\noindent 
{\bf Construction 3.0.}\quad We will show that there exist epimorphisms 
$\e \colon 2\text{T}_\piii(-1) \ra 3\sco_\piii(1)$ such that 
$\tH^0(\e (1)) \colon \tH^0(2\text{T}_\piii) \ra \tH^0(3\sco_\piii(2))$ is 
injective (hence bijective). If $F$ is the kernel of such an epimorphism then, 
using the exact sequence$\, :$ 
\[
0 \lra F \lra 2\text{T}_\piii(-1) \overset{\e}{\lra} 3\sco_\piii(1) \lra 0\, , 
\] 
one sees that $F$ is a rank 3 vector bundle with $c_1(F) = -1$, $c_2(F) = 3$, 
$c_3(F) = -3$, $\tH^i(F(1)) = 0$, $i = 0,\, 1$, and $\tH^0(F^\vee) = 0$. 
Moreover, $F$ is 2-regular hence $E := F(2)$ is globally generated. 

Let us prove, now, the existence of the epimorphisms $\e$ with the above 
property. We start by recalling that a general morphism 
$\text{T}_\piii(-2) \ra 3\sco_\piii$ is a monomorphism and its cokernel is 
isomorphic to $\sco_Q(0,2)$ for some nonsingular quadric surface $Q \subset 
\piii$ endowed with an isomorphism $Q \simeq \pj \times \pj$ (see, for 
example, Case 3 in the proof of \cite[Prop.~6.3]{acm1}). Actually, if 
$K$ is the kernel of the evaluation morphism$\, :$ 
\[
3\sco_\piii \simeq \tH^0(\sco_Q(0,2))\otimes_k \sco_\piii \lra \sco_Q(0,2)\, , 
\] 
then $K$ is a rank 3 vector bundle with $\tH^1_\ast(K) = 0$ and $\tH^2_\ast(K) 
\simeq {\underline k}(2)$ hence $K \simeq \text{T}_\piii(-2)$. It thus 
suffices to show that there exist epimorphisms 
$\eta \colon \text{T}_\piii \vb Q \ra \sco_Q(2,4)$ such that $\tH^0(\eta)$ is 
injective (hence bijective). For a proof of the existence of such 
epimorphisms, see Lemma~\ref{L:tvbqraoq(2,4)} in 
Appendix~\ref{A:miscellaneous}. 

\vskip2mm 

\noindent 
{\bf Construction 3.1.}\quad Let $F$ be a rank 3 vector bundle on $\piii$ 
such that $F^\vee$ can be realized as a nontrivial extension$\, :$ 
\[
0 \lra F^\prim \lra F^\vee \lra \sco_\piii(1) \lra 0\, , 
\]     
where $F^\prim$ is a general 3-instanton. Here ``general'' means that $F^\prim$ 
has no jumping line of maximal order 3. In particular, $\tH^0(F^\prim (1)) = 
0$ and, according to Gruson and Skiti \cite[Prop.~1.1.1]{gs}, $F^\prim (2)$ is 
globally generated. Dualizing the above extension and taking into account that 
$F^{\prim \vee} \simeq F^\prim$, one gets that $c_1(F) = -1$, $c_2(F) = 3$, 
$c_3(F) = -3$, $\h^0(F(1)) = 1$ and that $E := F(2)$ is globally generated. 

Let $\xi$ be the (nonzero) element of $\tH^1(F^\prim (-1))$ defining the above 
extension. One has an exact sequence$\, :$ 
\[
0 \lra \tH^0(F^\vee) \lra \tH^0(\sco_\piii(1)) \overset{\partial}{\lra} 
\tH^1(F^\prim )\, , 
\] 
where $\partial (h) = h\xi$, $\forall \, h \in \tH^0(\sco_\piii(1))$. 

Now, if $0 \neq h \in \tH^0(\sco_\piii(1))$ and $H \subset \piii$ is the plane 
of equation $h = 0$ then one has an exact sequence$\, :$ 
\[
0 \lra \tH^0(F_H^\prim ) \lra \tH^1(F^\prim (-1)) \overset{h}{\lra} 
\tH^1(F^\prim )\, . 
\]   
Since $\tH^1(F^\prim (-2)) = 0$ it follows that $\tH^0(F_H^\prim (-1)) = 0$ 
hence $\h^0(F_H^\prim ) \leq 1$. The set of planes $H \subset \piii$ for 
which $\h^0(F_H^\prim ) = 1$ form a subset of pure dimension 1 of the dual 
projective space $\p^{3\vee}$ (indeed, since $\h^1(F^\prim (-1)) = 3$ and 
$\h^1(F^\prim ) = 4$ every component of this subset has codimension $\leq 2$; 
on the other hand, the subset has dimension $\leq 1$, by the main result of 
\cite{co2}). Moreover, if $h,\, h^\prime \in \tH^0(\sco_\piii(1))$ are linearly 
independent and if $L \subset \piii$ is the line of equations $h = h^\prime = 
0$ then one has an exact sequence$\, :$ 
\[
0 = \tH^0(\sci_L(1) \otimes F^\prim ) \lra \tH^1(F^\prim (-1))  
\xra{\left(\begin{smallmatrix} h\\ h^\prime \end{smallmatrix}\right)} 
\tH^1(2F^\prim )  
\]
hence there is no non-zero element of $\tH^1(F^\prim (-1))$ annihilated by 
two linearly independent linear forms. 

One deduces that if $\xi$ is a general element of $\tH^1(F^\prim (-1))$ then 
$\tH^0(F^\vee) = 0$, while for extensions defined by some special elements 
of $\tH^1(F^\prim (-1))$ one has $\h^0(F^\vee) = 1$. 

\vskip2mm 

\noindent 
{\bf Construction 3.2.}\quad Let $L_0,\, L_1,\, L_2$ be mutually disjoint 
lines in $\piii$ and let $Y$ be their union. Consider an epimorphism 
$\delta : 2\sco_\piii \ra \omega_Y(3) \simeq \sco_Y(1)$. Then it is well known 
that there exists an extension$\, :$ 
\[
0 \lra 2\sco_\piii \lra F_1 \lra \sci_Y(1) \lra 0\, , 
\] 
with $F_1$ locally free and such that, dualizing it, one gets an exact 
sequence$\, :$ 
\[
0 \lra \sco_\piii(-1) \lra F_1^\vee \lra 2\sco_\piii \overset{\delta}{\lra} 
\omega_Y(3) \lra 0\, . 
\]
(This is Serre's method of extensions recalled in Thm.~\ref{T:serreext} from 
Appendix~\ref{A:serre}.) Put $F := F_1^\vee$. One has exact 
sequences$\, :$ 
\begin{gather*}
0 \lra 2\sco_\piii \lra F^\vee \lra \sci_Y(1) \lra 0\, ,\\ 
0 \lra \sco_\piii(-1) \lra F \lra 2\sco_\piii \overset{\delta}{\lra} 
\omega_Y(3) \lra 0\, . 
\end{gather*}
It follows that $F$ is a rank 3 vector bundle with $c_1(F) = -1$, $c_2(F) = 3$, 
$c_3(F) = -3$, $\tH^0(F) = 0$ and $\h^0(F^\vee) = 2$. 

It is shown, in Lemma~\ref{L:2oraoy(1)} from Appendix~\ref{A:miscellaneous}, 
that if $\sck$ is the kernel of a general epimorphism $\delta \colon 
2\sco_\piii \ra \sco_Y(1) \simeq \omega_Y(3)$ then $\sck(2)$ is globally 
generated. One deduces that if $F^\vee$ is the middle term of the extension 
corresponding to such a $\delta$ then $E := F(2)$ is globally generated. 
Notice, also, that the Claim from the proof of Lemma~\ref{L:2oraoy(1)} 
implies that $\tH^0(\sck(1)) = 0$ hence $\h^0(F(1)) = 1$. 

\vskip2mm 

\noindent 
{\bf Case 4.}\quad $F$ \emph{has spectrum} $(1,0,0)$. 

\vskip2mm 

\noindent 
In this case $r = 3$ (hence $E = F(2)$), $c_3(F) = -5$ and $c_3 = 5$. 
Using the spectrum one gets that $\h^1(E(l)) = \h^1(F(l+2)) = 0$ for $l \leq 
-5$, $\h^1(E(-4)) = 1$ and $\h^1(E(-3)) = 4$ hence $s := \h^1(E(-3)) - 
\h^1(E(-4)) = 3$.  
Moreover, $\tH^0(E^\vee(1)) = \tH^0(F^\vee(-1)) = 0$ (because $F$ is stable).  
It follows, from Remark~\ref{R:h2e(-3)=0}(ii), that the graded 
$S$-module $\tH^1_\ast(E)$ is generated in degrees $\leq -3$. But, using 
Remark~\ref{R:muh1e(-4)}, one sees that, actually, $\tH^1_\ast(E)$ is generated 
by $\tH^1(E(-4))$. 

\vskip2mm 

\noindent 
{\bf Claim 4.1.}\quad \emph{The graded} $S$-\emph{module} $\tH^1_\ast(E^\vee)$ 
\emph{is generated by} $\tH^1(E^\vee(1))$. 

\vskip2mm 

\noindent 
\emph{Indeed}, since $\tH^0(E^\vee(1)) = 0$ one gets, by Riemann-Roch, that 
$\h^1(E^\vee(1)) = -\chi(E^\vee(1)) = -\chi(F^\vee(-1)) = 1$. Assume, by 
contradiction, that $\tH^1_\ast(E^\vee)$ is not generated by $\tH^1(E^\vee(1))$. 
Then Lemma~\ref{L:h2e(-3)=0}(h) implies that $\tH^1_\ast(E^\vee) \simeq 
{\underline k}(-1) \oplus {\underline k}(-2)$. 
Consider, now, the universal extension$\, :$ 
\[
0 \lra E(-4) \lra E_4 \lra \sco_\piii \lra 0\, . 
\] 
One has $\tH^1_\ast(E_4) = 0$ (hence $\tH^2_\ast(E_4^\vee) = 0$) and 
$\tH^1_\ast(E_4^\vee) \simeq {\underline k}(2) \oplus {\underline k}(1)$. 
One deduces, from the Horrocks correspondence, that $E_4^\vee \simeq 
\Omega_\piii(2) \oplus \Omega_\piii(1) \oplus A$, where $A$ is a direct sum of 
line bundles. But \emph{this is not possible} because $E_4$ has rank 4. 
Claim 4.1 is proven. 

\vskip2mm 

\noindent 
One deduces, now, that $E(-2)$ is the cohomology of a Horrocks monad of the 
form$\, :$ 
\[
0 \lra \sco_\piii(-1) \lra B \lra \sco_\piii(2) \lra 0\, , 
\]  
with $B$ a direct sum of line bundles. $B$ has rank 5, $c_1(B) = 0$ and 
$\h^0(B(-1)) = 0$ hence $B \simeq 5\sco_\piii$.  
It follows that $E(-2)$ is the kernel of an epimorphism $\text{T}_\piii(-1) 
\oplus \sco_\piii \ra \sco_\piii(2)$.  

Conversely, for any epimorphism $\phi \colon \text{T}_\piii(-1) \oplus 
\sco_\piii \ra \sco_\piii(2)$, one has that $\Ker \phi(2)$ is globally generated 
as one can easily see by considering the (geometric) Koszul complex associated 
to $\phi$ and using the fact that ${\bigwedge}^2\left(\text{T}_\piii(-1) 
\oplus \sco_\piii\right)$ is globally generated.  

\vskip2mm 

\noindent 
{\bf Case 5.}\quad $F$ \emph{has spectrum} $(1,1,0)$. 

\vskip2mm 

\noindent 
This case \emph{cannot occur} in our context. \emph{Indeed}, assume, 
by contradiction, that there exists a vector bundle $E$ satisfying the 
hypotheses of the proposition, such that the associated (normalized) rank 3 
vector bundle $F$ has spectrum $(1,1,0)$. $E$ has rank $r = 3$ (hence $E = 
F(2)$), $c_3(F) = -7$ and $c_3 = 3$. 
Using the spectrum one gets that $\h^1(E(l)) = \h^1(F(l+2)) = 0$ for $l \leq 
-5$, $\h^1(E(-4)) = 2$ and $\h^1(E(-3)) = 5$ hence $s := \h^1(E(-3)) - 
\h^1(E(-4)) = 3$. Moreover, by Riemann-Roch, $\h^1(E(-2)) = \h^1(F) = 6$.  
Since $\tH^0(E^\vee(1)) = \tH^0(F^\vee(-1)) = 0$ (because 
$F$ is stable) one deduces, from Remark~\ref{R:h2e(-3)=0}(ii), that the 
graded $S$-module $\tH^1_\ast(E)$ is generated in degrees $\leq -3$. But,  
according to Remark~\ref{R:muh1e(-4)}, the multiplication map $\tH^1(E(-4)) 
\otimes \tH^0(\sco_\piii(1)) \ra \tH^1(E(-3))$ is surjective hence 
$\tH^1_\ast(E)$ is, actually, generated by $\tH^1(E(-4))$.  

\vskip2mm 

\noindent 
{\bf Claim 5.1.}\quad $E(-2)$ \emph{is the kernel of an epimorphism} 
$3\sco_\piii(1) \oplus 2\sco_\piii \ra 2\sco_\piii(2)$. 

\vskip2mm 

\noindent 
\emph{Indeed}, since $\h^0(E^\vee(1)) = 0$ one deduces 
that $\h^1(E^\vee(1)) = -\chi(E^\vee(1)) = -\chi(F^\vee(-1)) = 0$. It follows, 
from Lemma~\ref{L:h2e(-3)=0}(h), that either $\tH^1_\ast(E^\vee) = 0$ or 
$\tH^1_\ast(E^\vee) \simeq {\underline k}(-2)$. Anyway, $E(-2)$ is the 
cohomology of a Horrocks monad of the form$\, :$ 
\[
0 \lra \sco_\piii \overset{\beta}{\lra} B 
\overset{\alpha}{\lra} 2\sco_\piii(2) \lra 0\, , 
\] 
with $B$ a direct sum of line bundles. $B$ has rank 6, $c_1(B) = 3$, and  
$\h^0(B(-1)) = \h^0(2\sco_\piii(1)) - \h^1(E(-3)) = 3$ hence $B \simeq 
3\sco_\piii(1) \oplus 3\sco_\piii$. Since there is no locally split monomorphism 
$\sco_\piii \ra 3\sco_\piii(1)$, the component $\beta_2 \colon \sco_\piii \ra 
3\sco_\piii$ of $\beta$ must be non-zero and the claim is proven. 

\vskip2mm 

\noindent 
Now, since $\h^1(E(-2)) = 6$, Lemma~\ref{L:h1e(l)} implies that $\h^1(E) = 0$. 
One deduces that if there exists an epimorphism $\phi_0 \colon 3\sco_\piii(1) 
\oplus 2\sco_\piii \ra 2\sco_\piii(2)$ such that $\tH^0(\Ker \phi_0) = 0$ and 
$\Ker \phi_0(2)$ is globally generated then, for a \emph{general} 
epimorphism $\phi \colon 3\sco_\piii(1) \oplus 2\sco_\piii \ra 2\sco_\piii(2)$, 
$\tH^0(\Ker \phi) = 0$ and $\Ker \phi(2)$ is globally generated. But, for such  
a general epimorphism, the cokernel of the component $\phi_1 \colon 
3\sco_\piii(1) \ra 2\sco_\piii(2)$ of $\phi$ is isomorphic to $\omega_C(3)$, 
for some (nonsingular) twisted cubic curve $C \subset \piii$. One gets an 
exact sequence$\, :$ 
\[
0 \lra \sco_\piii(-1) \lra \Ker \phi \lra 2\sco_\piii 
\overset{{\overline \phi}_2}{\lra} \omega_C(3) \lra 0\, , 
\]    
with ${\overline \phi}_2$ induced by the other component $\phi_2 \colon 
2\sco_\piii \ra 2\sco_\piii(2)$ of $\phi$. Restricting ${\overline \phi}_2$ 
to $C$ one gets an epimorphism $2\sco_C \ra \omega_C(3)$ whose kernel is 
isomorphic to $\omega_C^{-1}(-3)$. Since $\omega_C^{-1}(-1)$ is not globally 
generated, it follows that $\Ker \phi(2)$ cannot be globally generated. 
This, finally, shows that Case 5 \emph{cannot occur} in our context.  
\end{proof}

\section{The case $c_2 = 12$}\label{S:c2=12} 

We denote, in this section, by $E$ a globally generated vector bundle on 
$\piii$, with $c_1 = 5$, $c_2 = 12$, and such that $\tH^i(E^\vee) = 0$, 
$i = 0,\, 1$, and $\tH^0(E(-2)) = 0$. If $E$ has rank $r \geq 3$, we denote 
by $E^\prim$ the rank 3 vector bundle associated to $E$ in the exact sequence 
\eqref{E:oeeprim} and by $F$ the normalized bundle $E^\prim (-2)$. 

\begin{prop}\label{P:h0e(-1)geq2c2=12} 
If ${\fam0 h}^0(E(-1)) \geq 2$ then one of the following holds$\, :$ 
\begin{enumerate} 
\item[(i)] $c_3 = 20$ and $E \simeq 3\sco_\piii(1) \oplus 
2{\fam0 T}_\piii(-1)$$\, ;$ 
\item[(ii)] $c_3 = 18$ and $E \simeq 2\sco_\piii(1) \oplus {\fam0 T}_\piii(-1) 
\oplus \Omega_\piii(2)$$\, ;$ 
\item[(iii)] $c_3 = 16$ and $E \simeq \sco_\piii(1) \oplus E_0$, where, up to a 
linear change of coordinates, $E_0$ is the cohomology of the monad$\, :$ 
\[
\sco_\piii(-1) \xra{\left(\begin{smallmatrix} s\\ u 
\end{smallmatrix}\right)} 2\sco_\piii(2) \oplus 
2\sco_\piii(1) \oplus 4\sco_\piii \xra{(p\, ,\, 0)} \sco_\piii(3) 
\] 
where $\sco_\piii(-1) \overset{s}{\ra} 2\sco_\piii(2) \oplus 2\sco_\piii(1) 
\overset{p}{\ra} \sco_\piii(3)$ is a subcomplex of the Koszul complex defined 
by $x_0,\, x_1,\, x_2^2,\, x_3^2$ and $u : \sco_\piii(-1) \ra 4\sco_\piii$ is 
defined by $x_0, \ldots ,x_3$. 
\end{enumerate}
\end{prop}

\begin{proof} 
Let $Y$ be the curve associated to $E$ in the exact sequence \eqref{E:oeiy}. 
It is nonsingular and connected, of degree 12 (see Lemma~\ref{L:yconnected}). 
Our hypotheses imply that $\tH^0(\sci_Y(3)) = 0$ and $\h^0(\sci_Y(4)) \geq 2$. 
It follows that $Y$ is directly linked to a curve $Y^\prime$ of degree 4 by a 
complete intersection of type$(4,4)$. $Y^\prime$ is locally complete 
intersection except at finitely many points, where it is locally 
Cohen-Macaulay. The fundamental exact sequence of liaison (recalled in 
\cite[Remark~2.6]{acm1})$\, :$ 
\[
0 \lra \sco_\piii(-8) \lra 2\sco_\piii(-4) \lra \sci_Y \lra \omega_{Y^\prime}(-4) 
\lra 0 
\] 
implies that $\omega_{Y^\prime}(1)$ is globally generated. It follows that a 
general global section of $\omega_{Y^\prime}(1)$ generates this sheaf except at 
finitely many points hence it defines an extension$\, :$ 
\[
0 \lra \sco_\piii(-2) \lra \scg \lra \sci_{Y^\prime}(1) \lra 0
\]
with $\scg$ a rank 2 reflexive sheaf with $c_1(\scg) = -1$, $c_2(\scg) = 
\text{deg}\, Y^\prime - 2 = 2$ (see \cite[Thm.~4.1]{ha}). Since $\chi (\scg) 
= \chi (\sci_{Y^\prime}(1)) = \chi (\sco_\piii(1)) - \chi (\sco_{Y^\prime}(1))$, 
using Riemann-Roch (see Remark~\ref{R:chern}) one gets that $c_3(\scg) = 
4 - 2\chi (\sco_{Y^\prime})$. We also know that $c_3 = -12 - 2\chi (\sco_Y)$. 
On the other hand, by a basic formula in liaison theory (deduced from the 
fundamental exact sequence of liaison)$\, :$ 
\[
\chi (\sco_{Y^\prime}) - \chi (\sco_Y) = 
\frac{1}{2}(4 + 4 - 4)(\text{deg}\, Y - \text{deg}\, Y^\prime) = 16\, . 
\]   
It follows that $c_3 = c_3(\scg) + 16$. 

Now, if $\tH^0(\sci_{Y^\prime}(1)) \neq 0$ then $Y^\prime$ is a complete 
intersection of type $(1,4)$, hence $\omega_{Y^\prime} \simeq \sco_{Y^\prime}(1)$. 
It follows that $\tH^0(\omega_{Y^\prime}(-1)) \neq 0$, hence $\tH^0(\sci_Y(3)) 
\neq 0$, a \emph{contradiction}. 

It remains that $\tH^0(\sci_{Y^\prime}(1)) = 0$ hence $\scg$ is \emph{stable}. 
\cite[Thm.~8.2(b)]{ha} implies, now, that $c_3(\scg) \in \{0,\, 2,\, 4\}$ 
hence $c_3 \in \{16,\, 18,\, 20\}$.   

\vskip2mm

\noindent 
{\bf Case 1.}\quad $c_3(\scg) = 4$. 

\vskip2mm 

\noindent 
In this case, by \cite[Lemma~9.6]{ha}, $\scg$ can be realized as an 
extension$\, :$ 
\[
0 \lra \sco_\piii(-1) \lra \scg \lra \sci_Z \lra 0 
\]  
where $Z$ is a complete intersection of type $(1,2)$ (actually, $\chi(\scg(1)) 
= 2$ hence $\h^0(\scg(1)) \geq 2$). It follows that $\scg$ 
has a resolution of the form$\, :$ 
\[
0 \lra \sco_\piii(-3) \lra 2\sco_\piii(-1) \oplus \sco_\piii(-2) \lra \scg 
\lra 0
\]
hence $\sci_{Y^\prime}$ has a resolution of the form$\, :$ 
\[
0 \lra \sco_\piii(-3) \oplus \sco_\piii(-4) \lra 2\sco_\piii(-2) \oplus 
\sco_\piii(-3) \lra \sci_{Y^\prime} \lra 0\, . 
\]
It follows, from Remark~\ref{R:monadsandliaison}, that $\sci_Y(8)$ has a 
resolution of the form$\, :$ 
\[
0 \lra \sco_\piii(3) \oplus 2\sco_\piii(2) \lra 3\sco_\piii(4) \oplus 
\sco_\piii(3) \lra \sci_Y(8) \lra 0\, . 
\]
One deduces that $E$ has a resolution of the form$\, :$ 
\[
0 \lra \sco_\piii \oplus 2\sco_\piii(-1) \lra 3\sco_\piii(1) \oplus r\sco_\piii 
\lra E \lra 0\, . 
\]
Dualizing this resolution and taking into account that $\tH^i(E^\vee) = 0$, 
$i = 0,\, 1$, one gets easily that $E$ is as in item (i) from the statement. 

\vskip2mm 

\noindent
{\bf Case 2.}\quad $c_3(\scg) = 2$. 

\vskip2mm 

\noindent
In this case, by \cite[Lemma~2.4]{ch}, $\scg$ can be realized as an 
extension$\, :$ 
\[
0 \lra \sco_\piii(-1) \lra \scg \lra \sci_Z \lra 0 
\]
where $Z$ is either the union of two disjoint lines $L$ and $L^\prime$ or is a 
divisor of the form $2L$ on a nonsingular quadric surface, $L$ being a line  
(actually, $\chi(\scg(1)) = 1$ hence $\tH^0(\scg(1)) \neq 0$ and $\chi(\scg) 
= -1$). In both cases, the resolution of $\sci_Z$ has the same numerical 
shape. In the first case, the tensor product of the resolutions of $\sci_L$ and 
$\sci_{L^\prime}$ is a resolution of $\sci_{L \cup L^\prime}$ (see, for example, 
\cite[Lemma~B.1]{acm2}). Consequently, in both cases, $\sci_Z$ has a resolution 
of the form$\, :$ 
\[
0 \lra \sco_\piii(-4) \lra 4\sco_\piii(-3) \lra 4\sco_\piii(-2) \lra \sci_Z 
\lra 0
\]  
hence $\scg$ has a resolution of the form$\, :$ 
\[
0 \lra \sco_\piii(-4) \lra 4\sco_\piii(-3) \lra \sco_\piii(-1) \oplus 
4\sco_\piii(-2) \lra \scg \lra 0\, . 
\]
The global section of $\scg(2)$ whose zero scheme is $Y^\prime$ cannot be the  
product of a linear form and of the unique non-zero global section of 
$\scg(1)$. One deduces that $\sci_{Y^\prime}$ has a resolution of the form$\, :$ 
\[
0 \lra \sco_\piii(-5) \lra 4\sco_\piii(-4) \lra \sco_\piii(-2) \oplus 
3\sco_\piii(-3) \lra \sci_{Y^\prime} \lra 0\, . 
\] 
One deduces, from Remark~\ref{R:monadsandliaison}, that $\sci_Y(8)$ has a 
monad of the form$\, :$ 
\[
0 \lra 3\sco_\piii(3) \oplus \sco_\piii(2) \lra 6\sco_\piii(4) \lra 
\sco_\piii(5) \lra 0\, . 
\] 
It follows that $E$ has a monad of the form$\, :$ 
\[
0 \lra 3\sco_\piii \oplus \sco_\piii(-1) \overset{d^{-1}}{\lra}  
6\sco_\piii(1) \oplus (r-1)\sco_\piii \overset{d^0}{\lra} \sco_\piii(2) 
\lra 0   
\]
where the component $(r-1)\sco_\piii \ra \sco_\piii(2)$ of $d^0$ is 0. 
Dualizing this monad and taking into account that $\tH^i(E^\vee) = 0$, 
$i = 0,\, 1$, one gets that $E$ has, actually, a monad of the form$\, :$ 
\[
0 \lra \sco_\piii(-1) \overset{d^{\prime -1}}{\lra}  
6\sco_\piii(1) \oplus 4\sco_\piii \overset{d^{\prime 0}}{\lra} \sco_\piii(2) 
\lra 0   
\]
where the component $\sco_\piii(-1) \ra 4\sco_\piii$ of $d^{\prime -1}$ is defined 
by $x_0, \ldots ,x_3$ and where the component $4\sco_\piii \ra \sco_\piii(2)$ 
of $d^{\prime 0}$ is 0. One deduces an exact sequence$\, :$ 
\[
0 \lra \sco_\piii(-1) \lra 2\sco_\piii(1) \oplus \Omega_\piii(2) \oplus 
4\sco_\piii \lra E \lra 0\, . 
\]
Since the multiplication maps $\tH^0(\sco_\piii(1)) \otimes S_1 \ra 
\tH^0(\sco_\piii(2))$ and $\tH^0(\Omega_\piii(2)) \otimes S_1 \ra 
\tH^0(\Omega_\piii(3))$ are surjective, it follows easily that $E$ is as in 
item (ii) from the statement. 

\vskip2mm 

\noindent
{\bf Case 3.}\quad $c_3(\scg) = 0$. 

\vskip2mm 

\noindent 
In this case, by \cite[Prop.~2.3]{hs}, $\scg$ can be realized as an 
extension$\, :$ 
\[
0 \lra \sco_\piii(-2) \lra \scg \lra \sci_Z(1) \lra 0 
\]    
where $Z$ is the union of two disjoint conics $C$ and $C^\prime$. The tensor 
product of the resolutions of $\sci_C$ and $\sci_{C^\prime}$ is a resolution 
of $\sci_Z$ (see, for example, \cite[Lemma~B.1]{acm2}). From this resolution 
one gets a resolution of $\scg$, and from the resolution of $\scg$ one gets a 
resolution of $\sci_{Y^\prime}$. Since the global section of $\scg(2)$ whose 
zero scheme is $Y^\prime$ is not the product of a linear form and of the unique 
non-zero global section of $\scg(1)$ one deduces that $\sci_{Y^\prime}$ has a 
resolution of the same numerical shape as that of $\sci_Z$, that is, a 
resolution of the form$\, :$ 
\[
0 \ra \sco_\piii(-6) \ra 2\sco_\piii(-4) \oplus 2\sco_\piii(-5) \ra 
\sco_\piii(-2) \oplus 2\sco_\piii(-3) \oplus \sco_\piii(-4) \ra \sci_{Y^\prime} 
\ra 0\, . 
\]
One deduces, using Remark~\ref{R:monadsandliaison}, a monad of $\sci_Y(8)$ of 
the form$\, :$ 
\[
0 \lra \sco_\piii(4) \oplus 2\sco_\piii(3) \oplus \sco_\piii(2) \lra 
2\sco_\piii(5) \oplus 4\sco_\piii(4) \lra \sco_\piii(6) 
\lra 0  
\]
hence a monad for $E$ of the form$\, :$ 
\[
\sco_\piii(1) \oplus 2\sco_\piii \oplus \sco_\piii(-1) 
\overset{d^{-1}}{\lra}  
2\sco_\piii(2) \oplus 4\sco_\piii(1) \oplus (r-1)\sco_\piii 
\overset{d^0}{\lra} \sco_\piii(3) 
\]
such that the component $(r-1)\sco_\piii \ra \sco_\piii(3)$ of $d^0$ is 0. 
Since there is no locally split monomorphism $\sco_\piii(1) \ra 2\sco_\piii(2)$ 
it follows that one can cancel the direct summand $\sco_\piii(1)$ from the 
left term of the monad and a direct summand $\sco_\piii(1)$ from the middle 
term, hence $E$ has a monad of the form$\, :$ 
\[
2\sco_\piii \oplus \sco_\piii(-1) \overset{d^{\prime -1}}{\lra}  
2\sco_\piii(2) \oplus 3\sco_\piii(1) \oplus (r-1)\sco_\piii 
\overset{d^{\prime 0}}{\lra} \sco_\piii(3)\, .  
\]  
Dualizing this monad and taking into account that $\tH^i(E^\vee) = 0$, 
$i = 0,\, 1$, one gets that $E$ has, actually, a monad of the form$\, :$ 
\[
\sco_\piii(-1) \xra{d^{\secund -1}} 
2\sco_\piii(2) \oplus 3\sco_\piii(1) \oplus 4\sco_\piii 
\xra{d^{\secund 0}} \sco_\piii(3)
\]
such that $\tH^0(d^{\secund -1 \vee})$ is surjective and the component 
$4\sco_\piii \ra \sco_\piii(3)$ of $d^{\secund 0}$ is 0. 
Since $\tH^0(E(-2)) = 0$, the component $2\sco_\piii(2) 
\ra \sco_\piii(3)$ of $d^{\secund 0}$ is defined by two linearly independent 
linear forms, $h_0$ and $h_1$. Since $\h^0(E(-1)) \geq 2$, it follows that the 
image of $\tH^0(d^{\secund 0}(-1))$ has dimension $\leq 9$. It follows that, up 
to an automorphism of $2\sco_\piii(2) \oplus 3\sco_\piii(1)$, one can assume 
that the last summand $\sco_\piii(1)$ is mapped to 0 by $d^{\secund 0}$. 
Moreover, up to an automorphism of $\sco_\piii(1) \oplus 4\sco_\piii$, one can 
assume that the component $\sco_\piii(-1) \ra \sco_\piii(1)$ of $d^{\secund -1}$ 
is 0. One deduces that $E \simeq \sco_\piii(1) \oplus E_0$, where $E_0$ is the 
cohomology of a monad$\, :$ 
\[
\sco_\piii(-1) \overset{d_1^{-1}}{\lra} 2\sco_\piii(2) \oplus 2\sco_\piii(1) 
\oplus 4\sco_\piii \overset{d_1^0}{\lra} \sco_\piii(3) 
\] 
with the component $4\sco_\piii \ra \sco_\piii(3)$ of $d_1^0$ equal to 0 and with 
$\tH^0(d_1^{-1 \vee})$ surjective. Let $K$ be the kernel of the restriction of 
$d_1^0$ to $2\sco_\piii(2) \oplus 2\sco_\piii(1) \ra \sco_\piii(3)$ (which is an 
epimorphism). One gets an exact sequence$\, :$ 
\[
0 \lra \sco_\piii(-1) \lra K \oplus 4\sco_\piii \lra E_0 \lra 0\, . 
\] 
Then Claim 6.2 from the proof of \cite[Prop.~6.3]{acm1} shows that $E_0$ is 
as in item (iii) from the statement.  
\end{proof}

\begin{lemma}\label{L:h1e(l)c2=12} 
If $E$ is as at the beginning of this section then$\, :$ 
\[ 
{\fam0 h}^1(E(l)) \leq {\fam0 max}(0,\, {\fam0 h}^1(E(l-1)) - 3)\, ,\  \forall 
\, l \geq 0\, .
\] 
\end{lemma}

\begin{proof} 
Let $H \subset \piii$ be an arbitrary plane, of equation $h = 0$. 
By Remark~\ref{R:c1=5onp2}, $E_H$ is 1-regular. 
In particular, $\tH^1(E_H(l)) = 0$, $\forall \, l \geq 0$. One deduces that 
multiplication by $h \colon \tH^1(E(l-1)) \ra \tH^1(E(l))$ is surjective, 
$\forall \, l \geq 0$. Applying the Bilinear Map Lemma 
\cite[Lemma~5.1]{ha} to $\tH^1(E(l))^\vee \otimes \tH^0(\sco_\piii(1)) \ra 
\tH^1(E(l-1))^\vee$ one gets the desired inequality. 
\end{proof}

\begin{prop}\label{P:eprimunstablec2=12} 
If $E$ (as at the beginning of the section) has rank $r \geq 3$ and the 
associated rank $3$ vector bundle $E^\prim$ is not stable then one of the 
following holds$\, :$ 
\begin{enumerate} 
\item[(i)] $c_3 = 8$ and $E$ can be realized as an extension$\, :$ 
\[
0 \lra F^\prim (2) \lra E \lra \sco_\piii(1) \lra 0
\] 
where $F^\prim$ is a $4$-instanton with ${\fam0 h}^0(F^\prim (1)) \leq 1$$\, ;$ 
\item[(ii)] $c_3 = 8$ and $E \simeq \sco_\piii(1) \oplus E_0$, where $E_0$ is 
the kernel of an epimorphism $4\sco_\piii(2) \ra \sco_\piii(4)$ (and 
$E^\prim \simeq \sco_\piii(1) \oplus E_0^\prim$ with $E_0^\prim$ defined by an 
exact sequence $0 \ra \sco_\piii \ra E_0 \ra E_0^\prim \ra 0$)$\, ;$ 
\item[(iii)] $c_3 = 12$ and there is an exact sequence$\, :$ 
\[
0 \lra \sco_\piii(-1) \xra{\left(\begin{smallmatrix} u\\ v 
\end{smallmatrix}\right)} F^\prim (2 ) \oplus 4\sco_\piii \lra E \lra 0
\]
where $F^\prim$ is a $3$-instanton with ${\fam0 h}^0(F^\prim (1)) \leq 1$ and 
with $v$ defined by four linearly independent linear forms. 
\end{enumerate} 
\end{prop} 

\begin{proof} 
Lemma~\ref{L:eprimunstable} implies that either $c_3 = 8$ and $E^\prim$ can be 
realized as an extension$\, :$ 
\[
0 \lra F^\prim(2) \lra E^\prim \lra \sco_\piii(1) \lra 0 
\]
where $F^\prim$ is a stable rank 2 vector bundle with $c_1(F^\prim) = 0$, 
$c_2(F^\prim) = 4$, or $c_3 = 12$ and $E^\prim$ can be realized as an 
extension$\, :$ 
\[
0 \lra F^\prim(2) \lra E^\prim \lra \sci_L(1) \lra 0 
\]
where $F^\prim$ is a stable rank 2 vector bundle with $c_1(F^\prim) = 0$, 
$c_2(F^\prim) = 3$ and $L$ is a line. 

\vskip2mm 

\noindent
$\bullet$\quad In the former case, $F^\prim$ has two possible spectra$\, :$ 
$(0,0,0,0)$ and $(1,0,0,-1)$. If $F^\prim$ has spectrum $(0,0,0,0)$ then it is 
a 4-instanton. It must have $\h^0(F^\prim (1)) \leq 1$ (by the first part of 
the proof of Prop.~\ref{P:h0e(-1)geq2c2=12}, if $\h^0(E(-1)) \geq 2$ then 
$c_3 \in \{16,\, 18,\, 20\}$). Moreover, by Remark~\ref{R:eprimunstable}, 
$r = 3$ (because $\tH^2(F^\prim (-2)) = 0$) hence $E (= E^\prim)$ is as in case 
(i) from the statement. The 4-instantons $F^\prim$ for which there exists an 
extension $0 \ra F^\prim (2) \ra E \ra \sco_\piii(1) \ra 0$ with $E$ globally 
generated are characterized in Remark~\ref{R:eprimunstablec2=12} below. 

\vskip2mm 

If $F^\prim$ has spectrum $(1,0,0,-1)$ then, according to Chang 
\cite[Prop.~1.5]{ch2}, either $F^\prim$ has an unstable plane $H$ of order 1 or 
it can be realized as the cohomology of a selfdual monad$\, :$ 
\[
0 \lra \sco_\piii(-2) \lra 4\sco_\piii \lra \sco_\piii(2) \lra 0\, . 
\] 
The former case cannot, however, occur because, in that case, there exists an 
epimorphism $F^\prim \ra \sci_{Z,H}(-1) \ra 0$ where $Z$ is a 0-dimensional 
subscheme of $H$, of length 5, and this would \emph{contradict} the fact that 
$F^\prim (3)$ must be globally generated (see Remark~\ref{R:eprimunstable2}(i)). 

It thus remains that $F^\prim$ is the cohomology of a monad as above. Let $K$ 
be the kernel of the epimorphism $4\sco_\piii \ra \sco_\piii(2)$ from the 
monad. $K$ admits a (Koszul) resolution of the form$\, :$ 
\[
0 \lra \sco_\piii(-6) \lra 4\sco_\piii(-4) \lra 6\sco_\piii(-2) \lra K \lra 0\, . 
\]     
One deduces that $\tH^1(F^\prim (1)) \simeq \tH^3(\sco_\piii(-5))$ and 
$\tH^1(F^\prim (2)) \simeq \tH^3(\sco_\piii(-4)) \simeq k$. It follows that the 
multiplication map $\tH^1(F^\prim (1)) \otimes_k S_1 \ra \tH^2(F^\prim(2))$ is a 
perfect pairing, that is, if $\xi \in \tH^1(F^\prim (1))$ is annihilated by 
every linear form $h \in S_1$ then $\xi = 0$. Remark~\ref{R:eprimunstable2}(i)  
implies, now, that $E^\prim \simeq \sco_\piii(1) \oplus F^\prim (2)$. Since, 
by Remark~\ref{R:eprimunstable}, $r = 4$ (because $\h^2(F^\prim(-2)) = 1$) and 
since $\h^1(F^{\prim \vee}(-2)) = \h^1(F^\prim (-2)) = 1$ it follows that 
$E \simeq \sco_\piii(1) \oplus K(2)$. 

\vskip2mm 

\noindent 
$\bullet$\quad In the latter case (considered at the beginning of the proof), 
$F^\prim$ has, \emph{a priori}, two possible spectra$\, :$ $(1,0,-1)$ and 
$(0,0,0)$. But if $F^\prim$ would have spectrum $(1,0,-1)$ then, by 
\cite[Lemma~9.15]{ha}, one would have $\h^0(F^\prim (1)) = 2$ and this would 
\emph{contradict} the fact, established in the first part of the proof of 
Prop.~\ref{P:h0e(-1)geq2c2=12}, that if $\h^0(E(-1)) \geq 2$ then $c_3 \in 
\{16,\, 18,\, 20\}$. It remains 
that $F^\prim$ has spectrum $(0,0,0)$ hence it is a 3-instanton. Moreover, it 
must satisfy $\h^0(F^\prim (1)) \leq 1$. One concludes, now, using  
Remark~\ref{R:eprimunstable2}(ii), that $E$ must be as in item (iii) of the 
statement.   
\end{proof} 

\begin{remark}\label{R:eprimunstablec2=12} 
(a) We want to say a few words about the 4-instantons $F^\prim$ 
with $\h^0(F^\prim (1)) \leq 1$ for which there exists an extension as in  
Prop.~\ref{P:eprimunstablec2=12}(i) with the middle term globally generated. 

\vskip2mm 

\noindent 
$\bullet$\quad Since, by Remark~\ref{R:eprimunstable2}, $F^\prim (3)$ must 
be globally generated it follows that $F^\prim$ cannot have any jumping line of 
maximal order 4. 

\vskip2mm 

\noindent
$\bullet$\quad $F^\prim$ can have only finitely many jumping lines of order 3. 
\emph{Indeed}, if $L$ and $L^\prime$ are two such lines then, by 
Lemma~\ref{L:twojumpinglines} from Appendix~\ref{A:instantons}, 
$L$ and $L^\prime$ must be disjoint. Then, by Lemma~\ref{L:fpriml}(b), 
the map $\tH^1(F^\prim (1)) \ra \tH^1(\sco_L(-2)) \oplus 
\tH^1(\sco_{L^\prime}(-2))$ is surjective hence the kernels of the 
maps $\tH^1(F^\prim (1)) \ra \tH^1(\sco_L(-2))$ and $\tH^1(F^\prim (1)) \ra 
\tH^1(\sco_{L^\prime}(-2))$ are distinct. It follows that if $F^\prim$ has a 
1-dimensional family $(L_t)_{t \in T}$ of jumping lines of order 3 then the 
kernels of the maps $\tH^1(F^\prim (1)) \ra \tH^1(\sco_{L_t}(-2))$ cover 
$\tH^1(F^\prim (2))$. Looking, now, at the arguments from 
Remark~\ref{R:eprimunstablec2=11}(ii) one sees that if $F^\prim$ has a 
1-dimensional family of jumping lines of order 3 then none of the extensions 
from Prop.~\ref{P:eprimunstablec2=12}(i) produces a globally generated vector 
bundle.    

\vskip2mm 

\noindent
$\bullet$\quad If $\tH^0(F^\prim(1)) = 0$ then one must have $\tH^1(F^\prim (2)) 
= 0$. \emph{Indeed}, if $F^\prim (2)$ is globally generated then we showed, in 
\cite[Remark~6.4]{acm1}, that $\tH^1(F^\prim (2)) = 0$. If $F^\prim (2)$ is not 
globally generated then the extension from Prop.~\ref{P:eprimunstablec2=12}(i) 
must be non-trivial. Since the map $\tH^0(E) \ra \tH^0(\sco_\piii(1))$ is 
surjective (because $E$ is globally generated), Lemma~\ref{L:h1fprim(2)=0} from 
Appendix~\ref{A:instantons} implies that $\tH^1(F^\prim (2)) = 0$. 

\vskip2mm 

\noindent
$\bullet$\quad Assume, finally, that $\h^0(F^\prim (1)) = 1$. The general such 
bundle can be realized as an extension$\, :$ 
\[
0 \lra \sco_\piii(-1) \lra F^\prim \lra \sci_{L_1 \cup \ldots \cup L_5}(1) 
\lra 0
\] 
where $L_1 , \ldots , L_5$ are mutually disjoint lines. Assume that 
$F^\prim$ has no jumping line of order 4, i.e., that $L_1 \cup \ldots \cup L_5$ 
has no 5-secant. It is classically known that, in this case, 
$\tH^i(\sci_{L_1 \cup \ldots \cup L_5}(3)) = 0$, $i = 0,\, 1$ (the argument 
is recalled, for example, in \cite[Lemma~5]{acm3}).       
It follows that $\tH^1(F^\prim (2)) = 0$ and that the cokernel of the 
evaluation morphism$\, :$ 
\[
\text{ev} : \tH^0(F^\prim (2)) \otimes_k \sco_\piii \lra F^\prim (2)
\]
is isomorphic to $\sci_{L_1 \cup \ldots \cup L_5}(3)$. Consequently, by 
Remark~\ref{R:eprimunstable2}(i), there exists, in this case, an extension as 
in Prop.~\ref{P:eprimunstablec2=12}(i) producing a globally generated vector 
bundle if and only if there exists an epimorphism$\, :$ 
\[
\Omega_\piii(1) \lra \sci_{L_1 \cup \ldots \cup L_5}(3) \lra 0\, . 
\]
One can show that such epimorphisms really exist$\, :$ see 
\cite[Lemma~2]{acm3}.  

\vskip2mm 

(b) Let $F^\prim$ be a 3-instanton with $\h^0(F^\prim (1)) \leq 1$ and let 
$v : \sco_\piii(-1) \ra 4\sco_\piii$ be a morphism defined by four linearly 
independent linear forms. We want to characterize the morphisms $u : 
\sco_\piii(-1) \ra F^\prim (2)$ for which the cokernel of $(u,v) : \sco_\piii(-1) 
\ra F^\prim (2) \oplus 4\sco_\piii$ is globally generated. Of course, this 
happens if and only if the morphism$\, :$ 
\[
(\text{ev}, u) : (\tH^0(F^\prim (2)) \otimes_k \sco_\piii) \oplus \sco_\piii(-1) 
\lra F^\prim (2) 
\]  
is an epimorphism. We use, now, the stratification, due to Gruson and Skiti 
\cite{gs}, of the moduli space of 3-instantons recalled in Remark~\ref{R:gs}.  

\vskip2mm 

(i) If $\tH^0(F^\prim (1)) = 0$ and $F^\prim$ has no jumping line of order 3 
then $F^\prim (2)$ is globally generated. Moreover, by \cite[Prop.~1.1.1]{gs}, 
the multiplication map $\tH^0(F^\prim (2)) \otimes_k S_1 \ra \tH^0(F^\prim (3))$ 
is surjective. It follows that, in this case, $E \simeq \text{T}_\piii(-1) 
\oplus F^\prim (2)$ (see Remark~\ref{R:eprimunstable2}(ii)).  

\vskip2mm 

(ii) If $\tH^0(F^\prim (1)) = 0$ and $F^\prim$ has a jumping line $L$ of order 3 
then the cokernel of the evaluation morphism of $F^\prim (2)$ is $\sco_L(-1)$. 
Consequently, in this case, the morphism $(\text{ev}, u)$ above is an 
epimorphism if and only if the composite morphism $\sco_\piii(-1) 
\overset{u}{\lra} F^\prim (2) \ra \sco_L(-1)$ is an epimorphism. 

\vskip2mm 

(iii) If $\h^0(F^\prim (1)) = 1$ then $F^\prim$ has two jumping lines $L$ and 
$L^\prime$ of order 3 and the cokernel of the evaluation morphism of 
$F^\prim (2)$ is $\sco_{L\cup L^\prime}(-1)$. 
Consequently, in this case, the morphism $(\text{ev}, u)$ above is an 
epimorphism if and only if the composite morphism $\sco_\piii(-1) 
\overset{u}{\lra} F^\prim (2) \ra \sco_{L\cup L^\prime}(-1)$ is an epimorphism.   
\end{remark}  

\begin{prop}\label{P:h1e(-3)=0c2=12} 
Let $E$ and $E^\prim$ be as at the beginning of this section. If $E^\prim$ is 
stable and ${\fam0 H}^1(E(-3)) = 0$ then one of the following holds$\, :$ 
\begin{enumerate} 
\item[(i)] $c_3 = 20$ and $E \simeq 3\sco_\piii(1) \oplus 
2{\fam0 T}_\piii(-1)$$\, ;$ 
\item[(ii)] $c_3 = 18$ and $E \simeq 2\sco_\piii(1) \oplus \Omega_\piii(2) 
\oplus {\fam0 T}_\piii(-1)$$\, ;$ 
\item[(iii)] $c_3 = 16$ and $E \simeq \sco_\piii(1) \oplus 2\Omega_\piii(2)$. 
\end{enumerate} 
\end{prop} 

\begin{proof} 
Let $F = E^\prim (-2)$ be the normalized rank 3 vector bundle associated to 
$E^\prim$ and let $k_F = (k_1,\, k_2,\, k_3,\, k_4)$ be the spectrum of $F$. One 
has $c_1(F) = -1$, $c_2(F) = 4$, $c_3(F) = c_3 - 12$ (see \eqref{E:chernf}) and 
$c_3(F) = -2\sum k_i - 4$ (see Remark~\ref{R:spectrumf}(iii)). Moreover,  
one has, from relation \eqref{E:r}, $r = 3 + \h^2(F(-2))$.  
Since $\tH^1(E(-3)) \simeq \tH^1(F(-1))$, the hypothesis $\tH^1(E(-3)) = 0$ is 
equivalent to $k_1 \leq -1$. Taking into account Lemma~\ref{L:spectrumf}, 
it follows that, under our hypotheses, the only possible spectra are 
$(-1,-1,-2,-2)$, $(-1,-1,-1,-2)$ and $(-1,-1,-1,-1)$.  

\vskip2mm 

\noindent 
{\bf Case 1.}\quad $F$ \emph{has spectrum} $(-1,-1,-2,-2)$. 

\vskip2mm 

\noindent 
In this case, $r = 9$, $c_3(F) = 8$ and $c_3 = 20$. It follows, from 
Riemann-Roch, that $\chi(E(-1)) = \chi(F(1)) = 3$ hence $\h^0(E(-1)) \geq 3$.  
Prop.~\ref{P:h0e(-1)geq2c2=12} implies, now, that $E \simeq 3\sco_\piii(1) 
\oplus 2\text{T}_\piii(-1)$.    

\vskip2mm 

\noindent 
{\bf Case 2.}\quad $F$ \emph{has spectrum} $(-1,-1,-1,-2)$. 

\vskip2mm 

\noindent 
In this case, $r = 8$, $c_3(F) = 6$ and $c_3 = 18$. It follows, from 
Riemann-Roch, that $\chi(E(-1)) = \chi(F(1)) = 2$ hence $\h^0(E(-1)) \geq 2$.  
One deduces, now, from Prop.~\ref{P:h0e(-1)geq2c2=12}, that 
$E \simeq 2\sco_\piii(1) \oplus \Omega_\piii(2) \oplus \text{T}_\piii(-1)$.     

\vskip2mm 

\noindent 
{\bf Case 3.}\quad $F$ \emph{has spectrum} $(-1,-1,-1,-1)$. 

\vskip2mm 

\noindent 
In this case, $r = 7$, $c_3(F) = 4$ and $c_3 = 16$. By 
Lemma~\ref{L:edual(1)}(a), $E^\vee (1)$ is globally generated. The Chern 
classes of $E^\vee (1)$ are $c_1(E^\vee (1)) = 2$, $c_2(E^\vee (1)) = 3$, 
$c_3(E^\vee (1)) = 4$. By the classification \cite{su} of globally generated 
vector bundles with $c_1 = 2$ (see, also, \cite[Prop.~2.2]{acm1}),  
there exist integers $s$ and $t$ such that $E^\vee (1) \simeq t\sco_\piii 
\oplus G$ with $G$ defined by an exact sequence$\, :$ 
\[
0 \lra s\sco_\piii \lra 2\text{T}_\piii(-1) \lra G \lra 0\, .  
\]    
It follows that $E \simeq t\sco_\piii(1) \oplus G^\vee (1)$. Since $E$ has 
rank 7 and $G$ has rank $\leq 6$, one gets that $t \geq 1$. We assert that 
$t = 1$. \emph{Indeed}, if $t \geq 2$ then $E \simeq 2\sco_\piii(1) \oplus K$, 
where $K := (t-2)\sco_\piii(1) \oplus G^\vee(1)$ is a globally generated 
vector bundle with $c_1(K) = 3$, $c_2(K) = 5$, $c_3(K) = 3$. In this case, the 
dual $P(K)$ of the kernel of the evaluation morphism $\tH^0(K) \otimes 
\sco_\piii \ra K$ has Chern classes $c_1(P(K)) = 3$, $c_2(P(K)) = 4$, 
$c_3(P(K)) = 0$. But, according to the classification \cite{am} or \cite{su2} 
of globally generated vector bundles with $c_1 = 3$ (see, also, 
\cite[Remark~2.12]{acm1}), there exists no globally 
generated vector bundle having these Chern classes. It thus remains that 
$t = 1$ hence $E \simeq \sco_\piii(1) \oplus 2\Omega_\piii(2)$.   
\end{proof}

\begin{prop}\label{P:h2e(-3)neq0c2=12} 
Assume that the vector bundle $E$ (as at the beginning of this section) has 
rank $r \geq 3$ and that the rank $3$ vector bundle $E^\prim$ associated to 
$E$ in the exact sequence \eqref{E:oeeprim} is stable. If 
${\fam0 H}^2(E(-3)) \neq 0$ then one of the following holds$\, :$ 
\begin{enumerate} 
\item[(i)] $c_3 = 20$ and $E \simeq 2{\fam0 T}_\piii(-1) \oplus 
3\sco_\piii(1)$$\, ;$ 
\item[(ii)] $c_3 = 18$ and $E \simeq {\fam0 T}_\piii(-1) \oplus \Omega_\piii(2) 
\oplus 2\sco_\piii(1)$$\, ;$ 
\item[(iii)] $c_3 = 16$ and $E \simeq {\fam0 T}_\piii(-1) \oplus E_1$ where, 
up to a linear change of coordinates, $E_1$ is the kernel of the 
epimorphism$\, :$ 
\[
(x_0,\, x_1,\, x_2^2,\, x_2x_3,\, x_3^2) : 2\sco_\piii(2) \oplus 3\sco_\piii(1) 
\lra \sco_\piii(3)\, ; 
\]
\item[(iv)] $c_3 = 16$ and $E \simeq \sco_\piii(1) \oplus E_0$, with $E_0$  
described in the statement of Prop.~\ref{P:h0e(-1)geq2c2=12}(iii)$\, ;$ 
\item[(v)] $c_3 = 14$ and $E \simeq {\fam0 T}_\piii(-1) \oplus E_1$ where, up 
to a linear change of coordinates, $E_1$ is the kernel of the 
epimorphism$\, :$ 
\[
\begin{pmatrix} x_0 & x_1 & x_2 & x_3 & 0\\ 
                0 & x_0 & x_1 & x_2 & x_3 
\end{pmatrix} : 5\sco_\piii(2) \lra 2\sco_\piii(3)\, ; 
\]
\item[(vi)] $c_3 = 14$ and $E$ has a resolution of the form$\, :$ 
\[
0 \lra \sco_\piii(-1) \xra{\left(\begin{smallmatrix} u\\ v 
\end{smallmatrix}\right)} E_1 \oplus 4\sco_\piii \lra E \lra 0 
\]
with $v$ defined by $x_0, \ldots ,x_3$ and $E_1^\vee(1)$ defined by an 
extension$\, :$ 
\[
0 \lra \sco_\piii(-1) \oplus \sco_\piii \lra E_1^\vee(1) \lra \sci_X \lra 0 
\]
where $X$ is either the union of two disjoint lines or its degeneration, 
a double line on a nonsingular quadric surface. 
\end{enumerate}
\end{prop} 

\begin{proof}
Let $F = E^\prim (-2)$ be the normalized rank 3 vector bundle associated to 
$E^\prim$ and let $k_F = (k_1,\, k_2,\, k_3,\, k_4)$ be the spectrum of $F$. 
Since $\tH^2(E(-3)) \simeq \tH^2(F(-1))$ the hypothesis $\tH^2(E(-3)) \neq 0$ 
is equivalent to $k_4 = -2$ (taking into account Lemma~\ref{L:spectrumf}). It 
follows that, under our hypotheses, only the following spectra can occur$\, :$ 
$(-1,-1,-2,-2)$, $(0,-1,-2,-2)$, $(-1,-1,-1,-2)$, $(0,-1,-1,-2)$, 
$(0,0,-1,-2)$ and $(1,0,-1,-2)$.  

Our main tool will be Lemma~\ref{L:h2e(-3)=1} which asserts that if 
$\h^2(E(-3)) = 1$, i.e., if $\h^2(F(-1)) = 1$ (which is equivalent to 
$k_4 = -2$ and $k_3 = -1$) then there exists an exact sequence$\, :$ 
\[
0 \lra \sco_\piii(-1) \xra{\left(\begin{smallmatrix} u\\ v 
\end{smallmatrix}\right)} E_1 \oplus 4\sco_\piii \lra E \lra 0\, , 
\] 
with $v \colon \sco_\piii(-1) \ra 4\sco_\piii$ defined by $4$ linearly 
independent linear forms, where $E_1$ is a vector bundle of rank $r - 3$, 
with $\tH^i(E_1^\vee) = 0$, $i = 0,\, 1$, with Chern classes $c_1(E_1) = 
4$, $c_2(E_1) = 7$, $c_3(E_1) = c_3 - 12$. Moreover, there exists an exact 
sequence$\, :$ 
\[
0 \lra (r-5)\sco_\piii \lra E_1 \lra \scf_1(2) \lra 0 
\]
with $\scf_1$ a stable rank 2 reflexive sheaf with Chern classes $c_1(\scf_1) 
= 0$, $c_2(\scf_1) = 3$, $c_3(\scf_1) = c_3 - 12$.   

\vskip2mm 

We show, firstly, that, although allowed by the general theory, two of the 
above spectra cannot occur in our context. 

\vskip2mm 

\noindent 
$\bullet$\quad If the spectrum of $F$ is $(0,-1,-2,-2)$ then $r = 8$, 
$c_3(F) = 6$ and $c_3 = 18$. By Riemann-Roch, $\chi(F(1)) = 2$ hence 
$\h^0(F(1)) \geq 2$ hence $\h^0(E(-1)) \geq 2$. 
Prop.~\ref{P:h0e(-1)geq2c2=12} implies that $E \simeq 2\sco_\piii(1) \oplus 
\Omega_\piii(2) \oplus \text{T}_\piii(-1)$. But, in this case, $\h^2(F(-1)) = 
\h^2(E(-3)) = 1$ which \emph{contradicts} the fact that $F$ has the above 
spectrum. Consequently, this spectrum cannot occur. 

\vskip2mm 

\noindent 
$\bullet$\quad If $F$ has spectrum $(1,0,-1,-2)$ then $\h^1(E(-4)) = 
\h^1(F(-2)) = 1$ and $\h^1(E(-3)) = \h^1(F(-1)) = 3$. But this 
\emph{contradicts} Remark~\ref{R:muh1e(-4)}. Consequently, this spectrum 
cannot occur, either.  

\vskip2mm 

We split, now, the rest of the proof into several cases according to the 
remaining spectra of $F$. 

\vskip2mm 

\noindent 
{\bf Case 1.}\quad $F$ \emph{has spectrum} $(-1,-1,-2,-2)$. 

\vskip2mm 

\noindent 
In this case, as we saw in the proof of Prop.~\ref{P:h1e(-3)=0c2=12}, $E$ 
is as in item (i) from the statement. 

\vskip2mm 

\noindent 
{\bf Case 2.}\quad $F$ \emph{has spectrum} $(-1,-1,-1,-2)$. 

\vskip2mm 

\noindent 
In this case, as we saw in the proof of Prop.~\ref{P:h1e(-3)=0c2=12}, $E$ 
is as in item (ii) from the statement.

\vskip2mm 

\noindent 
{\bf Case 3.}\quad $F$ \emph{has spectrum} $(0,-1,-1,-2)$. 

\vskip2mm 

\noindent 
In this case, one has $r = 7$, $c_3(F) = 4$ and $c_3 = 16$. Let $E_1$ (resp., 
$\scf_1$) be as in the statement (resp., proof) of Lemma~\ref{L:h2e(-3)=1} 
recalled at the beginning of the proof.  
In our case, $c_1(\scf_1) = 0$, $c_2(\scf_1) = 3$, $c_3(\scf_1) = 4$. 
The only possible spectrum for $\scf_1$ is $(0,-1,-1)$ (see \cite[\S~7]{ha}). 
Applying Riemann-Roch one gets $\h^0(\scf_1(1)) - \h^1(\scf_1(1)) = 1$. 

If $\tH^1(\scf_1(1)) = 0$ then $\scf_1$ is 2-regular hence $E_1$ is 0-regular. 
It follows that, in this subcase, $E \simeq \text{T}_\piii(-1) \oplus E_1$ 
(because the multiplication map $\tH^0(E_1) \otimes S_1 \ra \tH^0(E_1(1))$ is 
surjective). Since $E_1$ has Chern classes $c_1(E_1) = 4$, $c_2(E_1) = 7$, 
$c_3(E_1) = 4$, \cite[Prop.~5.1]{acm1} implies that $E_1$ is as in item (iii) 
of the statement. 

If $\tH^1(\scf_1(1)) \neq 0$ then $\h^0(\scf_1(1)) \geq 2$ hence 
$\h^0(E_1(-1)) \geq 2$ hence $\h^0(E(-1)) \geq 2$. 
Prop.~\ref{P:h0e(-1)geq2c2=12} implies, now, that, in this subcase, $E$ is as 
in item (iv) of the statement. 

\vskip2mm

\noindent
{\bf Case 4.}\quad $F$ \emph{has spectrum} $(0,0,-1,-2)$. 

\vskip2mm 

\noindent
In this case, $r = 6$, $c_3(F) = 2$ and $c_3 = 14$. 
It follows that $E_1$ (from the statement of Lemma~\ref{L:h2e(-3)=1}) 
is, in our case, a rank 3 vector 
bundle with Chern classes $c_1(E_1) = 4$, $c_2(E_1) = 7$, $c_3(E_1) = 2$. 
Put $G := E_1^\vee(1)$. It has Chern classes $c_1(G) = -1$, $c_2(G) = 2$, 
$c_3(G) = 2$. Moreover, $\tH^0(G(-1)) = \tH^0(E_1^\vee) = 0$ and 
$\tH^0(G^\vee(-1)) = \tH^0(E_1(-2)) = 0$ hence $G$ satisfies the hypothesis of 
Lemma~\ref{L:(3;-1,2,2)} from Appendix~\ref{A:(3;-1,2,-)}.  

If $\tH^0(G) = 0$ then, by Cor.~\ref{C:(3;-1,2,2)}(i), $G^\vee$ is 1-regular 
hence $E_1 \simeq G^\vee(1)$ is 0-regular. It follows that $E \simeq 
\text{T}_\piii(-1) \oplus E_1$ and, by \cite[Prop.~5.1]{acm1}, $E_1$ is as in 
item (v) of the statement. 

If $\tH^0(G) \neq 0$ then, by Lemma~\ref{L:(3;-1,2,2)}(ii), one has an exact 
sequence$\, :$ 
\[
0 \lra \sco_\piii \lra G \lra \scg \lra 0 
\]       
where $\scg$ is a stable rank 2 reflexive sheaf with the same Chern classes as 
$G$. \cite[Lemma~2.4]{ch} implies that $\scg$ can be realized as an 
extension$\, :$ 
\[
0 \lra \sco_\piii(-1) \lra \scg \lra \sci_X \lra 0 
\]
where $X$ is the union of two disjoint lines or its degeneration, a double line 
on a nonsigular quadric surface. It follows that $G$ can be realized as an 
extension$\, :$ 
\[
0 \lra \sco_\piii(-1) \oplus \sco_\piii \lra G \lra \sci_X \lra 0\, . 
\]
Since $G = E_1^\vee(1)$, $E$ is as in item (vi) of the statement (recalling,  
from the beginning of the proof, the exact sequence relating $E$ and $E_1$). 
Note that, by Cor.~\ref{C:(3;-1,2,2)}(ii), the cokernel of the evaluation 
morphism $\tH^0(E_1) \otimes \sco_\piii \ra E_1$ is isomorphic to $\sco_L(-1)$ 
for some line $L \subset \piii$ hence $E$ is globally generated if and only if 
the composite morphism $\sco_\piii(-1) \overset{u}{\ra} E_1 \ra \sco_L(-1)$ is 
an epimorphism. Morphisms $u$ satisfying this condition really exist because 
$E_1(1)$ is globally generated (see Cor.~\ref{C:(3;-1,2,2)}(ii)) hence the 
map $\tH^0(E_1(1)) \ra \tH^0(\sco_L)$ is surjective.  
\end{proof} 

\begin{lemma}\label{L:h0f(1)=0c2=12} 
Let $F$ be a stable rank $3$ vector bundle on $\piii$, with $c_1(F) = -1$, 
$c_2(F) = 4$, such that $F(2)$ is globally generated. If $c_3(F) \leq -6$ 
then ${\fam0 H}^0(F(1)) = 0$. 
\end{lemma}

\begin{proof} 
We will show that if $\tH^0(F(1)) \neq 0$ then $c_3(F) \geq -4$. A nonzero 
global section of $F(1)$ defines an exact sequence$\, :$ 
\begin{equation}\label{E:scgfveesciw(1)} 
0 \lra \scg \lra F^\vee \lra \sci_W(1) \lra 0\, , 
\end{equation}
with $W$ a closed subscheme of $\piii$ of dimension $\leq 1$ and with 
$\scg$ a rank 2 reflexive sheaf with $\tH^0(\scg(-1)) = 0$. Using 
Remark~\ref{R:chern}(c) one can compute the Chern classes of $\scg$ and one 
gets $c_1(\scg) = 0$, $c_2(\scg) = 4 - \text{deg}\, W_{\text{CM}}$. Moreover, 
\[
c_3(F) = c_3(\scg) - 4 - 2\text{deg}\, W_{\text{CM}} + 2\chi(W_{\text{CM}})\, . 
\]
Since $c_3(\scg) \geq 0$ (this is true for any rank 2 reflexive sheaf) it 
follows that if $\dim W \leq 0$ (hence $W_{\text{CM}} = \emptyset$) or if 
$\text{deg}\, W_{\text{CM}} = 1$ (i.e., if $W_{\text{CM}}$ is a line) then 
$c_3(F) \geq -4$. 

Now, since $F(2)$ is a globally generated vector bundle with $c_1(F(2)) = 5$ 
and $c_2(F(2)) = 12$, Remark~\ref{R:c1=5onp2} implies that  
$\tH^0(F_H(-1)) = 0$, for every plane $H \subset \piii$, hence 
$\text{deg}(H \cap W)_{\text{CM}} \leq 1$, for every plane $H \subset \piii$. 
In particular, the support of $W_{\text{CM}}$ cannot contain an irreducible 
plane curve of degree $\geq 2$ or two intersecting lines. We shall use, below, 
the classification of multiple structures on lines in $\piii$ from 
B\u{a}nic\u{a} and Forster \cite{bf} (see, for example, also 
\cite[Appendix~A]{acm2}$\, ;$ some results are recalled in 
Remark~\ref{R:multilines} from Appendix~\ref{A:h0e(-2)neq0}).   

\vskip2mm 

\noindent 
{\bf Case 1.}\quad $\text{deg}\, W_{\text{CM}} = 2$. 

\vskip2mm 

\noindent 
In this case, one of the following holds$\, :$ 
\begin{enumerate} 
\item[(i)] $W_{\text{CM}}$ is the union of two disjoint lines$\, ;$ 
\item[(ii)] $W_{\text{CM}}$ is a double structure $X$ on a line $L \subset \piii$ 
such that $\Ker(\sco_X \ra \sco_L) \simeq \sco_L(l)$, with $l \geq 0$. 
\end{enumerate}
In both cases, $\chi(\sco_{W_{\text{CM}}}) \geq 2 = \text{deg}\, W_{\text{CM}}$ 
hence $c_3(\scf) \geq -4$. 

\vskip2mm 

\noindent 
{\bf Case 2.}\quad $\text{deg}\, W_{\text{CM}} = 3$. 

\vskip2mm 

\noindent 
In this case, one of the following holds$\, :$ 
\begin{enumerate} 
\item[(iii)] $W_{\text{CM}}$ is a twisted cubic curve $C \subset \piii$$\, ;$ 
\item[(iv)] $W_{\text{CM}} = X \cup L^\prime$, with $X$ as in (ii) and $L^\prime$ 
another line not intersecting $L$$\, ;$ 
\item[(v)] $W_{\text{CM}}$ is a triple structure $Y$ on a line $L \subset \piii$, 
containing a double structure $X$ on $L$ as in (ii) and such that 
$\Ker(\sco_Y \ra \sco_X) \simeq \sco_L(2l+m)$, with $m \geq 0$. 
\end{enumerate}
In the cases (iv) and (v) one has $\chi(\sco_{W_{\text{CM}}}) \geq 3 = 
\text{deg}\, W_{\text{CM}}$ hence $c_3(F) \geq -4$. 

We show, now, that case (iii) cannot occur. \emph{Indeed}, dualizing 
the exact sequence \eqref{E:scgfveesciw(1)} (with $W_{\text{CM}}$ as in 
case (iii)) one gets an exact sequence$\, :$ 
\[
0 \lra \sco_\piii(-1) \lra F \lra \scg \lra \omega_C(3) \lra 0\, . 
\] 
Since $\scg$ is semistable with Chern classes $c_1(\scg) = 0$ and $c_2(\scg) 
= 1$, it follows, from Chang \cite[Lemma~2.1]{ch}, that either $\scg$ is a 
nullcorrelation bundle or it can be realized as an extension$\, :$ 
\[
0 \lra \sco_\piii \lra \scg \lra \sci_L \lra 0 
\] 
for some line $L \subset \piii$. In the latter case, $\scg$ admits a 
resolution of the form$\, :$ 
\[
0 \lra \sco_\piii(-2) \lra \sco_\piii \oplus 2\sco_\piii(-1) \lra \scg \lra 0\, . 
\]
In both cases, $\scg_C := \scg \otimes_{\sco_\piii} \sco_C$ is a rank 2 
$\sco_C$-module, with $\text{det}\, \scg_C \simeq \sco_C$. Now, $C$ is the 
image of an embedding $\nu \colon \pj \ra \piii$ such that 
$\nu^\ast \sco_\piii(1) \simeq \sco_\pj(3)$. It follows that $\scl := 
\Ker(\scg_C \ra \omega_C(3))$ is an $\sco_C$-module of rank 1 with 
$\text{det}\, \nu^\ast \scl \simeq \sco_\pj(-7)$. One must have 
$\nu^\ast \scl \simeq \sco_\pj(a) \oplus \sct$, where $\sct$ is a torsion 
$\sco_\pj$-module. Since $\text{det}\, \scl \simeq 
\sco_\pj(a + \text{length}\, \sct)$, one deduces that $a \leq -7$ hence  
$\nu^\ast(\scl(2)) \simeq \sco_\pj(a+6) \oplus \sct(6)$ is not globally 
generated. Using the exact sequence$\, :$ 
\[
F_C \lra \scg_C \lra \omega_C(3) \lra 0\, , 
\]
this \emph{contradicts} the fact that $F(2)$ is globally generated. 
Consequently, case (iii) \emph{cannot occur}. 

\vskip2mm 

\noindent 
{\bf Case 3.}\quad $\text{deg}\, W_{\text{CM}} = 4$. 

\vskip2mm 

\noindent 
In this case, $\scg$ is a semistable rank 2 reflexive sheaf with $c_1(\scg) 
= 0$ and $c_2(\scg) = 0$ hence $\scg \simeq 2\sco_\piii$. It follows that 
$W = W_{\text{CM}}$ hence one has an exact sequence$\, :$ 
\[
0 \lra 2\sco_\piii \lra F^\vee \lra \sci_W(1) \lra 0\, , 
\] 
which, by dualization, produces an exact sequence$\, :$ 
\[
0 \lra \sco_\piii(-1) \lra F \lra 2\sco_\piii \lra \omega_W(3) \lra 0\, . 
\] 
Moreover, one of the following holds$\, :$ 
\begin{enumerate} 
\item[(vi)] $W$ is a reduced and irreducible curve which is a complete 
intersection of type $(2,2)$$\, ;$ 
\item[(vii)] $W$ is a reduced and irreducible curve which is a divisor of 
type $(1,3)$ on a nonsingular quadric surface $Q \subset \piii$$\, ;$ 
\item[(viii)] $W = C \cup L$ where $C$ is a twisted cubic curve and $L$ is 
a line$\, ;$ 
\item[(ix)] $W$ is the union of four mutually disjoint lines$\, ;$ 
\item[(x)] $W = X \cup L^\prime \cup L^\secund$ with $X$ as in (ii) and with 
$L^\prime$ and $L^\secund$ lines such that $L$, $L^\prime$ and $L^\secund$ are 
mutually disjoint$\, ;$ 
\item[(xi)] $W = X \cup X^\prim$, with $X$ as in (ii) and with $X^\prim$ a 
double structure on a line $L^\prime$, not intersecting $L$, such that 
$\Ker(\sco_{X^\prim} \ra \sco_{L^\prime}) \simeq \sco_{L^\prime}(l^\prim)$ with 
$l^\prim \geq 0$$\, ;$ 
\item[(xii)] $W = Y \cup L^\prime$, with $Y$ as in (v) and with $L^\prime$ a 
line not intersecting $L$$\, ;$ 
\item[(xiii)] $W$ is a quadruple structure on a line $L \subset \piii$, 
containing a triple structure $Y$ as in (v) and such that 
$\Ker(\sco_W \ra \sco_Y) \simeq \sco_L(3l+m+n)$ with $n \geq 0$.    
\end{enumerate}
In the cases (ix)--(xiii), $\chi(\sco_W) \geq 4 = \text{deg}\, W$ hence 
$c_3(F) \geq -4$. We will show, now, that the cases (vi)--(viii) cannot occur. 

\emph{Indeed}, in case (vi), $\omega_W \simeq \sco_W$ hence $\Ker(2\sco_W \ra 
\omega_W(3)) \simeq \sco_W(-3)$. Taking into account the exact sequence$\, :$ 
\[
F_W \lra 2\sco_W \lra \omega_W(3) \lra 0\, , 
\]
this \emph{contradicts} the fact that $F(2)$ is globally generated. 

In case (vii), one deduces, using the exact sequence$\, :$ 
\[
0 \lra \sco_Q(-1,-3) \lra \sco_Q \lra \sco_W \lra 0\, , 
\]
that $\omega_W \simeq \sco_Q(-1,1) \vb W$ hence $\Ker(2\sco_W \ra \omega_W(3)) 
\simeq \sco_Q(-2,-4) \vb W$. Using the fact that $\tH^0(\sco_Q(0,-2) \vb W) = 
0$, this \emph{contradicts}, as above, the fact that $F(2)$ is globally 
generated. 

In case (viii), the proof of Lemma~\ref{L:omegaxverty} from 
Appendix~\ref{A:(2;-1,4,0)} shows that $\omega_C$ embeds into 
$\omega_W \vb C$. $C$ is the image of an embedding $\nu \colon \pj \ra \piii$ 
such that $\nu^\ast \sco_\piii(1) \simeq \sco_\pj(3)$. One deduces that 
$\nu^\ast(\omega_W(3) \vb C) \simeq \sco_\pj(b) \oplus \sct$, with $b \geq 7$ 
and with $\sct$ a torsion sheaf on $\pj$. If $\scl$ is the kernel of 
$2\sco_C \ra \omega_W(3) \vb C$ it follows that $\nu^\ast \scl \simeq 
\sco_\pj(c)$ with $c \leq -7$. Taking into account the exact sequence$\, :$ 
\[
F_C \lra 2\sco_C \lra \omega_W(3) \vb C \lra 0\, ,
\]  
this \emph{contradicts} the fact that $F(2)$ is globally generated. 
\end{proof}

\begin{lemma}\label{L:h0f(1)=1c2=12c3=8} 
Let $F$ be a stable rank $3$ vector bundle on $\piii$, with $c_1(F) = -1$, 
$c_2(F) = 4$, $c_3(F) = -4$ and spectrum $(0,0,0,0)$ such that $F(2)$ is 
globally generated. If ${\fam0 H}^0(F(1)) \neq 0$ then there exists an 
integer $m$ with $0 \leq m \leq 4$ and an exact sequence$\, :$ 
\[
0 \lra M \lra F^\vee \lra \sci_W(1) \lra 0\, , 
\]
with $M$ an $m$-instanton and $W$ a union of multiple structures on mutually 
disjoint lines with ${\fam0 deg}\, W = 4-m$ and such that $\omega_W \simeq 
\sco_W(-2)$. Moreover, if $1 \leq m \leq 4$ then ${\fam0 h}^0(F^\vee) \leq 1$.  
\end{lemma}

\begin{proof} 
This is actually a corollary of the proof of Lemma~\ref{L:h0f(1)=0c2=12}. 
\emph{Indeed}, looking at the proof of that lemma, one sees that $F^\vee$ 
can be realized as an extension$\, :$ 
\[
0 \lra \scg \lra F^\vee \lra \sci_W(1) \lra 0\, , 
\]
with $c_3(\scg) = 0$ hence with $\scg$ locally free which implies that 
$W = W_{\text{CM}}$ and with $W = \emptyset$ or $W$ a line or $W$ as in the 
items (i), (ii), (iv), (v), (ix)--(xiii) from the lists in that proof. 
Moreover, if some multiple structures appear in $W$ then they must 
have $l = 0$, $m = 0$, and $n = 0$. This implies that $W$ satisfies the 
conditions from the conclusion of the lemma (see B\u{a}nic\u{a} and Forster 
\cite[Prop.~2.3~and~\S 3.2]{bf}). Let $m := 4 - \text{deg}\, W$. Then 
$M := \scg$ is a semistable rank 2 vector bundle with $c_1(M) = 0$ and 
$c_2(M) = m$. Since the spectrum of $F$ is $(0,0,0,0)$ it follows that 
$\tH^1(F^\vee(-2)) \simeq \tH^2(F(-2))^\vee = 0$ hence $\tH^1(M(-2)) = 0$ 
which means that $M$ is an $m$-instanton. 

Finally, let us prove the last assertion from the statement. If $m = 1,\, 2$ 
then one has $\tH^0(\sci_W(1)) = 0$ hence $\tH^0(F^\vee) = 0$. 

If $m = 3$ and $\h^0(F^\vee) = 2$ then one has an exact sequence$\, :$ 
\[
0 \lra M \lra F^\vee \lra \sci_L(1) \lra 0\, , 
\]   
where $M$ is a 3-instanton, $L$ is a line and the map $\tH^0(F^\vee) \ra 
\tH^0(\sci_L(1))$ is bijective. Applying the Snake Lemma to the diagram$\, :$ 
\[
\begin{CD} 
0 @>>> \sco_\piii(-1) @>>> 2\sco_\piii @>>> \sci_L(1) @>>> 0\\ 
@. @VVV @VVV @\vert\\ 
0 @>>> M @>>> F^\vee @>>> \sci_L(1) @>>> 0 
\end{CD}
\]
one gets an exact sequence$\, :$ 
\[
0 \lra \sco_\piii(-1) \lra M \oplus 2\sco_\piii \lra F^\vee \lra 0\, ,  
\] 
which, by dualization, produces an exact sequence$\, :$ 
\[
0 \lra F \lra M \oplus 2\sco_\piii \lra \sco_\piii(1) \lra 0\, . 
\]
Since $\tH^1(F(2)) = 0$ (see Claim 5.1 in the proof of 
Prop.~\ref{P:h2e(-3)=0h1e(-3)neq0c2=12}) it follows that $M(2)$ is globally 
generated. But in this case $\tH^0(M(1)) = 0$ (if $\tH^0(M(1)) \neq 0$ then 
$M$ has a jumping line of order 3) hence $\text{Hom}(M,\sco_\piii(1)) = 0$ 
and this \emph{contradicts} the existence of an epimorphism $M \oplus 
2\sco_\piii \ra \sco_\piii(1)$. This contradiction shows that one cannot have 
$\h^0(F^\vee) = 2$. 

If $m = 4$ then one has an exact sequence$\, :$ 
\[
0 \lra M \lra F^\vee \lra \sco_\piii(1) \lra 0\, , 
\]
where $M$ is a 4-instanton. Since $F$ is stable, the extension is 
nontrivial, hence it is defined by a non-zero element $\xi \in \tH^1(M(-1))$. 
If $\h^0(F^\vee) \geq 2$ then $\xi$ is annihilated, in the graded 
$S$-module $\tH^1_\ast(M)$ by two linearly independent linear forms 
$h_0,\, h_1 \in S_1$. Let $L \subset \piii$ be the line of equations 
$h_0 = h_1 = 0$. Tensorizing by $M$ the exact sequence$\, :$ 
\[
0 \lra \sco_\piii(-1) \lra 2\sco_\piii \lra \sci_L(1) \lra 0\, , 
\]  
one deduces that $\tH^0(M \otimes \sci_L(1)) \neq 0$. But one gets, using 
the exact sequence$\, :$ 
\[
0 \lra \sco_\piii(-1) \lra F \lra M \lra 0\, , 
\]
that $M(2)$ is globally generated. This implies that $\tH^0(M(1)) = 0$ and 
this \emph{contradiction} shows that one cannot have $\h^0(F^\vee) = 2$. 
\end{proof}

\begin{remark}\label{R:h0fvee=2c2=12c3=8} 
It seems a difficult question to decide whether there exist vector bundles 
$F$ on $\piii$ satisfying the hypotheses of Lemma~\ref{L:h0f(1)=1c2=12c3=8} 
or not. In particular, we are not able to answer the following question$\, :$ 
let $L_1, \ldots , L_4$ be mutually disjoint lines in $\piii$, not contained 
in a quadric surface, let $W$ denote their union, and consider a general  
extension$\, :$ 
\[
0 \lra 2\sco_\piii \lra G \lra \sci_W(1) \lra 0\, , 
\]
with $G$ locally free. Is, then, $G^\vee(2)$ globally generated ? 

Dualizing the above extension, one gets an exact sequence$\, :$ 
\[
0 \lra \sco_\piii(-1) \lra G^\vee \lra 2\sco_\piii \overset{\delta}{\lra} 
\omega_W(3) \lra 0\, , 
\] 
and this operation establishes a bijection between this kind of extensions 
and the set of epimorphisms $\delta \colon 2\sco_\piii \ra \omega_W(3) \simeq 
\sco_W(1)$. So, we are led to ask$\, :$ 

\vskip2mm 

\noindent 
{\bf Question 1.}\quad \emph{Let} $W$ \emph{be the union of four mutually 
disjoint lines, not contained in a quadric surface. If} $\sck$ \emph{is the 
kernel of a general epimorphism} $\delta \colon 2\sco_\piii \ra \sco_W(1)$ 
\emph{is} $\sck(2)$ \emph{globally generated} ? 

\vskip2mm 

\noindent 
If $W$ is the union of only three lines then the answer is yes$\, :$ see 
Lemma~\ref{L:2oraoy(1)} from Appendix~\ref{A:miscellaneous}. 
One way to attack this question is to consider an elliptic quartic curve 
$C \subset \piii$ (that is, a complete intersection of type $(2,2)$ in 
$\piii$), four general points $P_1, \ldots , P_4$ of $C$ and, for each 
$i \in \{1, \ldots 4\}$, a general line $L_i$ passing through $P_i$. 
Putting $W := L_1 \cup \ldots \cup L_4$, one has an exact sequence$\, :$ 
\[
0 \lra \sci_{C \cup W} \lra \sci_C \lra \sci_{C \cap W , W} \lra 0\, , 
\] 
and $\sci_{C \cap W , W} \simeq \sco_W(-1)$ hence the composite morphism$\, :$ 
\[
\tH^0(\sci_C(2))\otimes_k\sco_\piii \overset{\text{ev}}{\lra} \sci_C(2) 
\lra \sci_{C \cap W , W}(2) 
\]
is an epimorphism $\delta \colon 2\sco_\piii \ra \sco_W(1)$. If $\sck$ is 
its kernel then one has an exact sequence$\, :$ 
\[
0 \lra \sco_\piii(-2) \lra \sck \lra \sci_{C \cup W}(2) \lra 0\, . 
\]
Consequently, an affirmative answer to the first question is a consequence 
of an affirmative answer to$\, :$ 

\vskip2mm 

\noindent 
{\bf Question 2.}\quad \emph{Under the above hypotheses, is} 
$\sci_{C \cup W}(4)$ \emph{globally generated} ? 

\vskip2mm 

\noindent 
Notice that, since $\tH^1(\sco_{C \cup W}(1)) = 0$, $C \cup W$ is smoothable 
in $\piii$ (see Hartshorne and Hirschowitz \cite[Cor.~1.2]{hh2} or 
Hartshorne \cite[Prop.~29.9]{hadt}) hence a positive answer to the second 
question would also prove the existence of elliptic curves $Y$ of degree 8 
in $\piii$ with $\sci_Y(4)$ globally generated, a result proven, by a quite 
different method, by Chiodera and Ellia \cite{ce}.   
\end{remark}

\begin{remark}\label{R:p(e)} 
Let $E$ be a globally generated vector bundle on $\piii$ of rank $r \geq 3$, 
with $c_1 = 5$, $c_2 = 12$, such that $\tH^i(E^\vee) = 0$, $i = 0,\, 1$, and 
$\tH^0(E(-2)) = 0$. As usual, $r-3$ general global sections of $E$ define an 
exact sequence$\, :$ 
\[
0 \lra (r - 3)\sco_\piii \lra E \lra E^\prim \lra 0\, , 
\] 
where $E^\prim$ is a rank 3 vector bundle.  Consider the normalized vector 
bundle $F := E^\prim(-2)$. It has Chern classes $c_1(F) = -1$, $c_2(F) = 4$, 
$c_3(F) = c_3 - 12$. 

Let, now, $P(E)$ be the dual of the kernel of the evaluation morphism 
$\tH^0(E) \otimes_k \sco_\piii \ra E$ of $E$. $P(E)$ is a globally generated 
vector bundle with Chern classes $c_1^\prim = 5$, $c_2^\prim = 13$, 
$c_3^\prim = c_3 + 5$. Since $c_3^\prim > 0$, $P(E)$ has rank $r^\prim \geq 3$.  
Then $r^\prim - 3$ general global sections of $P(E)$ define an exact 
sequence$\, :$ 
\[
0 \lra (r^\prim - 3)\sco_\piii \lra P(E) \lra F^\prim(2) \lra 0\, , 
\]
with $F^\prim$ a rank 3 vector bundle. One has $c_1(F^\prim) = -1$, 
$c_2(F^\prim) = 5$, $c_3(F^\prim) = c_3^\prim - 14 = c_3(F) + 3$ (see 
\eqref{E:chernf}). We want to relate the cohomological properties of $E$ 
and $P(E)$. 

Firstly, using the exact sequence 
\[
0 \lra E^\vee \lra \tH^0(E)^\vee \otimes_k \sco_\piii \lra P(E) \lra 0 
\]
and Serre duality, one gets that 
$\tH^1_\ast(F^\prim) \simeq \tH^1_\ast(F)^\vee$, where $\tH^1_\ast(F)^\vee$ denotes 
the graded $k$-vector space $\bigoplus_{l \in \z}\tH^1(F(-l))^\vee$ endowed with 
the obvious $S$-module structure. Moreover, 
\[
\h^2(F^\prim(-1)) = \h^2(P(E)(-3)) = \h^0(E(-1))\, ,\  \h^0(P(E)(-1)) = 
\h^2(E(-3)) = \h^2(F(-1))\, ,
\]
and $\tH^2(F^\prim(l)) = 0$, $\forall \, l \geq 0$. 
Since $r + r^\prim = \h^0(E)$, one deduces, from relation \eqref{E:r}, 
that$\, :$ 
\[
\h^2(F^\prim(-2)) + \h^2(F(-2)) = \h^0(E) - 6\, . 
\]
Finally, if $F$ and $F^\prim$ are stable, with spectrum 
$k_F = (k_i)_{1 \leq i \leq 4}$ and $k_{F^\prim} = (k_j^\prim)_{1 \leq j \leq 5}$, 
respectively, then, using Remark~\ref{R:spectrumf}(iii), one gets that$\, :$ 
\[
{\textstyle \sum}_{j=1}^5 k_j^\prim = {\textstyle \sum}_{i=1}^4 k_i - 2\, .  
\] 
\end{remark}

The following two observations will be used occasionally. 

\begin{lemma}\label{L:suppcok} 
Let $\sce \ra G \ra \scc \ra 0$ be an exact sequence of coherent sheaves on 
$\piii$. Assume that $G$ is locally free, with $c_1(G) = m \geq 0$ and that 
$\sce$ is globally generated. If $\scc(-m-1)$ is globally generated then the 
support of $\scc$ is $0$-dimensional or empty. 
\end{lemma} 

\begin{proof} 
Assume, by contradiction, that $\text{Supp}\, \scc$ contains a reduced and 
irreducible curve $Z$, of degree $d$. Let $\nu \colon C \ra Z$ be the 
normalization of $Z$. Recall that if $\scf$ is a coherent sheaf on $C$ one 
denotes by $c_1(\scf)$ the degree of $\text{det}\, \scf$. Now, $\nu^\ast\scc$ 
is a coherent sheaf of rank $\geq 1$ on $C$. Since $\scc(-m-1)$ is globally 
generated it follows that $c_1(\nu^\ast\scc) \geq d(m+1)$. Using the exact 
sequence $\nu^\ast\sce \ra \nu^\ast G \ra \nu^\ast\scc \ra 0$ and the fact that 
$\sce$ is globally generated one deduces that $c_1(\nu^\ast G) \geq d(m+1)$. 
Since $c_1(\nu^\ast G) = dm$ we have got a \emph{contradiction}.    
\end{proof}

\begin{lemma}\label{L:lengthcok}  
Consider, on $\piii$, an exact sequence $\sce \ra G \ra \scc \ra 0$ of 
coherent sheaves, with $G$ locally free of rank $r$, with Chern classes 
$c_1(G) = c_1$, $c_2(G) = c_2$, $c_3(G) = c_3$. If $\sce$ is globally generated 
and the support of $\scc$ is $0$-dimensional (or empty) then the length of 
$\scc$ is at most $c_3 + c_1(c_1^2 - 2c_2)$.   
\end{lemma}

\begin{proof} 
One also has an exact sequence $m\sco_\piii \ra G \ra \scc \ra 0$, for some 
integer $m \geq r + 2$. The kernel $\scf$ of $m\sco_\piii \ra G$ is a reflexive 
sheaf of rank $m - r$. Dualizing the exact sequence$\, :$ 
\[
0 \lra \scf \lra m\sco_\piii \lra G \lra \scc \lra 0\, , 
\]   
one gets an exact sequence$\, :$ 
\[
0 \lra G^\vee \lra m\sco_\piii \lra \scf^\vee \lra 0\, . 
\]
Since $\scf^\vee$ is globally generated, the 
dependency locus of $m - r - 2$ general global sections of $\scf^\vee$ 
(from the image of $\tH^0(m\sco_\piii) \ra \tH^0(\scf^\vee)$) is 
0-dimensional. They define an exact sequence$\, :$ 
\[
0 \lra (m-r-2)\sco_\piii \lra \scf^\vee \lra \scf^\prim \lra 0\, , 
\]
with $\scf^\prim$ a rank 2 reflexive sheaf. One deduces an exact sequence$\, :$ 
\[
0 \lra G^\vee \lra (r+2)\sco_\piii \lra \scf^\prim \lra 0\, . 
\]
Dualizing it, one deduces that $\scc$ is a quotient of 
$\sce xt^1_{\sco_\piii}(\scf^\prim , \sco_\piii)$. But, by \cite[Prop.~2.6]{ha}, 
the length of this Ext sheaf is equal to $c_3(\scf^\prim)$ which can be 
computed using the last exact sequence.   
\end{proof}

\begin{prop}\label{P:h2e(-3)=0h1e(-3)neq0c2=12} 
Let $E$ and $E^\prim$ be as at the beginning of this section. Assume that 
$E^\prim$ is stable, that ${\fam0 H}^2(E(-3)) = 0$, and that 
${\fam0 H}^1(E(-3)) \neq 0$. Then one of the following holds$\, :$  
\begin{enumerate} 
\item[(i)] $c_3 = 14$ and $E(-2)$ is the kernel of a general epimorphism$\, :$ 
\[
\phi \, \colon 2\sco_\piii \oplus 6\sco_\piii(-1) \lra \sco_\piii(1) \oplus 
\sco_\piii \, . 
\] 
In the typical case, $E(-1)$ is the kernel of the evaluation morphism 
$6\sco_\piii \ra \sco_H(2)$ of $\sco_H(2)$, for some plane $H \subset \piii$. 
Other two special cases are described during the proof$\, ;$ 
\item[(ii)] $c_3 = 12$ and $E(-2)$ is the cohomology of a (not necessarily 
minimal) monad of the form$\, :$  
\[
0 \lra \sco_\piii(-1) \lra 4\sco_\piii \oplus 4\sco_\piii(-1) \lra 
2\sco_\piii(1) \lra 0\, ; 
\]
\item[(iii)] $c_3 = 10$ and $E(-2)$ is the cohomology of a (not necessarily 
minimal) monad of the form$\, :$ 
\[
0 \lra 2\sco_\piii(-1) \lra 7\sco_\piii \oplus 2\sco_\piii(-1) \lra 
3\sco_\piii(1) \lra 0\, ; 
\]
\item[(iv)] $c_3 = 8$ and $E \simeq \sco_\piii(1) \oplus E_0$, where 
$E_0(-2)$ is the kernel of an arbitrary epimorphism $4\sco_\piii \ra 
\sco_\piii(2)$$\, ;$ 
\item[(v)] $c_3 = 8$ and $E(-2)$ is the cohomology of a monad of the 
form$\, :$ 
\[
0 \lra 3\sco_\piii(-1) \lra 10\sco_\piii \lra 4\sco_\piii(1) \lra 0\, ; 
\] 
\item[(vi)] $c_3 = 6$ and $E(-2)$ is the cohomology of a monad of the 
form$\, :$ 
\[
0 \lra 2\sco_\piii(-1) \lra 7\sco_\piii \lra \sco_\piii(2) \oplus \sco_\piii(1) 
\lra 0\, .  
\] 
\end{enumerate}
\end{prop} 

\begin{proof} 
Let $F = E^\prim (-2)$ be the normalized rank 3 vector bundle associated to 
$E^\prim$ and let $k_F = (k_1,\, k_2,\, k_3,\, k_4)$ be the spectrum of $F$. The 
condition $\tH^2(E(-3)) = 0$ is equivalent to $k_4 \geq -1$ and the condition 
$\tH^1(E(-3)) \neq 0$ is equivalent $k_1 \geq 0$. Taking into account 
Lemma~\ref{L:spectrumf}, the possible spectra of $F$ are $(0,-1,-1,-1)$, 
$(0,0,-1,-1)$, $(1,0,-1,-1)$, $(0,0,0,-1)$, $(1,0,0,-1)$, $(0,0,0,0)$, 
$(1,0,0,0)$, $(1,1,0,-1)$, $(1,1,0,0)$ and $(1,1,1,0)$. Using 
Remark~\ref{R:muh1e(-4)}, one can easily show that the spectrum 
$(1,0,-1,-1)$ cannot occur in our context.   

We analyse, now, case by case, the remaining spectra. Recall the formulae from 
the beginning of the proof of Prop.~\ref{P:h1e(-3)=0c2=12}. 

\vskip2mm 

\noindent 
{\bf Case 1.}\quad $F$ \emph{has spectrum} $(0,-1,-1,-1)$. 

\vskip2mm 

\noindent 
In this case, $r = 6$, $c_3(F) = 2$ and $c_3 = 14$. Using the spectrum, one 
gets that $\h^1(E(l)) = \h^1(F(l+2)) = 0$ for $l \leq -4$ and $\h^1(E(-3)) = 
1$. In particular, $\h^1(E(-3)) - \h^1(E(-4)) = 1$. Lemma~\ref{L:h2e(-3)=0}(e) 
implies that $\tH^2_\ast(E) = 0$. Using Riemann-Roch, $\h^1(E(-2)) = \h^1(F) = 
3$. Since $\tH^2(E(-3)) = 0$ and $\tH^3(E(-4)) \simeq \tH^0(E^\vee)^\vee = 0$, 
the graded $S$-module $\tH^1_\ast(E)$ is generated in degrees $\leq -2$ (see 
Remark~\ref{R:h2e(-3)=0}(i)).  

We assert that the multiplication map $\mu \colon \tH^1(E(-3)) \otimes S_1 
\ra \tH^1(E(-2))$ has rank $\geq 2$. \emph{Indeed}, 
Prop.~\ref{P:h0e(-1)geq2c2=12} implies that $\h^0(E(-1)) \leq 1$. If $\mu$ 
has rank $\leq 1$ then, by Remark~\ref{R:beilinson}, there exists an exact 
sequence$\, :$ 
\[
0 \lra \Omega_\piii^2(2) \lra \Omega_\piii^1(1) \oplus \sco_\piii \lra Q 
\lra 0\, , 
\] 
with $Q$ locally free. Since, in this case, $Q$ would have rank 1, such an 
exact sequence \emph{cannot exist}. 

Consequently, $\tH^1_\ast(E)$ has a minimal generator in degree $-3$ and at 
most one generator in degree $-2$. Consider the extension defined by these 
generators$\, :$ 
\[
0 \lra E(-2) \lra B \lra \sco_\piii(1) \oplus \sco_\piii \lra 0\, . 
\]
One has $\tH^1_\ast(B) = 0$ and $\tH^2_\ast(B) \simeq \tH^2_\ast(E) = 0$ hence 
$B$ is a direct sum of line bundles. $B$ has rank 8, $\tH^0(B(-1)) = 0$, 
$\h^0(B) = \h^0(\sco_\piii(1) \oplus \sco_\piii) - \h^1(E(-2)) = 2$, and 
$\tH^0(B^\vee(-2)) = 0$. It follows that $B \simeq 2\sco_\piii \oplus 
6\sco_\piii(-1)$.

\vskip2mm 

\noindent 
{\bf Description of epimorphisms} $\phi \colon 2\sco_\piii \oplus 6\sco_\piii(-1)
\ra \sco_\piii(1) \oplus \sco_\piii$ {\bf for which} $E := \Ker \phi(2)$  
{\bf is globally generated and} $\tH^0(E(-2)) = 0${\bf .}\quad 
(i) If the component $2\sco_\piii \ra \sco_\piii$ 
of $\phi$ is non-zero then $E \simeq \Ker \psi(2)$ for some epimorphism 
$\psi \colon \sco_\piii \oplus 6\sco_\piii(-1) \ra \sco_\piii(1)$. Since 
$\tH^0(E(-2)) = 0$ it follows that the component $\sco_\piii \ra \sco_\piii(1)$ 
of $\psi$ is non-zero hence there exists a plane $H \subset \piii$ and an 
epimorphism $\e \colon 6\sco_\piii(-1) \ra \sco_H(1)$ such that $E \simeq 
\Ker \e(2)$.  

Now, since $E$ globally generated implies $\h^0(E(-1)) \leq 1$ (as we saw at 
the beginning of Case 1), the map $\tH^0(\e(1)) \colon \tH^0(6\sco_\piii) 
\ra \tH^0(\sco_H(2))$ must have rank at least 5. 

If $\tH^0(\e(1))$ has rank 6 then $\e(1)$ can be identified with the 
evaluation morphism of $\sco_H(2)$ (as an $\sco_\piii$-module). In this case, 
$E$ is even 0-regular. 

If $\tH^0(\e(1))$ has rank 5 then $E \simeq \sco_\piii(1) \oplus E^\prim$, 
where $E^\prim(-1)$ is the kernel of an epimorphism $\e^\prim \colon 5\sco_\piii 
\ra \sco_H(2)$ with $\tH^0(\e^\prim)$ injective. The epimorphisms $\e^\prim$ 
with $\tH^0(\e^\prim)$ injective and such that $\Ker \e^\prim(1)$ is globally 
generated are described in Case 5 of the proof of \cite[Prop.~6.3]{acm1}.

\vskip2mm 

(ii) If the component $2\sco_\piii \ra \sco_\piii$ of $\phi$ is 0 then the 
condition $\tH^0(E(-2)) = 0$ implies that the component $2\sco_\piii \ra 
\sco_\piii(1)$ of $\phi$ must be defined by two linearly independent linear 
forms. Let $L \subset \piii$ be the line defined by these linear forms. 
Then one has an exact sequence$\, :$ 
\[
0 \lra \sco_\piii(-1) \lra E(-2) \lra 6\sco_\piii(-1) 
\overset{{\overline \phi}_2}{\lra} \sco_\piii \oplus \sco_L(1) \lra 0\, ,  
\]  
with ${\overline \phi}_2$ induced by the component $\phi_2 \colon 
6\sco_\piii(-1) \ra \sco_\piii(1) \oplus \sco_\piii$ of $\phi$. 
Let $\sck$ be the kernel of ${\overline \phi}_2$. Using the commutative 
diagram$\, :$ 
\[
\begin{CD}
0 @>>> \sck @>>> 6\sco_\piii(-1) @>>> \sco_\piii \oplus \sco_L(1) @>>> 0\\ 
@. @VVV @\vert @VVV\\ 
0 @>>> 2\sco_\piii(-1) \oplus \Omega_\piii @>>> 6\sco_\piii(-1) @>>> \sco_\piii  
@>>> 0 
\end{CD}
\]
one gets an exact sequence$\, :$ 
\[
0 \lra \sck \lra 2\sco_\piii(-1) \oplus \Omega_\piii \lra \sco_L(1) \lra 0\, , 
\]
whence an exact sequence$\, :$ 
\[
0 \lra \sco_\piii(-1) \lra E(-2) \lra 2\sco_\piii(-1) \oplus \Omega_\piii \lra 
\sco_L(1) \lra 0\, . 
\]
Dualizing the last exact sequence and, then, twisting by $-1$, one gets that 
$E^\vee (1)$ can be realized as an extension$\, :$ 
\[
0 \lra 2\sco_\piii \oplus \text{T}_\piii(-1) \lra E^\vee (1) \lra \sci_L 
\lra 0\, . 
\]
According to Serre's method of extensions (see Thm.~\ref{T:serreext} in 
Appendix~\ref{A:serre}) there exist such locally free extensions 
with a prescribed connecting epimorphism$\, :$ 
\[
\delta \colon 2\sco_\piii \oplus \Omega_\piii(1) \lra 
\sce xt_{\sco_\piii}^1(\sci_L , \sco_\piii) \simeq \sco_L(2)\, . 
\] 
Taking into account the exact sequence$\, :$ 
\[
0 \lra \sco_\piii \lra E(-1) \lra 2\sco_\piii \oplus \Omega_\piii(1) 
\overset{\delta}{\lra} \sco_L(2) \lra 0\, , 
\]
it follows that $E$ is globally generated if and only if $\Ker \delta (1)$ 
is globally generated. According to Lemma~\ref{L:2oomega(1)raol(2)}  
in Appendix~\ref{A:miscellaneous}, such epimorphisms $\delta$ really exist. 

\vskip2mm 

\noindent 
{\bf Case 2.}\quad $F$ \emph{has spectrum} $(0,0,-1,-1)$. 

\vskip2mm 

\noindent 
In this case, $r = 5$, $c_3(F) = 0$ and $c_3 = 12$. Using the spectrum, one 
gets that $\h^1(E(l)) = \h^1(F(l+2)) = 0$ for $l \leq -4$ and $\h^1(E(-3)) = 
2$. In particular, $\h^1(E(-3)) - \h^1(E(-4)) = 2$ and $\h^2(E^\vee) = 
\h^1(E(-4)) = 0$. Using Riemann-Roch, $\h^1(E(-2)) = \h^1(F) = 4$ and 
$\h^0(E(-1)) - \h^1(E(-1)) = \chi(F(1)) = -1$. Since, by 
Prop.~\ref{P:h0e(-1)geq2c2=12}, $\h^0(E(-1)) \leq 1$ it follows that 
$\h^1(E(-1)) \leq 2$ hence, by Lemma~\ref{L:h1e(l)c2=12}, $\tH^1(E) = 0$.    

Since $\tH^1(E(-4)) = 0$, Lemma~\ref{L:h2e(-3)=0}(b) implies that the graded 
$S$-module $\tH^1_\ast(E^\vee)$ is generated by $\tH^1(E^\vee(1))$. 
Lemma~\ref{L:h2e(-3)=0}(d),(f) implies that $\h^1(E^\vee(1)) \leq 1$. 
On the other hand, since $\tH^2(E(-3)) = 0$ and $\tH^3(E(-4)) \simeq 
\tH^0(E^\vee)^\vee = 0$, the graded $S$-module $\tH^1_\ast(E)$ is generated in 
degrees $\leq -2$ (see Remark~\ref{R:h2e(-3)=0}(i)).  

We assert that the multiplication map $\mu \colon \tH^1(E(-3)) \otimes S_1 
\ra \tH^1(E(-2))$ has rank at least 3. \emph{Indeed}, if $\mu$ has rank 
$\leq 2$ then Remark~\ref{R:beilinson} implies that one would have an 
exact sequence$\, :$ 
\[
0 \lra 2\Omega_\piii^2(2) \lra 2\Omega_\piii^1(1) \oplus \sco_\piii \lra Q 
\lra 0\, , 
\]
with $Q$ locally free. Since $Q$ would have rank 1, such an exact sequence 
\emph{cannot exist}. 

It remains that $\tH^1_\ast(E)$ has two minimal generators in degree $-3$ and 
at most one in degree $-2$. Consequently, $E(-2)$ is the cohomology of a 
Horrocks monad of the form$\, :$ 
\[
0 \lra \sco_\piii(-1) \overset{\beta}{\lra} B \overset{\alpha}{\lra}  
2\sco_\piii(1) \oplus \sco_\piii \lra 0\, , 
\] 
with $B$ a direct sum of line bundles. Since $B$ has rank 9, $\tH^0(B(-1)) 
= 0$, $\h^0(B) = \h^0(2\sco_\piii(1) \oplus \sco_\piii) - \h^1(E(-2)) = 5$ and 
$\tH^0(B^\vee(-2)) = 0$ it follows that $B \simeq 5\sco_\piii \oplus 
4\sco_\piii(-1)$.

\vskip2mm 

\noindent 
{\bf Claim 2.1.}\quad \emph{The component} $\alpha_{21} \colon 5\sco_\piii \ra 
\sco_\piii$ \emph{of} $\alpha$ \emph{is non-zero}. 

\vskip2mm 

\noindent 
\emph{Indeed}, assume, by contradiction, that $\alpha_{21} = 0$. Then the 
component $\beta_2 \colon \sco_\piii(-1) \ra 4\sco_\piii(-1)$ of $\beta$ must 
be, also, zero (because there is no epimorphism $3\sco_\piii(-1) \ra 
\sco_\piii$). One has, in this case, an exact sequence$\, :$ 
\[
0 \lra E(-2) \lra \text{T}_\piii(-1) \oplus \sco_\piii \oplus \Omega_\piii 
\lra 2\sco_\piii(1) \lra 0\, .  
\]
Since, as we saw at the beginning of Case 2, $E$ globally generated implies  
$\tH^i(E) = 0$, $i = 1,\, 2,\, 3$, one deduces that, for a 
\emph{general} epimorphism $\rho \colon \text{T}_\piii(-1) \oplus \sco_\piii 
\oplus \Omega_\piii \ra 2\sco_\piii(1)$, $\text{Ker}\, \rho(2)$ is globally 
generated. Let $\rho_i$, $i = 1,\, 2,\, 3$, be the components of $\rho$. 
The cokernel of a general morphism $\text{T}_\piii(-1) \ra 2\sco_\piii(1)$ is 
isomorphic to $\sco_{L_1 \cup L_2}(1)$, where $L_1$ and $L_2$ are disjoint 
lines (see Lemma~\ref{L:t(-2)ra2o} from Appendix~\ref{A:miscellaneous}).  
It follows that for a general epimorphism $\rho$, the cokernel of 
$(\rho_1 \, ,\, \rho_2) \colon \text{T}_\piii(-1) \oplus \sco_\piii \ra 
2\sco_\piii(1)$ is isomorphic to $\sco_{\{x\, ,\,y\}}(1)$, where $x$ and $y$ are 
two distinct points. Let ${\overline \rho}_3 : \Omega_\piii \ra 
\sco_{\{x\, ,\, y\}}(1)$ be the epimorphism induced by $\rho_3$ and let $\sck$ be 
its kernel. 
Assuming that $\text{Ker}\, \rho(2)$ is globally generated, it follows that 
$\sck(2)$ is globally generated. Now, $\tH^0(\sck(2)) \subseteq  
\tH^0(\Omega_\piii(2))$ and $\h^0(\Omega_\piii(2)) = 6$. One cannot have 
$\h^0(\sck(2)) = 6$ because $\sck(2) \neq \Omega_\piii(2)$. On the other hand, 
one cannot have $\h^0(\sck(2)) \leq 4$ because the degeneracy locus of any 
morphism $4\sco_\piii \ra \Omega_\piii(2)$ has codimension $\leq 2$ in $\piii$. 
It remains that $\h^0(\sck(2)) = 5$. Using Lemma~\ref{L:5oraomega(2)} from 
Appendix~\ref{A:miscellaneous} one gets, now, a \emph{contradiction}.  
This proves Claim 2.1.  

\vskip2mm 

It follows, from Claim 2.1, that $E(-2)$ is the cohomology of a monad of the 
form$\, :$ 
\[
0 \lra \sco_\piii(-1) \overset{\beta^\prim}{\lra} 
4\sco_\piii \oplus 4\sco_\piii(-1) \overset{\alpha^\prim}{\lra}  
2\sco_\piii(1) \lra 0\, , 
\]
If the component $\beta_2^\prim \colon \sco_\piii(-1) \ra 4\sco_\piii(-1)$ of 
$\beta^\prim$ is non-zero then $E(-2)$ is the kernel of an epimorphism 
$\phi \colon 4\sco_\piii \oplus 3\sco_\piii(-1) \ra 2\sco_\piii(1)$. If 
$\beta_2^\prim = 0$ then $E(-2)$ is the kernel of an epimorphism 
$\psi \colon \text{T}_\piii(-1) \oplus 4\sco_\piii(-1) \ra 2\sco_\piii(1)$.  

\vskip2mm 

\noindent 
{\bf Claim 2.2.}\quad \emph{If} $E(-2)$ \emph{is the kernel of an epimorphism} 
$\phi \colon 4\sco_\piii \oplus 3\sco_\piii(-1) \ra 2\sco_\piii(1)$ 
\emph{then the degeneracy locus of the component} 
$\phi_1 \colon 4\sco_\piii \ra 2\sco_\piii(1)$ \emph{of} $\phi$ 
\emph{has dimension} $\leq 1$. \emph{Moreover, if} $\tH^0(E(-1)) = 0$ 
\emph{then this locus has dimension} $0$.  

\vskip2mm 

\noindent 
\emph{Indeed}, let $A := \tH^0(\phi_1) \colon k^4 \ra k^2 \otimes S_1$. 
$A$ can be represented by a $2 \times 4$ matrix of linear forms in four 
indeterminates. These matrices have been classified by Mir\'{o}-Roig and 
Trautmann \cite{mrt}. Their result is recalled in \cite[Remark~B.1]{acm1}. 
Since $\tH^0(E(-2)) = 0$, $A$ must be injective. If $A$ is not \emph{stable} 
(in the sense of \cite[Lemma~1.1.1]{mrt}) then, as we noticed at the beginning 
of \cite[Remark~B.1]{acm1}, $A$ can be represented, up to the action of 
$\text{GL}(2) \times \text{GL}(4)$, by a matrix of the form$\, :$ 
\[
\begin{pmatrix} 
h_{00} & h_{01} & h_{02} & h_{03}\\ 
h_{10} & h_{11} & 0 & 0 
\end{pmatrix}\, . 
\]  
Since $A$ is injective, $h_{02}$ and $h_{03}$ must be linearly independent. 
Moreover, since, by Prop.~\ref{P:h0e(-1)geq2c2=12}, one has $\h^0(E(-1)) \leq 
1$, it follows that $h_{10}$ and $h_{11}$ are linearly independent. Let $L$ 
(resp., $L^\prime$) be the line of equations $h_{10} = h_{11} = 0$ (resp., 
$h_{02} = h_{03} = 0$). Then the degeneracy locus of $\phi_1$ is contained in 
$L \cup L^\prime$ hence it has dimension $\leq 1$. On the other hand, if $A$ is 
stable then (see \cite[Remark~B.1]{acm1}) the degeneracy locus of $\phi_1$ 
is 0-dimensional, unless it is a line or a conic.  

Assume, now, that $\tH^0(E(-1)) = 0$. Then the above arguments show that 
$A$ is stable. If the degeneracy locus of $\phi_1$ has positive dimension then 
(see \cite[Remark~B.1]{acm1}) either there exists a line $L \subset \piii$ 
such that $\Cok (\phi_1 \vb L) \simeq \sco_L(2)$ or there exists a 
(nonsingular) conic $C \subset \piii$ such that, identifying $C$ with $\pj$, 
$\Cok (\phi_1 \vb L) \simeq \sco_\pj(3)$. 

In the former case, one deduces, from the exact sequence$\, :$ 
\begin{equation}\label{E:phi1e(-2)phi2} 
0 \lra \Ker \phi_1 \lra E(-2) \lra 3\sco_\piii(-1) 
\overset{{\overline \phi}_2}{\lra} \Cok \phi_1 \lra 0 
\end{equation}
(with ${\overline \phi}_2$ induced by the component $\phi_2 \colon 
3\sco_\piii(-1) \ra 2\sco_\piii(1)$ of $\phi$), an exact sequence$\, :$ 
\[
E_L(-2) \lra 3\sco_L(-1) \lra \sco_L(2) \lra 0\, , 
\]
which \emph{contradicts} the fact that $E$ is globally generated. In the 
latter case, one gets a contradiction in a similar manner. Claim 2.2 is 
proven.  

\vskip2mm 
 
\noindent 
{\bf Claim 2.3.}\quad \emph{If} $E(-2)$ \emph{is the kernel of an epimorphism} 
$\psi \colon \text{T}_\piii(-1) \oplus 4\sco_\piii(-1) \ra 2\sco_\piii(1)$ 
\emph{then} $\tH^0(E(-1)) \neq 0$ \emph{and the degeneracy locus of the 
component} $\psi_1 \colon \text{T}_\piii(-1) \ra 2\sco_\piii(1)$ \emph{of} $\psi$ 
\emph{has dimension} $1$. 

\vskip2mm 

\noindent 
\emph{Indeed}, $\psi_1 \neq 0$ because there is no epimorphism 
$4\sco_\piii(-1) \ra 2\sco_\piii(1)$. If $\tH^0(\psi_1^\vee(1))$ is not injective 
then $\psi_1$ factorizes as $\text{T}_\piii(-1) \overset{\sigma}{\lra} 
\sco_\piii(1) \ra 2\sco_\piii(1)$. Then either $\sigma$ is an epimorphism or 
$\text{Im}\, \sigma = \sci_L(1)$ for some line $L \subset \piii$. In both 
cases $\h^0(\Ker \sigma(1)) \geq 5$ which \emph{contradicts} the fact that, 
by Prop.~\ref{P:h0e(-1)geq2c2=12}, $\h^0(E(-1)) \leq 1$. Consequently, 
$\tH^0(\psi_1^\vee(1))$ must be injective. 

One uses, now, Lemma~\ref{L:t(-2)ra2o} from Appendix~\ref{A:miscellaneous}  
(with $\rho = \psi_1^\vee(1)$). Case (c) of that lemma cannot occur in our 
context because, in that case, $\h^0(\Ker \psi_1) = 1$ which \emph{would 
contradict} the hypothesis $\tH^0(E(-2)) = 0$. In the cases (a) and (b) from 
Lemma~\ref{L:t(-2)ra2o} the degeneracy locus of $\psi_1$ has dimension 1 and 
$\h^0(\Ker \psi_1(1)) = 1$ hence $\tH^0(E(-1)) \neq 0$. Claim 2.3 is proven. 

\vskip2mm 

\noindent 
{\bf Construction 2.4.}\quad Let $P_0 = (1:0:0:0), \ldots , P_3 = (0:0:0:1)$ be 
the coordinate points of $\piii$. Consider four nonzero constants $a_1,\, 
a_2,\, b_1,\, b_2 \in k \setminus \{0\}$ such that $a_1b_2 - a_2b_1 \neq 0$ 
and let $\phi_1 \colon 4\sco_\piii \ra 2\sco_\piii(1)$ be the morphism defined 
by the matrix$\, :$ 
\[
\begin{pmatrix} 
x_0 & a_1x_1 & a_2x_2 & 0\\ 
0 & b_1x_1 & b_2x_2 & x_3
\end{pmatrix}\, . 
\]   
One has an exact sequence$\, :$ 
\[
0 \lra \scg \lra 4\sco_\piii \overset{\phi_1}{\lra} 2\sco_\piii(1) 
\overset{\pi}{\lra} \sco_{\{P_0\}}(1) \oplus \cdots \oplus \sco_{\{P_3\}}(1) 
\lra 0\, ,  
\]
where $\scg = \Ker \phi_1$ and $\pi$ is defined by the transpose of the 
matrix$\, :$ 
\[
\begin{pmatrix} 
0 & -b_1 & -b_2 & 1\\ 
1 & a_1 & a_2 & 0
\end{pmatrix}\, . 
\]
$\scg$ is a rank 2 reflexive sheaf with $c_1(\scg) = -2$ hence with 
$\scg^\vee \simeq \scg(2)$. Dualizing the above exact sequence, one gets a 
presentation$\, :$ 
\[
0 \lra 2\sco_\piii(-1) \lra 4\sco_\piii \lra \scg(2) \lra 0\, . 
\]
Consider a linear form $\lambda \in \tH^0(\sco_\piii(1))$ vanishing at none of 
the points $P_0 ,\ldots , P_3$ and a point $P \in \piii \setminus \{P_0, 
\dots , P_3\}$. Let ${\overline \phi}_2$ denote the composite 
epimorphism$\, :$ 
\[
3\sco_\piii(-1) \lra \sci_{\{P\}} \lra \sco_{\{P_0, \ldots , P_3\}} 
\overset{\lambda}{\lra} \sco_{\{P_0, \ldots , P_3\}}(1)\, ,  
\] 
and let $\sck$ be its kernel. One has an exact sequence$\, :$ 
\[
0 \lra \sco_\piii(-3) \lra 3\sco_\piii(-2) \lra \sck \lra 
\sci_{\{P, P_0 , \ldots , P_3\}} \lra 0\, . 
\]
Since $\tH^2(\scg(1)) = 0$, one has ${\overline \phi}_2 = \pi \circ \phi_2$, 
for some morphism $\phi_2 \colon 3\sco_\piii(-1) \ra 2\sco_\piii(1)$. 
Let $\phi \colon 4\sco_\piii \oplus 3\sco_\piii(-1) \ra 2\sco_\piii(1)$ be the 
epimorphism defined by $\phi_1$ and $\phi_2$ and let $E := \Ker \phi(2)$. The 
exact sequence \eqref{E:phi1e(-2)phi2} induces an exact sequence$\, :$  
\[
0 \lra \scg \lra E(-2) \lra \sck \lra 0\, . 
\] 
Now, if $P$ does not belong to any of the planes containing three of the 
points $P_0 , \ldots , P_3$ then Lemma~\ref{L:igamma(2)} implies that 
$\sci_{\{P, P_0 , \ldots , P_3\}}(2)$ is globally generated. In this case $E$ is 
globally generated and $\tH^0(E(-1)) = 0$. 

\vskip2mm 

\noindent 
{\bf Construction 2.5.}\quad Let $P_0 , \ldots , P_3$ and $a_1,\, a_2,\, 
b_1,\, b_2$ be as in Construction 2.6. Let $L \subset \piii$ be the line 
containing $P_0$ and $P_3$ (i.e., the line of equations $x_1 = x_2 = 0$). 
Denote by $\phi_1 \colon 4\sco_\piii \ra 2\sco_\piii(1)$ the morphism defined 
by the matrix$\, :$ 
\[
\begin{pmatrix} 
x_0 & a_1x_1 & a_2x_2 & x_3\\ 
0 & b_1x_1 & b_2x_2 & 0
\end{pmatrix}\, . 
\]   
One has an exact sequence$\, :$ 
\[
0 \lra \scg \lra 4\sco_\piii \overset{\phi_1}{\lra} 2\sco_\piii(1) 
\overset{\pi}{\lra} \sco_L(1) \oplus \sco_{\{P_1\}}(1) \oplus 
\sco_{\{P_2\}}(1) \lra 0\, , 
\]
where $\scg$ denotes the kernel of $\phi_1$ and $\pi$ is defined by the 
transpose of the matrix$\, :$ 
\[
\begin{pmatrix} 
0 & -b_1 & -b_2\\ 
1 & a_1 & a_2
\end{pmatrix}\, . 
\]
$\scg$ is a rank 2 reflexive sheaf with $c_1(\scg) = -2$ hence with 
$\scg^\vee \simeq \scg(2)$. Dualizing the above exact sequence, one gets an  
sequence$\, :$ 
\[
0 \lra 2\sco_\piii(-1) \lra 4\sco_\piii \lra \scg(2) \lra \omega_L(3) 
\lra 0\, .  
\]  
Of course, $\omega_L(3) \simeq \sco_L(1)$ hence $\scg$ is 2-regular. Let 
${\overline \phi}_2 \colon 3\sco_\piii(-1) \ra \sco_L(1) \oplus \sco_{\{P_1\}}(1) 
\oplus \sco_{\{P_2\}}(1)$ be the epimorphism defined by the matrix$\, :$ 
\[
\begin{pmatrix} 
x_0x_3 & x_3^2 & x_0^2\\ 
0 & -b_1x_1^2 & a_1x_1^2\\ 
0 & -b_2x_2^2 & a_2x_2^2
\end{pmatrix}
\] 
and let $\sck$ denote its kernel. Let ${\overline \phi}_{21} \colon 
\sco_\piii(-1) \ra \sco_L(1) \oplus \sco_{\{P_1\}}(1) \oplus \sco_{\{P_2\}}(1)$ be 
the morphism defined by the first column of the above matrix and 
${\overline \phi}_{22} \colon 
2\sco_\piii(-1) \ra \sco_L(1) \oplus \sco_{\{P_1\}}(1) \oplus \sco_{\{P_2\}}(1)$ 
the morphism defined by the other two columns. Of course$\, :$ 
\[
\Ker {\overline \phi}_{21} = \sci_L(-1)\, ,\  \Cok {\overline \phi}_{21} = 
\sco_{\{P_0\}}(1) \oplus \cdots \oplus \sco_{\{P_3\}}(1)\, , 
\] 
whence an exact sequence$\, :$ 
\[
0 \lra \sci_L(-1) \lra \sck \lra 2\sco_\piii(-1) \xra{{\widehat \phi}_{22}} 
\sco_{\{P_0\}}(1) \oplus \cdots \oplus \sco_{\{P_3\}}(1) \lra 0\, . 
\]
with ${\widehat \phi}_{22}$ induced by ${\overline \phi}_{22}$. 
${\widehat \phi}_{22}$ is defined by the transpose of the matrix$\, :$ 
\[
\begin{pmatrix} 
0 & -b_1x_1^2 & -b_2x_2^2 & x_3^2\\ 
x_0^2 & a_1x_1^2 & a_2x_2^2 & 0
\end{pmatrix}\, . 
\]
Now, since one has an exact sequence$\, :$ 
\[
4\sco_\piii(-2) \overset{\rho}{\lra} 2\sco_\piii(-1) \xra{{\widehat \phi}_{22}} 
\sco_{\{P_0\}}(1) \oplus \cdots \oplus \sco_{\{P_3\}}(1) \lra 0
\] 
with $\rho$ defined by the matrix$\, :$ 
\[
\begin{pmatrix} 
x_0 & a_1x_1 & a_2x_2 & 0\\ 
0 & b_1x_1 & b_2x_2 & x_3 
\end{pmatrix}
\] 
it follows that $\sck(2)$ \emph{is globally generated}. 

Finally, one can easily check that ${\overline \phi}_2 \colon 
3\sco_\piii(-1) \ra \sco_L(1) \oplus \sco_{\{P_1\}}(1) \oplus \sco_{\{P_2\}}(1)$ 
can be written as $\pi \circ \phi_2$, where 
$\phi_2 \colon 3\sco_\piii(-1) \ra 2\sco_\piii(1)$ is defined by the matrix$\, :$ 
\[
\begin{pmatrix} 
0 & x_1^2 + x_2^2 & 0\\ 
x_0x_3 & x_3^2 & x_0^2 + x_1^2 + x_2^2
\end{pmatrix}\, . 
\]
If $\phi \colon 4\sco_\piii \oplus 3\sco_\piii(-1) \ra 2\sco_\piii(1)$ is the 
morphism defined by $\phi_1$ and $\phi_2$ and if $E = \Ker \phi(2)$ then, 
using the exact sequence$\, :$ 
\[
0 \lra \scg \lra E(-2) \lra \sck \lra 0\, ,  
\]
one deduces that $E$ is globally generated. 

\vskip2mm 

\noindent 
{\bf Construction 2.6.}\quad According to Lemma~\ref{L:t(-2)ra2o}(a) in 
Appendix~\ref{A:miscellaneous}, for a general morphism $\psi_1 \colon 
\text{T}_\piii(-1) \ra 2\sco_\piii(1)$, one has an exact sequence$\, :$ 
\[
0 \lra \sco_\piii(-1) \lra \text{T}_\piii(-1) \overset{\psi_1}{\lra} 
2\sco_\piii(1) \lra \sco_{L \cup L^\prime}(1) \lra 0\, , 
\] 
where $L$ and $L^\prime$ are two disjoint lines. Choose $f \in 
\tH^0(\sco_L(2))$ (resp., $f^\prim \in \tH^0(\sco_{L^\prime}(2))$) vanishing in 
two distinct points $P_0$ and $P_1$ (resp., $P_2$ and $P_3$). One has an 
exact sequence$\, :$ 
\[
0 \lra \sci_L(-1) \oplus \sci_{L^\prime}(-1) \lra 2\sco_\piii(-1)  
\xra{f \oplus f^\prim} \sco_{L \cup L^\prime}(1) \lra 
\sco_{\{P_0, \ldots , P_3\}}(1) \lra 0\, . 
\] 
As we saw above (towards the final part of Construction 2.5) there exist 
epimorphisms ${\widehat \psi}_{22} \colon 2\sco_\piii(-1) \ra 
\sco_{\{P_0, \ldots , P_3\}}(1)$ such that $\Ker {\widehat \psi}_{22}(2)$ is 
globally generated. Using the above exact sequence one deduces that 
${\widehat \psi}_{22}$ lifts to a morphism ${\overline \psi}_{22} \colon 
2\sco_\piii(-1) \ra \sco_{L \cup L^\prime}(1)$. Since $\sci_L(1)$, 
$\sci_{L^\prime}(1)$ and $\Ker {\widehat \psi}_{22}(2)$ are globally generated, 
it follows that if ${\overline \psi}_2 \colon 4\sco_\piii(-1) \ra 
\sco_{L \cup L^\prime}(1)$ is the morphism with components $f$, $f^\prime$ and 
${\overline \psi}_{22}$ then $\Ker {\overline \psi}_2(2)$ is globally 
generated.  ${\overline \psi}_{22}$ lifts to a morphism $\psi_{22} \colon 
2\sco_\piii(-1) \ra 2\sco_\piii(1)$ and $f$ (resp., $f^\prim$) lifts to a 
quadratic form $q \in \tH^0(\sco_\piii(2))$ (resp., $q^\prim \in 
\tH^0(\sco_\piii(2))$). Let $\psi_2 \colon 4\sco_\piii(-1) \ra 2\sco_\piii(1)$ be 
the morphism with components $q$, $q^\prim$ and $\psi_{22}$. If $\psi \colon 
\text{T}_\piii(-1) \oplus 4\sco_\piii(-1) \ra 2\sco_\piii(1)$ is the morphism 
with components $\psi_1$ and $\psi_2$ then $\Ker \psi(2)$ is globally 
generated. 

\vskip2mm 

\noindent 
{\bf Case 3.}\quad $F$ \emph{has spectrum} $(0,0,0,-1)$. 

\vskip2mm 

\noindent 
In this case, $r = 4$, $c_3(F) = -2$ and $c_3 = 10$. Using the spectrum, one 
gets that $\h^1(E(l)) = \h^1(F(l+2)) = 0$ for $l \leq -4$ and $\h^1(E(-3)) = 
3$. In particular, $s := \h^1(E(-3)) - \h^1(E(-4)) = 3$ and $\h^2(E^\vee) = 
\h^1(E(-4)) = 0$. Lemma~\ref{L:h2e(-3)=0}(f) implies that $\h^1(E_H^\vee(1)) 
\leq s-1 = 2$, for any plane $H \subset \piii$, and assertion (d) of the same 
lemma implies, now, that $\h^1(E^\vee(1)) \leq 2$. By Riemann-Roch, 
$\h^1(E(-2)) = \h^1(F) = 5$ and $\h^0(E(-1)) - \h^1(E(-1)) = \h^0(F(1)) - 
\h^1(F(1)) = -2$. Since, by Prop.~\ref{P:h0e(-1)geq2c2=12}, $\h^0(E(-1)) \leq 
1$ it follows that $\h^1(E(-1)) \leq 3$. One deduces, from 
Lemma~\ref{L:h1e(l)c2=12}, that $\tH^1(E) = 0$. 

Since $\tH^1(E(-4)) = 0$, Lemma~\ref{L:h2e(-3)=0}(b) implies that the graded 
$S$-module $\tH^1_\ast(E^\vee)$ is generated by $\tH^1(E^\vee(1))$.  
Since $\tH^2(E(-3)) = 0$ and $\tH^3(E(-4)) \simeq \tH^0(E^\vee)^\vee = 0$, 
the graded $S$-module $\tH^1_\ast(E)$ is generated in 
degrees $\leq -2$ (see Remark~\ref{R:h2e(-3)=0}(i)).  
We assert that the multiplication map $\mu \colon 
\tH^1(E(-3)) \otimes S_1 \ra \tH^1(E(-2))$ has rank $\geq 4$. \emph{Indeed}, 
if it would have rank $\leq 3$ then, by Remark~\ref{R:beilinson}, it would 
exist an exact sequence$\, :$ 
\[
0 \lra 3\Omega_\piii^2(2) \lra 3\Omega_\piii^1(1) \oplus \sco_\piii \lra 
Q \lra 0\, , 
\]
with $Q$ locally free. $Q$ would have rank 1 but this is clearly \emph{not 
possible}. 

It remains that the graded $S$-module $\tH^1_\ast(E)$ has three minimal 
generators of degree $-3$ and at most one of degree $-2$. It follows that 
$E(-2)$ is the cohomology of a Horrocks monad of the form$\, :$ 
\[
0 \lra 2\sco_\piii(-1) \lra B \overset{\phi}{\lra} 
3\sco_\piii(1) \oplus \sco_\piii \lra 0\, , 
\] 
with $B$ a direct sum of line bundles. $B$ has rank 10, $\tH^0(B(-1)) = 0$, 
$\h^0(B) = \h^0(3\sco_\piii(1) \oplus \sco_\piii) - \h^1(E(-2)) = 8$ and 
$\tH^0(B^\vee(-2)) = 0$. One deduces that $B \simeq 8\sco_\piii \oplus 
2\sco_\piii(-1)$. Since there is no epimorphism $2\sco_\piii(-1) \ra \sco_\piii$, 
the component $\phi_{21} \colon 8\sco_\piii \ra \sco_\piii$ of $\phi$ must 
be non-zero. It follows that, actually, $E(-2)$ is the cohomology of a monad 
of the form$\, :$ 
\begin{equation}\label{E:monad3} 
0 \lra 2\sco_\piii(-1) \overset{\beta}{\lra} 7\sco_\piii \oplus 
2\sco_\piii(-1) \overset{\alpha}{\lra} 3\sco_\piii(1) \lra 0\, . 
\end{equation} 
Lemma~\ref{L:mora2o(1)} from Appendix~\ref{A:h0e(-2)neq0} implies that the 
monads of this form can be put toghether into a family with irreducible base.

\vskip2mm  

Consider, now, a monad of the form \eqref{E:monad3} such that $E := 
(\Ker \alpha/\text{Im}\, \beta)(2)$ is globally generated. 
Let $\alpha_1 \colon 7\sco_\piii \ra 3\sco_\piii(1)$ and $\alpha_2 \colon 
2\sco_\piii(-1) \ra 3\sco_\piii(1)$ be the components of $\alpha$. One has an 
exact sequence$\, :$ 
\begin{equation}\label{E:keralpha} 
0 \lra \Ker \alpha_1 \lra \Ker \alpha \lra 2\sco_\piii(-1) 
\overset{{\overline \alpha}_2}{\lra} \Cok \alpha_1 \lra 0\, , 
\end{equation}  
where ${\overline \alpha}_2$ is induced by $\alpha_2$. 

\vskip2mm 

\noindent 
{\bf Claim 3.1.}\quad \emph{The support of} $\Cok \alpha_1$ 
\emph{has dimension} $\leq 0$. 

\vskip2mm 

\noindent 
\emph{Indeed}, $\Ker \alpha(2)$ is globally generated, $c_1(2\sco_\piii(1)) 
= 2$ and $\Cok \alpha_1(-1)$ is globally generated (because $\Cok \alpha_1$ is 
a quotient of $3\sco_\piii(1)$). Claim 3.1 follows, now, from 
Lemma~\ref{L:suppcok}. 

\vskip2mm 

\noindent 
{\bf Construction 3.2.}\quad We want to show that, for a general epimorphism 
$\phi \colon 7\sco_\piii \ra 3\sco_\piii(1)$, $E := \Ker \phi(2)$ is globally 
generated and $\tH^0(E(-1)) = 0$. 

\vskip2mm 

Firstly, we construct an epimorphism $\psi \colon 7\sco_\piii \ra 
3\sco_\piii(1)$ such that $\Ker \psi(2)$ is globally generated. Consider, for 
that, a nonsigular quadric surface $Q \subset \piii$ and fix an isomorphism 
$Q \simeq \pj \times \pj$. If $N$ is the kernel of the evaluation epimorphism 
$3\sco_\piii \ra \sco_Q(0,2)$ then $\tH^1(N) = 0$, $\tH^2_\ast(N) \simeq 
{\underline k}(2)$ and $\tH^3(N(-2)) = 0$. In particular, $N$ is 1-regular.  
One deduces easily that $\tH^1_\ast(N) = 0$. It follows that 
$N \simeq \text{T}_\piii(-2)$ whence an exact sequence$\, :$ 
\[
0 \lra \text{T}_\piii(-1) \overset{\psi^\prim_1}{\lra} 3\sco_\piii(1) \lra 
\sco_Q(1,3) \lra 0\, . 
\] 
Let $\psi_1 \colon 4\sco_\piii \ra 3\sco_\piii(1)$ be the composite morphism 
$4\sco_\piii \ra \text{T}_\piii(-1) \overset{\psi^\prim_1}{\lra} 3\sco_\piii(1)$. 

Now, according to Claim 3.3 in the proof of \cite[Prop.~6.3]{acm1}, if 
$\psi^\prim_2 \colon 3\sco_\piii \ra \sco_Q(1,3)$ is a general epimorphism then 
$\Ker \psi^\prim_2(2)$ is globally generated. $\psi^\prim_2$ lifts to a morphism 
$\psi_2 \colon 3\sco_\piii \ra 3\sco_\piii(1)$. Let $\psi \colon 7\sco_\piii \ra 
3\sco_\piii(1)$ be the epimorphism defined by $\psi_1$ and $\psi_2$. Taking 
into account the exact sequence$\, :$ 
\[
0 \lra \sco_\piii(-1) \lra \Ker \psi \lra 3\sco_\piii 
\overset{\psi^\prim_2}{\lra} \sco_Q(1,3) \lra 0\, ,  
\]    
one deduces that $\Ker \psi(2)$ is globally generated. Since, as we noticed at 
the beginning of Case 3, $\tH^1(\Ker \psi(2)) = 0$ it follows that, for the 
general epimorphism $\phi \colon 7\sco_\piii \ra 3\sco_\piii(1)$, 
$\Ker \phi(2)$ is globally generated. 

\vskip2mm 

Secondly, we construct a morphism $\rho \colon 7\sco_\piii \ra 3\sco_\piii(1)$ 
with the property that the map $\tH^0(\rho(1)) \colon \tH^0(7\sco_\piii(1)) \ra 
\tH^0(3\sco_\piii(2))$ is \emph{injective}. We start by constructing a morphism 
$\eta \colon 6\sco_\pii \ra 3\sco_\pii(1)$ such that $\tH^0(\eta(1)) \colon 
\tH^0(6\sco_\pii(1)) \ra \tH^0(3\sco_\pii(2))$ is surjective, hence bijective. 
Let $W$ denote the 3-dimensional vector space $\tH^0(\sco_\pii(1))$, let 
$u \colon S^2W \ra W \otimes W$ be the injection defined by  
$
u(\ell_1\ell_2) := \ell_1 \otimes \ell_2 + \ell_2 \otimes \ell_1
$ 
and let $\pi \colon W \otimes W \ra S^2W$ be the canonical surjection. 
$u$ defines a morphism $\eta \colon S^2W \otimes_k \sco_\pii \ra W \otimes_k 
\sco_\pii(1)$. If $x_0,\, x_1,\, x_2$ is the canonical basis of $W$ then the 
composite map$\, :$ 
\[
S^2W \otimes W \xra{u \otimes \text{id}_W} W \otimes W \otimes W 
\xra{\text{id}_W \otimes \pi} W \otimes S^2W
\] 
maps $x_ix_j \otimes x_k - x_jx_k \otimes x_i + x_kx_i \otimes x_j$ to 
$2 x_i \otimes x_jx_k$ hence it is surjective hence $\tH^0(\eta(1))$ is 
surjective. One deduces that, for a general morphism $\rho_1 \colon 6\sco_\pii 
\ra 3\sco_\pii(1)$, $\tH^0(\rho_1(1))$ is bijective. 

Since, for a general morphism $\rho_1^\prim \colon 6\sco_\pii \ra 
2\sco_\pii(1)$, $\tH^0(\rho_1^\prim) \colon \tH^0(6\sco_\pii) \ra 
\tH^0(2\sco_\pii(1))$ is bijective it follows that there exists a matrix 
$A_1 = (\ell_{ij})_{\begin{subarray}{1} 1 \leq i \leq 3\\ 1 \leq j \leq 6 \end{subarray}}$,  
with entries $\ell_{ij} \in k[x_0,x_1,x_2]_1$, with the property that there is 
no linear relation (over $k[x_0,x_1,x_2]$) between its columns and such that 
the columns of the submatrix 
$A^\prime_1 = 
(\ell_{ij})_{\begin{subarray}{1} 1 \leq i \leq 2\\ 1 \leq j \leq 6 \end{subarray}}$ 
are linearly independent. Let $A$ be the matrix obtained by attaching to 
$A_1$ the seventh column $(0\, ,\, 0\, ,\, x_3)^{\text{t}}$. 

We assert that there is no linear relation (over $k[x_0, \ldots , x_3]$) 
between the columns of $A$. \emph{Indeed}, let $\mathbf{r} = 
(\ell_1 + a_1x_3\, ,\, \ldots \, ,\, \ell_7 + a_7x_3)^{\text{t}}$ be such a 
relation (with $\ell_j \in k[x_0,x_1,x_2]$ and $a_j \in k$). Multiplying the 
first two rows of $A$ against $\mathbf{r}$ one gets that $a_1 = \cdots = a_6 
= 0$. Multiplying the third row of $A$ against $\mathbf{r}$ one gets that 
$a_7 = 0$, $\ell_7 = 0$ and that 
$(\ell_1 \, ,\, \ldots \, ,\, \ell_6)^{\text{t}}$ is a linear relation between 
the columns of $A_1$ hence $\ell_1 = \cdots = \ell_6 = 0$. If $\rho \colon 
7\sco_\piii \ra 3\sco_\piii(1)$ is the morphism defined by $A$ then 
$\tH^0(\rho(1)) \colon \tH^0(7\sco_\piii(1)) \ra \tH^0(3\sco_\piii(2))$ is 
injective. One deduces that, for the general morphism $\phi \colon 
7\sco_\piii \ra 3\sco_\piii(1)$, $\tH^0(\phi(1))$ is injective. 

\vskip2mm 

\noindent 
{\bf Construction 3.3.}\quad As we saw at the beginning of Construction 3.2,  
if $\psi \colon \text{T}_\piii(-1) \oplus 3\sco_\piii \ra 
3\sco_\piii(1)$ is a general epimorphism then $E^\prim := \Ker \psi(2)$ is 
globally generated and $\tH^1(E^\prim) = 0$. If $E$ is a vector bundle such 
that $E(-1)$ can be realized as an extension$\, :$ 
\[
0 \lra E^\prim(-1) \lra E(-1) \lra \sco_\piii \lra 0 
\]       
then $E$ is globally generated and $E(-2)$ is the cohomology of a 
\emph{minimal} monad of the form$\, :$ 
\begin{equation}\label{E:o(-1)ra7o+o(-1)ra3o(1)}
0 \lra \sco_\piii(-1) \overset{\beta}{\lra} 7\sco_\piii \oplus \sco_\piii(-1) 
\overset{\alpha}{\lra} 3\sco_\piii(1) \lra 0\, . 
\end{equation} 

\vskip2mm 

\noindent 
{\bf Construction 3.4.}\quad We want to show that there exist globally 
generated vector bundles $E$ such that $E(-2)$ is the cohomolgy of a minimal 
monad of the form \eqref{E:o(-1)ra7o+o(-1)ra3o(1)} with the property that the 
component $\alpha_1 \colon 7\sco_\piii \ra 3\sco_\piii(1)$ of $\alpha$ is not an 
epimorphism. The minimality of the monad implies that one has an exact 
sequence$\, :$ 
\[
0 \lra E(-2) \lra \text{T}_\piii(-1) \oplus 3\sco_\piii \oplus \sco_\piii(-1) 
\overset{\phi}{\lra} 3\sco_\piii(1) \lra 0\, , 
\] 
and the assumption on $\alpha_1$ is equivalent to the fact that the component 
$\phi_1 \colon \text{T}_\piii(-1) \oplus 3\sco_\piii \ra 3\sco_\piii(1)$ of  
$\phi$ is not an epimorphism. Let $\phi_2 \colon \sco_\piii(-1) \ra 
3\sco_\piii(1)$ be the other component of $\phi$. Taking into account Claim 3.1  
and the exact sequence$\, :$ 
\[
0 \lra \Ker \phi_1 \lra E(-2) \lra \sco_\piii(-1) 
\overset{{\overline \phi}_2}{\lra} \Cok \phi_1 \lra 0\, , 
\]
with ${\overline \phi}_2$ the composite morphism $\sco_\piii(-1) 
\overset{\phi_2}{\lra} 3\sco_\piii(1) \ra \Cok \phi_1$, one deduces that, if  
$E$ is globally generated, then one must have $\Cok \phi_1 \simeq 
\sco_{\{x\}}(1)$, for some point $x \in \piii$. 

Now, to construct such an epimorphism $\phi$, recall, from the beginning of 
Construction 3.2, that there exist morphisms $\phi_{11} \colon 
\text{T}_\piii(-1) \ra 3\sco_\piii(1)$ such that $\Cok \phi_{11} \simeq 
\sco_Q(1,3)$, for some nonsingular quadric surface $Q \subset \piii$, with a 
fixed isomorphism $Q \simeq \pj \times \pj$. Take a point $x \in Q$. 
Lemma~\ref{L:3oqraix(1,3)} in Appendix~\ref{A:miscellaneous} shows that there 
exist epimorphisms $\psi \colon 3\sco_Q \ra \sci_{\{x\},Q}(1,3)$ such that 
$\Ker \psi(2,2)$ is globally generated. Denoting by ${\overline \phi}_{12}$ 
the composite morphism 
$3\sco_\piii \ra 3\sco_Q \overset{\psi}{\lra} \sci_{\{x\},Q}(1,3) \hookrightarrow 
\sco_Q(1,3)$, one deduces that $\Ker {\overline \phi}_{12}(2)$ is globally 
generated. Lift ${\overline \phi}_{12}$ to a morphism $\phi_{12} \colon 
3\sco_\piii \ra 3\sco_\piii(1)$ and let $\phi_1 \colon \text{T}_\piii(-1) \oplus 
3\sco_\piii \ra 3\sco_\piii(1)$ be the morphism defined by $\phi_{11}$ and 
$\phi_{12}$. Then $\Ker \phi_1(2) \simeq \Ker {\overline \phi}_{12}(2)$ is 
globally generated and $\Cok \phi_1 \simeq \sco_{\{x\}}$. Moreover, by 
Lemma~\ref{L:3oqraix(1,3)}, one has $\tH^1(\Ker \phi_1(2)) = 0$. 

Consider, finally, a map $\phi^\prim_2 \colon \sco_\piii(-1) \ra 
\sco_Q(1,3)$ defined by an element of $\tH^0(\sco_Q(2,4))$ which does not 
vanish at $x$, and lift $\phi^\prim_2$ to a morphism $\phi_2 \colon 
\sco_\piii(-1) \ra 3\sco_\piii(1)$. If $\phi \colon \text{T}_\piii(-1) \oplus 
3\sco_\piii \oplus \sco_\piii(-1) \ra 3\sco_\piii(1)$ is the epimorphism with 
components $\phi_1$ and $\phi_2$ then $\Ker \phi(2)$ is globally generated. 

\vskip2mm 

\noindent 
{\bf Construction 3.5.}\quad We would like to show the existence of globally 
generated vector bundles $E$ with $\tH^0(E(-2)) = 0$ and such that $E(-2)$ 
is the cohomology of a \emph{minimal} monad of the form$\, :$ 
\[
0 \lra 2\sco_\piii(-1) \overset{\beta}{\lra} 7\sco_\piii \oplus 2\sco_\piii(-1) 
\overset{\alpha}{\lra} 3\sco_\piii(1) \lra 0\, . 
\]    
For such an $E$, the middle cohomology of the subcomplex$\, :$ 
\[
0 \lra 2\sco_\piii(-1) \overset{\beta}{\lra} 7\sco_\piii  
\overset{\alpha_1}{\lra} 3\sco_\piii(1) \lra 0
\]
of the monad is a rank 2 reflexive sheaf $\scf$ with $c_1(\scf) = -1$ and 
$\tH^0(\scf) = 0$. One deduces an exact sequence$\, :$ 
\[
0 \lra \scf \lra E(-2) \lra 2\sco_\piii(-1) 
\overset{{\overline \alpha}_2}{\lra} \Cok \alpha_1 \lra 0\, , 
\]
with ${\overline \alpha}_2$ induced by the component $\alpha_2 \colon 
2\sco_\piii(-1) \ra 3\sco_\piii(1)$ of $\alpha$. By Claim 3.1, the support of 
$\Cok \alpha_1$ has dimension $\leq 0$. Dualizing, now, the last exact 
sequence and taking into account that $\scf^\vee \simeq \scf(1)$, one gets an 
exact sequence$\, :$ 
\[
0 \lra 2\sco_\piii \lra E^\vee(1) \lra \scf \lra 0\, . 
\]
Using it, one computes the Chern classes of $\scf$ and one gets $c_1(\scf) = 
0$, $c_2(\scf) = 3$, $c_3(\scf) = 3$. Recall, also, that $\tH^0(\scf) = 0$  
hence $\scf$ is stable. According to the proof of Chang \cite[Thm.~3.13]{ch}, 
a \emph{general} stable rank 2 reflexive sheaf with the above Chern classes 
can be realized as an extension$\, :$ 
\[
0 \lra \sco_\piii(-2) \lra \scf \lra \sci_C(1) \lra 0\, , 
\]
where $C$ is a nonsingular rational quintic curve, not contained in a 
quadric surface. It follows that, assuming $\scf$ general, $E^\vee(1)$ can be 
realized as an extension$\, :$ 
\[
0 \lra 2\sco_\piii \oplus \sco_\piii(-2) \lra E^\vee(1) \lra \sci_C(1) \lra 0\, . 
\]

%\vskip2mm 

We will show, now, that if $C \subset \piii$ is a nonsingular rational quintic 
curve, not contained in a quadric surface, then there exist extensions$\, :$ 
\[
0 \lra 2\sco_\piii \oplus \sco_\piii(-2) \lra E_1 \lra \sci_C(1) \lra 0
\]
with $E_1$ locally free such that $E := E_1^\vee(1)$ is globally generated (and 
$\tH^0(E(-2)) = 0$). 

\vskip2mm 

\emph{Indeed}, according to Serre's method of extensions (see 
Thm.~\ref{T:serreext} in Appendix~\ref{A:serre}), for every epimorphism 
$\delta \colon \sco_\piii(2) \oplus 2\sco_\piii \ra \omega_C(3)$ there 
exists an extension of the above form, with $E_1$ locally free, such that, 
dualizing the extension, one gets an exact sequence$\, :$ 
\[
0 \lra \sco_\piii(-1) \lra E(-1) \lra \sco_\piii(2) \oplus 2\sco_\piii  
\overset{\delta}{\lra} \omega_C(3) \lra 0\, . 
\] 
where $E := E_1^\vee(1)$. It thus remains to construct epimorphisms $\delta 
\colon \sco_\piii(2) \oplus 2\sco_\piii \ra \omega_C(3)$ with 
$\Ker \delta(1)$ globally generated. 

It is well known that $C$ admits a unique 4-secant line $L \subset \piii$ 
[$C$ is the image of an embedding $\nu \colon \pj \ra \piii$ defined by an 
epimorphism $\pi \colon 4\sco_\pj \ra \sco_\pj(5)$; since $C$ is not contained 
in a quadric surface, one must have $\Ker \pi \simeq \sco_\pj(-1) \oplus 
2\sco_\pj(-2)$; the 4-secant corresponds to the unique linear relation 
between the four binary quintics defining $\pi$]. Choose a plane $H_0 \subset 
\piii$, of equation $h_0 = 0$, intersecting $C$ in five distinct points 
$P_1 , \ldots , P_5$, any three of them noncollinear, and such that none of 
them belongs to $L$. Let $P_0$ be the intersection point of $H_0$ and $L$. 
For $1 \leq i < j \leq 5$, $P_0$ does not belong to the line 
$\overline{P_iP_j}$ [\emph{indeed}, if $P_0$ would belong to such a line then 
the plane spanned by that line and $L$ would cut $C$ in at least six points 
which is not possible]. 

Fix, now, an isomorphism $\omega_C \simeq \sci_{\{P_4 , P_5\} , C}$ and choose a 
2-dimensional vector subspace $W$ of $\tH^0(\sci_{\{P_4\, ,\, P_5\}\, ,\, H_0}(3))$ 
such that$\, :$ 
\[
W \cap \tH^0(\sci_{\{P_i\, ,\, P_j\, ,\, P_4\, ,\, P_5\}\, ,\, H_0}(3)) = 0\, ,\  
\text{for}\  0 \leq i < j \leq 3\, . 
\]
With this choice, the composite maps $W \hookrightarrow 
\tH^0(\sci_{\{P_4\, ,\, P_5\}\, ,\, H_0}(3)) \ra \tH^0(\sco_{\{P_i , P_j\}}(3))$, 
$0 \leq i < j \leq 3$, are all bijective. In particular, 
$W$ has a $k$-basis $f_1,\, f_2$ such that $f_i$ vanishes at $P_i$, $i = 1,\, 
2$. Lift $f_i$ to ${\widetilde f}_i \in \tH^0(\sci_{\{P_4\, ,\, P_5\}}(3))$, 
$i = 1,\, 2$. Consider, finally, the epimorphism $\delta \colon \sco_\piii(2) 
\oplus 2\sco_\piii \ra \omega_C(3) \simeq \sci_{\{P_4 , P_5\} , C}(3)$ defined 
by $h_0 \vb C$, ${\widetilde f}_1 \vb C$ and ${\widetilde f}_2 \vb C$. 
We will show that $\Ker \delta(1)$ is globally generated.    

\vskip2mm 

\noindent 
\emph{Indeed}, let $\delta_1 \colon \sco_\piii(2) \ra \omega_C(3)$ and 
$\delta_2 \colon 2\sco_\piii \ra \omega_C(3)$ be the components of $\delta$. 
$\Ker \delta_1 = \sci_C(2)$, $\Cok \delta_1 \simeq \sco_{\{P_1 , P_2 , P_3\}}$ and 
one has an exact sequence$\, :$ 
\[
0 \lra \sci_C(2) \lra \Ker \delta \lra 2\sco_\piii 
\overset{{\overline \delta}_2}{\lra} \sco_{\{P_1 , P_2 , P_3\}} \lra 0\, ,  
\]
where ${\overline \delta}_2$ is the composite morphism $2\sco_\piii 
\overset{\delta_2}{\lra} \omega_C(3) \ra \sco_{\{P_1 , P_2 , P_3\}}$. Put 
$\sck := \Ker \delta$ and $\sck_2 := \Ker {\overline \delta}_2$. 
Lemma~\ref{L:2oraogamma} from Appendix~\ref{A:miscellaneous} implies that 
$\sck_2(1)$ is globally generated. Let $\scn$ denote the kernel of the 
evaluation morphism $\tH^0(\sck_2(1)) \otimes_k \sco_\piii \ra \sck_2(1)$.  
One also knows, from Lemma~\ref{L:vahlenquintic}, that the cokernel of the 
evaluation morphism $\tH^0(\sci_C(3)) \otimes_k \sco_\piii \ra \sci_C(3)$ is 
$\sci_{L \cap C\, ,\, L}(3)$ and that $\tH^1(\sci_C(3)) = 0$. Applying, now, 
the Snake Lemma to the diagram$\, :$ 
\[
\begin{CD} 
0 @>>> \tH^0(\sci_C(3)) \otimes \sco_\piii @>>> \tH^0(\sck(1)) \otimes 
\sco_\piii @>>> \tH^0(\sck_2(1)) \otimes \sco_\piii @>>> 0\\ 
@. @VV{\text{ev}}V @VV{\text{ev}}V @VV{\text{ev}}V\\ 
0 @>>> \sci_C(3) @>>> \sck(1) @>>> \sck_2(1) @>>> 0 
\end{CD}
\] 
one sees that, in order to show that $\Ker \delta(1)$ is globally generated,  
it suffices to show that the connecting morphism $\partial \colon \scn \ra 
\sci_{L \cap C\, ,\, L}(3)$ (induced by the diagram) is an epimorphism. 
Since $\sci_{L \cap C\, ,\, L}(3) \simeq \sco_L(-1)$, it, actually, suffices to 
show that the map $\tH^0(\partial(1)) \colon \tH^0(\scn(1)) \ra 
\tH^0(\sci_{L \cap C\, ,\, L}(4))$ is non-zero.   

We shall emphasize, now, an element $\xi$ of $\tH^0(\scn(1))$ such that 
$\partial(1)(\xi) \neq 0$. Choose $h_1 \in \tH^0(\sco_\piii(1))$ vanishing in  
$P_2$ and $P_3$ but not in $P_1$. The elements $(h_0\, ,\, 0)^{\text{t}}$ and 
$(h_1\, ,\, 0)^{\text{t}}$ of $\tH^0(2\sco_\piii(1))$ belong to $\tH^0(\sck_2(1))$ 
hence$\, :$ 
\[
\xi := (h_1\, ,\, 0)^{\text{t}} \otimes h_0 - (h_0\, ,\, 0)^{\text{t}} \otimes h_1 
\in \tH^0(\scn(1))\, . 
\]
$(h_0\, ,\, 0)^{\text{t}}$ can be lifted to $(-{\widetilde f}_1\, ,\, h_0\, ,\, 
0)^{\text{t}} \in \tH^0(\sck(1))$. On the other hand, there exists $g_1 \in 
\tH^0(\sco_C(3))$ such that $h_1{\widetilde f}_1 \vb C = (h_0 \vb C)g_1$. 
Since $\tH^1(\sci_C(3)) = 0$, there exists ${\widetilde g}_1 \in 
\tH^0(\sco_\piii(3))$ such that ${\widetilde g}_1 \vb C = g_1$. Then 
$(h_1\, ,\, 0)^{\text{t}}$ can be lifted to $(-{\widetilde g}_1\, ,\, h_1\, ,\, 
0)^{\text{t}} \in \tH^0(\sck(1))$ hence $\xi$ can be lifted to the following 
element of $\tH^0(\sck(1)) \otimes \tH^0(\sco_\piii(1))$$\, :$ 
\[
(-{\widetilde g}_1\, ,\, h_1\, ,\, 0)^{\text{t}} \otimes h_0 - 
(-{\widetilde f}_1\, ,\, h_0\, ,\, 0)^{\text{t}} \otimes h_1\, . 
\]
It is clear, now, that$\, :$ 
\[
\partial(1)(\xi) = ({\widetilde f}_1h_1 - {\widetilde g}_1h_0) \vb L \in 
\tH^0(\sci_{L \cap C , L}(4))\, . 
\] 
It follows that $\partial(1)(\xi) \neq 0$ because ${\widetilde f}_1h_1 - 
{\widetilde g}_1h_0$ does not vanish in $P_0 \in L$ [\emph{indeed}, $h_0$ 
vanishes in $P_0$, $h_1$ does not vanish in $P_0$ (because $P_0 \notin 
\overline{P_2P_3}$), and ${\widetilde f}_1$ does not vanish in $P_0$ (because 
${\widetilde f}_1 \vb H_0 = f_1$ and $W \cap 
\tH^0(\sci_{\{P_0, P_1, P_4, P_5\} , H_0}(3)) = 0$)].

\vskip2mm 

\noindent 
{\bf Construction 3.6.}\quad We want to construct globally 
generated vector bundles $E$ with $\tH^0(E(-2)) = 0$ and such that $E(-2)$ 
is the cohomology of a \emph{minimal} monad of the form$\, :$ 
\[
0 \lra 2\sco_\piii(-1) \overset{\beta}{\lra} 7\sco_\piii \oplus 2\sco_\piii(-1) 
\overset{\alpha}{\lra} 3\sco_\piii(1) \lra 0\, . 
\]    
with the property that the middle cohomology of the subcomplex$\, :$ 
\[
0 \lra 2\sco_\piii(-1) \overset{\beta}{\lra} 7\sco_\piii  
\overset{\alpha_1}{\lra} 3\sco_\piii(1) \lra 0
\] 
is a \emph{special} stable rank 2 reflexive sheaf $\scf$ with $c_1(\scf) = -1$, 
$c_2(\scf) = 3$, $c_3(\scf) = 3$. Here special means that $\tH^0(\scf(1)) \neq 
0$. Examples of such reflexive sheaves are the ones that can be realized as 
extensions$\, :$ 
\[
0 \lra \sco_\piii(-1) \lra \scf \lra \sci_Y \lra 0\, , 
\] 
where $Y$ is the union of three mutually disjoint lines $L_1,\, L_2,\, L_3$. 
If $E$ is a vector bundle corresponding, as above, to such an $\scf$ then, as 
at the beginning of Construction 3.5, $E^\vee(1)$ can be realized as an 
extension$\, :$ 
\[
0 \lra 2\sco_\piii \oplus \sco_\piii(-1) \lra E^\vee(1) \lra \sci_Y \lra 0\, . 
\] 
Dualizing this extension, one gets an exact sequence$\, :$ 
\[
0 \lra \sco_\piii \lra E(-1) \lra \sco_\piii(1) \oplus 2\sco_\piii 
\overset{\delta}{\lra} \omega_Y(4) \lra 0\, . 
\]
Consequently, using Serre's method of extensions,  
we are left with the following question$\, :$ if $Y$ is the union of three 
mutually disjoint lines $L_1,\, L_2,\, L_3$, do there exist epimorphisms 
$\delta \colon \sco_\piii(1) \oplus 2\sco_\piii \ra \omega_Y(4) \simeq 
\sco_Y(2)$ such that $\Ker \delta(1)$ is globally generated$\, ?$ 

We show, now, that such epimorphisms really exist. Let $Q \subset \piii$ be 
the unique quadric surface containing $Y$. Fix an isomorphism $Q \simeq 
\pj \times \pj$. Assume that $L_1,\, L_2,\, L_3$ belong to the linear system 
$\vert \, \sco_Q(1,0) \, \vert$.  
Choose points $P_i \in L_i$, $i = 1,\, 2,\, 3$, such that 
none of the lines $\overline{P_iP_j}$, $1 \leq i < j \leq 3$ is contained in 
$Q$. The intersection of the plane $H_0$ containing $P_1,\, P_2,\, P_3$ with 
$Q$ is a nonsingular conic $C$. Let $h_0 = 0$ be an equation of $H_0$ and let 
$h_0 = h_i = 0$ be equations of the line containing $\{P_1,\, P_2,\, P_3\} 
\setminus \{P_i\}$, $i = 1,\, 2,\, 3$. 

Choose $f_1,\, f_2 \in \tH^0(\sco_C(2))$ such that the zero divisors of $f_1$ 
and $f_2$ have the form$\, :$ 
\[
(f_1)_0 = P_1 + Q_1 + Q_2 + Q_3 \, ,\  
(f_2)_0 = P_2 + R_1 + R_2 + R_3\, ,  
\] 
with $Q_i \in C \setminus \{P_1,\, P_2,\, P_3\}$ and $R_j \in C \setminus 
\{P_1,\, P_2,\, P_3,\, Q_1,\, Q_2,\, Q_3\}$. Lift $f_1,\, f_2$ to 
${\widetilde f}_1,\, {\widetilde f}_2 \in \tH^0(\sco_\piii(2))$. 

Let $\delta_1 \colon \sco_\piii(1) \ra \sco_Y(2)$ the morphism defined by 
$h_0 \vb Y$, $\delta_2 \colon 2\sco_\piii \ra \sco_Y(2)$ the morphism defined by 
${\widetilde f}_1 \vb Y$ and ${\widetilde f}_2 \vb Y$, and $\delta \colon 
\sco_\piii(1) \oplus 2\sco_\piii \ra \sco_Y(2)$ the epimorphism of components 
$\delta_1$ and $\delta_2$. We assert that $\Ker \delta(1)$ is globally 
generated. 

\vskip2mm 

\emph{Indeed}, $\Ker \delta_1 = \sci_Y(1)$, $\Cok \delta_1 = \sco_\Gamma(2)$, 
where $\Gamma := \{P_1,\, P_2,\, P_3\}$, and one has an exact sequence$\, :$ 
\[
0 \lra \sci_Y(1) \lra \Ker \delta \lra 2\sco_\piii 
\overset{{\overline \delta}_2}{\lra} \sco_\Gamma(2) \lra 0\, , 
\]     
where ${\overline \delta}_2$ is the composite morphism $2\sco_\piii 
\overset{\delta_2}{\lra} \sco_Y(2) \ra \sco_\Gamma(2)$. Let $\sck := 
\Ker \delta$ and $\sck_2 := \Ker {\overline \delta}_2$. Consider the 
commutative diagram$\, :$ 
\[
\begin{CD} 
0 @>>> \sco_\piii @>>> \tH^0(\sck(1)) \otimes \sco_\piii @>>> 
\tH^0(\sch_2(1)) \otimes \sco_\piii @>>> 0\\ 
@. @VV{\text{ev}}V @VV{\text{ev}}V @VV{\text{ev}}V\\ 
0 @>>> \sci_Y(2) @>>> \sck(1) @>>> \sck_2(1) @>>> 0 
\end{CD}
\]
($\h^0(\sci_Y(2)) = 1$ and $\tH^1(\sci_Y(2)) = 0$). 
By Lemma~\ref{L:2oraogamma} from Appendix~\ref{A:miscellaneous}, $\sck_2(1)$ 
is globally generated. Let $\scn$ be the kernel of its evaluation morphism. 
One also has an exact sequence$\, :$ 
\[
0 \lra \sco_\piii \overset{\text{ev}}{\lra} \sci_Y(2) \lra 
\sci_{Y\, ,\, Q}(2) \lra 0\, , 
\]
and $\sci_{Y\, ,\, Q}(2) \simeq \sco_Q(-1,2)$. Consequently, 
in order to show that $\sck(1)$ is globally generated, it suffices to 
show that the connecting morphism $\partial \colon \scn \ra 
\sci_{Y\, ,\, Q}(2)$ is an epimorphism. To prove that, we shall emphasize two 
elements $\xi_1$ and $\xi_2$ of $\tH^0(\scn(1))$ such that the zero divisors 
of $\partial(1)(\xi_1),\, \partial(1)(\xi_2) \in \tH^0(\sci_{Y\, ,\, Q}(3)) 
\subset \tH^0(\sco_Q(3,3))$ have no common component besides the lines 
$L_1,\, L_2,\, L_3$. 

According to the proof of Lemma~\ref{L:2oraogamma},  
$\tH^0(\sck_2(1)) \subset \tH^0(2\sco_\piii(1))$ is generated by the columns of 
the matrix$\, :$ 
\[
\begin{pmatrix} 
h_1 & 0 & a_1h_3 & h_0 & 0\\ 
0 & h_2 & a_2h_3 & 0 & h_0
\end{pmatrix}\, . 
\]    
Consider the following two global sections of $\scn(1) \subset 
\tH^0(\sck_2(1)) \otimes_k \sco_\piii(1)$$\, :$ 
\[
\xi_1 := (h_1\, ,\, 0)^{\text{t}} \otimes h_0 -  
(h_0\, ,\, 0)^{\text{t}} \otimes h_1 \, , \  
\xi_2 := (0\, ,\, h_2)^{\text{t}} \otimes h_0 -  
(0\, ,\, h_0)^{\text{t}} \otimes h_2 \, . 
\]
Now, $(h_1\, ,\, 0)^{\text{t}}$ (resp., $(0\, ,\, h_2)^{\text{t}}$) is the image 
of an element $(q_1\, ,\, h_1\, ,\, 0)^{\text{t}}$ (resp., 
$(q_2\, ,\, 0\, ,\, h_2)^{\text{t}}$) of $\tH^0(\sck(1)) \subset 
\tH^0(\sco_\piii(2) \oplus 2\sco_\piii(1))$. This means that $q_1h_0 + 
h_1{\widetilde f}_1$ (resp., $q_2h_0 + h_2{\widetilde f}_2$) vanishes on $Y$.  
Moreover, $(h_0\, ,\, 0)^{\text{t}}$ (resp., $(0\, ,\, h_0)^{\text{t}}$) 
is the image of the element $(-{\widetilde f}_1\, ,\, h_0\, ,\, 0)^{\text{t}}$ 
(resp., $(-{\widetilde f}_2\, ,\, 0\, ,\, h_0)^{\text{t}}$) of $\tH^0(\sck(1))$. 
One deduces that$\, :$ 
\[
\partial(1)(\xi_1) = q_1h_0 + {\widetilde f}_1h_1 \vb Q \ \   \text{and}\ \    
\partial(1)(\xi_2) = q_2h_0 + {\widetilde f}_2h_2 \vb Q\, . 
\]   
Now, $\partial(1)(\xi_i) \vb C = f_i(h_i \vb C)$, $i = 1,\, 2$. Denoting by 
$D_i$ the zero divisor of $\partial(1)(\xi_i)$ on $Q$, one gets that$\, :$ 
\begin{gather*} 
D_1 = L_1 + L_2 + L_3 + M_1 + M_2 + M_3\, ,\  
D_2 = L_1 + L_2 + L_3 +N_1 + N_2 + N_3\, , 
\end{gather*}
where $M_i$ (resp., $N_j$) is the line belonging to the linear system 
$\vert \, \sco_Q(0,1)\, \vert$ that passes through $Q_i$ (resp., $R_j$). 
It follows that the only common components of $D_1$ and $D_2$ are $L_1$, 
$L_2$ and $L_3$. 

\vskip2mm 

\noindent 
{\bf Case 4.}\quad $F$ \emph{has spectrum} $(1,0,0,-1)$. 

\vskip2mm 

\noindent 
In this case, $r = 4$, $c_3(F) = -4$ and $c_3 = 8$. It follows, using the 
spectrum, that $\h^1(E(l)) = \h^1(F(l+2)) = 0$ for $l \leq -5$, $\h^1(E(-4)) 
= 1$ (hence $\h^2(E^\vee) = 1$) and $\h^1(E(-3)) = 4$. In particular, 
$s := \h^1(E(-3)) - \h^1(E(-4)) = 3$. 
Applying Riemann-Roch to (the twists of) $F$, one gets that $\h^1(E(-2)) = 6$ 
and $\h^0(E(-1)) - \h^1(E(-1)) = -3$. Since, by Prop.~\ref{P:h0e(-1)geq2c2=12}, 
$\h^0(E(-1)) \leq 1$ it follows that $\h^1(E(-1)) \leq 4$. 

\vskip2mm 

\noindent 
{\bf Claim 4.1.}\quad \emph{The graded} $S$-\emph{module} $\tH^1_\ast(E)$ 
\emph{has a minimal generator in degree} $-4$ \emph{and at most one in degree}  
$-2$. 

\vskip2mm 

\noindent 
\emph{Indeed}, since $\tH^2(E(-3)) = 0$ and $\tH^3(E(-4)) \simeq 
\tH^0(E^\vee)^\vee= 0$, $\tH^1_\ast(E)$ is generated in degrees $\leq -2$ (see 
Remark~\ref{R:h2e(-3)=0}(i)). It follows, from Remark~\ref{R:muh1e(-4)}, that 
the multiplication map $\tH^1(E(-4)) \otimes S_1 \ra \tH^1(E(-3))$ is 
surjective. 

We assert, now, that the multiplication map $\mu \colon \tH^1(E(-3)) \otimes 
S_1 \ra \tH^1(E(-2))$ has rank $\geq 5$. \emph{Indeed}, if 
the rank of $\mu$ is $\leq 4$ then, by Remark~\ref{R:beilinson}, there exists 
an exact sequence$\, :$  
\[
0 \lra \Omega_\piii^3(3) \lra 4\Omega_\piii^2(2) \lra 4\Omega_\piii^1(1) \oplus 
\sco_\piii \lra Q \lra 0\, ,  
\]
with $Q$ locally free. $Q$ has rank 2 and Chern classes $c_1(Q) = 3$, $c_2(Q) 
= 6$, $c_3(Q) = 2$. But this \emph{is not possible}. It remains that $\mu$ 
has rank $\geq 5$ and Claim 4.1 is proven.   

\vskip2mm 

\noindent 
{\bf Claim 4.2.}\quad \emph{The graded} $S$-\emph{module} $\tH^1_\ast(E^\vee)$ 
\emph{is generated by} $\tH^1(E^\vee(1))$. 

\vskip2mm 

\noindent 
\emph{Indeed}, assume, by contradiction, that $\tH^1_\ast(E^\vee)$ 
is not generated by $\tH^1(E^\vee(1))$. Then (the proof of) 
Lemma~\ref{L:h2e(-3)=0}(h) implies that $\tH^1_\ast(E^\vee) \simeq 
{\underline k}(-1) \oplus {\underline k}(-2)$.   
Consider, now, the extension$\, :$ 
\[
0 \lra E(-2) \lra E_4 \lra \sco_\piii(2) \oplus \sco_\piii \lra 0 
\]   
defined by the generators of the graded $S$-module $\tH^1_\ast(E)$ from 
Claim 4.1. One has $\tH^1_\ast(E_4) = 0$. Dualizing, one gets an exact 
sequence$\, :$ 
\[
0 \lra \sco_\piii \oplus \sco_\piii(-2) \lra E_4^\vee \lra E^\vee(2) \lra 0\, . 
\]
One deduces that $\tH^1_\ast(E_4^\vee) \simeq {\underline k}(1) \oplus 
{\underline k}$ and $\tH^2_\ast(E_4^\vee) = 0$. Since $E_4^\vee$ has rank 6 it 
follows that $E_4^\vee \simeq \Omega_\piii(1) \oplus \Omega_\piii$. But, from the 
above exact sequence, $\tH^0(E_4^\vee) \neq 0$ and this is a 
\emph{contradiction}. It, thus, remains that 
$\tH^1_\ast(E^\vee)$ is generated by $\tH^1(E^\vee(1))$. 

\vskip2mm 

\noindent 
{\bf Claim 4.3.}\quad $E(-2)$ \emph{is the kernel of an epimorphism} 
$4\sco_\piii \oplus \sco_\piii(-1) \ra \sco_\piii(2)$. 

\vskip2mm 

\noindent 
\emph{Indeed}, it follows, from Claim 4.1 and Claim 4.2, that $E(-2)$ admits a 
Horrocks monad of the form$\, :$ 
\[
0 \lra \sco_\piii(-1) \overset{\beta}{\lra} B 
\overset{\alpha}{\lra} \sco_\piii(2) \oplus \sco_\piii \lra 0\, ,
\] 
where $B$ is a direct sum of line bundles. One has $\text{rk}\, B = 7$, 
$\h^0(B) = \h^0(\sco_\piii(2) \oplus \sco_\piii) - \h^1(E(-2)) = 
5$, $\h^0(B(-1)) = \h^0(\sco_\piii(1) \oplus \sco_\piii(-1)) - \h^1(E(-3)) = 0$, 
and $\tH^0(B^\vee(-2)) = 0$. It follows that $B \simeq 5\sco_\piii \oplus 
2\sco_\piii(-1)$. Since there is no epimorphism $2\sco_\piii(-1) \ra \sco_\piii$ 
one deduces that the component $5\sco_\piii \ra \sco_\piii$ of $\alpha$ is 
non-zero hence $E(-2)$ is, actually, the cohomology of a monad of the 
form$\, :$ 
\[
0 \lra \sco_\piii(-1) \overset{\beta^\prim}{\lra} 4\sco_\piii \oplus 
2\sco_\piii(-1) \overset{\alpha^\prim}{\lra} \sco_\piii(2) \lra 0\, . 
\]  
In order to prove the claim one has to show that the component $\sco_\piii(-1) 
\ra 2\sco_\piii(-1)$ of $\beta^\prim$ is nonzero. Assume, by contradiction, 
that this component is zero. Then one has an exact sequence$\, :$ 
\[
0 \lra E(-2) \lra \text{T}_\piii(-1) \oplus 2\sco_\piii(-1) 
\overset{\alpha^\secund}{\lra} \sco_\piii(2) \lra 0\, . 
\] 
Let $\alpha_1^\secund \colon \text{T}_\piii(-1) \ra \sco_\piii(2)$ and 
$\alpha_2^\secund \colon 2\sco_\piii(-1) \ra \sco_\piii(2)$ be the components of 
$\alpha^\secund$. $\Cok \alpha_1^\secund \simeq \sco_Z(2)$, for some closed 
subscheme $Z$ of $\piii$. Let $\pi$ denote the composite epimorphism$\, :$ 
\[
2\sco_\piii(-1) \overset{\alpha_2^\secund}{\lra} \sco_\piii(2) \lra \sco_Z(2)\, . 
\]
Restricting to $Z$ the exact sequence$\, :$ 
\[
0 \lra \Ker \alpha_1^\secund \lra E(-2) \lra 2\sco_\piii(-1) \overset{\pi}{\lra} 
\sco_Z(2) \lra 0 
\]
one gets an epimorphism $E_Z(-2) \ra \sco_Z(-4)$. Since $E$ is globally 
generated, it follows that $\dim Z \leq 0$. Since $c_3(\Omega_\piii(3)) = 5$, 
$Z$ is a 0-dimensional subscheme of $\piii$ of length 5. $\alpha_1^\secund$ 
can be extended to a Koszul resolution of $\sco_Z(2)$$\, :$ 
\[
0 \lra \sco_\piii(-3) \lra \Omega_\piii \lra \text{T}_\piii(-1) 
\overset{\alpha_1^\secund}{\lra} \sco_\piii(2) \lra \sco_Z(2) \lra 0
\] 
(we used the fact that ${\bigwedge}^2(\text{T}_\piii(-1)) \simeq 
\Omega_\piii(2)$). One gets an exact sequence$\, :$ 
\[
0 \lra \sco_\piii(-3) \lra \Omega_\piii \lra E(-2) \lra 2\sco_\piii(-1) 
\overset{\pi}{\lra} \sco_Z(2) \lra 0\, . 
\]
Since $\sci_Z(1)$ is not globally generated the map $\tH^0(\pi(1)) \colon 
\tH^0(2\sco_\piii) \ra \tH^0(\sco_Z(3))$ is injective. One deduces that 
$\tH^0(E(-1)) = 0$ hence $\h^1(E(-1)) = 3$ (by Riemann-Roch, as at the 
beginning of Case 4). Lemma~\ref{L:h1e(l)c2=12} implies that $\tH^1(E) = 0$. 
Using, now, the above exact sequence one gets that $\tH^1(\Ker \pi(2)) = 0$ 
and that $\Ker \pi(2)$ is globally generated. Using the exact sequence$\, :$ 
\[
0 \lra \Ker \pi(2) \lra 2\sco_\piii(1) \lra \sco_Z(4) \lra 0 
\]   
one deduces that $\h^0(\Ker \pi(2)) = \h^0(2\sco_\piii(1)) - \h^0(\sco_Z(4)) 
= 3$. One obtains, now, an exact sequence$\, :$ 
\[
3\sco_\piii \lra 2\sco_\piii(1) \lra \sco_Z(4) \lra 0\, . 
\]
But such an exact sequence cannot exist because $Z$ has codimension 3 in 
$\piii$. This \emph{contradiction} shows that the component $\sco_\piii(-1) 
\ra 2\sco_\piii(-1)$ of $\beta^\prim$ is nonzero and Claim 4.3 is proven. 

\vskip2mm 

\noindent 
{\bf Claim 4.4.}\quad \emph{Consider an epimorphism} $\phi \colon 4\sco_\piii 
\oplus \sco_\piii(-1) \ra \sco_\piii(2)$. \emph{If} $\Ker \phi(2)$ \emph{is 
globally generated then the component} $\phi_1 \colon 4\sco_\piii \ra 
\sco_\piii(2)$ \emph{of} $\phi$ \emph{is an epimorphism}. 

\vskip2mm 

\noindent 
\emph{Indeed}, let $E := \Ker \phi(2)$ and $\phi_2 \colon \sco_\piii(-1) \ra 
\sco_\piii(2)$ be the other component of $\phi$. $\Cok \phi_1 \simeq \sco_Z(2)$ 
for some closed subscheme $Z$ of $\piii$. Since one has an exact 
sequence$\, :$ 
\[
0 \lra \Ker \phi_1 \lra E(-2) \lra \sco_\piii(-1) 
\overset{{\overline \phi}_2}{\lra} \sco_Z(2) \lra 0\, ,  
\] 
with ${\overline \phi}_2$ the composite morphism $\sco_\piii(-1) 
\overset{\phi_2}{\lra} \sco_\piii(2) \ra \sco_Z(2)$, it follows that $\dim Z 
\leq 0$. 

Assume, by contradiction, that $Z \neq \emptyset$. Since 
$\Ker {\overline \phi}_2 = \sci_Z(-1)$ and since $E$ is globally generated, 
it follows that $Z$ consists of a simple point $x$, that is, $\text{Im}\, 
\phi_1 = \sci_{\{x\}}(2)$. If $W$ is a general 3-dimensional vector subspace of 
$\tH^0(4\sco_\piii)$ then the composite map $W \hookrightarrow 
\tH^0(4\sco_\piii) \ra (\sci_{\{x\}}/\sci_{\{x\}}^2)(2)$ is bijective. 
Consequently, one can assume that $\phi_1$ is defined by four quadratic 
forms $f_1,\, f_2,\, f_3,\, f \in \tH^0(\sci_{\{x\}}(2))$ such that $f_1,\, 
f_2,\, f_3$ define a closed subscheme $X$ of $\piii$ having $x$ as a simple, 
isolated point. Since the scheme theoretic intersection of $X$ with the 
quadric surface $\{ f = 0\}$ is the simple point $\{x\}$ it follows that 
$\dim X = 0$ hence $X$ is a complete intersection of type $(2,2,2)$ in 
$\piii$. Moreover, $\phi_2$ is defined by a cubic form $g \in 
\tH^0(\sco_\piii(3))$ not vanishing in $x$.   

Now, let $\psi_1 \colon 3\sco_\piii \ra \sco_\piii(2)$ be the morphism defined 
by $f_1,\, f_2,\, f_3$ and $\psi_2 \colon \sco_\piii \oplus \sco_\piii(-1) 
\ra \sco_\piii(2)$ the morphism defined by $f$ and $g$. Of course, 
$\Cok \psi_1 = \sco_X(2)$ and $\Ker \psi_1$ admits a Koszul resolution defined 
by $f_1,\, f_2,\, f_3$. One, thus, has an exact sequence$\, :$ 
\[
0 \ra \sco_\piii(-6) \ra 3\sco_\piii(-4) \ra 3\sco_\piii(-2) \ra E(-2) 
\ra \sco_\piii \oplus \sco_\piii(-1) \overset{{\overline \psi}_2}{\lra} 
\sco_X(2) \ra 0\, , 
\] 
with ${\overline \psi}_2$ the composite morphism $\sco_\piii \oplus 
\sco_\piii(-1) \overset{\psi_2}{\lra} \sco_\piii(2) \ra \sco_X(2)$. 
Let $\sck$ be the kernel of ${\overline \psi}_2$. Since $E$ is globally 
generated it follows that $\sck(2)$ is globally generated.
The components ${\overline \psi}_{21} \colon \sco_\piii \ra \sco_X(2)$ and 
${\overline \psi}_{22} \colon \sco_\piii(-1) \ra \sco_X(2)$ of 
${\overline \psi}_2$ are defined by $f \vb X$ and $g \vb X$, respectively. 
It follows that $\Cok {\overline \psi}_{21} = \sco_{\{x\}}(2)$ and 
$\Ker {\overline \psi}_{21} = \sci_Y$, where $Y := X \setminus \{x\}$. One 
deduces an exact sequence$\, :$ 
\[
0 \lra \sci_Y \lra \sck \lra \sci_{\{x\}}(-1) \lra 0\, . 
\] 
Considering a morphism of resolutions$\, :$ 
\[
\begin{CD} 
\sco_\piii(-6) @>>> 3\sco_\piii(-4) @>>> 3\sco_\piii(-2) @>>> \sci_X\\ 
@VVV @VVV @VVV @VVV\\ 
\sco_\piii(-3) @>>> 3\sco_\piii(-2) @>>> 3\sco_\piii(-1) @>>> \sci_{\{x\}} 
\end{CD}
\]
and applying a result atributed by Peskine and Szpiro \cite[Prop.~2.5]{ps} to 
D. Ferrand, one gets that $\sci_Y$ admits a resolution of the form$\, :$ 
\[
0 \lra 3\sco_\piii(-5) \lra 6\sco_\piii(-4) \lra 3\sco_\piii(-2) \oplus 
\sco_\piii(-3) \lra \sci_Y \lra 0\, . 
\] 
Noticing that $\tH^1(\sci_Y(2)) = 0$, consider the commutative diagram with 
exact rows$\, :$ 
\[
\begin{CD} 
0 @>>> \tH^0(\sci_Y(2)) \otimes \sco_\piii @>>> 
\tH^0(\sck(2)) \otimes \sco_\piii 
@>>> \tH^0(\sci_{\{x\}}(1)) \otimes \sco_\piii @>>> 0\\ 
@. @VV{\text{ev}}V @VV{\text{ev}}V @VV{\text{ev}}V\\ 
0 @>>> \sci_Y(2) @>>> \sck(2) @>>> \sci_{\{x\}}(1) @>>> 0 
\end{CD}
\]  
The cokernel of the evaluation morphism $\tH^0(\sci_Y(2)) \otimes \sco_\piii 
\ra \sci_Y(2)$ is $(\sci_Y/\sci_X)(2)$. Let $\scn$ be the kernel of the 
evaluation morphism $\tH^0(\sci_{\{x\}}(1)) \otimes \sco_\piii \ra 
\sci_{\{x\}}(1)$. We will show that if $\partial \colon \scn \ra 
(\sci_Y/\sci_X)(2)$ is the connecting morphism induced by the above diagram 
then $\partial = 0$. 

\emph{Indeed}, $\scn(1)$ is globally generated and $\tH^0(\scn(1))$ has a 
$k$-basis consisting of the elements $h_i \otimes h_j - h_j \otimes h_i$, 
$0 \leq i < j \leq 2$, where $h_0,\, h_1,\, h_2$ is a $k$-basis of 
$\tH^0(\sci_{\{x\}}(1))$. $h_i \in \tH^0(\sci_{\{x\}}(1))$ can be lifted to an 
element $(q_i\, ,\, h_i)^{\text{t}}$ of $\tH^0(\sck(2)) \subset 
\tH^0(\sco_\piii(2) \oplus \sco_\piii(1))$. It follows that, for $0 \leq i < j 
\leq 2$, $q_ih_j - q_jh_i \in \tH^0(\sci_Y(3))$ and $\, :$ 
\[
\partial(1)(h_i \otimes h_j - h_j \otimes h_i) = \text{the image of}\ 
q_ih_j - q_jh_i\  \text{into}\  \tH^0((\sci_Y/\sci_X)(3))\, . 
\]    
We, consequently, have to prove that $q_ih_j - q_jh_i$ vanishes in $x$. But  
one has$\, :$ 
\[ 
q_if + h_ig = 0\, ,\  q_jf + h_jg = 0\, .  
\]  
One deduces that $(q_ih_j - q_jh_i)g = 0$. Since $g$ does not vanish in $x$  
it follows that $q_ih_j - q_jh_i$ vanish in $x$. Consequently,  
$\partial(1)(h_i \otimes h_j - h_j \otimes h_i) = 0$, $0 \leq i < j \leq 2$, 
hence $\partial = 0$. But this \emph{contradicts} the fact that $\sck(2)$ is 
globally generated. Tracing back the origin of this contradiction one sees 
that it comes from our assumption that $Z \neq \emptyset$. It follows that 
$\phi_1$ is an epimorphism and Claim 4.4 is proven. 

\vskip2mm 

\noindent 
{\bf Claim 4.5.}\quad $E \simeq \sco_\piii(1) \oplus E_0$, \emph{where} 
$E_0(-2)$ \emph{is the kernel of an arbitrary epimorphism} $4\sco_\piii 
\ra \sco_\piii(2)$. 

\vskip2mm 

\noindent 
\emph{Indeed}, by Claim 4.3, $E(-2)$ is the kernel of an epimorphism 
$\phi \colon 4\sco_\piii \oplus \sco_\piii(-1) \ra \sco_\piii(2)$. By Claim 4.4, 
the component $\phi_1 \colon 4\sco_\piii \ra \sco_\piii(2)$ of $\phi$ is an 
epimorphism. Put $E_0 := \Ker \phi_1(2)$. It follows that $E$ can be 
realized as an extension$\, :$ 
\[
0 \lra E_0 \lra E \lra \sco_\piii(1) \lra 0\, . 
\]      
We will show that this extension must be trivial. \emph{Indeed}, let $\xi \in 
\tH^1(E_0(-1))$ be the element defining this extension. Since $E$ is 
globally generated, the map $\tH^0(E) \ra \tH^0(\sco_\piii(1))$ must be 
surjective. This implies that $\tH^0(\sco_\piii(1))\cdot \xi = 0$ inside 
$\tH^1(E_0)$. Using the (geometric) Koszul complex associated to the 
epimorphism $\phi_1$ one deduces that $E_0$ admits a resolution of the 
form$\, :$ 
\[
0 \lra \sco_\piii(-4) \lra 4\sco_\piii(-2) \lra 6\sco_\piii \lra E_0 \lra 0
\, . 
\] 
It follows that $\tH^1(E_0(-1)) \simeq \tH^3(\sco_\piii(-5))$ and 
$\tH^1(E_0) \simeq \tH^3(\sco_\piii(-4))$. Since the pairing 
$\tH^0(\sco_\piii(1)) \times \tH^3(\sco_\piii(-5)) \ra \tH^3(\sco_\piii(-4)) 
\simeq k$ is perfect, one gets that $\xi = 0$ and Claim 4.5 is proven.  

\vskip2mm 

\noindent 
{\bf Case 5.}\quad $F$ \emph{has spectrum} $(0,0,0,0)$. 

\vskip2mm 

\noindent 
In this case, $r = 3$ (hence $E = F(2)$), $c_3(F) = -4$ and $c_3 = 8$. It 
follows, using the spectrum, that $\h^1(E(l)) = \h^1(F(l+2)) = 0$ for $l \leq 
-4$, $\h^1(E(-3)) = 4$. By Riemann-Roch, $\h^1(E(-2)) = -\chi(F) = 6$ and 
$\h^0(E(-1)) - \h^1(E(-1)) = \chi(F(1)) = -3$. It follows, from 
Prop.~\ref{P:h0e(-1)geq2c2=12}, that $\h^0(E(-1)) \leq 1$. 
If $\h^0(E(-1)) = 0$ then $\h^1(E(-1)) = 3$ and Lemma~\ref{L:h1e(l)c2=12} 
implies that $\h^1(E) = 0$.

\vskip2mm 

\noindent 
{\bf Claim 5.1.}\quad \emph{If} $F$ \emph{is a stable rank} 3 \emph{vector 
bundle with} $c_1(F) = -1$, $c_2(F) = 4$, $c_3(F) = -4$ \emph{and spectrum} 
$k_F = (0,0,0,0)$ \emph{then} $F$ \emph{is the cohomology of a monad of the 
form}$\, :$ 
\begin{equation}\label{E:3o(-1)ra10ora4o(1)}
0 \lra 3\sco_\piii(-1) \overset{\beta}{\lra} 10\sco_\piii 
\overset{\alpha}{\lra} 4\sco_\piii(1) \lra 0\, . 
\end{equation}

%\vskip2mm 

\noindent 
\emph{Indeed}, using the spectrum, one gets that $\tH^1(F(l)) = 0$ for $l 
\leq -2$ and $\h^1(F(-1)) = 4$. Moreover, $\tH^2(F(l)) = 0$ for $l \geq -2$ 
and, by Riemann-Roch, $\h^2(F(-3)) = 3$ hence, by Serre duality, 
$\tH^1(F^\vee(l)) = 0$ for $l \leq -2$ and $\h^1(F^\vee(-1)) = 3$. Since 
$\tH^2(F(-2)) = 0$ and $\tH^3(F(-3)) \simeq \tH^0(F^\vee(-1))^\vee = 0$, the 
Castelnuovo-Mumford lemma (in the form stated in \cite[Lemma~1.21]{acm1}) 
implies that the graded $S$-module $\tH^1_\ast(F)$ is generated by 
$\tH^1(F(-1))$. Analogously, the graded $S$-module $\tH^1_\ast(F^\vee)$ is 
generated by $\tH^1(F^\vee(-1))$. One deduces that $F$ admits a Horrocks 
monad of the form$\, :$ 
\[
0 \lra 3\sco_\piii(-1) \lra B \lra 4\sco_\piii(1) \lra 0\, , 
\] 
with $B$ a direct sum of line bundles. Since $B$ has rank 10, $\tH^0(B(-1)) = 
0$, and $\tH^0(B^\vee(-1)) = 0$ (because $\tH^0(F^\vee(-1)) = 0$) one deduces 
that $B \simeq 10\sco_\piii$ and this proves Claim 5.1.  

\vskip2mm 

%\noindent 
Notice that the monads of form \eqref{E:3o(-1)ra10ora4o(1)}, with 
$\tH^0(\beta^\vee(1)) \colon \tH^0(10\sco_\piii(1)) \ra \tH^0(3\sco_\piii(2))$ 
surjective, can be put toghether into a family with irreducible base.  
The condition $\tH^0(\beta^\vee(1))$ surjective is equivalent to 
$\tH^1(F^\vee(1)) = 0$, $F$ being the cohomology sheaf of the monad.   

Recall that if $E$ is a globally generated vector bundle on $\piii$, we 
denote by $P(E)$ the dual of the evaluation epimorphism 
$\tH^0(E) \otimes_k \sco_\piii \ra E$ of $E$. 

\vskip2mm  

\noindent 
{\bf Claim 5.2.}\quad \emph{Let} $E$ \emph{be a globally generated vector 
bundle on} $\piii$, \emph{with} $c_1 = 5$, $c_2 = 12$, $c_3 = 8$, \emph{and 
such that} $\tH^i(E^\vee) = 0$, $i = 0,\, 1$, $\tH^0(E(-2)) = 0$, \emph{and} 
$\tH^1(E(-4)) = 0$. \emph{If} $\h^0(E(-1)) = 1$ \emph{then} $\tH^1(E) = 0$ 
\emph{and then there is an exact sequence}$\, :$ 
\[
0 \lra \sco_\piii(-1) \lra F^\prim(2) \oplus 4\sco_\piii \lra P(E) \lra 0\, ,
\]   
\emph{with} $F^\prim$ \emph{a} 4-\emph{instanton with} $\tH^0(F^\prim(1)) = 0$ 
\emph{and} $\tH^1(F^\prim(2)) = 0$.  

\vskip2mm 

\noindent 
\emph{Indeed}, it follows, from Prop.~\ref{P:eprimunstablec2=12},  
Prop.~\ref{P:h1e(-3)=0c2=12}, Prop.~\ref{P:h2e(-3)neq0c2=12}, and from the 
beginning of the proof of Prop.~\ref{P:h2e(-3)=0h1e(-3)neq0c2=12}, that 
either $E$ can be realized as an extension$\, :$ 
\[
0 \lra G(2) \lra E \lra \sco_\piii(1) \lra 0\, , 
\] 
where $G$ is a 4-instanton with $\h^0(G(1)) \leq 1$ or $E = F(2)$, where $F$ 
a stable rank 3 vector bundle with $c_1(F) = -1$, $c_2(F) = 4$, $c_3(F) = -4$ 
and spectrum $k_F = (0,0,0,0)$. 

Now, $P(E)$ has Chern classes $c_1(P(E)) = 5$, $c_2(P(E)) = 13$, 
$c_3(P(E)) = 13$. Moreover, $\h^0(P(E)(-1)) = \h^1(E^\vee(-1)) = \h^2(E(-3)) 
= 0$ and $\h^2(P(E)(-3)) = \h^3(E^\vee(-3)) = \h^0(E(-1)) = 1$. 
Lemma~\ref{L:h2e(-3)=1} implies, now, that one has exact sequences$\, :$ 
\begin{gather*} 
0 \lra \sco_\piii(-1) \lra E_1 \oplus 4\sco_\piii \lra P(E) \lra 0\, ,\\ 
0 \lra (r^\prim - 5)\sco_\piii \lra E_1 \lra \scf_1(2) \lra 0\, , 
\end{gather*}
where $r^\prim$ is the rank of $P(E)$ and $\scf_1$ is a stable rank 2 reflexive 
sheaf with Chern classes $c_1(\scf_1) = 0$, $c_2(\scf_1) = 4$, $c_3(\scf_1) 
= 0$. The last relation shows that $\scf_1$ is locally free. 

Now, $\h^1(\scf_1(2)) = \h^1(E_1) = \h^1(P(E)) = \h^2(E^\vee) = \h^1(E(-4)) = 
0$. Lemma~\ref{L:c1=0c2=4h1fprim(2)=0} from Appendix~\ref{A:instantons} 
implies that $\scf_1$ is a 4-instanton. It follows that$\, :$    
\[
\h^1(E) = \h^2(P(E)^\vee) = \h^1(P(E)(-4)) = \h^1(E_1(-4)) = \h^1(\scf_1(-2)) 
= 0\, ,  
\] 
hence, by Riemann-Roch, $\h^0(E) = \chi(E) = 8$. Since $E$ has rank 3, one 
deduces that $P(E)$ has rank $r^\prim = 5$. The other assertions from Claim 5.2 
are, now, clear.    

\vskip2mm 

\noindent 
{\bf Claim 5.3.}\quad \emph{Let} $E$ \emph{be a globally generated vector 
bundle on} $\piii$ \emph{with} $c_1 = 5$, $c_2 = 12$, $c_3 = 8$, \emph{such 
that} $\tH^i(E^\vee) = 0$, $i = 0,\, 1$, \emph{and} $\tH^0(E(-1)) = 0$. 
\emph{Then} $P(E)(-2)$ \emph{is the cohomology of a (not necessarily minimal) 
monad of the form}$\, :$ 
\[
0 \lra \sco_\piii(-1) \lra 6\sco_\piii \oplus 3\sco_\piii(-1) \lra 
3\sco_\piii(1) \lra 0\, . 
\]   

%\vskip2mm 

\noindent 
\emph{Indeed}, by Prop.~\ref{P:eprimunstablec2=12}, 
Prop.~\ref{P:h1e(-3)=0c2=12}, Prop.~\ref{P:h2e(-3)neq0c2=12}, and by Case 4 
above, either $E$ can be realized as a non-trivial extension$\, :$ 
\[
0 \lra G(2) \lra E \lra \sco_\piii(1) \lra 0, 
\] 
where $G$ is a 4-instanton with $\tH^0(G(1)) = 0$, or $E = F(2)$, where $F$ is 
a stable rank 3 vector bundle with $c_1(F) = -1$, $c_2(F) = 4$, $c_3(F) = -4$ 
and spectrum $k_F = (0,0,0,0)$. Since $\tH^0(E(-1)) = 0$, one gets, from 
Riemann-Roch, that $\h^1(E(-1)) = 3$. Lemma~\ref{L:h1e(l)c2=12} implies that 
$\tH^1(E) = 0$ and, by Riemann-Roch again, $\h^0(E) = 8$. It follows that 
$P(E)$ has rank $r^\prim = 5$ and Chern classes $c_1^\prim = 5$, 
$c_2^\prim = 13$, $c_3^\prim = 13$. Moreover, $\h^2(P(E)(-3)) = \h^0(E(-1)) = 
0$, $\h^0(P(E)(-1)) = \h^2(E(-3)) = 0$ and $\h^1(P(E)) = \h^1(E(-4)) = 0$.  
Two general global sections of $P(E)$ define an exact sequence$\, :$ 
\[
0 \lra 2\sco_\piii \lra P(E) \lra F^\prim(2) \lra 0\, , 
\]  
where $F^\prim$ is a rank 3 vector bundle with $c_1(F^\prim) = -1$, $c_2(F^\prim) 
= 5$, $c_3(F^\prim) = -1$. By Lemma~\ref{L:eprimunstable}, $F^\prim$ must be 
stable (if $L$ is a line in $\piii$ then $\h^2(\sci_L(-2)) = \h^1(\sco_L(-2)) 
= 1$). By Remark~\ref{R:p(e)}, $\h^2(F^\prim(-2)) = 2$, $\h^2(F^\prim(-1)) = 0$ 
and $\h^1(F^\prim(-1)) = \h^1(F(1)) = \h^1(E(-1)) = 3$. One deduces that 
$F^\prim$ has spectrum $(0,0,0,-1,-1)$.  

Now, $\h^1(P(E)(l)) = \h^1(F^\prim(l+2)) = 0$, for $l \leq -4$, and 
$\h^1(P(E)(-3)) = 3$. By Riemann-Roch, $\h^1(P(E)(-2)) = 6$ and 
$\h^1(P(E)(-1)) = 4$. Since $\tH^2(P(E)(-3)) = 0$ and $\tH^3(P(E)(-4)) \simeq 
\tH^0(P(E)^\vee)^\vee = 0$, the graded $S$-module $\tH^1_\ast(P(E))$ is generated 
in dergees $\leq -2$. We assert that the rank of the multiplication map 
$\mu \colon \tH^1(P(E)(-3)) \otimes S_1 \ra \tH^1(P(E)(-2))$ is $\geq 4$. 
\emph{Indeed}, this follows immediately from Remark~\ref{R:beilinson}. 
One deduces that the graded $S$-module $\tH^1_\ast(P(E))$ has three minimal  
generators in degree $-3$ and at most two minimal generators in degree $-2$. 

On the other hand, Lemma~\ref{L:h2e(-3)=0}(b) implies that the graded 
$S$-module $\tH^1_\ast(P(E)^\vee)$ is generated by $\tH^1(P(E)^\vee(1))$.  
Since $\h^2(P(E)^\vee) = \h^1(E) = 0$, one gets, from 
Lemma~\ref{L:h2e(-3)=0}(d)(f), that $\h^1(P(E)^\vee(1)) = \h^1(P(E)_H^\vee(1)) 
\leq 2$, for every plane $H \subset \piii$.  
It follows, now, that $P(E)(-2)$ is the cohomology of a Horrocks monad of the 
form$\, :$ 
\[
0 \lra 2\sco_\piii(-1) \lra B^\prim \lra 3\sco_\piii(1) \oplus 2\sco_\piii 
\lra 0\, ,  
\]
where $B^\prim$ is a direct sum of line bundles. $B^\prim$ has rank 12, 
$\tH^0(B^\prim(-1)) = 0$ and $\h^0(B^\prim) = \h^0(3\sco_\piii(1) \oplus 
2\sco_\piii) - \h^1(P(E)(-2)) = 8$ and $\h^0(B^{\prim \vee}(-2)) = 0$ hence 
$B^\prim \simeq 8\sco_\piii \oplus 4\sco_\piii(-1)$. Since there is no 
epimorphism $4\sco_\piii(-1) \ra 2\sco_\piii$, it follows that the component 
$8\sco_\piii \ra 2\sco_\piii$ of the morphism $B^\prim \ra 3\sco_\piii(1) \oplus 
2\sco_\piii$ is nonzero hence $P(E)(-2)$ is, actually, the cohomology of a 
monad of the form$\, :$ 
\[
0 \lra 2\sco_\piii(-1) \overset{\beta}{\lra} 7\sco_\piii \oplus 4\sco_\piii(-1) 
\overset{\alpha}{\lra} 3\sco_\piii(1) \oplus \sco_\piii \lra 0\, . 
\] 

\vskip2mm 

\noindent 
{\bf Claim 5.3.1.}\quad \emph{The component} $\alpha_{21} \colon 7\sco_\piii 
\ra \sco_\piii$ \emph{of} $\alpha$ \emph{is non-zero}. 

\vskip2mm 

\noindent 
\emph{Indeed}, assume, by contradiction, that it is 0. Since there is no 
epimorphism $3\sco_\piii(-1) \ra \sco_\piii$ one deduces that the component 
$\beta_2 \colon 2\sco_\piii(-1) \ra 4\sco_\piii(-1)$ of $\beta$ is 0, too. 
Denoting by $Q$ the cokernel of the component $\beta_1 \colon 2\sco_\piii(-1) 
\ra 7\sco_\piii$ of $\beta$, one gets an exact sequence$\, :$ 
\[
0 \lra P(E)(-2) \lra Q \oplus \Omega_\piii \overset{\phi}{\lra}  
3\sco_\piii(1) \lra 0\, .
\]   
Let $\phi_1 \colon Q \ra 3\sco_\piii(1)$ and $\phi_2 \colon \Omega_\piii \ra 
3\sco_\piii(1)$ be the components of $\phi$. One has an exact sequence$\, :$ 
\[
0 \lra \Ker \phi_1 \lra P(E)(-2) \lra \Omega_\piii 
\overset{{\overline \phi}_2}{\lra} \Cok \phi_1 \lra 0\, , 
\]
where ${\overline \phi}_2$ denotes the composite map $\Omega_\piii 
\overset{\phi_2}{\lra} 3\sco_\piii(1) \ra \Cok \phi_1$. 

Since $c_1(\Omega_\piii(2)) = 2$ and $\Cok \phi_1(-1)$ is globally generated 
(because $\Cok \phi_1$ is a quotient of $3\sco_\piii(1)$), 
Lemma~\ref{L:suppcok} implies that the support of $\Cok \phi_1$ is 
0-dimensional or empty. 
It follows that the kernel $\scf$ of $\phi_1$ is a rank 2 reflexive sheaf with 
$c_1(\scf) = -1$. Dualizing the exact sequence$\, :$ 
\[
0 \lra \scf \lra Q \overset{\phi_1}{\lra} 3\sco_\piii(1) \lra \Cok \phi_1 
\lra 0\, , 
\]
one gets an exact sequence$\, :$ 
\[
0 \lra 3\sco_\piii(-1) \lra Q^\vee \lra \scf(1) \lra 0 
\]
(because $\scf^\vee \simeq \scf(1)$). Dualizing again and using 
\cite[Prop.~2.6]{ha} one gets that the length of $\Cok \phi_1$ is equal to 
$c_3(\scf(1)) = c_3(\scf(2)) = c_3(Q^\vee(1)) = 3$. 

But, according to Lemma~\ref{L:5oraomega(2)} from 
Appendix~\ref{A:miscellaneous}, any morphism $m\sco_\piii \ra \Omega_\piii(2)$ 
is either an epimorphism or the support of its cokernel is at least 
1-dimensional. Using the exact sequence$\, :$ 
\[
P(E) \lra \Omega_\piii(2)  
\xra{{\overline \phi}_2(2)} \Cok \phi_1(2) \lra 0\, ,
\]    
one gets, now, a \emph{contradiction}. Claim 5.3.1 is proven.  

\vskip2mm 

One deduces, from Claim 5.3.1, that $P(E)(-2)$ is the cohomology of a 
monad of the form$\, :$ 
\begin{equation}\label{E:2o(-1)ra6o+4o(-1)ra3o(1)} 
0 \lra 2\sco_\piii(-1) \overset{\beta^\prim}{\lra} 6\sco_\piii \oplus 
4\sco_\piii(-1) \overset{\alpha^\prim}{\lra} 3\sco_\piii(1) \lra 0\, . 
\end{equation}

\vskip2mm 

\noindent 
{\bf Claim 5.3.2.}\quad \emph{The component} $\beta^\prim_2 \colon 
2\sco_\piii(-1) \ra 4\sco_\piii(-1)$ \emph{of} $\beta^\prim$ \emph{is non-zero}. 

\vskip2mm 

\noindent 
\emph{Indeed}, we will show that if $M$ is the cohomology of a monad of the 
above form with $\beta^\prim_2 = 0$ then $\tH^0(M(1)) \neq 0$. Since we 
noticed above that $\tH^0(P(E)(-1)) = 0$ this will imply Claim 5.3.2.   

It follows, from Lemma~\ref{L:mora2o(1)} from Appendix~\ref{A:h0e(-2)neq0}, 
that the monads of the form \eqref{E:2o(-1)ra6o+4o(-1)ra3o(1)} with 
$\beta_2^\prim = 0$ can be put toghether into a family with irreducible base. 

We show, next, that there exist monads of the form 
\eqref{E:2o(-1)ra6o+4o(-1)ra3o(1)} with $\beta_2^\prim = 0$ such that the 
support of the cokernel of the component $\alpha_1^\prim \colon 6\sco_\piii 
\ra 3\sco_\piii(1)$ of $\alpha^\prim$ is 1-dimensional. 

\emph{Indeed}, let $Y$ be the union of three mutually disjoint lines 
$L_0, \, L_1,\, L_2$ in $\piii$. Taking the direct sum of the resolutions of 
$\sco_{L_i}(1)$ one gets a resolution$\, :$ 
\[
0 \lra 3\sco_\piii(-1) \overset{d^{-2}}{\lra} 6\sco_\piii 
\overset{d^{-1}}{\lra} 3\sco_\piii(1) \overset{d^0}{\lra} \sco_Y(1) \lra 0\, . 
\]   
One has $\Cok d^{-2 \vee} \simeq \omega_Y(3) \simeq \sco_Y(1)$. Consider the 
exact sequence $0 \ra \sco_\piii(1) \overset{u}{\ra} 3\sco_\piii(1) 
\overset{v}{\ra} 2\sco_\piii(1) \ra 0$ with$\, :$ 
\[
u = (1\, ,\, 1\, ,\, 1)^{\text{t}} \  \text{and}\  
v = \begin{pmatrix} -1 & 1 & 0\\ 
                     0 & -1 & 1  
\end{pmatrix}\, . 
\]
Since $d^0 \circ u$ is an epimorphism it follows that $v \circ d^{-1}$ is an 
epimorphism. One thus gets a complex$\, :$ 
\[
2\sco_\piii(-1) \xra{(v \circ d^{-1})^\vee} 6\sco_\piii \xra{d^{-2 \vee}} 
3\sco_\piii(1) 
\]
with $(v \circ d^{-1})^\vee$ a locally split monomorphism and with 
$\Cok d^{-2 \vee} \simeq \sco_Y(1)$. Considering an epimorphism $4\sco_\piii(-1) 
\ra \Cok d^{-2 \vee}$ and lifting it to a morphism $4\sco_\piii(-1) \ra 
3\sco_\piii(1)$ one gets a monad of the form \eqref{E:2o(-1)ra6o+4o(-1)ra3o(1)} 
with $\beta_2^\prim = 0$ and $\Cok \alpha_1^\prim \simeq \sco_Y(1)$. 

We show, finally, that if $M$ is the cohomology of the form 
\eqref{E:2o(-1)ra6o+4o(-1)ra3o(1)} with $\beta_2^\prim = 0$ and such that 
the support of $\Cok \alpha_1^\prim$ is 1-dimensional then $\tH^0(M(1)) 
\neq 0$. 

\emph{Indeed}, let $\beta_1^\prim \colon 2\sco_\piii(-1) \ra 6\sco_\piii$ be 
the first component of $\beta^\prim$ and let $Q$ denote its cokernel. It is 
a rank 4 vector bundle, with $c_1(Q) = 2$. Then one has an exact 
sequence$\, :$ 
\[
0 \lra M \lra Q \oplus 4\sco_\piii(-1) \overset{\alpha^\secund}{\lra} 
3\sco_\piii(1) \lra 0\, , 
\] 
with $\alpha^\secund$ induced by $\alpha^\prim$. Since the support of the 
cokernel of the component $\alpha_1^\secund \colon Q \ra 3\sco_\piii(1)$ of 
$\alpha^\secund$ is 1-dimensional it follows that $\Ker \alpha_1^\secund \simeq 
\sco_\piii(-1)$. But $\Ker \alpha_1^\secund \subseteq M$ hence  
$\tH^0(M(1)) \neq 0$. 

Since the monads of the form \eqref{E:2o(-1)ra6o+4o(-1)ra3o(1)} with 
$\beta_2^\prim = 0$ and with the support of $\Cok \alpha_1^\prim$ 1-dimensional 
are general among the monads of the form \eqref{E:2o(-1)ra6o+4o(-1)ra3o(1)} 
with $\beta_2^\prim = 0$ it follows that if $M$ is the cohomology of any  
monad \eqref{E:2o(-1)ra6o+4o(-1)ra3o(1)} with $\beta_2^\prim = 0$  
then $\tH^0(M(1)) \neq 0$ and, as we noticed at the beginning of the 
proof of Claim 5.3.2, this suffices to prove the claim. 

\vskip2mm 

Claim 5.3 follows, now, from Claim 5.3.1 and Claim 5.3.2. 

\vskip2mm 

\noindent 
{\bf Construction 5.4.}\quad We want to show that there exist stable rank 3 
vector bundles $F$ on $\piii$, with $c_1(F) = -1$, $c_2(F) = 4$, $c_3(F) = 
-4$ and spectrum $(0,0,0,0)$, such that $\tH^0(F^\vee) = 0$, 
$\tH^1(F^\vee(1)) = 0$, $E := F(2)$ is globally generated and $\tH^1(E) = 0$.   

\vskip2mm 

Consider a rank 3 vector bundle $F$ on $\piii$ such that its dual $F^\vee$ can 
be realised as a non-trivial extension$\, :$ 
\[
0 \lra F^\prim \lra F^\vee \lra \sco_\piii(1) \lra 0\, , 
\]
defined by a non-zero element $\xi$ of $\tH^1(F^\prim(-1))$, where $F^\prim$ 
is a 4-instanton bundle such that $F^\prim(2)$ is globally generated. 
The existence of such 4-instantons was proven by Chiodera and Ellia 
\cite[Lemma~2.10]{ce}; a different proof was given in \cite{acm3}. 
Dualizing the extension, one gets an exact sequence$\, :$ 
\[
0 \lra \sco_\piii(-1) \lra F \lra F^\prim \lra 0\, . 
\]
The Chern classes of $F$ are $c_1(F) = -1$, $c_2(F) = 4$, $c_3(F) = -4$. 
Moreover, $\tH^0(F) = 0$ and $\tH^0(F^\vee(-1)) = 0$ hence $F$ 
is stable. Since $\tH^i(F(-2)) \simeq \tH^i(F^\prim (-2)) = 0$, $i = 1,\, 2$, 
the spectrum of $F$ is $(0,0,0,0)$. 
$F(2)$ is globally generated and $\tH^1(F(2)) \simeq \tH^1(F^\prim(2)) = 0$ 
(see \cite[Remark~6.4]{acm1}). 

\vskip2mm 

\noindent 
{\bf Claim 5.4.1.}\quad \emph{If} $\xi$ \emph{is a general element of} 
$\tH^1(F^\prim(-1))$ \emph{then} $\tH^0(F^\vee) = 0$. 

\vskip2mm 

\noindent 
\emph{Indeed}, we have to show that, for a general $\xi \in 
\tH^1(F^\prim(-1))$, $h\xi \neq 0$ in $\tH^1(F^\prim)$, $\forall \, 0 \neq h \in 
\tH^0(\sco_\piii(1))$. If $0 \neq h \in \tH^0(\sco_\piii(1))$ and if $H \subset 
\piii$ is the plane of equation $h = 0$ then one has an exact sequence$\, :$ 
\[
0 \lra \tH^0(F_H^\prim) \lra \tH^1(F^\prim(-1)) \overset{h}{\lra} 
\tH^1(F^\prim)\, . 
\]
Since $\tH^1(F^\prim(-2)) = 0$, one has $\tH^0(F_H^\prim(-1)) = 0$. If 
$\tH^0(F_H^\prim) \neq 0$ then one must have an exact sequence$\, :$ 
\begin{equation}\label{E:fhprim} 
0 \lra \sco_H \lra F_H^\prim \lra \sci_{Z,H} \lra 0\, , 
\end{equation}
where $Z$ is a 0-dimensional closed subscheme of $H$, of length 4. It 
follows that, in this case, $\h^0(F_H^\prim) = 1$. According to Barth's 
restriction theorem \cite{b} (see, also, Ein et al. \cite[Thm.~3.3]{ehv} 
or \cite[{\bf 2.7}]{co3} for different proofs) the set of planes $H \subset 
\piii$ for which $\tH^0(F_H^\prim) \neq 0$ is a proper closed subset of the 
dual projective space $\p^{3 \vee}$. Since, by Riemann-Roch, 
$\h^1(F^\prim(-1)) = 4$, Claim 5.4.1 is proven. 

\vskip2mm 

\noindent 
{\bf Claim 5.4.2.}\quad $\tH^1(F^\vee(1)) = 0$. 

\vskip2mm 

\noindent 
\emph{Indeed}, since $F^\prim(2)$ is globally generated, it follows that, for 
every line $L \subset \piii$, one must have $F_L^\prim \simeq \sco_L(a) \oplus 
\sco_L(-a)$, with $0 \leq a \leq 2$. One deduces that 
$\tH^1(F_L^\vee(1)) = 0$, for every line $L \subset \piii$.  

We assert, now, that $\h^1(F_H^\vee) \leq 2$, for every plane $H \subset 
\piii$. Indeed, if $\tH^0(F_H^\prim) = 0$ then, by Riemann-Roch, 
$\h^1(F_H^\prim) = 2$ and the assertion follows. If $\tH^0(F_H^\prim) \neq 0$ 
then, as we saw above, $F_H^\prim$ can be realized as an extension 
\eqref{E:fhprim}. Since $F^\prim(2)$ is globally generated, $Z$ must be a 
complete intersection of type $(2,2)$ in $H$. One deduces a presentation of 
the form$\, :$ 
\[
0 \lra \sco_H(-4) \lra \sco_H \oplus 2\sco_H(-2) \lra F_H^\prim \lra 0\, . 
\]  
It follows that $\tH^1(F_H^\prim(-1)) \subset \tH^2(\sco_H(-5))$ and that 
$\tH^1(F_H^\prim) \simeq \tH^2(\sco_H(-4))$. In particular, $\h^1(F_H^\prim) = 
3$. Since $\tH^i(F^\prim(-2)) = 0$, $i = 1,\, 2$, the restriction map 
$\tH^1(F^\prim(-1)) \ra \tH^1(F_H^\prim(-1))$ is bijective. In particular, 
$\xi \vb H \neq 0$ in $\tH^1(F_H^\prim(-1))$. One deduces that the vector 
space of linear forms $\ell \in \tH^0(\sco_H(1))$ annihilating $\xi \vb H$ is 
a proper subspace of $\tH^0(\sco_H(1))$. Using the exact sequence$\, :$ 
\[
0 \lra F_H^\prim \lra F_H^\vee \lra \sco_H(1) \lra 0\, ,  
\]
one gets that $\h^1(F_H^\vee) \leq 2$. 

Applying the Bilinear Map Lemma \cite[Lemma~5.1]{ha} to the bilinear map 
\[
\tH^1(F_H^\vee(1))^\vee \times \tH^0(\sco_H(1)) \ra 
\tH^1(F_H^\vee)^\vee 
\]
deduced from the multiplication map 
$\tH^1(F_H^\vee) \times \tH^0(\sco_H(1)) \ra \tH^1(F_H^\vee(1))$, 
one deduces, now, that $\tH^1(F_H^\vee(1)) = 0$, for every plane 
$H \subset \piii$. 

We assert, next, that $\h^1(F^\vee) \leq 3$. Indeed, assume that $\xi$ 
would be annihilated by two linearly independent linear forms $h_0,\, h_1 \in 
\tH^0(\sco_\piii(1))$. Let $L \subset \piii$ be the line of equations $h_0 = 
h_1 = 0$. Tensorizing by $F^\prim$ the exact sequence$\, :$ 
\[
0 \lra \sco_\piii(-1) \lra 2\sco_\piii \lra \sci_L(1) \lra 0\, , 
\] 
one would deduce that $\tH^0(F^\prim \otimes \sci_L(1)) \neq 0$. But \emph{this 
is not possible} because $\tH^0(F^\prim(1)) = 0$ (see \cite[Remark~6.4]{acm1}). 
It remains that $\xi$ is annihilated by at most one linear form. Using the 
extension defining $F^\vee$ and the fact that $\h^1(F^\prim) = 6$ one gets   
that $\h^1(F^\vee) \leq 3$. 

Applying the Bilinear Map Lemma to the bilinear map 
$\tH^1(F^\vee(1))^\vee \times \tH^0(\sco_\piii(1)) \ra \tH^1(F^\vee)^\vee$ 
deduced from the multiplication map $\tH^1(F^\vee) \times \tH^0(\sco_\piii(1)) 
\ra \tH^1(F^\vee(1))$ one gets, now, that $\tH^1(F^\vee(1)) = 0$. 

\vskip2mm 

\noindent 
{\bf Construction 5.5.}\quad We want to show that there exist stable rank 3 
vector bundles $F$ on $\piii$, with $c_1(F) = -1$, $c_2(F) = 4$, $c_3(F) = -4$ 
and spectrum $(0,0,0,0)$, such that $\tH^0(F(1)) = 0$, $\tH^1(F^\vee(1)) = 0$, 
$E := F(2)$ is 1-regular and the multiplication map $S_1 \otimes \tH^0(E) 
\ra \tH^0(E(1))$ is surjective (hence, in particular, $E$ is globally 
generated).   

\vskip2mm 

Consider a nonsingular rational curve $C \subset \piii$, of degree 
6, with $\h^0(\sci_C(3)) = 1$ (hence $\tH^1(\sci_C(3)) = 0$). The existence of 
such curves can be shown as follows$\, :$ 
let $X \subset \piii$ be a nonsigular cubic surface, obtained as the blow-up 
$\pi \colon X \ra \pii$ of $\pii$ in six general points $P_1, \ldots , P_6$,  
embedded in $\piii$ such that $\sco_X(1) \simeq \pi^\ast\sco_\pii(3) \otimes 
\sco_X[-E_1 - \cdots - E_6]$, where $E_i := \pi^{-1}(P_i)$.   
Let $C$ to be the strict transform of an irreducible quartic curve 
${\overline C} \subset \pii$, having nodes at $P_1,\, P_2,
\, P_3$ and containing none of the points $P_4,\, P_5,\, P_6$. Then $C \subset 
\piii$ is a rational curve of degree 6 and it is contained in only one cubic 
surface (namely $X$) because it has six 4-secants (namely the strict 
transforms of the lines $\overline{P_iP_j}$, $4 \leq i < j \leq 6$, and the 
strict transforms of the conics passing through $\{P_1, \ldots , P_6\} 
\setminus \{P_i\}$, $1 \leq i \leq 3$). Fix an isomorphism $\nu \colon \pj 
\izo C$.  

Consider an epimorphism $\delta \colon \sco_\piii(1) \oplus \sco_\piii \ra 
\omega_C(2)$ defined by a global section $s$ of $\omega_C(1) \simeq 
\sco_\pj(4)$ and a global section $t$ of $\omega_C(2) \simeq \sco_\pj(10)$.  
$\delta$ defines an extension$\, :$ 
\begin{equation}\label{E:fveescic(2)} 
0 \lra \sco_\piii \oplus \sco_\piii(-1) \lra F^\vee \lra \sci_C(2) \lra 0\, , 
\end{equation}     
with $F^\vee$ the dual of a rank 3 vector bundle $F$ with $c_1(F) = -1$, 
$c_2(F) = 4$, $c_3(F) = -4$. Dualizing the extension, one gets an exact 
sequence$\, :$ 
\[
0 \lra \sco_\piii(-2) \lra F \lra \sco_\piii(1) \oplus \sco_\piii 
\overset{\delta}{\lra} \omega_C(2) \lra 0\, . 
\]
One has $\tH^0(F^\vee(-1)) = 0$ and $\tH^0(F) = 0$ hence $F$ is stable. 
Moreover, $\tH^1(F(-2)) \simeq \tH^2(F^\vee(-2))^\vee = 0$ and 
$\tH^2(F(-2)) \simeq \tH^1(F^\vee(-2))^\vee = 0$ hence $F$ has spectrum 
$(0,0,0,0)$. 
Besides, $\tH^1(F^\vee(1)) \simeq \tH^1(\sci_C(3)) = 0$. 

\vskip2mm 

\noindent 
{\bf Claim 5.5.1.}\quad \emph{If the zero divisor} $Z := (s)_0$ \emph{of the 
global section} $s$ \emph{of} $\omega_C(1) \simeq \sco_\pj(4)$ \emph{consists 
of four simple points not contained in a plane, then} $\tH^0(F(1)) = 0$ 
\emph{and} $F$ \emph{is} 3-\emph{regular}. 

\vskip2mm 

\noindent 
\emph{Indeed}, if $\sck$ is the kernel of $\delta \colon \sco_\piii(1) \oplus 
\sco_\piii \ra \omega_C(2)$ then one has exact sequences$\, :$ 
\begin{gather*} 
0 \lra \sco_\piii(-2) \lra F \lra \sck \lra 0\, ,\\ 
0 \lra \sci_C(1) \lra \sck \lra \sci_Z \lra 0\, , 
\end{gather*} 
hence $\tH^0(F(1)) \izo \tH^0(\sck(1)) = 0$ and $\tH^1(F(2)) \izo 
\tH^1(\sck(2)) = 0$. Moreover, $\tH^2(F(1)) = 0$ (because $F$ has spectrum 
$(0,0,0,0)$) and $\tH^3(F) \simeq \tH^0(F^\vee(-4))^\vee = 0$ hence $F$ is 
3-regular. 

\vskip2mm 

\noindent 
{\bf Claim 5.5.2.}\quad \emph{For a general choice of the epimorphism} 
$\delta \colon \sco_\piii(1) \oplus \sco_\piii \ra \omega_C(2)$ \emph{defining 
the extension} \eqref{E:fveescic(2)}, \emph{the multiplication map} 
$S_1 \otimes \tH^0(F(2)) \ra \tH^0(F(3))$ \emph{is surjective}. 

\vskip2mm 

\noindent 
\emph{Indeed}, denoting, as above, by $\sck$ the kernel of $\delta$, it 
suffices to show that the multiplication map $S_1 \otimes \tH^0(\sck(2)) \ra 
\tH^0(\sck(3))$ is surjective. Let $\e$ denote the restriction of $\delta$ to 
the nonsingular cubic surface $X$ containing $C$. $\delta$ can be written as 
the composite morphism$\, :$ 
\[
\sco_\piii(1) \oplus \sco_\piii \lra \sco_X(1) \oplus \sco_X 
\overset{\e}{\lra} \omega_C(2)\, . 
\] 
The kernel $K$ of $\e$ is a rank 2 vector bundle on $X$ and one has an exact 
sequence$\, :$ 
\[
0 \lra \sco_\piii(-2) \oplus \sco_\piii(-3) \lra \sck \lra K \lra 0\, . 
\]

\noindent 
$\bullet$\quad We assert, firstly, that there exist epimorphisms $\e \colon 
\sco_X(1) \oplus \sco_X \ra \omega_C(2)$ such that $K := \Ker \e$ has the 
property that the multiplication map $S_1 \otimes \tH^0(K(2)) \ra \tH^0(K(3))$ 
is surjective. 

\vskip2mm 

\noindent 
\emph{Indeed}, consider a line $L_0 \subset \pii$, containing none of the 
points $P_1 , \ldots , P_6$, and intersecting $\overline C$ in four distinct 
points. $C_0 := \pi^{-1}(L_0) \subset X$ is a twisted cubic curve in $\piii$. 
Since $C_0 \cap C$ consists of four simple points it follows that 
$\sco_X[C_0] \vb C \simeq \omega_C(1)$. Moreover$\, :$ 
\[
\sco_X(1) \otimes \sco_X[C_0] \simeq \pi^\ast \sco_\pii(4) \otimes 
\sco_X[-E_1 - \cdots - E_6]\, . 
\] 
Let ${\overline D} \subset \pii$ be a nonsingular quartic curve containing 
$P_1 , \ldots , P_6$ and intersecting $L_0$ in four \emph{general} simple 
points [the map $\tH^0(\sci_{\{P_1, \ldots , P_6\} , \pii}(3)) \ra 
\tH^0(\sco_{L_0}(3))$ is injective (since $\{P_1, \ldots , P_6\}$ is not 
contained in a conic) hence bijective$\, ;$ it follows that the map 
$\tH^0(\sci_{\{P_1, \ldots , P_6\} , \pii}(4)) \ra \tH^0(\sco_{L_0}(4))$ is 
surjective]. 
The strict transform $D \subset X$ of $\overline D$ is a curve of degree 6, 
belonging to the linear system $\vert \, \sco_X(1) \otimes \sco_X[C_0] \, 
\vert$, and such that $\Gamma := D \cap C_0$ consists of four \emph{general} 
simple points of $C_0$. In particular, one can assume that none of the points 
of $\Gamma$ belongs to $C$. Let $\phi_0$ (resp., $\psi$) be a global section 
of $\sco_X[C_0]$ (resp., $\sco_X(1) \otimes \sco_X[C_0]$) such that its zero 
divisor on $X$ is $C_0$ (resp., $D$). One has an exact sequence$\, :$ 
\[
0 \lra \sco_X[-C_0] \xra{\left(\begin{smallmatrix} -\psi\\ \phi_0 
\end{smallmatrix}\right)} \sco_X(1) \oplus \sco_X \xra{(\phi_0 \, ,\, \psi)} 
\sci_{\Gamma , X}(1) \otimes \sco_X[C_0] \lra 0\, . 
\] 
Notice that the multiplication by $\phi_0 \colon \sco_X[-C_0] \ra \sco_X$ can 
be identified, modulo the isomorphism $\sco_X[-C_0] \simeq \sci_{C_0 , X}$,
with the canonical inclusion $\sci_{C_0 , X} \hookrightarrow \sco_X$. 

Let $\e$ denote the composite epimorphism$\, :$ 
\[
\sco_X(1) \oplus \sco_X \xra{(\phi_0 \, ,\, \psi)} \sci_{\Gamma , X}(1) \otimes 
\sco_X[C_0] \lra \left(\sci_{\Gamma , X}(1) \otimes \sco_X[C_0]\right) \vb C 
\simeq \omega_C(2)\, .  
\]
If $K$ is the kernel of $\e$ then one has an exact sequence$\, :$ 
\[
0 \lra \sco_X[-C_0] \lra K \lra \sci_{\Gamma , X}(1) \otimes \sco_X[C_0 - C] 
\lra 0\, . 
\]
Now, for $1 \leq l \leq 3$, let $L_l \subset X$ be the strict transform of the 
line in $\pii$ joining the points of $\{P_1,\, P_2,\, P_3\} \setminus \{P_l\}$. 
$L_l$ is, of course, a line in $\piii$. Since the effective divisor on 
$\pii$$\, :$ 
\[
L_0 + \overline{P_2P_3} + \overline{P_3P_1} + \overline{P_1P_2} 
\]
has nodes at $P_1,\, P_2,\, P_3$ and contains none of the points $P_4,\, P_5,\, 
P_6$, it follows that$\, :$ 
\[
C_0 + L_1 + L_2 + L_3 \sim C 
\]
as divisors on $X$. Putting $Y := L_1 \cup L_2 \cup L_3$, one gets an 
exact sequence$\, :$ 
\begin{equation}\label{E:ic0kiygamma}
0 \lra \sci_{C_0 , X} \lra K \lra \sci_{Y \cup \Gamma , X}(1) \lra 0\, , 
\end{equation} 
where the component $\sci_{C_0 , X} \ra \sco_X$ of the composite map 
$\sci_{C_0 , X} \ra K \ra \sco_X(1) \oplus \sco_X$ is the canonical inclusion. 
Since the ideal sheaf $\sci_{C_0} \subset \sco_\piii$ is 2-regular, one deduces 
that, in order to show that the multiplication map $S_1 \otimes \tH^0(K(2)) 
\ra \tH^0(K(3))$ is surjective, it suffices to show that the multiplication 
map $S_1 \otimes \tH^0(\sci_{Y \cup \Gamma , X}(3)) \ra 
\tH^0(\sci_{Y \cup \Gamma , X}(4))$ is surjective. In order to verify the latter 
fact it suffices to check that $L_1,\, L_2,\, L_3$ and $\Gamma$ satisfy the 
hypotheses of Lemma~\ref{L:3lines+4points} from Appendix~\ref{A:miscellaneous} 
and that is exactly what we are going to do next. 

Since $\Gamma$ consists of four simple points on the twisted cubic curve $C_0$ 
it is not contained in a plane. We assert that, for $1 \leq l \leq 3$, there 
is no quadric surface $Q^\prim \subset \piii$ containing $(Y \setminus L_l) \cup 
C_0$. \emph{Indeed}, if such a surface would exist it would be nonsingular. 
Fix, in this case, an isomorphism $Q^\prim \simeq \pj \times \pj$ such that 
the components of $Y \setminus L_l$ belong to the linear system $\vert \, 
\sco_{Q^\prim}(1,0) \, \vert$. Since $L_p \cap C_0$ consists of a simple point, 
$p = 1,\, 2,\, 3$, $C_0$ must belong to the linear system $\vert \, 
\sco_{Q^\prim}(2,1) \, \vert$ hence the divisor $(Y \setminus L_l) + C_0$ 
belongs to $\vert \, \sco_{Q^\prim}(4,1) \, \vert$. But $(Y \setminus L_l) + 
C_0 \subset Q^\prim \cap X$ which is a divisor of type $(3,3)$ on $Q^\prim$ and 
this is a \emph{contradiction}. 

It follows that the restriction map $\tH^0(\sci_{Y \setminus L_l}(2)) \ra 
\tH^0(\sco_{C_0}(2))$ is injective, $l = 1,\, 2,\, 3$. Since $\Gamma$ consists 
of four \emph{general} points of $C_0$, one can assume that 
$\tH^0(\sci_{(Y \setminus L_l) \cup \Gamma}(2)) = 0$, $l = 1,\, 2,\, 3$. Moreover, 
one can assume that none of the points belongs to the quadric surface 
containing $Y$. This completes the verification of the hypotheses of 
Lemma~\ref{L:3lines+4points} and, with it, the proof of the assertion that the 
multiplication map $S_1 \otimes \tH^0(K(2)) \ra \tH^0(K(3))$ is surjective.  

\vskip2mm 

\noindent 
$\bullet \bullet$\quad We show, finally, that if $\sck$ is the kernel of a 
composite epimorphism$\, :$ 
\[
\sco_\piii(1) \oplus \sco_\piii \lra \sco_X(1) \oplus \sco_X \overset{\e}{\lra} 
\omega_C(2)\, , 
\]       
with $\e$ defined as above, then the multiplication map $\mu_\sck \colon 
S_1 \otimes \tH^0(\sck(2)) \ra \tH^0(\sck(3))$ is surjective. 

\vskip2mm 

\noindent 
\emph{Indeed}, consider the commutative diagram$\, :$ 
\[
\begin{CD} 
0 @>>> S_1 \otimes \tH^0(\sco_\p \oplus \sco_\p(-1)) @>>> 
S_1 \otimes \tH^0(\sck(2)) @>>> S_1 \otimes \tH^0(K(2)) @>>> 0\\ 
@. @VVV @VV{\mu_\sck}V @VV{\mu_K}V\\ 
0 @>>> \tH^0(\sco_\p(1) \oplus \sco_\p) @>>> \tH^0(\sck(3)) @>>> 
\tH^0(K(3)) @>>> 0 
\end{CD}
\]
We have just shown that $\mu_K$ is surjective. Let $N$ be its kernel. In 
order to show that $\mu_\sck$ is surjective it suffices to show that the 
connecting morphism $\partial \colon N \ra \tH^0(\sco_\piii)$ induced by the 
above diagram is non-zero (hence surjective). Consider, for that, a (cubic) 
equation $f = 0$ of $X$ in $\piii$ and let $q_0,\, q_1,\, q_2$ be a 
$k$-basis of $\tH^0(\sci_{C_0}(2))$. Since $C_0 \subset X$, there exist linear 
forms $h_0,\, h_1,\, h_2 \in S_1$ such that $f = h_0q_0 + h_1q_1 + h_2q_2$.    
Recall the exact sequence \eqref{E:ic0kiygamma} which shows, in particular, 
that there is a monomorphism $\sci_{C_0 , X} \ra K$ such that the component 
$\sci_{C_0 , X} \ra \sco_X$ of the composite morphism $\sci_{C_0 , X} \ra K \ra 
\sco_X(1) \oplus \sco_X$ is the canonical inclusion. Now, $q_i \vb X \in 
\tH^0(\sci_{C_0 , X}(2))$ defines a global section $\sigma_i$ of $K(2)$, 
$i = 0,\, 1,\, 2$. $\sigma_i$ can be lifted to a global section 
${\widetilde \sigma}_i$ of $\sck(2)$ whose image into $\tH^0(\sco_\piii(3) 
\oplus \sco_\piii(2))$ is of the form $(f_i , q_i)$. Since $h_0(q_0 \vb X) + 
h_1(q_1 \vb X) + h_2(q_2 \vb X) = f \vb X = 0$ it follows that 
$h_0 \otimes \sigma_0 + h_1 \otimes \sigma_1 + h_2 \otimes \sigma_2$ belongs to 
$N$. Since the composite morphism$\, :$ 
\[
\sco_\piii(-2) \oplus \sco_\piii(-3) \lra \sck \lra \sco_\piii(1) \oplus 
\sco_\piii
\]
is defined by the matrix $\left(\begin{smallmatrix} f & 0\\ 0 & f 
\end{smallmatrix}\right)$, one gets that$\, :$ 
\[
\partial(h_0 \otimes \sigma_0 + h_1 \otimes \sigma_1 + h_2 \otimes \sigma_2) 
= 1
\] 
and this completes the proof of the surjectivity of the multiplication map 
$S_1 \otimes \tH^0(\sck(2)) \ra \tH^0(\sck(3))$ and, with it, the proof of 
Claim 5.5.2.

\vskip2mm 
 
Taking into account Claim 5.1, the Constructions 5.4 and 5.5 show that the 
cohomology sheaf $F$ of a general monad of the form 
\eqref{E:3o(-1)ra10ora4o(1)} with $\tH^0(\beta^\vee(1))$ surjective 
has the following properties$\, :$ $\tH^0(F(1)) = 0$, $\tH^0(F^\vee) = 0$, 
$E := F(2)$ is 1-regular and the multiplication map $S_1 \otimes \tH^0(E) \ra 
\tH^0(E(1))$ is surjective (hence, in particular, $E$ is globally generated). 

\vskip2mm 

\noindent 
{\bf Case 6.}\quad $F$ \emph{has spectrum} $(1,0,0,0)$. 

\vskip2mm 

\noindent 
In this case, $r = 3$ (hence $E = F(2)$), $c_3(F) = -6$ and $c_3 = 6$. It 
follows, using the spectrum, that $\h^1(E(l)) = \h^1(F(l+2)) = 0$ for $l \leq 
-5$, $\h^1(E(-4)) = 1$ and $\h^1(E(-3)) = 5$. Since $\tH^0(E^\vee(1)) = 
\tH^0(F^\vee(-1)) = 0$ (because $F$ is stable) it follows, from 
Remark~\ref{R:h2e(-3)=0}(ii), that the graded $S$-module $\tH^1_\ast(E)$ is 
generated in degrees $\leq -3$. Actually, Remark~\ref{R:muh1e(-4)} implies  
that $\tH^1_\ast(E)$ has a minimal generator of degree $-4$ and another one of 
degree $-3$. Let 
\[
0 \lra E(-3) \lra E_3 \lra \sco_\piii(1) \oplus \sco_\piii \lra 0 
\]
be the extension defined by these two minimal generators. One has 
$\tH^1_\ast(E_3) = 0$. Since $\tH^2(E_3(-1)) \simeq \tH^2(E(-4)) = 0$ and 
$\tH^3(E_3(-2)) \simeq \tH^3(E(-5)) \simeq \tH^0(E^\vee(1))^\vee = 0$ it follows 
that $E_3$ is 1-regular. In particular, $E_3(1)$ is globally generated. Using 
the exact sequence$\, :$ 
\[
0 \lra F \lra E_3(1) \lra \sco_\piii(2) \oplus \sco_\piii(1) \lra 0\, , 
\] 
one gets easily that $E_3(1)$ has rank 5, $c_1(E_3(1)) = 2$, $c_2(E_3(1)) = 3$, 
$c_3(E_3(1)) = 4$, hence $c_1(P(E_3(1))) = 2$, $c_2(P(E_3(1))) = 1$, 
$c_3(P(E_3(1))) = 0$. Using the results of Sierra and Ugaglia \cite{su} (see, 
also, \cite[Pop.~2.2]{acm1}) one deduces that $P(E_3(1)) \simeq 2\sco_\piii(1)$. 
Put $t := \h^0((E_3(1))^\vee) = \h^0(F^\vee)$. By Lemma~\ref{L:h0fdual}, 
$\h^0(F^\vee) \leq 1 + \h^0(F(1))$ and, by Lemma~\ref{L:h0f(1)=0c2=12}, 
$\h^0(F(1)) = 0$ hence $t \in \{0,\, 1\}$. 

Now, recalling that $P(\sco_\piii(1)) \simeq \text{T}_\piii(-1)$ and that 
$E_3(1)$ has rank 5, it follows, from \cite[Lemma~1.2]{acm1}, that 
$E_3(1) \simeq G \oplus t\sco_\piii$, where $G$ is defined by an exact 
sequence$\, :$ 
\[
0 \lra (t + 1)\sco_\piii \lra 2\text{T}_\piii(-1) \lra G \lra 0\, . 
\]   
Consequently, $E(-2) = F$ is the cohomology of a monad of the form$\, :$ 
\[
0 \lra (t+1)\sco_\piii \lra 2\text{T}_\piii(-1) \oplus t\sco_\piii \lra 
\sco_\piii(2) \oplus \sco_\piii(1) \lra 0\, . 
\] 
Recalling the exact sequence $0 \ra \sco_\piii(-1) \ra 4\sco_\piii \ra 
\text{T}_\piii(-1) \ra 0$, one deduces that $E(-2)$ is the cohomology of a 
monad of the form$\, :$ 
\[  
0 \lra 2\sco_\piii(-1) \oplus (t + 1)\sco_\piii \lra (8 + t)\sco_\piii \lra 
\sco_\piii(2) \oplus \sco_\piii(1) \lra 0\, , 
\]
hence of a monad of the form$\, :$ 
\begin{equation}\label{E:monad6} 
0 \lra 2\sco_\piii(-1) \lra 7\sco_\piii \lra \sco_\piii(2) \oplus \sco_\piii(1) 
\lra 0\, . 
\end{equation} 
By Lemma~\ref{L:mora2o(1)} from Appendix~\ref{A:h0e(-2)neq0}, the monads of 
this form can be put toghether into a family with irreducible base.

\vskip2mm 

\noindent 
{\bf Claim 6.1.}\quad $\tH^1(E) = 0$. 

\vskip2mm 

\noindent 
\emph{Indeed}, Lemma~\ref{L:h0f(1)=0c2=12} implies that $\tH^0(E(-1)) = 0$ 
hence, by Riemann-Roch, $\h^1(E(-1)) = 4$. One deduces, from 
Lemma~\ref{L:h1e(l)c2=12}, that $\h^1(E) \leq 1$. 

Assume, by contradiction, that $\h^1(E) = 1$. It follows, from Riemann-Roch, 
that $\h^0(E) = 8$. Consider the globally generated vector bundle $P(E)$ whose  
dual is the kernel of the evaluation morphism $\tH^0(E) \otimes_k \sco_\piii 
\ra E$. $P(E)$ has rank 5 and 
Chern classes $c_1^\prim = 5$, $c_2^\prim = 13$, $c_3^\prim = 11$ hence, by 
Lemma~\ref{L:eprimunstable}, the normalized rank 3 vector bundle $F^\prim$ 
associated to $P(E)$ must be stable. Since, by Remark~\ref{R:p(e)},  
$c_1(F^\prim) = -1$, $c_2(F^\prim) = 5$, $c_3(F^\prim) = -3$, 
$\h^1(F^\prim(-2)) = \h^1(F(2)) = \h^1(E) = 1$, $\h^2(F^\prim(-2)) = \h^0(E) -6 -  
\h^2(F(-2)) = 2$ and $\h^2(F^\prim(-1)) = \h^0(E(-1)) = 0$ it follows that 
$F^\prim$ must have spectrum $(1,0,0,-1,-1)$. Moreover, $\h^0(P(E)(-1)) = 
\h^2(F(-1)) = 0$ and $\h^2(P(E)(-3)) = \h^2(F^\prim(-1)) = 0$. 
Remark~\ref{R:h2e(-3)=0}(i) implies that the graded $S$-module 
$\tH^1_\ast(P(E))$ is generated in degrees $\leq -2$. 

Now, $\tH^1(P(E)(l)) = 0$ for $l \leq -5$, $\h^1(P(E)(-4)) = 1$, 
$\h^1(P(E)(-3)) = 4$ hence $s^\prim := \h^1(P(E)(-3)) - \h^1(P(E)(-4)) = 3$. 
Moreover, by Riemann-Roch, $\h^1(P(E)(-2)) = \h^1(F^\prim) = 7$.  
Remark~\ref{R:muh1e(-4)} implies that the multiplication map $\tH^1(P(E)(-4)) 
\otimes S_1 \ra \tH^1(P(E)(-3))$ is bijective. 

We assert that the multiplication map $\mu \colon \tH^1(P(E)(-3)) 
\otimes S_1 \ra \tH^1(P(E)(-2))$ has rank $\geq 5$. 
\emph{Indeed}, if $\mu$ has rank $\leq 4$ then  
Remark~\ref{R:beilinson} implies that there exists an exact sequence$\, :$ 
\[
0 \lra \Omega^3(3) \lra 4\Omega^2(2) \lra 4\Omega^1(1) \lra Q \lra 0\, . 
\]
with $Q$ locally free. $Q$ must have rank 1. Computing Chern classes one 
deduces that $Q \simeq \sco_\piii(3)$. But this is clearly \emph{not possible} 
as one can see, for example, by applying $\tH^0(-)$ to the exact sequence. 
It thus remains that $\mu$ has rank $\geq 5$. 

We have proved, so far, that $\tH^1_\ast(P(E))$ has a minimal generator in 
degree $-4$ and at most two minimal generators in degree $-2$. 

We assert, now, that the graded $S$-module $\tH^1_\ast(P(E)^\vee)$ is generated 
by $\tH^1(P(E)^\vee(1))$. \emph{Indeed}, if this is not the case then (the 
proof of) Lemma~\ref{L:h2e(-3)=0}(h) shows that $\tH^1_\ast(P(E)^\vee) 
\simeq {\underline k}(-1) \oplus {\underline k}(-2)$.  
Consider the extension defined by the above mentioned generators of 
$\tH^1_\ast(P(E))$$\, :$  
\[
0 \lra P(E)(-2) \lra E_2^\prim \lra \sco_\piii(2) \oplus 2\sco_\piii \lra 0\, . 
\] 
One has $\tH^1_\ast(E_2^\prim) = 0$ (hence $\tH^2_\ast(E_2^{\prim \vee}) = 0$) 
and $\tH^1_\ast(E_2^{\prim \vee}) \simeq {\underline k}(1) \oplus {\underline k}$. 
It follows, from the correspondence of Horrocks, that $E_2^{\prim \vee} 
\simeq \Omega_\piii(1) \oplus \Omega_\piii \oplus A$, with $A$ a direct sum of 
line bundles, hence $E_2^\prim \simeq \text{T}_\piii \oplus \text{T}_\piii(-1) 
\oplus A^\vee$. But \emph{this is not possible} because any morphism 
$\text{T}_\piii \oplus \text{T}_\piii(-1) \ra 2\sco_\piii$ is the zero morphism 
and there is no epimorphism $A^\vee \ra 2\sco_\piii$, because $A^\vee$ has 
rank 2. It thus remains that $\tH^1_\ast(P(E)^\vee)$ is generated by 
$\tH^1(P(E)^\vee(1))$. Moreover, by Lemma~\ref{L:h2e(-3)=0}(c),(d),(f), 
$\h^1(P(E)^\vee(1)) \leq 1$.  

Consequently, $P(E)(-2)$ has a Horrocks monad of the form$\, :$ 
\[
0 \lra \sco_\piii(-1) \lra B \lra \sco_\piii(2) \oplus 2\sco_\piii \lra 0\, , 
\]
with $B$ a direct sum of line bundles. $B$ has rank 9, $c_1(B) = -4$, 
$\h^0(B(-1)) = 0$ and $\h^0(B) = \h^0(\sco_\piii(2) \oplus 2\sco_\piii) - 
\h^1(P(E)(-2)) = 5$ hence $B \simeq 5\sco_\piii \oplus 4\sco_\piii(-1)$. 
Since there is no epimorphism $4\sco_\piii(-1) \ra 2\sco_\piii$ it follows that 
$P(E)(-2)$ is, actually, the cohomology of a monad of the form$\, :$  
\begin{equation}\label{E:monadp(e)}
0 \lra \sco_\piii(-1) \lra 4\sco_\piii \oplus 4\sco_\piii(-1) \lra 
\sco_\piii(2) \oplus \sco_\piii \lra 0\, .
\end{equation} 
Assume, firstly, that this monad is minimal. One gets an exact sequence$\, :$ 
\[
0 \lra P(E)(-2) \lra \text{T}_\piii(-1) \oplus \Omega_\piii 
\overset{\psi}{\lra} \sco_\piii(2) \lra 0\, . 
\]
Let $\psi_1 \colon \text{T}_\piii(-1) \ra \sco_\piii(2)$ and $\psi_2 \colon 
\Omega_\piii \ra \sco_\piii(2)$ be the components of $\psi$.  
$\Cok \psi_1 \simeq \sco_Z(2)$ for some 
closed subscheme $Z$ of $\piii$. Let ${\overline \psi}_2$ be the composite 
morphism $\Omega_\piii \overset{\psi_2}{\lra} \sco_\piii(2) \ra \sco_Z(2)$. 
One has an exact sequence$\, :$ 
\[
P(E) \lra \Omega_\piii(2) \xra{{\overline \psi}_2(2)} \sco_Z(4) \lra 0\, .  
\]
Since $c_1(\Omega_\piii(2)) = 2$, Lemma~\ref{L:suppcok} implies that 
$\dim Z \leq 0$. But, on one hand, $Z$ is the zero scheme of a global section 
of $\Omega_\piii(3)$ hence it has length $c_3(\Omega_\piii(3)) = 5$ and, on 
the other hand, applying Lemma~\ref{L:lengthcok} to the last exact sequence, 
one deduces that $Z = \emptyset$. This \emph{contradiction} shows that the 
monad \eqref{E:monadp(e)} \emph{cannot be minimal}.   

If the monad \eqref{E:monadp(e)} is not minimal then one has an exact 
sequence$\, :$ 
\[
0 \lra P(E)(-2) \lra 3\sco_\piii \oplus 3\sco_\piii(-1)  
\overset{\phi}{\lra} \sco_\piii(2) \lra 0\, . 
\]   
Let $\phi_1 \colon 3\sco_\piii \ra \sco_\piii(2)$ and $\phi_2 \colon 
3\sco_\piii(-1) \ra \sco_\piii(2)$ be the components of $\phi$. $\Cok \phi_1 
\simeq \sco_Z(2)$ for some closed subscheme $Z$ of $\piii$. Let 
${\overline \phi}_2$ be the composite morphism $3\sco_\piii(-1) 
\overset{\phi_2}{\ra} \sco_\piii(2) \ra \sco_Z(2)$. One has an exact 
sequence$\, :$ 
\[
P(E) \lra 3\sco_\piii(1) \xra{{\overline \phi}_2(2)} \sco_Z(4) \lra 0\, . 
\] 
Since $c_1(3\sco_\piii(1)) = 3$, Lemma~\ref{L:suppcok} implies that 
$\dim Z \leq 0$. $Z$ is a complete intersection of type 
$(2,2,2)$ in $\piii$ hence it has length 8 and one has an exact sequence$\, :$ 
\[
0 \lra \sco_\piii(-2) \lra 3\sco_\piii \lra P(E) \lra 3\sco_\piii(1) 
\xra{{\overline \phi}_2} \sco_Z(4) \lra 0\, . 
\] 
Since $P(E)$ is globally generated and $\h^0(P(E)) = 8$, one deduces an 
exact sequence$\, :$ 
\[
5\sco_\piii \lra 3\sco_\piii(1) \lra \sco_Z(4) \lra 0\, . 
\]
Applying the argument from the proof of Lemma~\ref{L:lengthcok} to this 
exact sequence one gets that $\text{length}\, Z = 10$. This \emph{final 
contradiction} shows that the normalized rank 3 vector bundle $F^\prim$ 
associated to $P(E)$ cannot have spectrum $(1,0,0,-1,-1)$ and this implies 
that $\h^1(E) = 0$. Claim 6.1 is proven.   

\vskip2mm 

\noindent 
{\bf Claim 6.2.}\quad \emph{Let} $F$ \emph{be the cohomology of a monad of the 
form} \eqref{E:monad6}.   
\emph{It is, obviously, also the cohomology of a monad of the form}$\, :$ 
\[
0 \lra 2\sco_\piii(-1) \overset{\beta}{\lra} 3\sco_\piii \oplus \Omega_\piii(1) 
\overset{\alpha}{\lra} \sco_\piii(2) \lra 0\, . 
\]
\emph{If the degeneracy locus of the component} $\beta_2 \colon 
2\sco_\piii(-1) \ra \Omega_\piii(1)$ \emph{of} $\beta$ \emph{is} 
$1$-\emph{dimensional then} $F$ \emph{is the cohomology of a monad of the 
form}$\, :$ 
\[
0 \lra \sco_\piii(-1) \lra 3\sco_\piii \oplus N \lra \sco_\piii(2) \lra 0\, , 
\]
\emph{where} $N$ \emph{is a nullcorrelation bundle}. 

\vskip2mm 

\noindent 
\emph{Indeed}, $\beta_2$ is defined by two global sections $s_1$ and $s_2$ 
of $\Omega_\piii(2)$. Lemma~\ref{L:t(-2)ra2o} from 
Appendix~\ref{A:miscellaneous} implies that, for general constants $a_1,\, 
a_2 \in k$, the global section $a_1s_1 + a_2s_2$ of $\Omega_\piii(2)$ 
vanishes at no point of $\piii$. We can assume that $s_2$ has this property. 
Then one has an exact sequence$\, :$ 
\[
0 \lra \sco_\piii(-1) \overset{s_2}{\lra} \Omega_\piii(1) \lra N \lra 0\, , 
\]   
where $N$ is a nullcorrelation bundle. Since $\text{Ext}^i(N , 
\sco_\piii) \simeq \tH^i(N^\vee) \simeq \tH^i(N) = 0$, $i = 0,\, 1$, $s_2$ 
induces an isomorphism$\, :$ 
\[
\text{Hom}(\Omega_\piii(1) , \sco_\piii) \Izo \text{Hom}(\sco_\piii(-1) , 
\sco_\piii)\, . 
\]
It follows that one can assume that the component of $\beta$ mapping the 
second summand $\sco_\piii(-1)$ of $2\sco_\piii(-1)$ to $3\sco_\piii$ is 0. 
Claim 6.2 becomes, now, clear.

\vskip2mm 

\noindent 
{\bf Construction 6.3.}\quad We want to construct a monad of the form$\, :$ 
\begin{equation}\label{E:monadf}
0 \lra \sco_\piii(-1) \overset{\psi}{\lra} 3\sco_\piii \oplus N 
\overset{\phi}{\lra} \sco_\piii(2) \lra 0\, ,  
\end{equation} 
with $N$ a nullcorrelation bundle, such that its cohomology sheaf $F$ has the 
property that $F(2)$ is globally generated, $\tH^0(F(1)) = 0$, $\tH^1(F(2)) = 
0$, and the multiplication map $\tH^0(F(2)) \otimes \tH^0(\sco_\piii(1)) \ra 
\tH^0(F(3))$ has corank 1.   

\vskip2mm 

We explain, firstly, \emph{the idea of the construction}.   
Let $F$ be the cohomology sheaf of a monad of the form 
\eqref{E:monadf}. The component $\psi_1 \colon \sco_\piii(-1) \ra 
3\sco_\piii$ (resp., $\psi_2 \colon \sco_\piii(-1) \ra N$) of the differential 
$\psi$ of the monad is defined by three linear forms $h_0,\, h_1,\, h_2$  
(resp., by a global section $s$ of $N(1)$), while the component $\phi_1 
\colon 3\sco_\piii \ra \sco_\piii(2)$ (resp., $\phi_2 \colon N \ra 
\sco_\piii(2)$) of the differential $\phi$ is defined by three quadratic 
forms $q_0,\, q_1,\, q_2$ (resp., by the exterior multiplication $- \wedge t$ 
by a global section $t$ of $N(2)$). The condition $\phi \circ \psi = 0$ is 
equivalent to$\, :$ 
\[
q_0h_0 + q_1h_1 + q_2h_2 + s \wedge t = 0\, . 
\] 
The cubic form $f := q_0h_0 + q_1h_1 + q_2h_2$ defines a cubic surface $X 
\subset \piii$, containing the zero schemes $Z(s)$ and $Z(t)$ of $s$ and $t$. 
$Z(s)$ is the union of two disjoint lines or a double structure on a line. 
If $X$ has no plane as a component then $Z(t)$ has codimension 2 in $\piii$ 
and it is a locally complete intersection curve in $\piii$, of degree 5 and 
with $\omega_{Z(t)} \simeq \sco_{Z(t)}$. Moreover, if $X$ is nonsigular then 
$Z(s) \cap Z(t) = \emptyset$. 

Assume, now, that $q_0,\, q_1,\, q_2$ define a 0-dimensional complete 
intersection $\Gamma \subset \piii$. $\phi$ is an epimorphism if and only 
if $Z(t) \cap \Gamma = \emptyset$. Using the exact sequences$\, :$ 
\begin{gather*} 
0 \lra \sco_\piii(-4) \overset{\delta_3}{\lra} 3\sco_\piii(-2) 
\overset{\delta_2}{\lra} 3\sco_\piii 
\overset{\phi_1}{\lra} \sci_\Gamma(2) \lra 0\, ,\\
0 \lra \sco_\piii(-2) \overset{t}{\lra} N \xra{- \wedge t} 
\sci_{Z(t)}(2) \lra 0\, ,   
\end{gather*} 
and the commutative diagram$\, :$ 
\[
\begin{CD}
0 @>>> \sco_\piii(-1) @>{\psi}>> 3\sco_\piii \oplus N @>{\phi}>>  
\sco_\piii(2) @>>> 0\\ 
@. @V{f}VV @VV{\phi_1 \oplus \phi_2}V @\vert\\ 
0 @>>> \sci_{Z(t) \cup \Gamma}(2) 
@>{\left(\begin{smallmatrix} u\\ -u \end{smallmatrix}\right)}>> 
\sci_{Z(t)}(2) \oplus \sci_\Gamma(2) 
@>>> \sco_\piii(2) @>>> 0
\end{CD}
\]
where $u \colon \sci_{Z(t) \cup \Gamma}(2) \hookrightarrow \sco_\piii(2)$ is the 
inclusion map, one gets an exact sequence$\, :$  
\[
0 \lra \sco_\piii(-4) \xra{(\delta_3\, ,\, 0)^{\text{t}}} 4\sco_\piii(-2) 
\xra{\delta_2 \oplus \, t} F \lra 
\sci_{Z(t) \cup \Gamma , X}(2) \lra 0\, . 
\]
It follows that $F(2)$ is globally generated, $\tH^0(F(1)) = 0$, 
$\tH^1(F(2)) = 0$ and the multiplication map 
$\tH^0(F(2)) \otimes \tH^0(\sco_\piii(1)) \ra \tH^0(F(3))$ 
has corank $1$ if and only if $\sci_{Z(t) \cup \Gamma , X}(4)$ is globally 
generated, $\tH^0(\sci_{Z(t) \cup \Gamma , X}(3)) = 0$,  
$\tH^1(\sci_{Z(t) \cup \Gamma , X}(4)) = 0$, and the multiplication map 
$\tH^0(\sci_{Z(t) \cup \Gamma , X}(4)) \otimes \tH^0(\sco_\piii(1)) \ra 
\tH^0(\sci_{Z(t) \cup \Gamma , X}(5))$ has corank $1$.  

\vskip2mm 

We make, secondly, a \emph{general remark}. Let $G$ be a rank 2 vector bundle 
on $\piii$ with $c_1(G) = 0$ and let $b \geq a \geq 0$ be two integers. Let 
$s$ be a global section of $G(a)$ whose zero scheme $Z(s)$ has codimension 2 
in $\piii$, let $X \subset \piii$ be a surface of degree $a + b$ containing 
$Z(s)$ as a subscheme and let $f = 0$ be an equation of $X$. Using the exact 
sequence$\, :$ 
\[
0 \lra \sco_\piii(-a) \overset{s}{\lra} G \xra{s \wedge -} \sci_{Z(s)}(a) 
\lra 0\, , 
\] 
one gets a global section $t_0$ of $G(b)$ such that $s \wedge t_0 = f$. One 
deduces, from the diagram$\, :$ 
\[
\begin{CD} 
@. @. \sco_\piii(-b) @= \sco_\piii(-b)\\ 
@. @. @VV{t_0}V @VV{f}V\\ 
0 @>>> \sco_\piii(-a) @>{s}>> G @>{s \wedge -}>> \sci_{Z(s)}(a) @>>> 0  
\end{CD}
\]
and exact sequence$\, :$ 
\[
0 \lra \sco_\piii(-a) \oplus \sco_\piii(-b) \xra{(s\, ,\, t_0)} G \lra 
\sci_{Z(s) , X}(a) \lra 0\, . 
\]
Assume, now, that $X$ is nonsingular. Using the commutative diagram$\, :$ 
\[
\begin{CD}
0 @>>> \sco_\piii(-a) \oplus \sco_\piii(-b) @>{(s\, ,\, t_0)}>> G @>>> 
\sci_{Z(s) , X}(a) @>>> 0\\ 
@. @VVV @VVV @\vert\\ 
0 @>>> \sco_X(-a) \otimes \sco_X[Z(s)] @>>> G_X @>>> \sci_{Z(s) , X}(a) @>>> 0 
\end{CD} 
\]
one gets that the vector subspace $k(t_0 \vb X) + \tH^0(\sco_X(b-a))(s \vb X)$ 
of $\tH^0(\sco_X(b-a) \otimes \sco_X[Z(s)])$ generates globally the line 
bundle $\sco_X(b-a) \otimes \sco_X[Z(s)]$ on $X$. It follows that, for a 
general form $g \in \tH^0(\sco_\piii(b - a))$, the zero scheme $Z(t)$ of the 
global section $t := t_0 + gs$ of $G(b)$ is a nonsingular curve contained in 
$X$ (because $s \wedge t = f$). Moreover, 
$\sco_X[Z(t)] \simeq \sco_X(b - a) \otimes \sco_X[Z(s)]$.   

\vskip2mm 

Now, we effectively \emph{begin the construction} by considering a nonsingular 
cubic surface $X \subset \piii$, which is the blow-up $\pi \colon X \ra \pii$ 
of $\pii$ in six general points $P_1, \ldots , P_6$, embedded in $\piii$ such 
that $\sco_X(1) \simeq \pi^\ast\sco_\pii(3) \otimes \sco_X[-E_1 - \cdots - 
E_6]$, where $E_i := \pi^{-1}(P_i)$. Consider, also, for $1 \leq i \leq 6$, the 
line $L_i \subset X$ which is the strict transform of the conic $C_i \subset 
\pii$ containing $\{P_1 , \ldots , P_6\} \setminus \{P_i\}$. Let $f = 0$ be a 
(cubic) equation of $X$. There exists a nullcorrelation bundle $N$ such that 
$N(1)$ has a global section $s$ whose zero scheme is $E_2 \cup E_3$. 
Applying the above considerations to $N$, $s$ and $X$, one gets a global 
section $t$ of $N(2)$, with $s \wedge t = -f$, and whose zero scheme is a 
nonsingular curve $Z(t)$ contained in $X$ such that $\sco_X[Z(t)] \simeq 
\sco_X(1) \otimes \sco_X[E_2 + E_3]$. Consider, finally, a general nonsingular 
cubic curve ${\overline C} \subset \pii$, containing $\{P_2, \ldots , P_6\}$ 
but not $P_1$, and such that ${\overline C} \cap C_j$ consists of six distinct 
points, $j = 4,\, 5,\, 6$. This is possible because the map 
$\tH^0(\sci_{\{P_2 , \ldots , P_6\}}(3)) \ra 
\tH^0(\sci_{\{P_2 , \ldots , P_6\} \setminus \{P_j\} , C_j}(3))$ is surjective, for each 
$j \geq 2$ (compute the dimension of the kernel of this map). The strict 
transform $C \subset X$ of $\overline C$ is an elliptic curve of degree 4 in 
$\piii$, hence a complete intersection of type $(2,2)$, and $\sco_X[C] 
\simeq \sco_X(1) \otimes \sco_X[E_1]$.    

\vskip2mm 

$C \subset \piii$ is described by two quadratic equations $q_0 = q_1 
= 0$. Since $C$ is contained in $X$, the cubic form $f$ vanishing on $X$ can 
be written as $f = q_0h_0 + q_1h_1$, with $h_0$ and $h_1$ linear forms. 
Since $q_0$ and $q_1$ have no common zero on $E_1$ (because $C \cap E_1 = 
\emptyset$) one deduces that $h_0$ and $h_1$ vanish on $E_1$ hence 
$h_0 = h_1 = 0$ are, actually, equations decribing $E_1$. It follows that 
$h_0,\, h_1,\, h_2:=0$ and $s$ define a locally split monomorphism 
$\psi \colon \sco_\piii(-1) \ra 3\sco_\piii \oplus N$ (recall that the zero 
scheme of $s$ is $E_2 \cup E_3$). On the other hand, choosing a general 
quadratic form $q_2 \in \tH^0(\sco_\piii(2))$, vanishing at none of the points 
of $C \cap (Z(t) \cup L_4 \cup L_5 \cup L_6)$, one gets that $q_0,\, q_1,\, 
q_2$ and exterior multiplication $- \wedge t$ by $t$ define an epimorphism 
$\phi \colon 3\sco_\piii \oplus N \ra \sco_\piii(2)$ such that $\phi \circ 
\psi = 0$. Denoting by $\Gamma$ the 0-dimensional complete intersection of 
equations $q_0 = q_1 = q_2 = 0$ and recalling that $\sco_X[Z(t)] \simeq 
\sco_X(1) \otimes \sco_X[E_2 + E_3]$, what we actually have to prove is that 
$\sci_{\Gamma , X}(3) \otimes \sco_X[-E_2 - E_3]$ is globally generated, 
$\tH^0(\sci_{\Gamma , X}(2) \otimes \sco_X[-E_2 - E_3]) = 0$, 
$\tH^1(\sci_{\Gamma , X}(3) \otimes \sco_X[-E_2 - E_3]) = 0$, and that the 
multiplication map $\tH^0(\sci_{\Gamma , X}(3) \otimes \sco_X[-E_2 - E_3]) 
\otimes \tH^0(\sco_\piii(1)) \ra \tH^0(\sci_{\Gamma , X}(4) \otimes 
\sco_X[-E_2 - E_3])$ has corank 1.   

\vskip2mm 

One can further reduce the problem as follows$\, :$ 
the unique quadric surface $Q$ containing $E_1 \cup E_2 \cup E_3$ must also 
contain $L_4$, $L_5$ and $L_6$ (because they are trisecants of the union 
$E_1 \cup E_2 \cup E_3$) hence$\, :$ 
\[
Q \cap X = E_1 + E_2 + E_3 + L_4 + L_5 + L_6 
\]      
as divisors on $X$. It follows that$\, :$ 
\[
\scl := \sco_X(3) \otimes \sco_X[-E_2 - E_3] \simeq 
\sco_X(1) \otimes \sco_X[E_1] \otimes \sco_X[L_4 + L_5 + L_6] 
\]  
hence the complete linear system $\vert \, \scl \, \vert$ contains the 
divisor $\Delta := C + L_4 + L_5 + L_6$. One gets an exact sequence$\, :$ 
\[
0 \lra \sco_X \lra \sci_{\Gamma , X} \otimes \scl \lra 
\sci_{\Gamma , \Delta} \otimes (\scl \vb \Delta) \lra 0\, .  
\]
Putting $\scm := \sci_{\Gamma , \Delta} \otimes (\scl \vb \Delta)$, one deduces 
that it suffices to prove that $\scm$ is globally generated, $\tH^0(\scm(-1)) 
= 0$, $\tH^1(\scm) = 0$, and that the multiplication map $\mu \colon 
\tH^0(\scm) \otimes \tH^0(\sco_\piii(1)) \ra \tH^0(\scm(1))$ has corank 1. 
Notice that, since $\Gamma \subset C \setminus (L_4 \cup L_5 \cup L_6)$, 
$\Gamma$ is an effective Cartier divisor on $\Delta$ hence $\scm$ is an 
invertible $\sco_\Delta$-module. 

\vskip2mm   

Let us prove, now, the above assertions about $\scm$. We begin by   
noticing that, for $2 \leq i \leq 3$ and $4 \leq j \leq 6$, $E_i$ intersects 
$C$ in one point $P_i^\prime$, $L_j$ intersects $C$ in two points $A_j$ and 
$B_j$, and $E_i$ intersects $L_j$ in one point not situated on $C$ (this 
follows from the fact that the cubic curve ${\overline C}\subset \pii$ 
intersects the conic $C_j \subset \pii$ in six distinct points). Let 
$L \subset \piii$ be the line joining $P_2^\prime$ and $P_3^\prime$. Since $L$ 
and $L_j$ are 2-secants of the union $E_2 \cup E_3$ and since none of the 
points $P_2^\prime$ and $P_3^\prime$ belongs to $L_j$ it follows that $L \cap 
L_j = \emptyset$, $j = 4,\, 5,\, 6$. One deduces that $\sco_C[A_j + B_j + 
P_2^\prime + P_3^\prime]$ is not isomorphic to $\sco_C(1)$, $j = 4,\, 5,\, 6$. 
Moreover, let $Q_j \subset \piii$ be the unique quadric surface containing 
$C \cup L_j$. One has$\, :$ 
\[
Q_j \cap X = C + L_j + L_j^\prime 
\] 
as divisors on $X$, where $L_j^\prime \subset X$ is the strict transform of 
the line $\overline{P_1P_j} \subset \pii$ (because $L_j^\prime$ is a 3-secant 
of $C \cup L_j$). Since $\overline{C} \subset \pii$ is a general cubic curve 
containing $\{P_2 , \ldots , P_6\}$, the points $P_2^\prime$ and $P_3^\prime$ are 
general points of $E_2$ and $E_3$, respectively, hence one can assume that the 
line $L$ does not intersect $L_j^\prime$ which implies that $L$ is not 
contained in $Q_j$. In this case, $\sco_C[P_2^\prime + P_3^\prime]$ is not 
isomorphic to $\sco_C[A_j + B_j]$, $j = 4,\, 5,\, 6$. 

Next, $\scm \vb C \simeq \sco_C(1) \otimes \sco_C[-P_2^\prime - P_3^\prime]$ is 
a line bundle of degree 2 on $C$, and $\scm \vb L_j \simeq \sco_{L_j}(1)$, 
$j = 4,\, 5,\, 6$. Applying $- \otimes_{\sco_\Delta} \scm(-1)$ to the exact 
sequence$\, :$ 
\begin{equation}\label{E:odeltaraoc} 
0 \lra {\textstyle \bigoplus_{j = 4}^6}\sco_{L_j}(-2) \lra \sco_\Delta \lra 
\sco_C \lra 0\, , 
\end{equation} 
one deduces that $\tH^0(\scm(-1)) = 0$. Tensorizing this exact sequence by 
$\scm$ one gets that $\tH^0(\scm) \izo \tH^0(\scm \vb C)$ and that 
$\tH^1(\scm) = 0$. On the other hand, tensorizing by $\scm$ the exact 
sequences$\, :$ 
\begin{gather*} 
0 \lra \sco_{L_5}(-2) \oplus \sco_{L_6}(-2) \lra \sco_\Delta \lra 
\sco_{C \cup L_4} \lra 0\, ,\\ 
0 \lra \sco_C[-A_4 - B_4] \lra \sco_{C \cup L_4} \lra \sco_{L_4} \lra 0\, ,  
\end{gather*} 
one gets that $\tH^0(\scm) \izo \tH^0(\scm \vb L_4)$ (one uses the fact that 
$(\scm \vb C) \otimes \sco_C[-A_4 - B_4] \simeq \sco_C(1) \otimes 
\sco_C[-P_2^\prime - P_3^\prime] \otimes \sco_C[-A_4 - B_4]$ is a non-trivial line 
bundle of degree 0 on $C$). Analogously, 
$\tH^0(\scm) \izo \tH^0(\scm \vb L_j)$, $j = 5,\, 6$. Since $\scm \vb C$ and 
$\scm \vb L_j$, $j = 4,\, 5,\, 6$, are globally generated it follows that 
$\scm$ is globally generated. 

Finally, one has, as we saw above, $\h^0(\scm) = 2$ and, tensorizing by 
$\scm(1)$ the exact sequence \eqref{E:odeltaraoc}, $\h^0(\scm(1)) = 9$. It 
follows that the corank of the multiplication map $\mu \colon 
\tH^0(\scm) \otimes \tH^0(\sco_\piii(1)) \ra \tH^0(\scm(1))$ is 1 if and only 
if $\mu$ is injective. But $\tH^0(\sco_\piii(1)) \izo \tH^0(\sco_\Delta(1))$ 
(tensorize the exact sequence \eqref{E:odeltaraoc} by $\sco_\Delta(1)$). 
Applying the ``base point free pencil trick'' one deduces that $\mu$ is 
injective if an only if $\tH^0(\scm^{-1}(1)) = 0$. 

\vskip2mm 

Now, in order to prove the last vanishing, we shall give an alternative 
description of $\scm^{-1}(1)$. Let $Q^\prime \subset \piii$ be a quadric surface 
containing $E_2 \cup E_3$ and intersecting $C$ in eight distinct points 
situated on $C \setminus (L_4 \cup L_5 \cup L_6)$ (two of these points are, of 
course, $P_2^\prime$ and $P_3^\prime$). One has, as divisors on $X$$\, :$ 
\[
Q^\prime \cap X = E_2 + E_3 + D^\prime \, , 
\]    
where $D^\prime$ is an effective divisor of degree 6 on $X$. $D^\prime$ does not 
intersect $L_j$, $j = 4,\, 5,\, 6$ (compute the intersection multiplicities). 
It follows that$\, :$ 
\[
\scl \vb \Delta \simeq \sco_\Delta(1) \otimes \sco_\Delta[C \cap D^\prime] 
\]
hence $\scm^{-1}(1) \simeq \sco_\Delta[\Gamma - (C \cap D^\prime)]$. Since, by 
construction, $q_2$ is a general element of $\tH^0(\sco_\piii(2))$ and 
$\Gamma$ is the zero divisor of $q_2 \vb C$, it remains to prove the 
following$\, :$ 

\vskip2mm 

\noindent 
{\bf Claim 6.3.1.}\quad \emph{If} $\Gamma$ \emph{is a general element of the 
complete linear system} $\vert \, \sco_C(2) \, \vert$ \emph{then one has}  
$\tH^0(\sco_\Delta[\Gamma - (C \cap D^\prime)]) = 0$. 

\vskip2mm 

\noindent 
\emph{Indeed}, put $\scn := \sco_\Delta[\Gamma - (C \cap D^\prime)]$. One has 
$\scn \vb C \simeq \sco_C[P_2^\prime + P_3^\prime]$ and $\scn \vb L_j \simeq 
\sco_{L_j}$, $j = 4,\, 5,\, 6$. Using the fact that 
$\sco_C[P_2^\prime + P_3^\prime] \otimes \sco_C[-A_j - B_j]$ is a non-trivial 
line bundle of degree 0 on $C$, one deduces, as above, that $\tH^0(\scn) 
\ra \tH^0(\scn \vb L_j)$ is injective, $j = 4,\, 5,\, 6$. It follows that if 
$\scn$ has a non-zero global section then its vanishing locus is contained in 
$C \setminus (L_4 \cup L_5 \cup L_6)$. Consequently, if $\tH^0(\scn) \neq 0$ 
then $\scn \simeq \sco_\Delta[R_2 + R_3]$, where $R_2$ and $R_3$ are two (not 
necessarily distinct) points of $C \setminus (L_4 \cup L_5 \cup L_6)$. One 
must have $\sco_C[R_2 + R_3] \simeq \scn \vb C \simeq \sco_C[P_2^\prime + 
P_3^\prime]$. Moreover$\, :$ 
\[
\sco_\Delta[\Gamma] \simeq \sco_\Delta[(C \cap D^\prime) + R_2 + R_3]\, . 
\]      
In other words, there exists a global section of 
$\sco_\Delta[(C \cap D^\prime) + R_2 + R_3]$ whose zero divisor is exactly 
$\Gamma$. 

Now, denoting by $\theta$ a global section of $\sco_C[C \cap D^\prime]$ whose 
zero divisor is $C \cap D^\prime$ and by $\tau$ a global section of 
$\sco_C[P_2^\prime + P_3^\prime]$ whose zero divisor is $R_2 + R_3$, the image of 
the restriction map$\, :$ 
\[
\tH^0(\sco_\Delta[(C \cap D^\prime) + R_2 + R_3]) \lra 
\tH^0(\sco_\Delta[(C \cap D^\prime) + R_2 + R_3] \vb C) \simeq \tH^0(\sco_C(2)) 
\]
consists of the elements $\sigma$ of $\tH^0(\sco_C(2))$ satisfying$\, :$ 
\[
(\theta \tau)(A_j) \otimes \sigma(B_j) - \sigma(A_j) \otimes 
(\theta \tau)(B_j) = 0 \  \text{in}\  (\sco_C(2))(A_j) \otimes 
(\sco_C(2))(B_j)\, ,\  j = 4,\, 5,\, 6\, . 
\]
Notice that the zero divisor of any non-zero global section $\tau^\prim$ of 
$\sco_C[P_2^\prime + P_3^\prime]$ is different from $A_j + B_j$, 
$j = 4,\, 5,\, 6$. 

Consider, now, the bilinear map$\, :$ 
\[
\tH^0(\sco_C(2)) \times \tH^0(\sco_C(2)) \lra 
{\textstyle \bigoplus_{j = 4}^6}(\sco_C(2))(A_j) \otimes (\sco_C(2))(B_j) 
\] 
with components $(\rho , \sigma) \mapsto \rho(A_j) \otimes \sigma(B_j) - 
\sigma(A_j) \otimes \rho(B_j)$. It is easy to see that if $\rho$ does not 
vanish simultaneously in $A_j$ and $B_j$, $j = 4,\, 5,\, 6$, then the induced 
linear map $\tH^0(\sco_C(2)) \ra \bigoplus_{j = 4}^6(\sco_C(2))(A_j) \otimes 
(\sco_C(2))(B_j)$ is surjective (use sections $\sigma$ of $\sco_C(2)$ 
vanishing at five of the points $A_4,\, B_4,\, A_5,\, B_5,\, A_6,\, B_6$). 
Considering the induced bilinear map$\, :$ 
\[
\theta\tH^0(\sco_C[P_2^\prime + P_3^\prime]) \times \tH^0(\sco_C(2)) \lra 
{\textstyle \bigoplus_{j = 4}^6}(\sco_C(2))(A_j) \otimes (\sco_C(2))(B_j)\, , 
\]
one deduces that a general global section of $\sco_C(2)$ does not belong to 
any of the images of the restriction maps 
$\tH^0(\sco_\Delta[(C \cap D^\prime) + R_2 + R_3]) \ra \tH^0(\sco_C(2))$, with 
$R_2 + R_3$ effective divisor on $C$ such that $R_2 + R_3 \sim P_2^\prime + 
P_3^\prime$ and $R_2,\, R_3 \in C \setminus (L_4 \cup L_5 \cup L_6)$. This 
completes the proof of the claim and, with it, the verification of the fact 
that $F$ satisfies the properties stated at the beginning of 
Construction 6.3.   

\vskip2mm 

Notice, however, that the bundle $F$ satisfies $\h^0(F^\vee) = 1$. One can 
construct stable rank 3 vector bundles $F$ with $c_1(F) = -1$, $c_2(F) = 4$, 
$c_3(F) = -6$ such that $F(2)$ is globally generated, $\tH^1(F(2)) = 0$ and 
$\tH^0(F^\vee) = 0$ by deforming the above monad. More precisely, looking 
carefully at the above construction, one sees that one can assume that the  
third quadratic form $q_2$ belongs to $\tH^0(\sci_{E_2 \cup E_3}(2))$. In this 
case, the arguments preceding Claim 6.3.1 (with $Q^\prime$ the quadric surface 
of equation $q_2 = 0$) show that$\, :$ 
\[
\sci_{\Gamma , \Delta} \otimes (\scl \vb \Delta) \simeq \sco_\Delta(1) \otimes 
\sco_\Delta[-P_2^\prime - P_3^\prime] \simeq \sci_{L \cap \Delta , \Delta}(1)\, , 
\]
where $L$ is, as above, the line joining $P_2^\prime$ and $P_3^\prime$$\, ;$ we 
used the fact that $L$ meets none of the lines $L_j$, $j = 4,\, 5,\, 6$.  

Now, since $q_2$ belongs to $\tH^0(\sci_{E_2 \cup E_3}(2))$, there exists 
$s_2 \in \tH^0(N(1))$ such that $q_2 = s \wedge s_2$. 
Choose a linear form $h_2$ such that $h_0,\, h_1,\, h_2$ are linearly 
independent. For $c \in k$, $h_0,\, h_1,\, ch_2,\, s$ define a locally split 
monomorphism $\psi_c \colon \sco_\piii(-1) \ra 3\sco_\piii \oplus N$ and 
$q_0,\, q_1,\, q_2,\, t - ch_2s_2$ define a morphism $\phi_c \colon 
3\sco_\piii \oplus N \ra \sco_\piii(2)$ such that $\phi_c \circ \psi_c = 0$. 
If $c \in k$ is general then $\phi_c$ is an epimorphism. In this case, 
denoting by $F_c$ is the cohomology of the monad defined by $\phi_c$ and 
$\psi_c$, one has $\tH^0(F_c^\vee) = 0$ if, moreover, $c \neq 0$. 

\vskip2mm 

\noindent 
{\bf Construction 6.4.}\quad We provide, here, another kind of argument 
for the existence of globally generated vector bundles $E$ on $\piii$ with 
Chern classes $c_1 = 5$, $c_2 = 12$, $c_3 = 6$. We will, actually, construct a 
globally generated rank 4 vector bundle $E^\prim$ with ``complementary'' Chern 
classes $c_1^\prim = 5$, $c_2^\prim = 13$, $c_3^\prim = 11$ and we will take 
$E := P(E^\prim)$. 

Consider a nonsigular rational curve $C \subset \piii$, of degree 6, such 
that $\h^0(\sci_C(3)) = 2$ (hence $\h^1(\sci_C(3)) = 1$). Such a curve can be 
constructed on a nonsingular cubic surface $X \subset \piii$ (i.e., the 
blow-up $\pi \colon X \ra \pii$ of $\pii$ in six general points $P_1, \ldots , 
P_6$, embedded in $\piii$ such that $\sco_X(1) \simeq \pi^\ast\sco_\pii(3) 
\otimes \sco_X[-E_1 - \cdots - E_6]$, where $E_i := \pi^{-1}(P_i)$) as the 
strict transform of an irreducible cubic curve ${\overline C} \subset \pii$ 
having a node at $P_1$, containing $P_2$ but none of the points $P_3, \ldots , 
P_6$. Let $L_i \subset X$ (resp., $L_{ij} \subset X$) be the strict transform 
of the conic $\Gamma_i \subset \pii$ containing $\{P_1 , \ldots , P_6\} 
\setminus \{P_i\}$, $i = 1, \ldots , 6$ (resp., of the line 
$\overline{P_iP_j} \subset \pii$, $1 \leq i < j \leq 6$). Using the usual 
$\z$-basis $\pi^\ast\sco_\pj(1)$, $\sco_X[E_i]$, $i = 1, \ldots , 6$, of 
$\text{Pic}\, X$ one sees easily that$\, :$ 
\[
\sco_X(3) \otimes \sco_X[-C] \simeq \sco_X[2L_1 + L_2]\, . 
\]     
Notice that $L_1$ is a 5-secant of $C$ while $L_2$ is a 4-secant. Since 
$\h^0(\sco_X[2L_1 + L_2]) = 1$ it follows that $\h^0(\sci_{C , X}(3)) = 1$ 
hence $\h^0(\sci_C(3)) = 2$. 

An epimorphism $\delta \colon \sco_\piii(1) \oplus 2\sco_\piii(-1) \ra 
\omega_C(2)$ determines an extension$\, :$ 
\[
0 \lra 2\sco_\piii(1) \oplus \sco_\piii(-1) \lra F^{\prim \vee} \lra 
\sci_C(2) \lra 0\, , 
\] 
with $F^{\prim \vee}$ the dual of a rank 4 vector bundle $F^\prim$, such that, 
dualizing the extension, one gets an exact sequence$\, :$ 
\[
0 \lra \sco_\piii(-2) \lra F^\prim \lra \sco_\piii(1) \oplus 2\sco_\piii(-1) 
\overset{\delta}{\lra} \omega_C(2) \lra 0\, . 
\] 
Using Remark~\ref{R:chern}(b) one deduces immediately that 
$c_1(F^{\prim \vee}(-2)) = -5$, $c_2(F^{\prim \vee}(-2)) = 13$, 
$c_3(F^{\prim \vee}(-2)) = -11$ hence that $E^\prim := F^\prim(2)$ has the Chern 
classes indicated at the beginning of the construction. Denoting by $\sck$ 
the kernel of $\delta$, one has an exact sequence$\, :$ 
\[
0 \lra \sco_\piii(-2) \lra F^\prim \lra \sck \lra 0\, . 
\] 

We will use epimorphisms $\delta$ that can be written as composite maps$\, :$ 
\[
\sco_\piii(1) \oplus 2\sco_\piii(-1) \lra \sco_\piii(1) \oplus \sci_L 
\overset{\delta^\prim}{\lra} \omega_C(2)\, , 
\]
with $L$ a line in $\piii$ and $\delta^\prim$ an epimorphism. Denoting by 
$\sck^\prim$ the kernel of $\delta^\prim$, one has an exact sequence$\, :$ 
\[
0 \lra \sco_\piii(-2) \lra \sck \lra \sck^\prim \lra 0\, . 
\] 
We shall assume, moreover, that $L \subset X$ and that $\delta^\prim$ is a 
composite map$\, :$ 
\[
\sco_\piii(1) \oplus \sci_L \lra \sco_X(1) \oplus \sco_X[-L] 
\overset{\e}{\lra} \omega_C(2)\, , 
\]
for some epimorphism $\e$. In this case, the kernel $K$ of $\e$ is a rank 2 
vector bundle on $X$ and one has an exact sequence$\, :$ 
\[
0 \lra \sco_\piii(-2) \oplus \sco_\piii(-3) \lra \sck^\prim \lra K 
\lra 0\, . 
\]

We \emph{effectively begin the construction} by considering a general conic 
$\Gamma_0 \subset \pii$ passing through $P_1$, $P_3$ and $P_4$ but not through 
$P_2$, $P_5$ and $P_6$. The strict transform $C_0 \subset X$ of $\Gamma_0$ is a 
twisted cubic curve in $\piii$. We can assume that $C_0$ intersects $C$ in 
four simple points hence $\sco_X[C_0] \vb C \simeq \omega_C(1)$. We take $L$ 
to be the line $L_{56}$. Notice that $L$ is a 2-secant of $C_0$. Using the 
exact sequence$\, :$ 
\[
0 \lra \sco_X(1) \otimes \sco_X[L] \lra \sco_X(1) \otimes \sco_X[C_0 + L] 
\lra (\sco_X(1) \otimes \sco_X[C_0 + L]) \vb C_0 \lra 0  
\]
(and the exact sequence $0 \ra \sco_X \ra \sco_X[L] \ra \sco_L(-1) \ra 0$) one 
deduces easily that the restriction map$\, :$ 
\[
\tH^0(\sco_X(1) \otimes \sco_X[C_0 + L]) \lra 
\tH^0((\sco_X(1) \otimes \sco_X[C_0 + L] ) \vb C_0) 
\]
is surjective. It follows that if $\psi$ is a general global section of 
$\sco_X(1) \otimes \sco_X[C_0 + L]$ then the intersection scheme 
$W := \{\psi = 0\} \cap C_0$ on $X$ consists of six general simple points of 
$C_0$, none of them situated on $C$ (notice that the selfintersection of 
$C_0$ on $X$ is 1). Denoting by $\phi_0$ a global section of $\sco_X[C_0]$ 
whose zero divisor is $C_0$, one has an exact sequence$\, :$ 
\[
0 \lra \sco_X[-C_0 - L] \xra{\left(\begin{smallmatrix} -\psi \\ \phi_0 
\end{smallmatrix}\right)} \sco_X(1) \oplus \sco_X[-L] 
\xra{(\phi_0 \, ,\, \psi)} \sci_{W , X}(1) \otimes \sco_X[C_0] \lra 0\, . 
\] 
We take $\e$ to be the composite map$\, :$ 
\[
\sco_X(1) \oplus \sco_X[-L] \xra{(\phi_0 \, ,\, \psi)}  
\sci_{W , X}(1) \otimes \sco_X[C_0] \lra 
(\sci_{W , X}(1) \otimes \sco_X[C_0]) \vb C \simeq \omega_C(2)\, . 
\]
In this case, the kernel $K$ of $\e$ sits into an exact sequence$\, :$ 
\[
0 \lra \sco_X[-C_0 - L] \xra{\left(\begin{smallmatrix} -\psi \\ \phi_0 
\end{smallmatrix}\right)} K \lra 
\sci_{W , X}(1) \otimes \sco_X[C_0 - C] \lra 0\, . 
\]

\noindent 
{\bf Claim 6.4.1.}\quad \emph{If} $K(2)$ \emph{is globally generated then} 
$\sck^\prim(2)$ \emph{is globally generated, too}. 

\vskip2mm 

\noindent 
\emph{Indeed}, applying the Snake Lemma to the diagram$\, :$ 
\[
\SelectTips{cm}{12}\xymatrix @C=1.5pc{ 
0\ar[r] & \tH^0(\sco_\piii \oplus \sco_\piii(-1)) \otimes \sco_\piii\ar[r]\ar[d]  
& \tH^0(\sck^\prim(2)) \otimes \sco_\piii\ar[r]\ar[d] 
& \tH^0(K(2)) \otimes \sco_\piii\ar[r]\ar[d] & 0\\ 
0\ar[r] & \sco_\piii \oplus \sco_\piii(-1)\ar[r] & 
\sck^\prim(2)\ar[r] & K(2)\ar[r] & 0}
\]   
one sees that it suffices to show that the connecting morphism $\partial 
\colon \scn \ra \sco_\piii(-1)$ is an epimorphism, $\scn$ being the kernel of 
the evaluation morphism $\tH^0(K(2)) \otimes \sco_\piii \ra K(2)$. We will show 
that $\tH^0(\partial(1)) \colon \tH^0(\scn(1)) \ra \tH^0(\sco_\piii)$ is 
surjective. Recall, from the exact sequence preceding Claim 6.4.1, that one 
has a morphism $\sco_X[-C_0 - L] \ra K$ such that the component 
$\sco_X[-C_0 - L] \ra \sco_X[-L]$ of the composite map$\, :$ 
\[
\sco_X[-C_0 - L] \lra K \hookrightarrow \sco_X(1) \oplus \sco_X[-L] 
\] 
can be identified with the inclusion map $\sci_{C_0 \cup L , X} \hookrightarrow 
\sci_{L , X}$. Since $L$ is a 2-secant of the twisted cubic curve 
$C_0$ it follows that $C_0 \cup L$ is a complete intersection of type $(2,2)$ 
in $\piii$. Let $q_0 = q_1 = 0$ be (quadratic) equations of $C_0 \cup L$ in 
$\piii$ and $f = 0$ a cubic equation of $X$. Since $C_0 \cup L \subset X$, 
there exist linear forms $h_0$ and $h_1$ such that $f = q_0h_0 + q_1h_1$. Let 
$\sigma_i$ be the global section of $K(2)$ which is the image of 
$q_i \vb X \in \tH^0(\sci_{C_0 \cup L , X}(2))$, $i = 0,\, 1$.  
One deduces that$\, :$ 
\[
\sigma_0 \otimes h_0 + \sigma_1 \otimes h_1 \in \tH^0(\scn(1))\, . 
\]
$\sigma_i$ lifts to a global section of $\tH^0(\sck^\prim(2)) \subset 
\tH^0(\sco_\piii(3) \oplus \sci_L(2))$ which must be of the form 
$(f_i\, ,\, q_i)$, $i = 0,\, 1$. Since the composite morphism$\, :$ 
\[
\sco_\piii(-2) \oplus \sco_\piii(-3) \lra \sck^\prim \hookrightarrow 
\sco_\piii(1) \oplus \sci_L 
\]
is defined by the matrix $\left(\begin{smallmatrix} f & 0\\ 0 & f 
\end{smallmatrix}\right)$, one deduces that $\partial(1)(\sigma_0 \otimes h_0 + 
\sigma_1 \otimes h_1) = 1 \in \tH^0(\sco_\piii)$. 

\vskip2mm 

\noindent 
{\bf Claim 6.4.2.}\quad $K(2)$ \emph{is globally generated (for a suitable 
choice of} $W \subset C_0$\emph{)}. 

\vskip2mm 

\noindent 
\emph{Indeed}, using the exact sequence preceding Claim 6.4.1 and the fact, 
noticed in the proof of Claim 6.4.1, that $C_0 \cup L$ is a complete 
intersection of type $(2,2)$ in $\piii$, one sees that it suffices to prove 
that, choosing 
$W \subset C_0$ carefully, $\sci_{W , X}(3) \otimes \sco_X[C_0 - C]$ is globally 
generated. Since the plane curve $\Gamma_0 \cup \overline{P_1P_2} \subset 
\pii$ has degree 3, has a node at $P_1$, passes through $P_2$, $P_3$ and $P_4$ 
but not through $P_5$ and $P_6$, one deduces that$\, :$ 
\[
C \sim C_0 + L_{12} + E_3 + E_4 
\]  
as divisors on $X$ hence $\sco_X(3) \otimes \sco_X[C_0 - C] \simeq 
\sci_{L_{12} \cup E_3 \cup E_4 , X}(3)$. Since $L_{12}$, $E_3$, $E_4$ are mutually 
disjoint lines, $\sco_X(3) \otimes \sco_X[C_0 - C]$ is globally generated. 
Take a global section $\sigma$ of $\sco_X(3) \otimes \sco_X[C_0 - C]$, not 
vanishing identically on $L_1$ or $L_2$, and such that the intersection scheme 
$\{\sigma = 0\} \cap C_0$ on $X$ consists of six simple points of $C_0$, none 
of them situated on $C$. Take $W := \{\sigma = 0\} \cap C_0$ (as we saw above, 
when we effectively begun the construction, there exists a global section 
$\psi$ of $\sco_X(1) \otimes \sco_X[C_0 + L]$ such that $\{\psi = 0\} \cap 
C_0 = W$ hence the construction can be performed with this particular $W$). 
Since $\sco_X(3) \otimes \sco_X[-C] \simeq \sco_X[2L_1 + L_2]$, there exists a 
global section $\tau$ of $\sco_X(3) \otimes \sco_X[C_0 - C]$ whose zero 
divisor is $C_0 + 2L_1 + L_2$. Since the intersection multiplicity of $C_0$ 
and $L_1$ (resp., $L_2$) on $X$ is 2 (resp., 1), it follows that 
$\{\sigma = 0\} \cap L_i = \emptyset$, $i = 1,\, 2$. One deduces that 
$\{\sigma = 0\} \cap \{\tau = 0\} = \{\sigma = 0\} \cap C_0 = W$ hence 
$\sigma$ and $\tau$ define an epimorphism $2\sco_X \ra \sci_{W , X}(3) \otimes 
\sco_X[C_0 - C]$ which shows that $\sci_{W , X}(3) \otimes \sco_X[C_0 - C]$ is 
globally generated and this completes the proof of Claim 6.4.2.      

\vskip2mm 

\noindent 
{\bf Case 7.}\quad $F$ \emph{has spectrum} $(1,1,0,-1)$. 

\vskip2mm 

\noindent 
We will show that this case \emph{cannot occur}. Indeed, assume the contrary. 
Then $r = 4$ and $c_3(F) = -6$ hence $c_3 = 6$. Using the spectrum, 
one gets that $\tH^1(E(l)) = \tH^1(F(l+2)) = 0$ for $l \leq -5$, 
$\h^1(E(-4)) = 2$, $\h^1(E(-3)) = 5$. In particular, $s := \h^1(E(-3)) - 
\h^1(E(-4)) = 3$. Moreover, by Riemann-Roch, $\h^1(E(-2)) = 7$. 

\vskip2mm 

\noindent 
{\bf Claim 7.1.}\quad \emph{The graded} $S$-\emph{module} $\tH^1_\ast(E)$ 
\emph{has two minimal generators in degree} $-4$ \emph{and at most one in 
degree} $-2$. 

\vskip2mm 

\noindent 
\emph{Indeed}, by Remark~\ref{R:h2e(-3)=0}(i), the graded $S$-module 
$\tH^1_\ast(E)$ is generated in degrees $\leq -2$. Moreover, 
Remark~\ref{R:muh1e(-4)} implies that the multiplication map $\tH^1(E(-4)) 
\otimes S_1 \ra \tH^1(E(-3))$ is surjective.  

We want to show, now, that the multiplication map $\mu \colon \tH^1(E(-3)) 
\otimes S_1 \ra \tH^1(E(-2))$ has rank at least 6. 
\emph{Indeed}, assume, by contradiction, that $\mu$ has rank $\leq 5$. 
In this case, Remark~\ref{R:beilinson} implies that there exists an exact 
sequence$\, :$  
\[
0 \lra 2\Omega^3(3) \lra 5\Omega^2(2) \lra 5\Omega^1(1) \lra Q \lra 0\, ,  
\]
with $Q$ locally free. $Q$ must have rank 2 and Chern classes $c_1(Q) = 3$ 
and $c_2(Q) = 6$. Consider the normalized rank 2 vector bundle $Q^\prim := 
Q(-2)$. It has $c_1(Q^\prim) = -1$ and $c_2(Q^\prim) = 4$. 
It is well known that such a vector bundle must have $\tH^0(Q^\prim(2)) \neq 0$ 
(see Hartshorne and Hirschowitz \cite[Ex.~1.6.3]{hh}; see, also,   
Lemma~\ref{L:h0g(2)neq0} from Appendix~\ref{A:(2;-1,4,0)} for an 
alternative argument). But $\tH^0(Q^\prim(2)) = \tH^0(Q)$ and, from the above 
exact sequence, $\tH^0(Q) = 0$. This \emph{contradiction} shows that $\mu$ 
must have rank $\geq 6$ and Claim 7.1 is proven. 

\vskip2mm 

\noindent 
{\bf Claim 7.2.}\quad $\tH^1_\ast(E^\vee) = 0$. 

\vskip2mm 

\noindent 
\emph{Indeed}, Lemma~\ref{L:h2e(-3)=0}(c),(f) implies that $\h^1(E_H^\vee(1)) 
\leq 2$, for every plane $H \subset \piii$. Since $\h^2(E^\vee) = \h^1(E(-4)) 
= 2$, one gets, from Lemma~\ref{L:h2e(-3)=0}(d), that $\h^1(E^\vee(1)) = 0$.   

Lemma~\ref{L:h2e(-3)=0}(h) implies, now, that if $\tH^1_\ast(E^\vee) \neq 0$  
then $\tH^1_\ast(E^\vee) \simeq {\underline k}(-2)$. Consider the extension$\, :$ 
\begin{equation}\label{E:e(-2)rae2ra2o(2)+o} 
0 \lra E(-2) \lra E_2 \lra 2\sco_\piii(2) \oplus \sco_\piii \lra 0 
\end{equation} 
defined by the generators of $\tH^1_\ast(E)$ from Claim 7.1. One has, by 
construction, $\tH^1_\ast(E_2) = 0$ hence $\tH^2_\ast(E_2^\vee) = 0$. Dualizing 
the extension, one gets that $\tH^1_\ast(E_2^\vee) \simeq {\underline k}$ hence 
$E_2^\vee \simeq \Omega_\piii \oplus A$, where $A$ is a direct sum of line 
bundles. It follows that $E_2 \simeq \text{T}_\piii \oplus A^\vee$. But 
\emph{this is not possible} because, using the extension 
\eqref{E:e(-2)rae2ra2o(2)+o}, $\h^0(E_2) = \h^0(2\sco_\piii(2) \oplus \sco_\piii) 
- \h^1(E(-2)) = 14$ while $\h^0(\text{T}_\piii) = 15$. It remains that 
$\tH^1_\ast(E^\vee) = 0$ and Claim 7.2 is proven. 

\vskip2mm 

Claim 7.2 and Horrocks' criterion imply that the vector bundle $E_2$ defined 
by the extension \eqref{E:e(-2)rae2ra2o(2)+o} is a direct sum of line 
bundles. $E_2$ has rank 7, $c_1(E_2) = 1$, $\h^0(E_2) = 14$ (see above), 
$\h^0(E_2(-1)) = \h^0(2\sco_\piii(1) \oplus \sco_\piii(-1)) - \h^1(E(-3)) = 3$ 
and $\h^0(E_2(-2)) = 0$. It follows that $E_2 \simeq 3\sco_\piii(1) \oplus 
2\sco_\piii \oplus 2\sco_\piii(-1)$. One thus has an exact sequence$\, :$ 
\[
0 \lra E(-2) \lra 3\sco_\piii(1) \oplus 2\sco_\piii \oplus 2\sco_\piii(-1) 
\overset{\phi}{\lra} 2\sco_\piii(2) \oplus \sco_\piii \lra 0\, . 
\]    
Since there is no epimorphism $2\sco_\piii(-1) \ra \sco_\piii$, the component 
$2\sco_\piii \ra \sco_\piii$ of $\phi$ must be nonzero hence one has, actually, 
an exact sequence$\, :$ 
\[
0 \lra E(-2) \lra 3\sco_\piii(1) \oplus \sco_\piii \oplus 2\sco_\piii(-1) 
\overset{\psi}{\lra} 2\sco_\piii(2) \lra 0\, . 
\]
Let $\psi_1 \colon 3\sco_\piii(1) \oplus \sco_\piii \ra 2\sco_\piii(2)$ and 
$\psi_2 \colon 2\sco_\piii(-1) \ra 2\sco_\piii(2)$ be the components of $\psi$. 
The composite map $2\sco_\piii(-1) \overset{\psi_2}{\lra} 2\sco_\piii(2) \ra 
\Cok \psi_1$ is an epimorphism. Let us denote it by ${\overline \psi}_2$.  

\vskip2mm 

\noindent 
{\bf Claim 7.3.}\quad \emph{The degeneracy locus of} $\psi_1 \colon 
3\sco_\piii(1) \oplus \sco_\piii \ra 2\sco_\piii(2)$ \emph{has dimension} $0$. 

\vskip2mm 

\noindent 
\emph{Indeed}, one has an exact sequence$\, :$ 
\begin{equation}\label{E:era2o(1)racokpsi1(2)} 
E \lra 2\sco_\piii(1) \xra{{\overline \psi}_2(2)} \Cok \psi_1(2) \lra 0\, . 
\end{equation}
Since $c_1(2\sco_\piii(1)) = 2$ and $\Cok \psi_1(-2)$ is globally generated 
(because $\Cok \psi_1$ is a quotient of $2\sco_\piii(2)$), 
Lemma~\ref{L:suppcok} implies that the support of $\Cok \psi_1$ is 
0-dimensional. Claim 7.3 is proven.

\vskip2mm 

The kernel of $\psi_1$ is a rank 2 reflexive sheaf $\scf$ on $\piii$. 
Dualizing the exact sequence$\, :$ 
\[
0 \lra \scf \lra 3\sco_\piii(1) \oplus \sco_\piii \overset{\psi_1}{\lra} 
2\sco_\piii(2) \lra \Cok \psi_1 \lra 0\, , 
\]     
and taking into account Claim 7.3, one gets an exact sequence$\, :$ 
\[
0 \lra 2\sco_\piii(-2) \lra \sco_\piii \oplus 3\sco_\piii(-1) \lra \scf^\vee 
\lra 0\, . 
\]
Dualizing this sequence again, one gets that $\Cok \psi_1 \simeq 
\sce xt_{\sco_\piii}^1(\scf^\vee , \sco_\piii)$. Applying \cite[Prop.~2.6]{ha}, 
one deduces that $\Cok \psi_1$ has length $c_3(\scf^\vee) = 7$. 

But, on the other hand, applying Lemma~\ref{L:lengthcok} to the exact sequence 
\eqref{E:era2o(1)racokpsi1(2)}, one gets that $\Cok \psi_1$ has length at 
most 4. 
This \emph{contradiction} shows that Case 7 cannot occur in our context.  

\vskip2mm 

\noindent 
{\bf Case 8.} $F$ \emph{has spectrum} $(1,1,0,0)$. 

\vskip2mm 

\noindent 
We will show that this case \emph{cannot occur}. Indeed, assume the contrary. 
Then $r = 3$ (hence $E = F(2)$) and $c_3(F) = -8$ hence $c_3 = 4$. Using the 
spectrum, one gets that $\tH^1(E(l)) = \tH^1(F(l+2)) = 0$ for $l \leq -5$, 
$\h^1(E(-4)) = 2$, $\h^1(E(-3)) = 6$. Since $\tH^0(E^\vee(1)) = 
\tH^0(F^\vee(-1)) = 0$ (because $F$ is stable) Remark~\ref{R:h2e(-3)=0}(ii) 
implies that the graded $S$-module $\tH^1_\ast(E)$ is generated in degrees 
$\leq -3$. By Remark~\ref{R:muh1e(-4)}, the multiplication map $\mu \colon 
\tH^1(E(-4)) \otimes S_1 \ra \tH^1(E(-3))$ has rank $\geq 5$. It follows that 
$\tH^1_\ast(E)$ has two minimal generators in degree $-4$ and at most one in 
degree $-3$. 
 
We assert that the graded $S$-module $\tH^1_\ast(E^\vee)$ is generated by 
$\tH^1(E^\vee(1))$. \emph{Indeed}, one has $\h^1(E^\vee(1)) = \h^1(F^\vee(-1)) 
= 1$ (by Riemann-Roch). Lemma~\ref{L:h0f(1)=0c2=12} implies that 
$\tH^0(F(1)) = 0$. One deduces, from Lemma~\ref{L:h0fdual}, that 
$\h^0(F^\vee) \leq 1$. Since $\chi(F^\vee) = 0$, it follows that $\h^1(F^\vee) = 
\h^0(F^\vee) \leq 1$ hence $\h^1(E^\vee(2)) \leq 1$. Since $\h^2(E^\vee(1)) = 
\h^1(E(-5)) = 0$ one gets that $\h^1(E_H^\vee(2)) \leq 1$, for every plane 
$H \subset \piii$. Lemma~\ref{L:h2e(-3)=0} implies, now, that 
$\h^1(E_H^\vee(l)) = 0$ for $l \geq 3$, for each plane $H \subset \piii$. 
One deduces that $\h^1(E^\vee(l)) = 0$, $\forall \, l \geq 3$. Consequently, 
if $\tH^1_\ast(E^\vee)$ is not generated by $\tH^1(E^\vee(1))$ then 
$\tH^1_\ast(E^\vee) \simeq {\underline k}(-1) \oplus {\underline k}(-2)$.  
Consider the extension defined by the above mentioned generators of 
$\tH^1_\ast(E)$$\, :$  
\[
0 \lra E(-3) \lra E_3 \lra 2\sco_\piii(1) \oplus \sco_\piii \lra 0\, . 
\]          
The vector bundle $E_3$ has rank 6 and $\tH^1_\ast(E_3) = 0$ hence 
$\tH^2_\ast(E_3^\vee) = 0$. Moreover, $\tH^1_\ast(E_3^\vee) \simeq 
{\underline k}(2) \oplus {\underline k}(1)$. It follows that $E_3^\vee 
\simeq \Omega_\piii(2) \oplus \Omega_\piii(1)$ hence $E_3 \simeq 
\text{T}_\piii(-1) \oplus \text{T}_\piii(-2)$, by Horrocks theory.   
But this \emph{contradicts} the fact that $c_1(E_3) = -2$. It thus remains 
that $\tH^1_\ast(E^\vee)$ is generated by $\tH^1(E^\vee(1))$.    

It follows, now, that $F = E(-2)$ is the cohomology of a Horrocks monad of the 
form$\, :$ 
\[
0 \lra \sco_\piii(-1) \overset{\beta}{\lra} B 
\overset{\alpha}{\lra} 2\sco_\piii(2) \oplus \sco_\piii(1) \lra 0\, , 
\]
where $B$ is a direct sum of line bundles. $B$ has rank 7, $c_1(B) = 3$ and 
$\h^0(B(-1)) = \h^0(2\sco_\piii(1) \oplus \sco_\piii) - \h^1(E(-3)) = 3$, hence 
$B \simeq 3\sco_\piii(1) \oplus 4\sco_\piii$.  

\vskip2mm 

\noindent 
{\bf Claim 8.1.}\quad \emph{The component} $\alpha_{21} \colon 3\sco_\piii(1) 
\ra \sco_\piii(1)$ \emph{of} $\alpha$ \emph{is non-zero}. 

\vskip2mm 

\noindent 
\emph{Indeed}, assume, by contradiction, that $\alpha_{21} = 0$. Then 
$F$ is the cohomology of a monad of the form$\, :$ 
\[
0 \lra \sco_\piii(-1) \overset{\beta^\prim}{\lra} 3\sco_\piii(1) \oplus 
\Omega_\piii(1) \overset{\alpha^\prim}{\lra} 2\sco_\piii(2) \lra 0\, . 
\]
The component $\beta^\prim_2 \colon \sco_\piii(-1) \ra \Omega_\piii(1)$ must 
be non-zero. Then the zero scheme of the global section of $\Omega_\piii(2)$ 
defining $\beta_2^\prim$ is either empty or it is a line $L \subset \piii$. 
Since the cokernel of $\tH^0(\beta_2^{\prim \vee})$ is isomorphic to 
$\tH^1(F^\vee)$ which, as we saw above, has dimension at most 1, this zero 
scheme must be empty. It follows that one has an exact sequence$\, :$ 
\[
0 \lra \sco_\piii(-1) \overset{\beta_2^\prim}{\lra} \Omega_\piii(1) \lra 
N \lra 0\, , 
\]  
where $N$ is a nullcorrelation bundle. Since $\text{Ext}^1(N , \sco_\piii(1)) 
\simeq \tH^1(N^\vee(1)) \simeq \tH^1(N(1)) = 0$, $\beta_2^\prim$ 
induces a surjection$\, :$ 
\[
\text{Hom}(\Omega_\piii(1) , \sco_\piii(1)) \twoheadrightarrow  
\text{Hom}(\sco_\piii(-1) , \sco_\piii(1))\, . 
\]
One can, thus, assume, modulo an automorphism of $3\sco_\piii(1) \oplus 
\Omega_\piii(1)$, that the other component $\beta_1^\prim \colon \sco_\piii(-1) 
\ra 3\sco_\piii(1)$ of $\beta^\prim$ is 0. One deduces an exact sequence$\, :$ 
\[
0 \lra F \lra 3\sco_\piii(1) \oplus N \lra 2\sco_\piii(2) \lra 0\, . 
\]

We assert, now, that, for any morphism $\gamma \colon 3\sco_\piii \ra 
2\sco_\piii(1)$, $\tH^0(\gamma(2)) \colon \tH^0(3\sco_\piii(2)) \ra 
\tH^0(2\sco_\piii(3))$ is not injective. This will imply that $\tH^0(F(1)) 
\neq 0$ which \emph{contradicts} Lemma~\ref{L:h0f(1)=0c2=12}, proving the 
claim. $\gamma$ is defined by a $2 \times 3$ matrix of linear forms 
\[
A := \begin{pmatrix} 
h_{11} & h_{12} & h_{13}\\ 
h_{21} & h_{22} & h_{23} 
\end{pmatrix}\, . 
\]   
We have to show that there exists a non-trivial relation with quadratic 
entries among the columns of this matrix. 
Assume, firstly, that all the $2 \times 2$ minors of $A$ are 0. 
Suppose, to fix the ideas, that $h_{11} \neq 0$. If $h_{21}$ is proportional 
to $h_{11}$ then one can assume that $h_{21} = 0$. It follows, in this case, 
that $h_{22} = 0$ and $h_{23} = 0$. If $h_{11}$ and $h_{21}$ are linearly 
independent then, using a divisibility argument, one deduces that the second 
and the third columns of $A$ are proportional to the first one. In both cases, 
there exists a non-trivial relation with linear (or constant) entries 
among the columns of $A$. 

If at least one of the $2 \times 2$ minors of $A$ is not 0 then there exists 
a nontrivial relation among the columns of $A$ whose entries are exactly 
these $2 \times 2$ minors. This proves our assertion above about $\gamma$ and, 
with it, Claim 8.1.

\vskip2mm 

It follows, from Claim 8.1, that $F = E(-2)$ is the cohomology of a monad of 
the form$\, :$ 
\begin{equation}\label{E:monad8}
0 \lra \sco_\piii(-1) \overset{\beta^\secund}{\lra} 2\sco_\piii(1) \oplus 
4\sco_\piii \overset{\alpha^\secund}{\lra} 2\sco_\piii(2) \lra 0\, . 
\end{equation}

\vskip2mm 

\noindent 
{\bf Claim 8.2.}\quad \emph{Let} $F$ \emph{be the cohomology of a monad of the 
form} \eqref{E:monad8}. \emph{Assume that} $\tH^0(F) = 0$ \emph{and} 
$\tH^0(F^\vee) = 0$. \emph{Then} $F(2)$ \emph{is not globally generated}. 

\vskip2mm 

\noindent 
\emph{Indeed}, assume, by contradiction, that there exists $F$ satisfying the 
hypothesis of Claim 8.2 such that $F(2)$ is globally generated. Since, as we 
noticed at the beginning of Case 8, $\h^1(F^\vee) = \h^0(F^\vee)$ it follows 
that $\h^1(F^\vee) = 0$. This is equivalent to the fact that 
$\tH^0(\beta^{\secund \vee}) \colon \tH^0(2\sco_\piii(-1) \oplus 4\sco_\piii) \ra 
\tH^0(\sco_\piii(1))$ is surjective. In this case, the component 
$\beta_2^\secund \colon \sco_\piii(-1) \ra 4\sco_\piii$ of $\beta^\secund$ is 
defined by four linearly independent linear forms. Up to an automorphism of 
$2\sco_\piii(1) \oplus 4\sco_\piii$, one can assume that the other component 
$\beta_1^\secund \colon \sco_\piii(-1) \ra 2\sco_\piii(1)$ of $\beta^\secund$ is 0. 
It follows that one has an exact sequence$\, :$ 
\[
0 \lra F \lra 2\sco_\piii(1) \oplus \text{T}_\piii(-1) \overset{\phi}{\lra} 
2\sco_\piii(2) \lra 0\, . 
\]  

We show, firstly, that the degeneracy locus $Q$ of the component $\phi_1 
\colon 2\sco_\piii(1) \ra 2\sco_\piii(2)$ of $\phi$ is either a nonsingular 
quadric surface or a quadric cone. \emph{Indeed}, $\phi_1$ is defined 
by a $2\times 2$ matrix of linear forms $(h_{ij})_{1 \leq i,\, j \leq 2}$. This 
matrix has the property that, for any nontrivial linear combination$\, :$ 
\[
\begin{pmatrix} h_1^\prime\\ h_2^\prime \end{pmatrix} = 
a_1 \begin{pmatrix} h_{11}\\ h_{21} \end{pmatrix} + 
a_2 \begin{pmatrix} h_{12}\\ h_{22} \end{pmatrix} 
\]  
of its columns, $h_1^\prime$ and $h_2^\prime$ are linearly independent. 
Indeed, if $h_1^\prime$ and $h_2^\prime$ would be multiples of a 
linear form $h^\prime$ then, denoting by $H^\prim$ the plane of equation 
$h^\prime = 0$, one would get that $\tH^0(F_{H^\prim}(-1)) \neq 0$, which would 
contradict Remark~\ref{R:c1=5onp2}. In particular, $h_{11}$ and 
$h_{21}$ are linearly independent, hence they define a line $L \subset \piii$. 
The degeneracy locus $Q$, which is defined by the determinant of the matrix 
$(h_{ij})$, is an effective divisor of degree 2 on $\piii$, containing the 
line $L$. If $Q$ would be the union of two planes (or a double plane) then 
one of the planes would contain $L$ hence one would have$\, :$ 
\[
h_{11}h_{22} - h_{12}h_{21} = h(b_1h_{11} + b_2h_{21}) 
\]   
for some linear form $h$ and for some constants $b_1,\, b_2 \in k$. 
Since $h_{11}$ and $h_{21}$ are linearly independent, it would follow that there 
exists $b \in k$ such that $h_{22} - b_1h = bh_{21}$ and $h_{12} + b_2h = 
bh_{11}$. But this would \emph{contradict} the above mentioned property of the 
columns of $(h_{ij})$. It remains that $Q$ is a prime divisor, as asserted. 

The above argument also shows that $\Cok \phi_1 \simeq \sci_{L,Q}(3)$ and that 
one has an exact sequence$\, :$ 
\[
0 \lra F \lra \text{T}_\piii(-1) \overset{{\overline \phi}_2}{\lra} 
\sci_{L,Q}(3) \lra 0\, , 
\]
with ${\overline \phi}_2$ defined by the other component  
$\phi_2 \colon \text{T}_\piii(-1) \ra 2\sco_\piii(2)$ of $\phi$. Denoting by 
$\sck$ the kernel of the restriction $\text{T}_\piii(-1) \vb Q \ra 
\sci_{L,Q}(3)$ of ${\overline \phi}_2$ to $Q$, one gets an exact sequence$\, :$ 
\begin{equation}\label{E:t(-3)rafrasck} 
0 \lra \text{T}_\piii(-3) \lra F \lra \sck \lra 0\, . 
\end{equation}
$\sck$ is a rank 2 reflexive $\sco_Q$-module. We will show that 
$\h^0(\sck(2)) = 2$ which will imply that $\sck(2)$ is not globally generated, 
hence that $F(2)$ is not globally generated. 

\vskip2mm 

If $Q$ is a nonsingular quadric surface, fix an isomorphism $Q \simeq \pj 
\times \pj$ such that $L$ belongs to the linear system 
$\vert \, \sco_Q(1,0)\, \vert$. Then $\sci_{L,Q}(3) \simeq \sco_Q(2,3)$ hence 
$\text{det}\, \sck \simeq \sco_Q(-1,-2)$ and $\sck^\vee \simeq \sck(1,2)$. 
Dualizing the exact sequence $0 \ra \sck \ra \text{T}_\piii(-1) \vb Q \ra 
\sco_Q(2,3) \ra 0$ and then tensorizing by $\sco_Q(1,0)$ one gets an exact 
sequence$\, :$ 
\[
0 \lra \sco_Q(-1,-3) \lra (\Omega_\piii(1) \vb Q) \otimes \sco_Q(1,0) \lra 
\sck(2,2) \lra 0\, . 
\]     
But $(\Omega_\piii(1) \vb Q) \otimes \sco_Q(1,0) \simeq 
(\Omega_\piii(2) \vb Q) \otimes \sco_Q(0,-1) \simeq 
(\Omega_\piii(2) \vb Q) \otimes \sci_{L_1,Q}$, where $L_1$ is a line belonging 
to the linear system $\vert \, \sco_Q(0,1)\, \vert$. Tensorizing by 
$\Omega_\piii(2)$ the exact sequence$\, :$ 
\[
0 \lra \sco_\piii(-2) \lra \sci_{L_1} \lra \sci_{L_1,Q} \lra 0\, , 
\]
one deduces that $\h^0((\Omega_\piii(2) \vb Q) \otimes \sci_{L_1,Q}) = 2$ hence 
$\h^0(\sck(2,2)) = 2$. 

\vskip2mm 

If $Q$ is a quadric cone, the argument is technically more complicated due to 
the singular point $P$ of $Q$. We begin by stating a general fact : let $X$ 
be a scheme of finite type over $k$, and let $\e \colon \sce \ra \sci$ be an 
epimorphism, with $\sce$ a locally free $\sco_X$-module and $\sci$ an ideal 
sheaf on $X$. Put $\sci^{-1} := \sch om_{\sco_X}(\sci , \sco_X)$ and 
$\sci^{-i} := \sch om_{\sco_X}(\sci , \sci^{-i+1})$ for $i \geq 2$. Then 
one can define a complex $\scc^\bullet$ of Koszul type$\, :$ 
\[
\cdots \lra {\textstyle \bigwedge}^3 \sce \otimes \sci^{-2} \lra 
{\textstyle \bigwedge}^2 \sce \otimes \sci^{-1} \lra 
\sce \overset{\e}{\lra} \sci \lra 0\, . 
\]  
If $U$ is an open subset of $X$ such that $\sci \vb U$ is the ideal sheaf of an 
effective Cartier divisor on $U$ then $\scc^\bullet \vb U$ is exact. 

Now, taking into account that $\sch om_{\sco_Q}(\sci_{L,Q},\sco_Q) \simeq 
\sci_{L,Q}(1)$ (both are reflexive $\sco_Q$-modules of rank 1 and are 
isomorphic over $Q \setminus \{P\}$), that $\sco_Q \izo 
\sch om_{\sco_Q}(\sci_{L,Q},\sci_{L,Q})$ (if $U \subset Q$ is a neighbourhood of 
$P$ then two morphisms $\sci_{L,Q} \vb U \ra \sci_{L,Q} \vb U$ coinciding over 
$U \setminus \{P\}$ coincide over $U$ and $\sco_Q(U) \izo 
\sco_Q(U\setminus \{P\})$) and that 
${\bigwedge}^2(\text{T}_\piii(-1)) \simeq \Omega_\piii(2)$, one 
gets a complex$\, :$ 
\[
0 \lra \sco_Q(-4) \overset{d^{-3}}{\lra} 
(\Omega_\piii \vb Q) \otimes \sci_{L,Q} \overset{d^{-2}}{\lra} 
\text{T}_\piii(-1) \vb Q \overset{d^{-1}}{\lra} \sci_{L,Q}(3) \lra 0\, ,  
\]
with $d^{-1}$ the restriction of ${\overline \phi}_2$ to $Q$. 
This complex is exact over $Q \setminus \{P\}$ hence its cohomology sheaves 
are concentrated in $P$. It follows, in particular, that $d^{-3}$ is a 
monomorphism. One has $\Ker d^{-1} = \sck$. 
Let $\sck^\prim$ be the cokernel of $d^{-3}$. 
$d^{-2}$ induces a morphism $\gamma \colon \sck^\prim \ra \sck$, which is 
an isomorphism over $Q \setminus \{P\}$. Using the exact sequence$\, :$ 
\[
0 \lra \sco_Q(-2) \lra (\Omega_\piii(2) \vb Q) \otimes \sci_{L,Q} 
\lra \sck^\prim(2) \lra 0\, , 
\] 
one deduces as above (in the case where $Q$ was a nonsingular quadric) 
that $\tH^1(\sck^\prim(2)) = 0$. Using the exact sequences$\, :$ 
\[ 
0 \ra \Ker \gamma \ra \sck^\prim \ra \text{Im}\, \gamma \ra 0\, ,\  
0 \ra \text{Im}\, \gamma \ra \sck \ra \Cok \gamma \ra 0\, , 
\]
one gets that $\tH^1(\sck(2)) = 0$. It follows, using the exact sequence 
\eqref{E:t(-3)rafrasck}, that $\tH^1(F(2)) = 0$. Applying Riemann-Roch to 
$F(2)$ one gets that $\h^0(F(2)) = 6$ hence, using again the exact 
sequence \eqref{E:t(-3)rafrasck}, $\h^0(\sck(2)) = 2$. This completes the 
proof of Claim 8.2.  

\vskip2mm 

\noindent 
{\bf Claim 8.3.}\quad \emph{Let} $F$ \emph{be the cohomology of a monad of the 
form} \eqref{E:monad8}. \emph{Assume that} $\tH^0(F) = 0$ \emph{and} 
$\h^0(F^\vee) = 1$. \emph{Then} $F(2)$ \emph{is not globally generated}.  

\vskip2mm 

\noindent 
\emph{Indeed}, assume, by contradiction, that there exists $F$ satisfying the 
hypothesis of Claim 8.3 such that $F(2)$ is globally generated. Since, as 
we noticed at the beginning of Case 8, $\h^1(F^\vee) = \h^0(F^\vee)$, it follows 
that $\h^1(F^\vee) = 1$. This means that $\tH^0(\beta^{\secund \vee}) \colon 
\tH^0(2\sco_\piii(-1) \oplus 4\sco_\piii) \ra \tH^0(\sco_\piii(1))$ has corank 1. 
One can easily show that, up to an automorphism of $2\sco_\piii(1) \oplus 
4\sco_\piii$, one can assume that the components of $\beta^\secund$ are 
$0,\, q,\, h_0,\, h_1,\, h_2,\, 0$, where $h_0,\, h_1,\, h_2$ are linearly 
independent linear forms, vanishing at a point $x \in \piii$, and $q$ is a 
quadratic form, not vanishing at $x$. If $G$ is the rank 3 vector bundle on 
$\piii$ defined by the exact sequence$\, :$ 
\[
0 \lra \sco_\piii(-1) \xra{(q,\, h_0,\, h_1,\, h_2)^{\text{t}}} \sco_\piii(1) 
\oplus 3\sco_\piii \lra G \lra 0\, , 
\]  
then one has an exact sequence$\, :$ 
\[
0 \lra F \lra \sco_\piii(1) \oplus G \oplus \sco_\piii 
\overset{\phi}{\lra} 2\sco_\piii(2) \lra 0\, . 
\]

Let $\phi_1$, $\phi_2$ and $\phi_3$ be the restrictions of $\phi$ to 
$\sco_\piii(1)$, $G$ and $\sco_\piii$, respectively, and let $\scc_{12}$ be the 
cokernel of the restriction of $\phi$ to $\sco_\piii(1) \oplus G$.  
Since the composite map $\sco_\piii \overset{\phi_3}{\lra} 2\sco_\piii(2) 
\ra \scc_{12}$ is an epimorphism, it follows that $\scc_{12} \simeq \sco_W$, for 
some closed subscheme $W$ of $\piii$. We assert that $W$ has dimension 0. 
\emph{Indeed}, $W$ cannot contain a reduced and irreducible curve $\Gamma$ 
because there is no epimorphism $2\sco_\piii(2) \ra \sco_\Gamma$. Now, one has 
an exact sequence$\, :$ 
\[
0 \lra \scf \lra \sco_\piii(1) \oplus G \lra 2\sco_\piii(2) \lra \sco_W 
\lra 0\, , 
\] 
where $\scf$ is a rank 2 reflexive sheaf. Dualizing this exact sequence, one 
gets an exact sequence$\, :$ 
\[
0 \lra 2\sco_\piii(-2) \lra \sco_\piii(-1) \oplus G^\vee \lra \scf^\vee \lra 0\, ,  
\]
from which one can compute the Chern classes of $\scf^\vee$. One gets 
$c_1(\scf^\vee) = 1$, $c_2(\scf^\vee) = 4$ and $c_3(\scf^\vee) = 8$. Since 
$\sce xt^1_{\sco_\piii}(\scf^\vee , \sco_\piii) \simeq \sco_W$, it follows, 
from \cite[Prop.~2.6]{ha}, that $W$ has length 8. 

We assert, next, that the degeneracy locus of $\phi_2 \colon G \ra 
2\sco_\piii(2)$ has codimension 2 in $\piii$. Indeed, if $\scc_2$ is the 
cokernel of $\phi_2$ then one has an exact sequence $\sco_\piii(1) 
\overset{{\overline \phi}_1}{\lra} \scc_2 \ra \scc_{12} \ra 0$, whence an exact 
sequence$\, :$ 
%\label{E:scc2}
\begin{equation*}
0 \lra \sco_Z(1) \lra \scc_2 \lra \sco_W \lra 0\, , 
\end{equation*} 
where $Z$ is a closed subscheme of $\piii$. If one would have $\dim Z \geq 2$ 
then one could choose a reduced and irreducible curve $\Gamma \subset Z$ such 
that $\Gamma \cap W = \emptyset$. Then one could lift the epimorphism 
$\sco_Z(1) \ra \sco_\Gamma(1)$ to an epimorphism $\scc_2 \ra \sco_\Gamma(1)$ (one 
glues the restriction of $\sco_Z(1) \ra \sco_\Gamma(1)$ to $\piii \setminus W$ 
and the zero morphism $\scc_2 \ra \sco_\Gamma(1)$ restricted to $\piii 
\setminus \Gamma$). Since there is no epimorphism $2\sco_\piii(2) \ra 
\sco_\Gamma(1)$, one would get a \emph{contradiction}. It thus remains that 
$\dim Z \leq 1$. 

Let, now, $C$ be the degeneracy scheme of $\phi_2$. By definition, $\sci_C(4)$ 
is the image of $\bigwedge^2\phi_2 \colon \bigwedge^2G \ra 
\bigwedge^2(2\sco_\piii(2)) \simeq \sco_\piii(4)$. By the theory of the 
Eagon-Northcott complex, one has an exact sequence$\, :$ 
\begin{equation}\label{E:2o(-2)ragdualraic(2)}
0 \lra 2\sco_\piii(-2) \overset{\phi_2^\vee}{\lra} G^\vee \overset{\pi}{\lra} 
\sci_C(2) \lra 0\, , 
\end{equation} 
where $\pi(2) \colon G^\vee(2) \ra \sci_C(4)$ can be identified with 
$\bigwedge^2\phi_2$ via the isomorphisms $\bigwedge^2G \simeq 
\sch om_{\sco_\piii}(G , \bigwedge^3G) \simeq G^\vee(2)$. It follows that $C$ is 
a locally Cohen-Macaulay curve of degree 6. Dualizing the above exact 
sequence, one gets an exact sequence$\, :$ 
\[
0 \lra \sco_\piii(-2) \overset{\pi^\vee}{\lra} G \overset{\phi_2}{\lra} 
2\sco_\piii(2) \overset{\e}{\lra} \omega_C(2) \lra 0\, , 
\] 
hence $\scc_2 \simeq \omega_C(2)$. Using the exact sequence$\, :$ 
\[
\sco_\piii(1) \xra{\e \circ \phi_1} \omega_C(2) \lra \sco_W \lra 0\, , 
\] 
one deduces that $C$ is locally complete intersection outside $W$. 

Now, one has an exact sequence$\, :$ 
\[
0 \lra \sco_\piii(-2) \overset{\pi^\vee}{\lra} F \lra \sco_\piii(1) \oplus 
\sco_\piii \overset{\delta}{\lra} \omega_C(2) \lra 0\, , 
\] 
where the components of $\delta$ are $\e \circ \phi_1$ and $\e \circ \phi_3$. 
If $\sck$ is the kernel of $\delta$ then $F(2)$ globally generated implies 
that $\sck(2)$ is globally generated. One has an exact sequence$\, :$ 
\begin{equation}\label{E:ic(1)rasckraiw}
0 \lra \sci_C(1) \lra \sck \lra \sci_W \lra 0\, .  
\end{equation} 
Since $\sck(2)$ is globally generated it follows that $\sci_W(2)$ is 
globally generated. But $W$ is a 0-dimensional subscheme of $\piii$ of 
length 8 hence $\sci_W(2)$ globally generated implies that $W$ is a complete 
intersection of type $(2,2,2)$ in $\piii$. Taking into account that 
$\tH^1(\sci_C(3)) = 0$ (from the exact sequence 
\eqref{E:2o(-2)ragdualraic(2)}), one gets a commutative diagram$\, :$   
\[
\begin{CD} 
0 @>>> \tH^0(\sci_C(3))\otimes \sco_\piii @>>> 
\tH^0(\sck(2)) \otimes \sco_\piii @>>> 
\tH^0(\sci_W(2)) \otimes \sco_\piii @>>> 0\\ 
@. @V{\text{ev}}VV @V{\text{ev}}VV @V{\text{ev}}VV\\ 
0 @>>> \sci_C(3) @>>> \sck(2) @>>> \sci_W(2) @>>> 0 
\end{CD}
\] 
Let $\scq$ (resp., $\scn$) be the cokernel (resp., kernel) of the evaluation 
morphism of $\sci_C(3)$ (resp., $\sci_W(2)$). Then $\sck(2)$ is globally 
generated if and only if the connecting morphism $\partial \colon \scn \ra 
\scq$ associated to the above diagram is an epimorphism. 

We describe, now, $\scq$ and $\partial$ explicitly. One has an exact 
sequence$\, :$ 
\[
0 \lra \sco_\piii(1) \overset{\sigma}{\lra} G \overset{\tau}{\lra} 
\scg \lra 0\, , 
\]
where $\scg$ is the rank 2 reflexive sheaf defined by the exact sequence$\, :$ 
\[
0 \lra \sco_\piii(-1) \xra{(h_0,\, h_1,\, h_2)^{\text{t}}} 3\sco_\piii \lra 
\scg \lra 0\, . 
\]
$\phi_2 \circ \sigma \colon \sco_\piii(1) \ra 2\sco_\piii(2)$ is defined by 
two linear forms $h_{12}$ and $h_{22}$. Since $\tH^0(F_H(-1)) = 0$, for every 
plane $H \subset \piii$, by Remark~\ref{R:c1=5onp2}, it follows that $h_{12}$ 
and $h_{22}$ are linearly independent. Let $L \subset \piii$ be the line of 
equations $h_{12} = h_{22} = 0$. One has an exact sequence$\, :$ 
\[
0 \lra \sco_\piii(1) \xra{\phi_2 \circ \sigma} 2\sco_\piii(2) 
\xra{(-h_{22}\, ,\, h_{12})} \sci_L(3) \lra 0\, .  
\]
Using the commutative diagram$\, :$ 
\[
\begin{CD} 
@. @. \sco_\piii(1) @= \sco_\piii(1)\\ 
@. @. @VV{\sigma}V @VV{\phi_2 \circ \sigma}V\\ 
0 @>>> \sco_\piii(-2) @>{\pi^\vee}>> G @>{\phi_2}>> 2\sco_\piii(2)\\ 
@. @\vert @VV{\tau}V @VVV\\ 
0 @>>> \sco_\piii(-2) @>{\tau \circ \pi^\vee}>> \scg @>{\phi_2^\prime}>> 
\sci_L(3)
\end{CD}
\]
with $\phi_2^\prime$ induced by $\phi_2$, one gets an exact sequence$\, :$ 
\[
0 \lra \sco_\piii(-2) \xra{\tau \circ \pi^\vee} \scg 
\overset{\phi_2^\prime}{\lra} \sci_L(3) \overset{\e^\prim}{\lra} \omega_C(2) 
\lra 0\, , 
\]  
with $\e^\prim$ induced by $\e \colon 2\sco_\piii(2) \ra \omega_C(2)$. It 
follows that the image of $\phi_2^\prime$ is of the form $\sci_{C^\prim}(3)$, 
where $C^\prim$ is a closed subscheme of $\piii$ such that $C^\prim = C \cup L$, 
set theoretically. Since one has an exact sequence$\, :$ 
\[
0 \lra \sco_\piii(-1) \oplus \sco_\piii(-2) \lra 3\sco_\piii \lra 
\sci_{C^\prim}(3) \lra 0\, , 
\]
$C^\prim$ is locally Cohen-Macaulay, of degree 7. Considering, now, the dual 
of the above commutative diagram$\, :$ 
\[
\begin{CD} 
@. \sci_{\{x\}}(-1) @<{}<{}< \sci_L(-1)\\ 
@. @AA{\sigma^\vee}A @AA{\sigma^\vee \circ \phi_2^\vee}A\\ 
\sco_\piii(2) @<{\pi}<{}< G^\vee @<{\phi_2^\vee}<{}< 2\sco_\piii(-2) 
@<{}<{}< 0\\
@\vert @AA{\tau^\vee}A @AAA\\ 
\sco_\piii(-2) @<{\pi \circ \tau^\vee}<{}< \scg^\vee 
@<{\phi_2^{\prime \vee}}<{}< 
\sco_\piii(-3) @<{}<{}< 0    
\end{CD}
\]  
and taking into account that the image of $\pi \circ \tau^\vee$ must be 
$\sci_{C^\prim}(2)$ (because $\scg$ is reflexive of rank 2) one gets an exact 
sequence$\, :$ 
\[
0 \lra \sci_{C^\prim}(2) \lra \sci_C(2) \lra \sci_{\{x\}, L}(-1) \lra 0\, . 
\]
Since $\sci_{C^\prim}(3)$ is globally generated, one deduces that the cokernel 
$\scq$ of the evaluation morphism of $\sci_C(3)$ is isomorphic to 
$\sci_{\{x\}, L} \simeq \sco_L(-1)$. 

Let us denote, now, by $\phi_1^\prime$ (resp., $\phi_3^\prime$) the composite 
map$\, :$ 
\[
\sco_\piii(1) \overset{\phi_1}{\lra} 2\sco_\piii(2) \xra{(-h_{22}\, ,\, h_{21})} 
\sci_L(3)\  (\text{resp.,}\  \sco_\piii \overset{\phi_3}{\lra} 2\sco_\piii(2) 
\xra{(-h_{22}\, ,\, h_{21})} \sci_L(3))\, . 
\]
$\phi_1^\prime$ (resp., $\phi_3^\prime$) is defined by multiplication with a 
form $q \in \tH^0(\sci_L(2))$ (resp., $f \in \tH^0(\sci_L(3))$) and one has 
$\e \circ \phi_i = \e^\prim \circ \phi_i^\prime$, $i = 1,\, 3$. Recalling the 
exact sequence$\, :$ 
\[
0 \lra \sci_{C^\prim}(3) \lra \sci_L(3) \overset{\e^\prim}{\lra} \omega_C(2) 
\lra 0\, , 
\]
one gets that$\, :$ 
\[
\tH^0(\sck(2)) = \{(f^\prim\, ,\, q^\prim) \in \tH^0(\sco_\piii(3) \oplus 
\sco_\piii(2)) \vb f^\prim q + q^\prim f \in \tH^0(\sci_{C^\prim}(5))\}\, . 
\]
Notice that if $(f^\prim\, ,\, q^\prim) \in \tH^0(\sck(2))$ then, by the exact 
sequence \eqref{E:ic(1)rasckraiw}, $q^\prim \in \tH^0(\sci_W(2))$. Since 
$(-f\, ,\, q)$ obviously belongs to $\tH^0(\sck(2))$, it follows that 
$q \in \tH^0(\sci_W(2))$. $q$ belongs to a $k$-basis $q,\, q_1,\, q_2$ of 
$\tH^0(\sci_W(2))$ (recall that $W$ is a complete intersection of type 
$(2,2,2)$ in $\piii$). $\scn(2)$ is globally generated and $\tH^0(\scn(2))  
\subset \tH^0(\sci_W(2)) \otimes \tH^0(\sco_\piii(2))$ has a $k$-basis 
consisting of the elements$\, :$ 
\[
\rho_i := q_i \otimes q - q \otimes q_i\, ,\  i =1,\,2\, ,\ \text{and}\   
\rho_{12} := q_2 \otimes q_1 - q_1 \otimes q_2\, . 
\] 
$q_i$ can be lifted to an element $(f_i\, ,\, q_i)$ of $\tH^0(\sck(2))$, 
$i = 1,\, 2$, while $q$ lifts to the element $(-f\, ,\,q)$ of 
$\tH^0(\sck(2))$. Recalling that the cokernel $\scq$ of the evaluation 
morphism of $\sci_C(3)$ equals $(\sci_C/\sci_{C^\prim})(3) \simeq \sco_L(-1)$, 
one gets that$\, :$ 
\[
\partial(2)(\rho_i) = \text{the image of } -fq_i - qf_i\  \text{in}\  
\tH^0((\sci_C/\sci_{C^\prim})(5))\, ,\  i =1,\, 2\, .  
\]
Since, by the above description of $\tH^0(\sck(2))$, $-fq_i - qf_i \in 
\tH^0(\sci_{C^\prim}(5))$, it follows that $\partial(2)(\rho_i) = 0$, 
$i = 1,\, 2$. Since $\scq(2) \simeq \sco_L(1)$ cannot be generated by only 
one global section, one deduces that $\partial \colon \scn \ra \scq$ is not 
an epimorphism, hence $\sck(2)$ is not globally generated. Concluding, the 
assumption that $F(2)$ is globally generated leads to a \emph{contradiction} 
and Claim 8.3 is proven. 

For a result related to the above considerations, see Lemma~\ref{L:d=6g=2} 
from Appendix~\ref{A:miscellaneous}.

\vskip2mm 

\noindent 
{\bf Case 9.}\quad $F$ \emph{has spectrum} $(1,1,1,0)$. 

\vskip2mm 

\noindent 
This case \emph{cannot occur}. Indeed, assume the contrary. 
Then $r = 3$ (hence $E = F(2)$) and $c_3(F) = -10$ hence $c_3 = 2$. Using the 
spectrum, one gets that $\tH^1(E(l)) = \tH^1(F(l+2)) = 0$ for $l \leq 
-5$, $\h^1(E(-4)) = 3$, $\h^1(E(-3)) = 7$. Since $\tH^0(E^\vee(1)) = 
\tH^0(F^\vee(-1)) = 0$ (because $F$ is stable) it follows, from 
Remark~\ref{R:h2e(-3)=0}(ii), that the graded $S$-module $\tH^1_\ast(E)$ is 
generated in degrees $\leq -3$. By Remark~\ref{R:muh1e(-4)},  
the multiplication map $\mu \colon \tH^1(E(-4)) \otimes S_1 \ra \tH^1(E(-3))$  
has rank $\geq 6$. It follows that $\tH^1_\ast(E)$ has three minimal generators 
in degree $-4$ and at most one in degree $-3$. 

We will show, now, that $\tH^2_\ast(E) = 0$. Indeed, by Serre duality, this is 
equivalent to $\tH^1_\ast(E^\vee) = 0$.  By Riemann-Roch, $\h^1(E^\vee(1)) = 
\h^1(F^\vee(-1)) = 0$ and $\h^0(F^\vee) - \h^1(F^\vee) = \chi(F^\vee) = 1$. Since, 
by Lemma~\ref{L:h0fdual} and by Lemma~\ref{L:h0f(1)=0c2=12}, $\h^0(F^\vee) \leq 
1$, one deduces that $\h^1(F^\vee) = 0$, i.e., $\h^1(E^\vee(2)) = 0$. 
Lemma~\ref{L:h2e(-3)=0}(a),(b) implies, now, that $\tH^1_\ast(E^\vee) = 0$. 

Consider, now, the extension defined by the above mentioned generators of  
$\tH^1_\ast(E)$$\, :$ 
\[
0 \lra E(-3) \lra E_3 \lra 3\sco_\piii(1) \oplus \sco_\piii \lra 0\, . 
\]          
One has $\tH^1_\ast(E_3) = 0$. Moreover, the fact that $\tH^2_\ast(E) = 0$ 
implies that $\tH^2_\ast(E_3) = 0$. It follows, from Horrocks' criterion, that 
$E_3$ is a direct sum of line bundles. 
$E_3$ has rank 7, $c_1(E_3) = -1$, $\h^0(E_3(-1)) = 0$ and $\h^0(E_3) = 
\h^0(3\sco_\piii(1) \oplus \sco_\piii) - \h^1(E(-3)) = 6$. One deduces that 
$E_3 \simeq 6\sco_\piii \oplus \sco_\piii(-1)$. It follows that $F = E(-2)$ can 
be described by an exact sequence$\, :$ 
\[
0 \lra F \lra 6\sco_\piii(1) \oplus \sco_\piii \overset{\phi}{\lra}  
3\sco_\piii(2) \oplus \sco_\piii(1) \lra 0\, . 
\]        
Since there is no epimorphism $\sco_\piii \ra \sco_\piii(1)$, the component 
$6\sco_\piii(1) \ra \sco_\piii(1)$ of $\phi$ must be nonzero, whence an exact 
sequence$\, :$ 
\[
0 \lra F \lra 5\sco_\piii(1) \oplus \sco_\piii \overset{\psi}{\lra}  
3\sco_\piii(2) \lra 0\, . 
\]
Let $\psi_1 \colon 5\sco_\piii(1) \ra 3\sco_\piii(2)$ and $\psi_2 \colon 
\sco_\piii \ra 3\sco_\piii(2)$ be the components of $\psi$. Since the composite 
map $\sco_\piii \overset{\psi_2}{\lra} 3\sco_\piii(2) \ra \Cok \psi_1$ is an 
epimorphism, it follows that $\Cok \psi_1 \simeq \sco_W$, for some closed 
subscheme $W$ of $\piii$. We assert that $W$ has dimension 0. Indeed, $W$ 
cannot contain a reduced and irreducible curve $\Gamma$ because there is no 
epimorphism $3\sco_\piii(2) \ra \sco_\Gamma$. 

Now, the image of the composite map $F \ra 5\sco_\piii(1) \oplus \sco_\piii 
\ra \sco_\piii$ must be $\sci_W$. Since $W$ has codimension 3 in $\piii$, its 
length must be $c_3(F^\vee) = 10$. This implies that $\sci_W(2)$ cannot be 
globally generated, hence $E = F(2)$ cannot be either and this 
\emph{contradicts} the assumption made at the beginning of Case 9. 
Consequently, this case \emph{cannot occur in our context}. 
\end{proof}

\appendix 
\section{The case $\tH^0(E(-2)) \neq 0$}\label{A:h0e(-2)neq0} 

We prove, in this appendix, the following$\, :$ 

\begin{prop}\label{P:h0e(-2)neq0} 
Let $E$ be a globally generated vector bundle on $\piii$, with Chern classes 
$c_1 = 5$, $c_2 \leq 12$ and $c_3$, and such that ${\fam0 H}^i(E^\vee) = 0$, 
$i = 0,\, 1$. If ${\fam0 H}^0(E(-3)) = 0$ and ${\fam0 H}^0(E(-2)) \neq 0$ then 
either $\sco_\piii(2)$ is a direct summand of $E$ or $E \simeq G(3)$, for some  
stable rank $2$ vector bundle $G$ with $c_1(G) = -1$, $c_2(G) = 2$. 
\end{prop}

This result is a particular case of \cite[Thm.~0.1]{acm2} but we provide 
here a different argument, keeping to a minimum the use of information about  
non-reduced space curves of small degree. We use, instead, some results of 
Chang \cite{ch} about stable rank 2 reflexive sheaves with small $c_2$. 
We prove (or recall), firstly, a number of auxiliary facts. 

Let $E$ be a vector bundle as in the statement of the proposition. If $r$ is 
the rank of $E$ then $r - 1$ general global sections $s_1, \ldots , 
s_{r-1}$ of $E$ define an exact sequence$\, :$ 
\begin{equation}\label{E:oeiy(5)} 
0 \lra (r-1)\sco_\piii \lra E \lra \sci_Y(5) \lra 0\, , 
\end{equation}
where, as we saw at the beginning of Section~\ref{S:prelim}, $Y$ is a 
nonsingular connected curve of degree $c_2$. Dualizing this exact sequence, 
one gets an exact sequence$\, :$ 
\[
0 \lra \sco_\piii(-5) \lra E^\vee \lra (r - 1)\sco_\piii 
\overset{\delta}{\lra} \omega_Y(-1) \lra 0\, , 
\]  
where $\delta$ is defined by $r - 1$ global sections $\sigma_1, \ldots , 
\sigma_{r-1}$ of $\omega_Y(-1)$. Since $\tH^0(E^\vee) \neq 0$, each of these 
global sections must be nonzero. Now, $s_1, \ldots , s_{r-2}$ define an 
exact sequence$\, :$ 
\[
0 \lra (r - 2)\sco_\piii \lra E \lra \sce^\prim \lra 0\, , 
\]  
where $\sce^\prim$ can be realized as an extension$\, :$ 
\[
0 \lra \sco_\piii \lra \sce^\prim \lra \sci_Y(5) \lra 0\, , 
\]
which, by dualization, produces an exact sequence$\, :$ 
\[
0 \lra \sco_\piii(-5) \lra \sce^{\prim \vee} \lra \sco_\piii 
\xra{\sigma_{r-1}} \omega_Y(-1) \lra 
\sce xt^1_{\sco_\piii}(\sce^\prim , \sco_\piii) \lra 0\, . 
\]
Since $\sigma_{r-1} \neq 0$, one deduces that the support of 
$\sce xt^1_{\sco_\piii}(\sce^\prim , \sco_\piii)$ is 0-dimensional (or empty). 
Moreover, $\sce xt^i_{\sco_\piii}(\sce^\prim , \sco_\piii) = 0$ for $i \geq 2$. 
It follows that $\sce^\prim$ is a rank 2 \emph{reflexive sheaf}. Since, by 
the hypothesis of Prop.~\ref{P:h0e(-2)neq0}, one has $\tH^0(\sce^\prim(-3)) = 
0$ and $\tH^0(\sce^\prim(-2)) \neq 0$ any nonzero global section of 
$\sce^\prim(-2)$ defines an exact sequence$\, :$ 
\[
0 \lra \sco_\piii(2) \lra \sce^\prim \lra \sci_Z(3) \lra 0\, , 
\]
where $Z$ is a locally Cohen-Macaulay closed subscheme of $\piii$, of pure 
dimension 1, locally complete intersection except at finitely many points, and 
such that $\sci_Z(3)$ is globally generated. One has $\text{deg}\, Z = 
c_2(\sce^\prim) - 6 = c_2 - 6$. Since $Z$ cannot be empty (because this would 
imply that $E \simeq \sco_\piii(3) \oplus \sco_\piii(2)$ which would contradict 
the fact that $\tH^0(E(-3)) = 0$) it follows that $c_2 \geq 7$. One also 
deduces an exact sequence$\, :$ 
\begin{equation}\label{E:erasciz(3)} 
0 \lra \sco_\piii(2) \oplus (r - 2)\sco_\piii \lra E \lra \sci_Z(3) \lra 0\, . 
\end{equation}
Applying $\sch om_{\sco_X}(- , \sco_X)$, one gets an exact sequence$\, :$ 
\[
0 \lra \sco_\piii(-3) \lra E^\vee \lra (r - 2)\sco_\piii \oplus \sco_\piii(-2) 
\lra \omega_Z(1) \lra 0\, . 
\]
Since $\tH^i(E^\vee) = 0$, $i = 0,\, 1$, it follows that 
$\tH^0((r - 2)\sco_\piii) \izo \tH^0(\omega_Z(1))$. In particular, $r = 2 + 
\h^0(\omega_Z(1))$.  

\begin{remark}\label{R:multilines} 
We recall here, from B\u{a}nic\u{a} and Forster \cite{bf}, some general 
facts about double and triple structures on a line in $\piii$. 

(a) Let $X$ be a locally Cohen-Macaulay curve of degree 2 in $\piii$, 
such that $X_{\text{red}}$ is a line $L$. Then $\sci_L^2 \subset \sci_X \subset 
\sci_L$ and $\sci_L/\sci_X \simeq \sco_L(l)$, for some integer $l$. 
Since $\sci_L/\sci_X$ is a quotient of $\sci_L/\sci_L^2 \simeq 2\sco_L(-1)$ it 
follows that $l \geq -1$. Moreover, if $l = -1$ then $X$ is the divisor $2L$ 
on some plane $H \supset L$ because $\h^0(\sco_X(1)) = 3$.  

By \cite[Prop.~2.2]{bf}, the canonical morphism $(\sci_L/\sci_X)^{\otimes j} 
\ra \sci_L^j/\sci_L^{j-1}\sci_X$ is an isomorphism, $\forall \, j \geq 1$. In 
particular, $\sci_L^2/\sci_L\sci_X \simeq \sco_L(2l)$. 

Finally, applying $\sch om_{\sco_\piii}(- , \omega_\piii)$ to the exact sequence 
$0 \ra \sco_L(l) \ra \sco_X \ra \sco_L \ra 0$, one gets an exact sequence$\, :$ 
\[
0 \lra \omega_L \lra \omega_X \lra \omega_L(-l) \lra 0\, . 
\]

(b) Let $Y$ be a locally Cohen-Macaulay curve of degree 3 in $\piii$ such 
that $Y_{\text{red}}$ is a line $L$. Then $Y$ contains, as a closed subscheme, 
a double structure $X$ on $L$. As we recalled in (a), $\sci_L/\sci_X \simeq 
\sco_L(l)$, for some $l \geq -1$. If, moreover, $Y$ is locally complete 
intersection except at finitely many points then $\sci_X/\sci_Y \simeq 
\sco_L(l^\prim)$, for some integer $l^\prim$. A local computation shows that 
the canonical morphism $\sci_L^2/\sci_L\sci_X \ra \sci_X/\sci_Y$ is an 
isomorphism at every point $x \in L$ at which $Y$ is locally complete 
intersection. Since, as we recalled in (a), $\sci_L^2/\sci_L\sci_X \simeq 
\sco_L(2l)$ it follows that $\sci_X/\sci_Y \simeq \sco_L(2l + m)$, for some 
$m \geq 0$.  

Finally, applying $\sch om_{\sco_\piii}(- , \omega_\piii)$ to the exact sequence 
$0 \ra \sco_L(2l + m) \ra \sco_Y \ra \sco_X \ra 0$, one gets an exact 
sequence$\, :$ 
\[
0 \lra \omega_X \lra \omega_Y \lra \omega_L(-2l-m) \lra 0\, . 
\] 
\end{remark}

\begin{lemma}\label{L:h0e(-1)>4} 
Under the hypothesis of Prop.~\ref{P:h0e(-2)neq0}, assume that $9 \leq c_2 
\leq 12$. If one has ${\fam0 h}^0(E(-1)) > 4$ then either $\sco_\piii(2) \oplus 
\sco_\piii(1)$ is a direct summand of $E$ or $c_2 = 9$ and $E \simeq 
2\sco_\piii(2) \oplus {\fam0 T}_\piii(-1)$.   
\end{lemma}

\begin{proof} 
The hypothesis implies that $\h^0(\sci_Y(4)) > 4$ ($Y$ being, of course, the 
curve from the exact sequence \eqref{E:oeiy(5)}). Since $\tH^0(\sci_Y(3)) 
\neq 0$ it follows that $Y$ is directly linked, by a complete intersection 
of type $(3,4)$, to a curve $Y^\secund$ which is locally Cohen-Macaulay and 
locally complete intersection except at finitely many points. One has 
$\text{deg}\, Y^\secund = 12 - c_2$. Using the fundamental exact sequence of 
liaison (recalled in \cite[Remark~2.6]{acm1})$\, :$ 
\[
0 \lra \sco_\piii(-7) \lra \sco_\piii(-3) \oplus \sco_\piii(-4) \lra 
\sci_Y \lra \omega_{Y^\secund}(-3) \lra 0\, , 
\]
one deduces that $\omega_{Y^\secund}(2)$ is globally generated. A general 
global section of $\omega_{Y^\secund}(2)$ generates this sheaf except at 
finitely many points and defines an extension$\, :$ 
\begin{equation}\label{E:scfsecundysecund}
0 \lra \sco_\piii(-1) \lra \scf^\secund \lra \sci_{Y^\secund}(1) \lra 0\, , 
\end{equation}
with $\scf^\secund$ a rank 2 reflexive sheaf with Chern classes $c_1^\secund = 
0$, $c_2^\secund = \text{deg}\, Y^\secund - 1 = 11 - c_2$. Using Ferrand's 
result about free resolutions under liaison (also recalled in 
\cite[Remark~2.6]{acm1}), one gets an exact sequence$\, :$ 
\[
0 \lra E^\vee \lra (r-1)\sco_\piii \oplus \sco_\piii(-1) \oplus \sco_\piii(-2) 
\lra \sci_{Y^\secund}(2) \lra 0\, , 
\] 
whence an exact sequence$\, :$ 
\begin{equation}\label{E:edualscfsecund} 
0 \lra E^\vee \lra r\sco_\piii \oplus \sco_\piii(-1) \oplus \sco_\piii(-2) 
\lra \scf^\secund(1) \lra 0\, . 
\end{equation}
Since $\tH^i(E^\vee) = 0$, $i = 0,\, 1$, it follows that $r = 
\h^0(\scf^\secund(1))$. One also sees that if the graded $S$-module 
$\tH^0_\ast(\scf^\secund)$ is generated in degrees $\leq 1$ then $\sco_\piii(2) 
\oplus \sco_\piii(1)$ is a direct summand of $E$.  

\vskip2mm 

\noindent 
{\bf Case 1.}\quad $c_2 = 12$. 

\vskip2mm 

\noindent 
In this case, $Y$ must be a complete intersection of type $(3,4)$ hence, 
by \cite[Lemma~2.1(i)]{acm1}, $E \simeq \sco_\piii(2) \oplus \sco_\piii(1) 
\oplus P(\sco_\piii(2))$. 

\vskip2mm 

\noindent 
{\bf Case 2.}\quad $c_2 = 11$. 

\vskip2mm 

\noindent 
In this case, $Y^\secund$ is a line hence $\scf^\secund \simeq 2\sco_\piii$. 
Using the exact sequence \eqref{E:edualscfsecund} one gets that $E \simeq 
\sco_\piii(2) \oplus \sco_\piii(1) \oplus 2\text{T}_\piii(-1)$. 

\vskip2mm 

\noindent 
{\bf Case 3.}\quad $c_2 = 10$. 

\vskip2mm 

\noindent 
In this case, $\scf^\secund$ has Chern classes $c_1^\secund = 0$ and $c_2^\secund 
= 1$. Since $\tH^0(\scf^\secund(-1)) = 0$, $\scf^\secund$ is semistable. It 
follows, from \cite[Thm.~8.2(a)]{ha}, that $c_3^\secund = 0$ or $c_3^\secund = 
2$. 

\vskip2mm 

\noindent 
{\bf Subcase 3.1.}\quad $c_3^\secund = 0$. 

\vskip2mm 

\noindent 
In this subcase, $\scf^\secund$ is a nullcorrelation bundle, i.e., one 
has an exact sequence$\, :$ 
\[
0 \lra \sco_\piii(-1) \lra \Omega_\piii(1) \lra \scf^\secund \lra 0\, . 
\]  
It follows easily that the graded $S$-module $\tH^0_\ast(\scf^\secund)$ is 
generated by $\tH^0(\scf^\secund(1))$. Moreover, the kernel of the evaluation 
morphism $\tH^0(\scf^\secund(1)) \otimes_k \sco_\piii \ra \scf^\secund(1)$ is 
isomorphic to $\text{T}_\piii(-2)$. One deduces that $E \simeq \sco_\piii(2) 
\oplus \sco_\piii(1) \oplus \Omega_\piii(2)$. 

\vskip2mm 

\noindent 
{\bf Subcase 3.2.}\quad $c_3^\secund = 2$. 

\vskip2mm 

\noindent 
In this subcase, by \cite[Lemma~2.1]{ch}, $\scf^\secund$ can be realized as an 
extension$\, :$ 
\[
0 \lra \sco_\piii \lra \scf^\secund \lra \sci_L \lra 0\, , 
\]   
where $L$ is a line. One deduces an exact sequence$\, :$ 
\[
0 \lra \sco_\piii(-2) \lra \sco_\piii \oplus 2\sco_\piii(-1) \lra \scf^\secund 
\lra 0\, . 
\]
It follows that the graded $S$-module $\tH^0_\ast(\scf^\secund)$ is generated in 
degrees $\leq 1$ and that the kernel of the evaluation morphism 
$\tH^0(\scf^\secund(1)) \otimes_k \sco_\piii \ra \scf^\secund(1)$ is isomorphic to 
$\sco_\piii(-1) \oplus \Omega_\piii(1)$. One deduces that $E \simeq 
\sco_\piii(2) \oplus 2\sco_\piii(1) \oplus \text{T}_\piii(-1)$. 

\vskip2mm 

\noindent 
{\bf Case 4.}\quad $c_2 = 9$. 

\vskip2mm 

\noindent 
In this case, if $\h^0(E(-2)) \geq 2$, i.e., if $\h^0(\sci_Y(3)) \geq 2$ then 
$Y$ is a complete intersection of type $(3,3)$ hence, by 
\cite[Lemma~2.1(i)]{acm1}, $E \simeq 2\sco_\piii(2) \oplus \text{T}_\piii(-1)$. 

Assume, from now on, that $\h^0(E(-2)) = 1$, i.e., that $\h^0(\sci_Y(3)) = 1$. 
One deduces, from the above fundamental exact sequence of liaison, that 
$\tH^0(\omega_{Y^\secund}) = 0$. It follows that $\tH^0(\sci_{Y^\secund}(1)) = 0$ 
(because $\text{deg}\, Y^\secund = 3$). Using the exact sequence 
\eqref{E:scfsecundysecund}, one gets that $\tH^0(\scf^\secund) = 0$, hence 
$\scf^\secund$ is stable. It has Chern classes $c_1^\secund = 0$ and 
$c_2^\secund = 2$ hence, by \cite[Thm.~8.2(b)]{ha}, $c_3^\secund \leq 4$. 

\vskip2mm 

\noindent 
{\bf Subcase 4.1.}\quad $c_3^\secund = 0$. 

\vskip2mm 

\noindent 
In this subcase, $\scf^\secund$ is a 2-instanton, hence the zero scheme $W$ of 
a general global section of $\scf^\secund(1)$ is the union of three mutually 
disjoint lines $L_1,\, L_2,\, L_3$. One thus has an exact sequence$\, :$ 
\[
0 \lra \sco_\piii \lra \scf^\secund(1) \lra \sci_W(2) \lra 0\, . 
\]   
Let $Q \subset \piii$ be the unique (nonsingular) quadric surface containing 
$W$ and fix an isomorphism $Q \simeq \pj \times \pj$. It follows that the 
cokernel of the evaluation morphism $\tH^0(\scf^\secund(1)) \otimes_k \sco_\piii 
\ra \scf^\secund(1)$ is isomorphic to $\sci_{W , Q}(2) \simeq \sco_Q(-1 , 2)$. 
Since there is no epimorphism $\sco_\piii(-1) \oplus \sco_\piii(-2) \ra 
\sco_Q(-1 , 2)$, the exact sequence \eqref{E:edualscfsecund} shows that 
\emph{this subcase cannot occur}. 

\vskip2mm 

\noindent 
{\bf Subcase 4.2.}\quad $c_3^\secund = 2$. 

\vskip2mm 

\noindent 
In this subcase, by the proof of \cite[Lemma~2.7]{ch}, there exists a plane 
$H \subset \piii$, a 0-dimensional subscheme $\Gamma$ of $H$ and an 
exact sequence$\, :$ 
\[
0 \lra \scg \lra \scf^\secund \lra \sci_{\Gamma , H}(-1) \lra 0\, , 
\] 
where $\scg$ is a stable rank 2 reflexive sheaf with Chern classes 
$c_1(\scg) = -1$, $c_2(\scg) = 1$, $c_3(\scg) = -1 + 2\text{length}\, \Gamma$ 
(see \cite[Prop.~9.1]{ha}). Since $\scg$ is stable, \cite[Thm.~8.2(d)]{ha} 
implies that $c_3(\scg) = 1$ hence $\Gamma$ consists of one simple point. 
Moreover, by \cite[Lemma~9.4]{ha}, one has an exact sequence$\, :$ 
\[
0 \lra \sco_\piii \lra \scg(1) \lra \sci_L(2) \lra 0\, , 
\] 
for some line $L \subset \piii$. One deduces that the cokernel of the 
evaluation morphism $\tH^0(\scf^\secund(1)) \otimes_k \sco_\piii \ra 
\scf^\secund(1)$ is isomorphic to $\sci_{\Gamma , H}$. Since there is no 
epimorphism $\sco_\piii(-1) \oplus \sco_\piii(-2) \ra \sci_{\Gamma , H}$, 
the exact sequence \eqref{E:edualscfsecund} shows that \emph{this subcase 
cannot occur, either}. 

\vskip2mm 

\noindent 
{\bf Subcase 4.3.}\quad $c_3^\secund = 4$. 

\vskip2mm 

\noindent 
In this subcase, \cite[Lemma~2.9]{ch} implies that $\scf^\secund(1)$ admits a 
resolution of the form$\, :$ 
\[
0 \lra 2\sco_\piii(-1) \lra 4\sco_\piii \lra \scf^\secund(1) \lra 0\, . 
\]    
One deduces easily, from the exact sequence \eqref{E:edualscfsecund}, that 
$E \simeq \sco_\piii(2) \oplus 3\sco_\piii(1)$. 
\end{proof} 

\begin{lemma}\label{L:h1k} 
Let $K$ be a vector bundle on $\p^n$, $n \geq 2$. If ${\fam0 H}^1(K_H) = 0$ for 
every hyperplane $H \subset \p^n$ then ${\fam0 h}^1(K) \leq 
{\fam0 max}(0 , {\fam0 h}^1(K(-1)) - n)$. 
\end{lemma}

\begin{proof} 
This follows immediately by applying the Bilinear Map Lemma 
\cite[Lemma~5.1]{ha} to the map $\tH^1(K)^\vee \otimes \tH^0(\sco_{\p^n}(1)) 
\ra \tH^1(K(-1))^\vee$ deduced from the multiplication map $\tH^1(K(-1)) 
\otimes \tH^0(\sco_{\p^n}(1)) \ra \tH^1(K)$.  
\end{proof}

\begin{lemma}\label{L:mora2o(1)} 
Let $n \geq 1$ and $m$ be integers. Then, for any epimorphism $\e \colon 
m\sco_{\p^n} \ra 2\sco_{\p^n}(1)$, the map ${\fam0 H}^0(\e(1)) \colon 
{\fam0 H}^0(m\sco_{\p^n}(1)) \ra {\fam0 H}^0(2\sco_{\p^n}(2))$ is surjective.   
\end{lemma}

\begin{proof} 
We note, firstly, that, since $\e$ is an epimorphism, the map $\tH^0(\e) 
\colon \tH^0(m\sco_{\p^n}) \ra \tH^0(2\sco_{\p^n}(1))$ must have rank at least 
$n + 2$. The kernel $K$ of $\e$ is a vector bundle on $\p^n$ with $\h^1(K) 
\leq 2(n+1) - (n+2) = n$. We have to show that $\tH^1(K(1)) = 0$. We use 
induction on $n$. 

\vskip2mm 

\noindent 
$\bullet$\quad If $n = 1$ then either $K \simeq (m - 4)\sco_\pj \oplus 
2\sco_\pj(-1)$ or $K \simeq (m - 3)\sco_\pj \oplus \sco_\pj(-2)$. 

\vskip2mm 

\noindent 
$\bullet$\quad Assume that the conclusion is true on $\p^{n-1}$ and let us 
prove it on $\p^n$. By the induction hypothesis, $\tH^1(K_H(1)) = 0$, for 
every hyperplane $H \subset \p^n$. Since, as we saw at the beginning of the 
proof, $\h^1(K) \leq n$, Lemma~\ref{L:h1k} implies that $\h^1(K(1)) = 0$.   
\end{proof}

\begin{lemma}\label{L:kvbl} 
Let $K$ be the kernel of an epimorphism $\e \colon m\sco_{\p^n} \ra 
2\sco_{\p^n}(1)$, $n \geq 1$. If ${\fam0 H}^1(K) \neq 0$, i.e., if the map 
${\fam0 H}^0(\e) \colon {\fam0 H}^0(m\sco_{\p^n}) \ra 
{\fam0 H}^0(2\sco_{\p^n}(1))$ is not surjective, then there is a line $L 
\subset \p^n$ such that $K_L \simeq (m - 3)\sco_L \oplus \sco_L(-2)$. In 
particular, $K(1)$ is not globally generated. 
\end{lemma}

\begin{proof} 
We use induction on $n$. 

\vskip2mm 

\noindent 
$\bullet$\quad If $n = 1$ then, as we saw in the proof of 
Lemma~\ref{L:mora2o(1)}, one must have $K \simeq (m - 3)\sco_\pj \oplus 
\sco_\pj(-2)$. 

\vskip2mm 

\noindent 
$\bullet$\quad Assume that the conclusion of the lemma has been verified on 
$\p^{n-1}$ and let us prove it on $\p^n$. Since $\h^1(K(-1)) = 2$ and 
$\h^1(K) \neq 0$, Lemma~\ref{L:h1k} implies that there exists a hyperplane 
$H \subset \p^n$ such that $\tH^1(K_H) \neq 0$. Applying the induction 
hypothesis to $K_H$ one gets that there exists a line $L \subseteq H$ such 
that $K_L \simeq (m - 3)\sco_L \oplus \sco_L(-2)$. 
\end{proof}

The following two lemmata complement the results of 
Chang \cite[Thm.~3.12]{ch} and \cite[Thm.~3.13]{ch}. 

\begin{lemma}\label{L:(2;-1,3,1)} 
Let $\scf$ be a rank $2$ reflexive sheaf on $\piii$ with $c_1(\scf) = -1$, 
$c_2(\scf) = 3$, $c_3(\scf) = 1$. If ${\fam0 H}^0(\scf(1)) = 0$ then 
$\scf(1)$ is the cohomology of a monad of the form$\, :$ 
\[
0 \lra 2\sco_\piii(-1) \overset{\beta}{\lra} \sco_\piii(1) \oplus 4\sco_\piii 
\overset{\alpha}{\lra} \sco_\piii(2) \lra 0\, . 
\]
\end{lemma} 

\begin{proof} 
Since, in particular, $\tH^0(\scf) = 0$, $\scf$ is stable. There is only one 
possibility for its spectrum, namely $k_\scf = (0,-1,-1)$ (see Hartshorne 
\cite[Sect.~7]{ha}). In particular, $\tH^1(\scf(l)) = 0$ for $l \leq -2$, 
$\h^1(\scf(-1)) = 1$ and $\tH^2(\scf(l)) = 0$ for $l \geq -1$. Moreover, 
by Riemann-Roch, $\h^1(\scf) = 3$. Since $\tH^2(\scf(-1)) = 0$ and 
$\tH^3(\scf(-2)) \simeq \tH^0(\scf^\vee(-2))^\vee \simeq \tH^0(\scf(-1))^\vee = 
0$, the Castelnuovo-Mumford lemma (in the slightly more general form stated 
in \cite[Lemma~1.21]{acm1}) implies that the graded $S$-module 
$\tH^1_\ast(\scf)$ is generated in degrees $\leq 0$. We assert that, in fact, 
it is generated by $\tH^1(\scf(-1))$. 

\emph{Indeed}, we have to show that the multiplication map $\tH^1(\scf(-1)) 
\otimes \tH^0(\sco_\piii(1)) \ra \tH^1(\scf)$ is surjective. Assume it is not. 
Then, since $\h^1(\scf(-1)) = 1$ and $\h^1(\scf) = 3$, $\tH^1(\scf(-1))$ is 
annihilated by two linearly independent linear forms $h_0$ and $h_1$. Let 
$L \subset \piii$ be the line of equations $h_0 = h_1 = 0$. Since $h_0,\, h_1$ 
is an $\scf$-regular sequence, it follows that tensorizing by $\scf$ the 
exact sequence$\, :$ 
\[
0 \lra \sco_\piii(-1) \lra 2\sco_\piii \lra \sco_\piii(1) \lra \sco_L(1) 
\lra 0\, , 
\] 
one gets an exact sequence$\, :$ 
\[
0 \lra \scf(-1) \lra 2\scf \lra \scf(1) \lra \scf_L(1) \lra 0\, . 
\]
Since $h_0$ and $h_1$ annihilate $\tH^1(\scf(-1))$ (inside $\tH^1_\ast(\scf)$) 
one gets that $\tH^0(\scf(1)) \neq 0$ which \emph{contradicts} our hypothesis. 
It thus remains that $\tH^1_\ast(\scf)$ is generated by $\tH^1(\scf(-1))$. 

Consider, now, the universal extension$\, :$ 
\[
0 \lra \scf \lra \scg \lra \sco_\piii(1) \lra 0\, . 
\]
$\scg$ is a rank 3 reflexive sheaf with $\tH^1_\ast(\scg) = 0$. Moreover, 
since $\tH^2(\scg(-1)) \simeq \tH^2(\scf(-1)) = 0$ and $\tH^3(\scg(-2)) 
\simeq \tH^3(\scf(-2)) = 0$, $\scg$ is 1-regular. Using Riemann-Roch one 
gets $\h^0(\scg) = \h^0(\sco_\piii(1)) + \chi(\scf) = 1$ and 
$\h^0(\scg(1)) = \h^0(\sco_\piii(2)) + \chi(\scf(1)) = 8$ hence there exists 
an epimorphism $\sco_\piii(1) \oplus 4\sco_\piii \ra \scg(1)$. The kernel $K$ 
of this epimorphism is a rank 2 vector bundle $K$ with $\tH^1_\ast(K) = 0$ and 
$\tH^2_\ast(K) \simeq \tH^1_\ast(\scg(1)) = 0$ hence $K$ is a direct sum of 
line bundles. Since $c_1(K) = -2$ and $\tH^0(K) = 0$ it follows that $K 
\simeq 2\sco_\piii(-1)$.  
\end{proof}

\begin{cor}\label{C:(2;-1,3,1)} 
If $\scf$ is as in Lemma~\ref{L:(2;-1,3,1)}, then there exists a line 
$L \subset \piii$ such that $\sco_L(-3)$ is a direct summand of $\scf_L$. 
\end{cor}

\begin{proof} 
Let $K$ be the kernel of the epimorphism $\alpha$ of the monad from the 
conclusion of Lemma~\ref{L:(2;-1,3,1)}, and let $\alpha_0 \colon \sco_\piii(1) 
\ra \sco_\piii(2)$ and $\alpha_1 \colon 4\sco_\piii \ra \sco_\piii(2)$ be the 
components of $\alpha$. There exists a plane $H \subset \piii$ such that the 
cokernel of $\alpha_0$ is $\sco_H(2)$. Applying an argument similar to that 
used in the proof of Lemma~\ref{L:kvbl} to $\alpha_1 \vb H \colon 4\sco_H \ra 
\sco_H(2)$, one deduces that there exists a line $L \subset H$ such 
that $K_L \simeq \sco_L(1) \oplus 2\sco_L \oplus \sco_L(-2)$. Since one has 
an exact sequence $0 \ra 2\sco_L(-1) \ra K_L \ra \scf_L(1) \ra 0$, the 
conclusion of the lemma follows.  
\end{proof}

\begin{remark}\label{R:(2;-1,3,1)} 
The monads of the form$\, :$ 
\[
0 \lra 2\sco_\piii(-1) \overset{\beta}{\lra} \sco_\piii(1) \oplus 4\sco_\piii 
\overset{\alpha}{\lra} \sco_\piii(2) \lra 0\, , 
\]
with $\tH^0(\alpha) \colon \tH^0(\sco_\piii(1) \oplus 4\sco_\piii) \ra 
\tH^0(\sco_\piii(2))$ injective can be put together into a family with 
irreducible base. 

\vskip2mm 

\noindent 
\emph{Indeed}, it suffices to show that $\tH^0(\alpha(1))$ is surjective. 
Let $\alpha_0 \colon \sco_\piii(1) \ra \sco_\piii(2)$ and $\alpha_1 \colon 
4\sco_\piii \ra \sco_\piii(2)$ be the components of $\alpha$. 
Applying an argument similar to that used in the proof of 
Lemma~\ref{L:mora2o(1)} to the epimorphism $\alpha_1 \vb H \colon 4\sco_H \ra 
\sco_H(2)$, one deduces that $\tH^0((\alpha_1 \vb H)(1)) \colon 
\tH^0(4\sco_H(1)) \ra \tH^0(\sco_H(3))$ is surjective and this implies that 
$\tH^0(\alpha(1))$ is surjective. 
\end{remark} 

\begin{lemma}\label{L:(2;-1,3,3)} 
Let $\scf$ be a rank $2$ reflexive sheaf on $\piii$ with $c_1(\scf) = -1$, 
$c_2(\scf) = 3$, $c_3(\scf) = 3$. If ${\fam0 H}^0(\scf) = 0$ and 
${\fam0 H}^2(\scf(-1)) = 0$ then $\scf(1)$ is the cohomology of a monad of 
the form$\, :$ 
\[
0 \lra 3\sco_\piii(-1) \lra 7\sco_\piii \lra 2\sco_\piii(1) \lra 0\, . 
\]
\end{lemma} 

The proof is similar to that of Lemma~\ref{L:(2;-1,3,1)} and can be found in 
\cite[Sect.~5]{acm4} (the idea is that the reflexive sheaf $\scf$ from 
Lemma~\ref{L:(2;-1,3,3)} is stable with spectrum $k_\scf = (-1,-1,-1)$).  

\begin{cor}\label{C:(2;-1,3,3)} 
If $\scf$ is as in Lemma~\ref{L:(2;-1,3,3)}, then there exists a line 
$L \subset \piii$ such that $\sco_L(-3)$ is a direct summand of $\scf_L$. 
\end{cor} 

\begin{proof} 
Let K be the kernel of the epimorphism $7\sco_\piii \ra 2\sco_\piii(1)$ of the 
monad from the conclusion of Lemma~\ref{L:(2;-1,3,3)}. Lemma~\ref{L:kvbl} 
implies that there exists a line $L \subset \piii$ such that $K_L \simeq 
4\sco_L \oplus \sco_L(-2)$. Using the exact sequence $0 \ra 3\sco_L(-1) 
\ra K_L \ra \scf_L(1) \ra 0$ one gets the conclusion. 
\end{proof} 

\begin{proof}[Proof of Prop.~\ref{P:h0e(-2)neq0}]  
We split the proof into several cases according to the values of $c_2$. 
For $7 \leq c_2 \leq 9$ we use the exact sequence \eqref{E:erasciz(3)} and 
the fact, noticed right after that exact sequence, that the rank $r$ of $E$ 
can be expressed by the formula $r = 2 + \h^0(\omega_Z(1))$. 

\vskip2mm 

\noindent 
{\bf Case 1.}\quad $c_2 = 7$. 

\vskip2mm 

\noindent 
In this case, $Z$ must be a line. It follows that $\h^0(\omega_Z(1)) = 0$, 
hence $r = 2$. But \emph{this is not possible} because the Chern classes of 
$E$ must satisfy the relation $c_3 \equiv c_1c_2 \pmod{2}$. Consequently, 
\emph{this case cannot occur}. 

\vskip2mm 

\noindent 
{\bf Case 2.}\quad $c_2 = 8$. 

\vskip2mm 

\noindent 
In this case, $Z$ has degree 2. If $\tH^0(\sci_Z(1)) \neq 0$ then $Z$ is a 
complete intersection of type $(1,2)$. In particular, $\h^0(\omega_Z(1)) = 1$ 
hence $r = 3$. One deduces, from the exact sequence \eqref{E:erasciz(3)}, an 
exact sequence$\, :$ 
\[
0 \lra \sco_\piii \lra 2\sco_\piii(2) \oplus \sco_\piii(1) \oplus \sco_\piii 
\lra E \lra 0\, ,  
\]    
hence $E \simeq 2\sco_\piii(2) \oplus \sco_\piii(1)$. 

If $\tH^0(\sci_Z(1)) = 0$ then either $Z$ is the disjoint union of two lines 
or $Z$ is a double structure on a line $L$ such that $\sci_L/\sci_Z \simeq 
\sco_L(l)$, for some $l \geq 0$ (one cannot have $l = -1$ because, in that 
case, $\tH^0(\sci_Z(1)) \neq 0$). Using the last part of 
Remark~\ref{R:multilines}(a), one deduces that $\h^0(\omega_Z(1)) = 0$ hence 
$r = 2$. Then $G := E(-3)$ is a rank 2 vector bundle with $c_1(G) = -1$, 
$c_2(G) = 2$. Since $\tH^0(E(-3)) = 0$ one has $\tH^0(G) = 0$ hence $G$ is 
stable. 

\vskip2mm 

\noindent 
{\bf Case 3.}\quad $c_2 = 9$. 

\noindent 
In this case, $Z$ has degree 3. Using Lemma~\ref{L:h0e(-1)>4}, one can assume 
that $\h^0(E(-1)) = 4$, hence that $\tH^0(\sci_Z(2)) = 0$. This implies, in 
particular, that $Z$ cannot be reduced. It follows that either $Z = X \cup 
L^\prime$, where $X$ is a double structure on a line $L$ and $L^\prime$ is 
another line, or $Z$ is a triple structure on a line $L$. 

In the former case, $L$ and $L^\prime$ must be disjoint (otherwise 
$\tH^0(\sci_Z(2)) \neq 0$ because $\sci_L^2 \subset \sci_X$) and, moreover, 
$\sci_L/\sci_X \simeq \sco_L(l)$ with $l \geq 0$ (if one would have $l = -1$ 
then $X$ would be contained in a plane and this would imply that 
$\tH^0(\sci_Z(2)) \neq 0$). 

In the latter case, as we saw in Remark~\ref{R:multilines}(b), $Z$ contains 
a double structure $X^\prim$ on $L$ such that $\sci_L/\sci_{X^\prim} \simeq 
\sco_L(l)$, with $l \geq -1$, and $\sci_{X^\prim}/\sci_Z \simeq \sco_L(2l + m)$, 
with $m \geq 0$. Actually, one cannot have $l = -1$ because, then, $X^\prim$ 
would be contained in a plane and this would imply that $\tH^0(\sci_Z(2)) 
\neq 0$ (since $\sci_L\sci_{X^\prim} \subset \sci_Z$). 

One deduces that, in both cases, $\h^0(\omega_Z(1)) = 0$ hence $r = 2$. But 
\emph{this is not possible} because $c_3 \equiv c_1c_2 \pmod{2}$. 
Consequently, \emph{the case} $c_2 = 9$ \emph{cannot occur} under the 
assumption that $\h^0(E(-1)) = 4$.    

\vskip2mm 

For $10 \leq c_2 \leq 12$ we use another approach. Taking into account  
Lemma~\ref{L:h0e(-1)>4}, \emph{we assume that} $\h^0(E(-1)) = 4$. The curve $Y$ 
from the exact sequence \eqref{E:oeiy(5)} has the property that 
$\tH^0(\sci_Y(2)) = 0$, $\tH^0(\sci_Y(3)) \neq 0$, and $\sci_Y(5)$ is globally 
generated. It follows that $Y$ is directly linked, by a complete intersection 
of type $(3,5)$ to a curve $Y^\prim$, of degree $15 - \text{deg}\, Y = 15 - 
c_2$, locally Cohen-Macaulay, and locally complete intersection except at 
finitely many points. From the fundamental exact sequence of liaison$\, :$ 
\[
0 \lra \sco_\piii(-8) \lra \sco_\piii(-3) \oplus \sco_\piii(-5) \lra \sci_Y 
\lra \omega_{Y^\prim}(-4) \lra 0\, , 
\]  
one gets that $\tH^0(\omega_{Y^\prim}) = 0$ (because $\h^0(E(-1)) = 4$) and that 
$\omega_{Y^\prim}(1)$ is globally generated. A general global section of 
$\omega_{Y^\prim}(1)$ generates this sheaf except at finitely many points hence 
it defines an extension$\, :$ 
\[
0 \lra \sco_\piii \lra \scf^\prim(2) \lra \sci_{Y^\prim}(3) \lra 0\, , 
\]
where $\scf^\prim$ is a rank 2 reflexive sheaf with Chern classes $c_1^\prim = 
-1$, $c_2^\prim = \text{deg}\, Y^\prim - 2 = 13 - c_2$ and $c_3^\prim$.  

Moreover, using Ferrand's result about free resolution under liaison, one 
deduces, from the exact sequence \eqref{E:oeiy(5)}, an exact sequence$\, :$ 
\[
0 \lra E^\vee \lra r\sco_\piii \oplus \sco_\piii(-2) \lra \sci_{Y^\prim}(3) 
\lra 0\, , 
\] 
from which one gets an exact sequence$\, :$ 
\begin{equation}\label{E:edualscfprim}  
0 \lra E^\vee \lra (r + 1)\sco_\piii \oplus \sco_\piii(-2) \lra \scf^\prim(2) 
\lra 0\, . 
\end{equation}

\noindent 
{\bf Case 4.}\quad $c_2 = 10$. 

\vskip2mm 

\noindent 
In this case, $Y^\prim$ has degree 5. We assert that $\tH^0(\sci_{Y^\prim}(2)) 
= 0$. \emph{Indeed}, if $\tH^0(\sci_{Y^\prim}(2)) \neq 0$ then, since $Y^\prim$ 
is contained in an irreducible surface of degree 3, it is directly linked, 
be a complete intersection of type $(2,3)$, to a line $L^\prime$. One deduces, 
then, from the fundamental exact sequence of liaison$\, :$ 
\[
0 \lra \sco_\piii(-5) \lra \sco_\piii(-2) \oplus \sco_\piii(-3) \lra 
\sci_{L^\prime} \lra \omega_{Y^\prim}(-1) \lra 0\, , 
\]  
that $\tH^0(\omega_{Y^\prim}) \neq 0$, which \emph{contradicts} our assumption 
that $\h^0(E(-1)) = 4$. 

It thus remains that $\tH^0(\sci_{Y^\prim}(2)) = 0$ hence $\tH^0(\scf^\prim(1)) 
= 0$. In particular, $\scf^\prim$ is stable. One also has $\tH^2(\sci_{Y^\prim}) 
\simeq \tH^1(\sco_{Y^\prim}) \simeq \tH^0(\omega_{Y^\prim})^\vee = 0$, hence 
$\tH^2(\scf^\prim(-1)) = 0$. Since $c_1^\prim = -1$ and $c_2^\prim = 3$, one 
deduces that the possible spectra of $\scf^\prim$ are $k_{\scf^\prim} = 
(0,-1,-1)$ and $k_{\scf^\prim} = (-1,-1,-1)$ (see Hartshorne \cite[Sect.~7]{ha}  
for information about the spectrum of a stable rank 2 reflexive sheaf). In 
particular, either $c_3^\prim = 1$ or $c_3^\prim = 3$. 

Now, Cor.~\ref{C:(2;-1,3,1)} and Cor.~\ref{C:(2;-1,3,3)} imply that there 
exists a line $L \subset \piii$ and an epimorphism $\scf^\prim(2) \ra 
\sco_L(-1)$. One deduces that there is no epimorphism $(r + 1)\sco_\piii \oplus 
\sco_\piii(-2) \ra \scf^\prim(2)$ which \emph{contradicts} the existence of an 
exact sequence \eqref{E:edualscfprim}. Consequently, \emph{the case} $c_2 = 10$ 
\emph{cannot occur} under the assumption that $\h^0(E(-1)) = 4$. 

\vskip2mm 

\noindent 
{\bf Case 5.}\quad $c_2 = 11$. 

\vskip2mm 

\noindent 
In this case, $Y^\prim$ has degree 4 and $\scf^\prim$ has Chern classes 
$c_1^\prim = -1$, $c_2^\prim = 2$. One cannot have $\tH^0(\sci_{Y^\prim}(1)) \neq 
0$ because this would contradict the fact that $\tH^0(\omega_{Y^\prim}) = 0$ 
(which is a consequence of our assumption that $\h^0(E(-1)) = 4$).  
It follows that $\tH^0(\scf^\prim) = 0$ hence $\scf^\prim$ is stable. By 
\cite[Thm.~8.2(d)]{ha}, one must have $c_3^\prim \in \{0,\, 2,\, 4\}$. 

\vskip2mm 

\noindent 
{\bf Subcase 5.1.}\quad $c_3^\prim = 0$. 

\vskip2mm 

\noindent 
In this subcase, by Hartshorne and Sols \cite[Prop.~4.1(d)]{hs}, there 
exists a line $L \subset \piii$ and an epimorphism $\scf^\prim(2) \ra 
\sco_L(-1)$. It follows, as at the end of Case 4, that \emph{this subcase 
cannot occur}.  

\vskip2mm 

\noindent 
{\bf Subcase 5.2.}\quad $c_3^\prim = 2$. 

\vskip2mm 

\noindent 
In this subcase, by Chang~\cite[Lemma~2.4]{ch}, one has an exact 
sequence$\, :$ 
\[
0 \lra \sco_\piii(-1) \lra \scf^\prim \lra \sci_X \lra 0\, ,
\]         
where $X$ is either the union of two disjoint lines or a double structure on 
a line $L$ such that $\sci_L/\sci_X \simeq \sco_L$. One deduces that 
$\scf^\prim$ is 2-regular. Using, now, the exact sequence 
\eqref{E:edualscfprim}, one gets that $\sco_\piii(2)$ is a direct summand of 
$E$. One can be, actually, more specific$\, :$ the kernel $K$ of the evaluation 
morphism $\tH^0(\sci_X(2)) \otimes_k \sco_\piii \ra \sci_X(2)$ has $\tH^1_\ast(K) 
= 0$ and $\tH^2_\ast(K) \simeq {\underline k}(2)$ hence it is isomorphic to 
$\text{T}_\piii(-2)$. It follows that the kernel of the evaluation morphism 
of $\scf^\prim(2)$ is isomorphic to $\Omega_\piii(1) \oplus \text{T}_\piii(-2)$ 
and one deduces that $E \simeq \sco_\piii(2) \oplus \text{T}_\piii(-1) \oplus 
\Omega_\piii(2)$. 

\vskip2mm 

\noindent 
{\bf Subcase 5.3.}\quad $c_3^\prim = 4$. 

\vskip2mm 

\noindent 
In this subcase, by Hartshorne \cite[Lemma~9.6]{ha}, there is an exact 
sequence$\, :$ 
\[
0 \lra \sco_\piii(-1) \lra \scf^\prim \lra \sci_X \lra 0\, , 
\]
where $X$ is a complete intersection of type $(1,2)$. It follows that 
\emph{this subscase cannot occur} because, on one hand, $\h^2(\scf^\prim(-1)) 
= \h^2(\sci_X(-1)) = \h^1(\sco_X(-1)) = 1$ while, on the other hand, 
$\h^2(\scf^\prim(-1)) = \h^2(\sci_{Y^\prim}) = \h^1(\sco_{Y^\prim}) = 
\h^0(\omega_{Y^\prim}) = 0$. 

\vskip2mm 

\noindent 
{\bf Case 6.}\quad $c_2 = 12$. 

\vskip2mm 

\noindent 
In this case, $Y^\prim$ has degree 3 and $\scf^\prim$ has Chern classes 
$c_1^\prim = -1$, $c_2^\prim = 1$. One cannot have $\tH^0(\sci_{Y^\prim}(1)) \neq 
0$ because this would contradict the fact $\tH^0(\omega_{Y^\prim}) = 0$ (which 
is a consequence of our assumption that $\h^0(E(-1)) = 4$). It follows that 
$\tH^0(\scf^\prim) = 0$ hence $\scf^\prim$ is stable. By \cite[Thm.~8.2(d)]{ha}, 
one must have $c_3^\prim = 1$. Now, \cite[Lemma~9.4]{ha} implies that one has 
an exact sequence$\, :$ 
\[
0 \lra \sco_\piii(-1) \lra \scf^\prim \lra \sci_L \lra 0\, , 
\]    
for some line $L \subset \piii$. It follows that $\scf^\prim$ is 1-regular. 
Using the exact sequence \eqref{E:edualscfprim}, one deduces that 
$\sco_\piii(2)$ is a direct summand of $E$. Actually, the kernel of the 
evaluation morphism $\tH^0(\sci_L(2)) \otimes_k \sco_\piii \ra \sci_L(2)$ is 
isomorphic to $2\Omega_\piii(1)$ hence $E \simeq \sco_\piii(2) \oplus 
3\text{T}_\piii(-1)$. 
\end{proof}

\section{The spectrum of a stable rank 3 vector 
bundle}\label{A:spectrum} 

We prove, in this appendix, the properties of the spectrum of a stable rank 
3 vector bundle with $c_1 = -1$ on $\piii$ stated in Remark~\ref{R:spectrumf} 
following the method used by Hartshorne \cite[Sect.~7]{ha} and 
\cite[Prop.~5.1]{ha2} in the case of stable rank 2 reflexive sheaves. These 
results can be found in Okonek and Spindler \cite{oks2}, \cite{oks3} and in 
\cite{coa} but we give a quite different proof of the key technical point of 
the construction (Prop.~\ref{P:nsubseth1astf}(a) below) using arguments 
extracted from the proof of a result of Drezet and Le Potier 
\cite[Prop.~(2.3)]{dlp}. The rest of the arguments imitate Hartshorne's 
arguments.  

We prove firstly, as Hartshorne does in \cite[Sect.~5]{ha}, some auxiliary 
results about stable vector bundles on $\pii$. 

\begin{lemma}\label{L:c1(sct)} 
Let $\sct$ be a torsion coherent sheaf on $\pii$. Then $c_1(\sct) \geq 0$ and 
$c_1(\sct) = 0$ if and only if the support of $\sct$ consists of finitely 
many points. 
\end{lemma}

\begin{proof} 
One restricts $\sct$ to a general line $L \subset \pii$. 
\end{proof}

\begin{lemma}\label{L:gggc1=0} 
Let $F$ be a vector bundle of rank $r$ on $\pii$. If the evaluation morphism 
${\fam0 H}^0(F) \otimes_k \sco_\pii \ra F$ is generically an epimorphism then 
$c_1(F) \geq 0$ and if $c_1(F) = 0$ then $F$ is trivial, that is, $F \simeq 
r\sco_\pii$.  
\end{lemma}

\begin{proof} 
Choose a point $x \in \pii$ such the evaluation map $\tH^0(F) \ra F(x)$ is 
surjective. Choose, now, an $r$-dimensional vector subspace $W \subset 
\tH^0(F)$ such that the composite map $W \hookrightarrow \tH^0(F) \ra F(x)$ is 
bijective. It follows that the composite morphism $W \otimes_k \sco_\pii 
\hookrightarrow \tH^0(F) \otimes_k \sco_\pii \ra F$ is generically an 
isomorphism. Taking the determinant map of this morphism one gets that 
$\tH^0(\text{det}\, F) \neq 0$ hence $c_1(F) \geq 0$. Moreover, if 
$\text{det}\, F \simeq \sco_\pii$ then the determinant map must be an 
isomorphism hence the above composite morphism is an isomorphism.      
\end{proof}

\begin{lemma}\label{L:rogmss} 
Let $\scf$ be a non-zero coherent subsheaf of the trivial vector bundle 
$r\sco_\pii$ of rank $r$ on $\pii$. Then $c_1(\scf) \leq 0$ and $c_1(\scf) = 0$ 
if and only if $\scf^{\vee \vee}$ is a trivial subbundle of $r\sco_\pii$. 
\end{lemma}

\begin{proof} 
Lemma~\ref{L:gggc1=0} implies that $c_1(\scf^\vee) \geq 0$ and if 
$c_1(\scf^\vee) = 0$ then $\scf^\vee$ is a trivial quotient bundle of 
$(r\sco_\pii)^\vee$. Since the canonical morphism $\scf \ra \scf^{\vee \vee}$ is 
an isomorphism except at finitely many points one deduces that $c_1(\scf) = 
c_1(\scf^{\vee \vee}) \leq 0$. 
\end{proof} 

\begin{prop}\label{P:nsubseth1astf}  
Let $F$ be a vector bundle on $\pii$ and $N$ a graded submodule of the graded 
module ${\fam0 H}^1_\ast(F)$ over the homogeneous coordinate ring $R := 
{\fam0 H}^0_\ast(\sco_\pii) \simeq k[x_0,x_1,x_2]$ of $\pii$. Put $n_i := 
\dim_kN_i$, $i \in \z$. Assume that $F$ satisfies the following condition$\, :$ 
$c_1(\scf) < 0$, $\forall \, \scf \subseteq F$ non-zero coherent subsheaf. 
Then$\, :$ 

\emph{(a)} $n_{-1} > n_{-2}$ if $N_{-2} \neq 0$$\, ;$ 

\emph{(b)} $n_{-i} > n_{-i-1}$ if $N_{-i-1} \neq 0$, $\forall \, i \geq 2$$\, ;$ 

\emph{(c)} If $n_{-i} - n_{-i-1} = 1$ for some $i \geq 2$ then there exists a 
non-zero linear form $\ell \in R_1$ such that $\ell N_{-j} = (0)$ in 
${\fam0 H}^1(F(-j+1))$, $\forall \, j \geq i$. 
\end{prop}

\begin{proof} 
(a) As we said at the beginning of the appendix, the arguments we use are 
extracted from the proof of a result of Drezet and Le Potier 
\cite[Prop.~(2.3)]{dlp}. Consider the universal extension$\, :$ 
\[
0 \lra F \lra G \lra \tH^1(F) \otimes_k \sco_\pii \lra 0\, . 
\]   
One has $\tH^1(F(l)) \izo \tH^1(G(l))$ for $l < 0$ and $\tH^1(G) = 0$. 

\vskip2mm 

\noindent 
{\bf Claim 1.}\quad $G$ \emph{satisfies the condition from the hypothesis of 
the proposition, that is} $c_1(\scg) < 0$, $\forall \, \scg \subseteq G$ 
\emph{non-zero coherent subsheaf}. 

\vskip2mm 

\noindent 
\emph{Indeed}, put $\scg^\prim := \scg \cap F$ and let $\scg^\secund$ be 
the image of the composite map $\scg \hookrightarrow G \ra \tH^1(F) \otimes 
\sco_\pii$. One has an exact sequence $0 \ra \scg^\prim \ra \scg \ra \scg^\secund 
\ra 0$. Lemma~\ref{L:rogmss} implies that $c_1(\scg^\secund) \leq 0$ hence 
$c_1(\scg) < 0$ if $\scg^\prim \neq (0)$. Assume, now, that $\scg^\prim = (0)$,  
hence that $\scg \izo \scg^\secund$, and that $c_1(\scg^\secund) = 0$. One gets, 
from Lemma~\ref{L:rogmss}, that $\scg^{\secund \vee \vee}$ is a trivial subbundle 
of $\tH^1(F) \otimes \sco_\pii$ hence $\scg^{\vee \vee}$ is a trivial subbundle 
of $G$. But this \emph{contradicts} the fact that $\tH^0(G) = 0$ (the 
condition from the hypothesis of the lemma implies that $\tH^0(F) = 0$ and the 
extension defining $G$ is the universal one). It thus remains that 
$c_1(\scg^\secund) < 0$ hence $c_1(\scg) < 0$ and Claim 1 is proven. 

\vskip2mm 

\noindent 
Now, the Beilinson monad of $G$ has the following shape$\, :$ 
\[
\tH^1(G(-2)) \otimes \Omega^2(2) \overset{d^{-1}}{\lra} 
\begin{matrix} \tH^1(G(-1)) \otimes \Omega^1(1)\\ \oplus\\ 
\tH^2(G(-2)) \otimes \Omega^2(2) \end{matrix} 
\overset{d^0}{\lra} \tH^2(G(-1)) \otimes \Omega^1(1) 
\overset{d^1}{\lra} \tH^2(G) \otimes \sco 
\]  
By the properties of Beilinson monads, the image of $d^{-1}$ is 
contained in $\tH^1(G(-1)) \otimes \Omega^1(1)$ and $d^0$ vanishes on 
$\tH^1(G(-1)) \otimes \Omega^1(1)$. Moreover, the component $\tH^1(G(-2)) 
\otimes \Omega^2(2) \ra \tH^1(G(-1)) \otimes \Omega^1(1)$ of $d^{-1}$ can be 
identified with the composite morphism$\, :$ 
\[
\begin{CD} 
\tH^1(G(-2)) \otimes \Omega^2(2) @>>> \tH^1(G(-2)) \otimes R_1 \otimes 
\Omega^1(1) @>{\mu \otimes \text{id}}>> \tH^1(G(-1)) \otimes \Omega^1(1)\\ 
@VVV @VVV\\
\tH^1(G(-2)) \otimes \overset{2}{\textstyle \bigwedge}R_1 \otimes \sco 
@>>> \tH^1(G(-2)) \otimes R_1 \otimes R_1 \otimes \sco
\end{CD}
\]     
where $\mu \colon \tH^1(G(-2)) \otimes R_1 \ra \tH^1(G(-1))$ is the 
multiplication map. We make, at this point, a \emph{technical assumption}.  
Put, for any non-zero vector subspace $W$ of 
$\tH^1(G(-2))$, $\delta(W) := \dim_k\mu(W \otimes R_1) - \dim_kW$ and let 
$\delta_{\text{min}}$ be the minimum of $\delta(W)$. We, actually, assume that 
$\delta(N_{-2}) = \delta_{\text{min}}$, that $N_{-2}$ is maximal 
among the non-zero vector subspaces $W$ of $\tH^1(G(-2))$ for which 
$\delta(W) = \delta_{\text{min}}$, and that $N_{-1} = \mu(N_{-2} \otimes R_1)$.  

\vskip2mm 

\noindent 
{\bf Claim 2.}\quad \emph{Under the previous assumptions, the morphism} 
$\left(\tH^1(G(-2))/N_{-2}\right) \otimes \Omega^2(2) \ra 
\left(\tH^1(G(-1))/N_{-1}\right) \otimes \Omega^1(1)$ \emph{induced by} 
$d^{-1}$ \emph{is a locally split monomorphism}. 

\vskip2mm 

\noindent 
\emph{Indeed}, if it is not then there exists $\xi \in \tH^1(G(-2)) \setminus 
N_{-2}$ and two linearly independent linear forms $\ell_0,\, \ell_1 \in R_1$ 
such that $\ell_i\xi \in N_{-1}$, $i = 0,\, 1$. Put $W := N_{-2} + k\xi$. 
Then $\delta(W) \leq \delta(N_{-2})$ and $W$ is strictly larger than $N_{-2}$ 
which \emph{contradicts} the assumptions preceding the claim. Consequently, 
Claim 2 is proven. 

\vskip2mm 

\noindent 
Let $\scg$ be the cokernel of the morphism $N_{-2} \otimes \Omega^2(2) \ra 
N_{-1} \otimes \Omega^1(1)$ induced by $d^{-1}$. Claim 2 implies that $\scg$ 
injects into $G$ (which is the cohomology of the Beilinson monad) and, by 
Claim 1, $-n_{-1} + n_{-2} = c_1(\scg) < 0$. 

\vskip2mm 

(b) can be proven by applying (a) to $F(-i+1)$, which obviously satisfies 
the condition from the hypothesis of the proposition. Alternatively, one can 
use induction on $i \geq 1$, the case $i = 1$ being the content of (a). 
More precisely, assume that $i \geq 2$ and that (b) has been verified for 
$i - 1$ (and any graded submodule of $\tH^1_\ast(F)$).  
Applying the Bilinear Map Lemma \cite[Lemma~5.1]{ha} to the 
multiplication map $R_1 \otimes N_{-i-1} \ra N_{-i}$ one sees that  
$n_{-i} \geq n_{-i-1} + 2$ or there exists a non-zero linear form $\ell \in 
R_1$ and a non-zero element $\xi$ of $N_{-i-1}$ such that $\ell \xi = 0$. 
Let $L \subset \pii$ be the line of equation $\ell = 0$. Consider the graded 
submodule $\ell N$ of $\tH^1_\ast(F)$. One has an exact sequence$\, :$ 
\[
0 \lra N^\prim \lra N \lra (\ell N)(1) \lra 0\, , 
\]     
where $N^\prim := \{ \eta \in N \vb \ell \eta = 0\}$. Using the exact sequence 
$0 \ra F(-1) \overset{\ell}{\lra} F \ra F_L \ra 0$ and the fact that 
$\tH^0(F) = 0$ one sees that $N^\prim_{-j}$ can be identified with a vector 
subspace of $\tH^0(F_L(-j+1))$ for $j \geq 1$. One has $N^\prim_{-i-1} \neq 0$ 
because it contains $\xi$. Applying the Bilinear Map Lemma to the 
multiplication map $\tH^0(\sco_L(1)) \otimes N^\prim_{-i-1} \ra N^\prim_{-i}$ one 
gets that $n^\prim_{-i} \geq n^\prim_{-i-1} + 1$. Since, by the induction 
hypothesis, $\dim_k(\ell N)_{-i+1} \geq \dim_k(\ell N)_{-i}$ (with equality only 
if both spaces are 0) it follows that $n_{-i} \geq n_{-i-1} + 1$. 

\vskip2mm 

(c) Using the notation from the proof of (b), one sees that if $n_{-i} = 
n_{-i-1} + 1$ then $(\ell N)_{-i+1} = 0$. Applying (b) to $\ell N$, one deduces 
that $(\ell N)_{-j+1} = 0$, $\forall j \geq i$.   
\end{proof} 

\begin{remark}\label{R:conditiononf} 
(i) If a vector bundle $F$ on $\pii$ satisfies the condition from the 
hypothesis of Prop.~\ref{P:nsubseth1astf} then $c_1(F) < 0$, $\tH^0(F) = 0$ 
and $\tH^0(F^\vee(c_1(F))) = 0$. 

%\vskip2mm 

(ii) Any rank 3 vector bundle $F$ on $\pii$ with $c_1(F) < 0$, $\tH^0(F) = 0$ 
and $\tH^0(F^\vee(c_1(F))) = 0$ satisfies the condition from the hypothesis of 
Prop.~\ref{P:nsubseth1astf}. 

(iii) More generally, any vector bundle $F$ on $\pii$ with $c_1(F) < 0$ and 
such that $\tH^0(\bigwedge^i F) = 0$, for $0 < i < \text{rk}\, F$,  
satisfies the condition from the hypothesis of Prop.~\ref{P:nsubseth1astf}.  
\end{remark}

\begin{remark}\label{R:nsubseth1astfc1=0} 
Prop.~\ref{P:nsubseth1astf} can be extended to the case $c_1(F) = 0$ with the 
following modifications$\, :$ 

%\vskip2mm 

\noindent 
$\bullet$\quad Replace the condition on $F$ from the hypothesis of 
Prop.~\ref{P:nsubseth1astf} by the condition ``$F$ \emph{stable in the sense of 
Mumford and Takemoto}'' which means that $c_1(\scf) < 0$, $\forall \scf 
\subset F$ coherent subsheaf with $0 < \text{rk}\, \scf < \text{rk}\, F$$\, ;$ 

%\vskip2mm 

\noindent 
$\bullet$\quad Replace item (a) of the conclusion of 
Prop.~\ref{P:nsubseth1astf} by$\, :$ ``$n_{-1} \geq n_{-2}$ \emph{with equality 
if and only if} $N_{-1} = 0$ \emph{and} $N_{-2} = 0$ \emph{or} $N_{-1} = 
\tH^1(F(-1))$ \emph{and} $N_{-2} = \tH^1(F(-2))$''.  

%\vskip2mm 

\noindent 
$\bullet$\quad Replace Claim 1 in the proof of Prop.~\ref{P:nsubseth1astf}(a)  
by the following statement$\, :$ ``$G$ \emph{is stable in the sense of 
Gieseker and Maruyama}'' which, taking into account that $c_1(G) = 0$ and 
$\chi(G) = 0$, means that, $\forall \, \scg \subseteq G$ coherent subsheaf, 
one has $c_1(\scg) \leq 0$ and if $c_1(\scg) = 0$ then $\chi(\scg) < 0$, 
unless $\scg = (0)$ or $\scg = G$. 

%\vskip2mm 
 
The proof requires only minor changes.  
\end{remark} 

\begin{defin}\label{D:spectrum} 
Let $F$ be a stable rank 3 vector bundle on $\piii$ with $c_1(F) = -1$,  
$c_2(F) = c_2 \geq 2$, $c_3(F) = c_3$. The stability is equivalent, in this 
case, to the fact that $\tH^0(F) = 0$ and $\tH^0(F^\vee(-1)) = 0$. According to 
the restriction theorem of Schneider \cite{sch} (see, also, Ein et al. 
\cite[Thm.~3.4]{ehv}), the restriction $F_H$ of $F$ to a general plane $H 
\subset \piii$ is stable. Let $h = 0$ be an equation of such a plane. Put$\, :$ 
\[
N := \text{Im}\left(\tH^1_\ast(F) \ra \tH^1_\ast(F_H)\right)\, ,\  
Q := \text{Coker}\left(\tH^1_\ast(F) \ra \tH^1_\ast(F_H)\right)\, . 
\]   
$N$ and $Q$ are graded modules over the homogeneous coordinate ring $S := 
\tH^0_\ast(\sco_\piii) \simeq k[x_0, \ldots , x_3]$ of $\piii$ (in fact, even 
over $S/hS$). Using the exact sequence$\, :$ 
\[
0 \lra F(-1) \overset{h}{\lra} F \lra F_H \lra 0\, , 
\] 
one sees that$\, :$ 
\[
N \simeq \text{Coker}\left(\tH^1_\ast(F(-1)) \overset{h}{\lra} 
\tH^1_\ast(F)\right)\, , 
\  Q \simeq \text{Ker}\left(\tH^2_\ast(F(-1)) \overset{h}{\lra} 
\tH^2_\ast(F)\right)\, . 
\]
Put $n_i := \dim_kN_i$ and $q_i := \dim_kQ_i$. Since $\tH^0(F_H) = 0$ and since  
$\tH^2(F_H(-2)) \simeq \tH^0(F_H^\vee(-1))^\vee = 0$, one gets that 
$n_i = \h^1(F(i)) - \h^1(F(i-1))$ for $i \leq 0$ and $q_i = \h^2(F(i-1)) - 
\h^2(F(i))$ for $i \geq -2$. 

\vskip2mm 

We invoke, now, Prop.~\ref{P:nsubseth1astf}. Since $F_H$ (which is a vector 
bundle on $H \simeq \pii$) satisfies the condition from the hypothesis of 
Prop.~\ref{P:nsubseth1astf}, it follows that $n_{-i} > n_{-i-1}$ if $n_{-i-1} > 
0$, $\forall \, i \geq 1$. On the other hand, $Q_i^\vee \subseteq 
\tH^1(F_H(i))^\vee \simeq \tH^1(F_H^\vee(-i-3))$. Since $F_H^\vee(-1)$ satisfies 
the condition from the hypothesis of Prop.~\ref{P:nsubseth1astf}, it follows 
that $q_i > q_{i+1}$ if $q_{i+1} > 0$, $\forall \, i \geq -1$. Moreover, by 
Prop.~\ref{P:nsubseth1astf}(c), if $q_i - q_{i+1} = 1$ for some $i \geq 0$ then 
there exists a linear form $\lambda \in S_1 \setminus kh$ such that the 
multiplication by $\lambda \colon Q_j^\vee \ra Q_{j-1}^\vee$ is the zero map 
for $j \geq i$ or, equivalently, the multiplication by $\lambda \colon Q_{j-1} 
\ra Q_j$ is the zero map for $j \geq i$. 

\vskip2mm 

Consider, now, the following vector bundles on $\pj$$\, :$ 
\[
K_+ := {\textstyle \bigoplus}_{i \geq 1}(n_{-i} - n_{-i-1})\sco_\pj(i - 1)\, \  
K_- := {\textstyle \bigoplus}_{i \geq -1}(q_i - q_{i+1})\sco_\pj(-i-2)\, , 
\]     
and put $K := K_+ \oplus K_-$. One has$\, :$ 
\begin{gather*} 
\h^1(F(-l)) - \h^1(F(-l-1)) = n_{-l} = \h^0(K(-l+1)) - \h^0(K(-l))\  
\text{for} \  l \geq 1\, ,\\ 
\h^2(F(l-1)) - \h^2(F(l)) = q_l = \h^1(K(l)) - \h^1(K(l+1))\  \text{for}\   
l \geq -1\, , 
\end{gather*}
hence $\h^1(F(l)) = \h^0(K(l+1))$ for $l \leq -1$ and $\h^2(F(l)) = 
\h^1(K(l+1))$ for $l \geq -2$. One can write $K \simeq 
\bigoplus_{i=1}^m\sco_\pj(k_i)$ with $k_1 \geq k_2 \geq \cdots \geq k_m$. 
The \emph{spectrum} of $F$ is, by definition, $k_F := (k_1, \ldots , k_m)$. 
One has$\, :$ 
\begin{gather*}
m = \text{rk}\, K = \text{rk}\, K_+ + \text{rk}\, K_- = n_{-1} + q_{-1} 
= \h^1(F_H(-1)) = c_2(F_H) = c_2\, ,\\ 
{\textstyle \sum}_{i=1}^m k_i = \chi(K(-1)) = -\chi(F(-2)) = 
-\frac{1}{2}(c_3 + c_2)\, .  
\end{gather*}

\vskip2mm 

The items (iv) and (v) from Remark~\ref{R:spectrumf} are quite clear. As for 
item (vi), if 0 does not occur in the spectrum then no positive integer occurs 
hence $n_{-1} = 0$. It follows that $q_{-1} = \h^1(F_H(-1)) = c_2$. Since 
$q_0 \leq \h^1(F_H) = c_2 - 2$, one deduces that $q_{-1} - q_0 \geq 2$, hence 
$-1$ occurs twice in the spectrum. The difficult point of 
Remark~\ref{R:spectrumf} is item (vii) which is the content of 
Prop.~\ref{P:unstableplane} below. This proposition is analogous to 
Hartshorne's result \cite[Prop.~5.1]{ha2} which treats the case of rank 2 
reflexive sheaves.    
\end{defin}  

\begin{prop}\label{P:unstableplane} 
Let $F$ be a stable rank $3$ vector bundle on $\piii$ with $c_1(F) = -1$, 
$c_2(F) = c \geq 4$ and let $k_F = (k_1, \ldots , k_c)$ be its spectrum. 
Assume that, for some $i$ with $2 \leq i \leq c - 1$, one has $-1 \geq k_{i-1} 
> k_i > k_{i+1}$. Then $F$ has an unstable plane $H_0$ of order $-k_c$ (which 
means that ${\fam0 H}^0(F_{H_0}^\vee(k_c)) \neq 0$ and 
${\fam0 H}^0(F_{H_0}^\vee(k_c - 1)) = 0$) and $k_{i+1} > k_{i+2} > \cdots > k_c$.   
\end{prop}

\begin{proof} 
Let $j$ (resp., $p$) be the integer defined by the relation $-j-2 = k_i$ 
(resp., $-p-2 = k_c$), that is, $j := -k_i - 2$ (resp., $p := -k_c - 2$). 
Using the notation from Definition~\ref{D:spectrum}, one has $j \geq 0$, 
$p \geq j + 1$, and $q_j > q_{j+1} > \cdots > q_p > q_{p+1} = 0$. The hypothesis 
implies that $q_j - q_{j+1} = 1$ hence (see, again, 
Definition~\ref{D:spectrum}) there exists a linear form $\lambda \in S_1 
\setminus kh$ such that multiplication by $\lambda \colon Q_{l-1} \ra Q_l$ is 
the zero map, $\forall \, l \geq j$. 

\vskip2mm 

\noindent 
{\bf Claim 1.}\quad $F$ \emph{has an unstable plane} $H_0$ \emph{of order} 
$-k_c$. 

\vskip2mm 

\noindent 
\emph{Indeed}, $\tH^2(F(l)) = 0$ for $l \geq p := -k_c - 2$ (by the 
definition of the spectrum), $Q_p \izo \tH^2(F(p-1))$ and one has an exact 
sequence$\, :$ 
\[
0 \lra Q_{p-1} \lra \tH^2(F(p-2)) \overset{h}{\lra} \tH^2(F(p-1)) \lra 0 
\]  
(recall that $h \in S_1$ is a linear form vanishing on a general plane 
$H \subset \piii$ for which $F_H$ is stable). By what has been said above, 
multiplication by $\lambda \colon Q_{p-1} \ra Q_p$ is the zero map hence  
multiplication by $\lambda \colon \tH^2(F(p-2)) \ra \tH^2(F(p-1))$
induces a map ${\overline \lambda} \colon \tH^2(F(p-2))/Q_{p-1} \ra 
\tH^2(F(p-1))$. On the other hand, multiplication by $h \colon \tH^2(F(p-2)) 
\ra \tH^2(F(p-1))$ induces an isomorphism ${\overline h} \colon 
\tH^2(F(p-2))/Q_{p-1} \izo \tH^2(F(p-1))$. Then there exists $c \in k$ such 
that ${\overline \lambda} + c{\overline h}$ is not an isomorphism. In this 
case, multiplication by $\lambda + ch \colon \tH^2(F(p-2)) \ra \tH^2(F(p-1))$ 
is not surjective. Let $H_0 \subset \piii$ be the plane of equation 
$\lambda + ch = 0$. Using the exact sequence$\, :$ 
\[
\tH^2(F(p-2)) \xra{\lambda + c\, h} \tH^2(F(p-1)) \lra \tH^2(F_{H_0}(p-1))\, ,  
\] 
one gets that $\tH^2(F_{H_0}(p-1)) \neq 0$. On the other hand, since 
$\tH^2(F(p)) = 0$ and $\tH^3(F(p-1)) \simeq \tH^0(F^\vee(-p-3))^\vee = 0$, it 
follows that $\tH^2(F_{H_0}(p)) = 0$. By Serre duality on $H_0$, 
$\tH^0(F_{H_0}^\vee(-p-2)) \neq 0$ and $\tH^0(F_{H_0}^\vee(-p-3)) = 0$. Since 
$-p-2 = k_c$, Claim 1 is proven. 

\vskip2mm 

According to Claim 1, there exists a non-zero morphism $\phi \colon F_{H_0} 
\ra \sco_{H_0}(k_c)$. Since $\tH^0(F_{H_0}^\vee(k_c - 1)) = 0$, the image of 
$\phi$ has the form $\sci_{Z , H_0}(k_c)$ for some 0-dimensional subscheme 
$Z$ of $H_0$. One can assume that the general plane $H \subset \piii$, used in 
the above definition to construct the spectrum of $F$, contains no point of 
$Z$. Put $L_0 := H \cap H_0$ and let $\e$ denote the composite epimorphism 
$F \ra F_{H_0} \overset{\phi}{\lra} \sco_{H_0}(k_c) \ra \sco_{L_0}(k_c)$. 
Let $F^\prim$ be the kernel of $\e \vb H \colon F_H \ra \sco_{L_0}(k_c)$. 
$F^\prim$ is a rank 3 vector bundle on $H \simeq \pii$, with $c_1(F^\prim) = 
-2$. 

\vskip2mm 

\noindent 
{\bf Claim 2.}\quad $F^{\prim \vee}(-2)$ \emph{satisfies the condition from the 
hypothesis of Prop.}~\ref{P:nsubseth1astf}. 

\vskip2mm 

\noindent 
\emph{Indeed}, the kernel $G$ of $\e \vb L_0 \colon F_{L_0} \ra \sco_{L_0}(k_c)$ 
is a rank 2 vector bundle on $L_0 \simeq \pj$. Comparing the exact 
sequences$\, :$ 
\[
0 \lra F^\prim \lra F_H \xra{\e \vb H} \sco_{L_0}(k_c) \lra 0\ \text{and} \  
0 \lra G \lra F_{L_0} \xra{\e \vb L_0} \sco_{L_0}(k_c) \lra 0\, ,  
\] 
one deduces an exact sequence$\, :$ 
\[
0 \lra F_H(-1) \lra F^\prim \lra G \lra 0\, . 
\]
It follows that $F^{\prim \vee}$ embeds into $F_H^\vee(1)$. Since $F_H^\vee(-1)$ 
satisfies the condition from the hypothesis of Prop.~\ref{P:nsubseth1astf}, 
the same must be true for $F^{\prim \vee}(-2)$. Claim 2 is proven. 

\vskip2mm 

\noindent 
{\bf Claim 3.}\quad $k_i > k_{i+1} > \cdots > k_c$. 

\vskip2mm 

\noindent 
\emph{Indeed}, using the notation from Definition~\ref{D:spectrum}, one knows, 
by hypothesis, that $q_j - q_{j+1} = 1$ and one wants to show that $q_l - q_{l+1} 
= 1$ for $j \leq l \leq p$. Recall that $Q$ is the cokernel of the morphism 
$\tH^1_\ast(F) \ra \tH^1_\ast(F_H)$. Since the composite morphism$\, :$ 
\[
\tH^1_\ast(F) \lra \tH^1_\ast(F_H) \xra{\tH^1_\ast(\e \vb H)} 
\tH^1_\ast(\sco_{L_0}(k_c))  
\]
is the zero morphism (because $\e \colon F \ra \sco_{L_0}(k_c)$ factorizes 
through $\sco_{H_0}(k_c)$), one gets, from the exact sequence$\, :$ 
\[
\tH^1_\ast(F^\prim) \lra \tH^1_\ast(F_H) \xra{\tH^1_\ast(\e \vb H)} 
\tH^1_\ast(\sco_{L_0}(k_c)) \lra \tH^2_\ast(F^\prim)\, , 
\] 
an exact sequence$\, :$ 
\[
\tH^1_\ast(F^\prim) \lra Q \lra \tH^1_\ast(\sco_{L_0}(k_c)) \lra 
\tH^2_\ast(F^\prim)\, .  
\]  
Let $Q^\prim$ denote the image of $\tH^1_\ast(F^\prim) \ra Q$ and put 
$q^\prim_l := \dim_kQ^\prim_l$. Since $\tH^2(F^\prim(l)) = 0$ for $l \geq -1$ 
(because, as we saw in the proof of Claim 2, $\tH^0(F^{\prim \vee}(-2)) = 0$) 
it follows that$\, :$ 
\[
q_l = q^\prim_l + \h^1(\sco_{L_0}(l + k_c))\, ,\  \forall \, l \geq -1\, . 
\]
Note, also, that$\, :$ 
\[
\h^1(\sco_{L_0}(l + k_c)) - \h^1(\sco_{L_0}(l + 1 + k_c)) = 1 \  \text{for}\  
l \leq -k_c - 2 = p\, . 
\]
Now, since $q_j - q_{j+1} = 1$ it follows that $q^\prim_j = q^\prim_{j+1}$. 
But $Q_l^{\prim \vee} \subseteq \tH^1(F^\prim(l))^\vee \simeq 
\tH^1(F^{\prim \vee}(-l - 3))$, $\forall \, l$. Applying 
Prop.~\ref{P:nsubseth1astf} to $F^{\prim \vee}(-2)$ one gets that $q^\prim_l = 0$, 
$\forall \, l \geq j$, hence $q_l - q_{l+1} = 1$ for $j \leq l \leq p$ and 
Claim 3 is proven.  
\end{proof}

\section{Rank $3$ vector bundles with $c_1 = -1$, 
$c_2 = 2$}\label{A:(3;-1,2,-)} 

The stable rank $3$ reflexive sheaves on $\piii$ with Chern classes 
$c_1 = -1$, $c_2 = 2$ were studied by Okonek and Spindler \cite{oks}. We 
prove here a number of complementary results that are needed in the main 
part of the paper. 

\begin{lemma}\label{L:(3;-1,2,2)} 
Let $F$ be a rank $3$ vector bundle on $\piii$ with $c_1(F) = -1$, $c_2(F) = 
2$, $c_3(F) = 2$ and such that ${\fam0 H}^0(F(-1)) = 0$ and 
${\fam0 H}^0(F^\vee(-1)) = 0$. 

\emph{(i)} If ${\fam0 H}^0(F) = 0$, i.e., if $F$ is stable then it has a 
resolution of the form$\, :$ 
\[
0 \lra 2\sco_\piii(-2) \lra 5\sco_\piii(-1) \lra F \lra 0\, . 
\]
Moreover, in this case, ${\fam0 H}^0(F^\vee) = 0$. 

\emph{(ii)} If ${\fam0 H}^0(F) \neq 0$ then a non-zero global section of $F$ 
defines an exact sequence$\, :$ 
\[
0 \lra \sco_\piii \lra F \lra \scg \lra 0 
\]
where $\scg$ is a stable rank $2$ reflexive sheaf. 
\end{lemma}

\begin{proof}
(i) The first assertion is due to Okonek and Spindler \cite[Lemma~1.12]{oks}. 
We include a variant of their argument. The spectrum of $F$ must be $(-1,-1)$. 
It follows that $\tH^2(F(-1)) = 0$. Moreover, $\tH^3(F(-2)) \simeq 
\tH^0(F^\vee (-2))^\vee = 0$. Finally, by Riemann-Roch, $\tH^1(F) = 0$. It 
follows that $F$ is 1-regular. By Riemann-Roch, again, $\h^0(F(1)) = 5$. 
The kernel $K$ of the evaluation morphism $5\sco_\piii \ra F(1)$ is a 
1-regular rank 2 vector bundle. Since, by Riemann-Roch, $\h^0(K(1)) = 20 - 
\h^0(F(2)) = 2$, one deduces that $K(1) \simeq 2\sco_\piii$. 

The fact that $\tH^0(F^\vee) = 0$ follows using the exact sequence$\, :$ 
\[
0 \lra F^\vee \lra 5\sco_\piii(1) \lra 2\sco_\piii(2) \lra 0\, , 
\]
and the fact that, by Lemma~\ref{L:mora2o(1)}, the map $\tH^0(5\sco_\piii(1)) 
\ra \tH^0(2\sco_\piii(2))$ is surjective hence bijective.

\vskip2mm 

(ii) A non-zero global section of $F$ defines an exact sequence$\, :$ 
\begin{equation}\label{E:gprimfveeiz} 
0 \lra \scg^\prime \lra F^\vee \lra \sci_Z \lra 0 
\end{equation} 
where $Z$ is a closed subscheme of $\piii$ of dimension $\leq 1$ and 
$\scg^\prime$ is a rank 2 reflexive sheaf. Put $\scg := \scg^\prime (-1)$. 
It follows, from Remark~\ref{R:chern}(c), that $\scg$ has Chern classes$\, :$ 
\[
c_1(\scg) = -1\, ,\  c_2(\scg) = c_2(\scg^\prime) = 2 - 
\text{deg}\, Z_{\text{CM}}\, ,\  
c_3(\scg) = c_3(\scg^\prime) = 2 + 5 \text{deg}\, Z_{\text{CM}} - 
2\chi(\sco_{Z_{\text{CM}}})\, .  
\]
Since $\tH^0(\scg) = 0$, $\scg$ is stable hence $c_2(\scg) \geq 1$ and 
$c_3(\scg) \leq c_2(\scg)^2$ (see \cite[Thm.~8.2(d)]{ha}). One deduces that 
either $Z_{\text{CM}}$ is a line or it is the empty set. In the former case one 
gets that $c_2(\scg) = 1$ and $c_3(\scg) = 5$ which \emph{is not possible}. 
It remains that $Z_{\text{CM}} = \emptyset$ hence $c_2(\scg) = 2$, $c_3(\scg) 
= 2$ and $Z$ is a 0-dimensional subscheme of $\piii$ of length 2. 
Dualizing, now, the exact sequence \eqref{E:gprimfveeiz} one gets the exact 
sequence from the statement. 
\end{proof}

\begin{cor}\label{C:(3;-1,2,2)} 
Under the hypothesis of Lemma~\ref{L:(3;-1,2,2)}, one has$\, :$ 

\emph{(i)} If ${\fam0 H}^0(F) = 0$ then $F^\vee$ is $1$-regular. 

\emph{(ii)} If ${\fam0 H}^0(F) \neq 0$ then the cokernel of the evaluation 
morphism ${\fam0 H}^0(F^\vee (1))\otimes_k \sco_\piii \ra F^\vee (1)$ is 
isomorphic to $\sco_L(-1)$, for some line $L \subset \piii$. Moreover, 
$F^\vee(2)$ is globally generated.  
\end{cor}

\begin{proof} 
(i) $\tH^2(F^\vee (-1)) \simeq \tH^1(F(-3))^\vee = 0$ (from the spectrum) and 
$\tH^3(F^\vee (-2)) \simeq \tH^0(F(-2))^\vee = 0$. Finally, Riemann-Roch and the 
fact that $\tH^0(F^\vee) = 0$ imply that $\tH^1(F^\vee) = 0$. 

\vskip2mm 

(ii) Twisting by 1 the exact sequence \eqref{E:gprimfveeiz} from the proof of 
Lemma~\ref{L:(3;-1,2,2)}, one gets an exact sequence$\, :$ 
\[
0 \lra \scg(2) \lra F^\vee(1) \lra \sci_Z(1) \lra 0\, . 
\] 
By Chang \cite[Lemma~2.4]{ch}, $\scg(2)$ is globally generated and 
$\tH^1(\scg(2)) = 0$. Since $Z$ is a 0-dimensional subscheme of $\piii$ of 
length 2 it is contained in a line $L \subset \piii$. The cokernel of the 
evaluation morphism $\tH^0(\sci_Z(1)) \otimes_k \sco_\piii \ra \sci_Z(1)$ is 
isomorphic to $\sco_L(-1)$.  
\end{proof} 

\begin{lemma}\label{L:(3;-1,2,0)} 
Let $F$ be a stable rank $3$ vector bundle on $\piii$ with $c_1(F) = -1$, 
$c_2(F) = 2$, $c_3(F) = 0$. Then $F$ is the cohomology of a monad of the 
form$\, :$ 
\[
0 \lra \sco_\piii(-2) \lra 3\sco_\piii \oplus 2\sco_\piii(-1) \lra 
\sco_\piii(1) \lra 0\, . 
\]  
\end{lemma} 

\begin{proof}
$F$ has spectrum $(0,-1)$. It follows that $\tH^1(F(l)) = 0$ for $l \leq -2$, 
$\h^1(F(-1)) = 1$ and $\tH^2(F(l)) = 0$ for $l \geq -1$. Moreover, by 
Riemann-Roch, $\h^1(F) = 1$. 
The restriction theorem of 
Schneider~\cite{sch} implies that, for a general plane $H \subset \piii$, 
$F_H$ is stable. In particular, $\tH^0(F_H) = 0$. If $h = 0$ is the equation 
of such a plane $H$ then the multiplication by $h \colon \tH^1(F(-1)) \ra 
\tH^1(F)$ is injective, hence bijective. 
Consider, now, the universal extension$\, :$ 
\[
0 \lra F \lra Q \overset{\e}{\lra} \sco_\piii(1) \lra 0\, . 
\]
$Q$ is a rank 4 vector bundle with $\tH^1(Q) = 0$,  
$\tH^2(Q(-1)) \simeq \tH^2(F(-1)) = 0$ and $\tH^3(Q(-2)) \simeq \tH^3(F(-2)) 
\simeq \tH^0(F^\vee(-2))^\vee = 0$. It follows $Q$ is 1-regular. 
One has $\tH^0(Q(-1)) = 0$, $\h^0(Q) = 3$ and 
$\h^0(Q(1)) = \chi(\sco_\piii(2)) + \chi(F(1)) = 14$. The image of 
$\tH^0(\e) \colon \tH^0(Q) \ra \tH^0(\sco_\piii(1))$ is a 3-dimensional vector 
subspace $W$ of $\tH^0(\sco_\piii(1))$. Consider the commutative diagram$\, :$ 
\[
\begin{CD} 
0 @>>> 0 @>>> \tH^0(Q) \otimes_k \sco_\piii @>{\sim}>> W \otimes_k \sco_\piii 
@>>> 0\\ 
@. @VVV @VV{\text{ev}}V @VV{\phi}V\\ 
0 @>>> F @>>> Q @>{\e}>> \sco_\piii(1) @>>> 0 
\end{CD}
\] 
The kernel $\scg$ of $\phi$ is a stable rank 2 reflexive sheaf with 
$c_1(\scg) = -1$. The connecting morphism $\partial \colon \scg \ra F$ induced 
by the above diagram is non-zero because the bottom row of the diagram does 
not split. Since $F$ is a stable rank 3 vector bundle with $c_1(F) = -1$, 
$\partial$ cannot have, generically, rank 1. It thus must have, generically,  
rank 2 hence it is a monomorphism. One deduces that the evaluation morphism 
$\text{ev} \colon \tH^0(Q) \otimes_k \sco_\piii \ra Q$ is a monomorphism. 

Consequently, the graded $S$-module $\tH^0_\ast(Q)$ has three minimal 
generators in degree 0 and two minimal generators in degree 1. Let $K$ be the 
kernel of the epimorphism $3\sco_\piii \oplus 2\sco_\piii(-1) \ra Q$ defined by 
these minimal generators. $K$ is a vector bundle of rank 1 and since $c_1(Q) 
= 0$ it follows that $K \simeq \sco_\piii(-2)$. 
\end{proof}

\begin{remark}\label{R:(3;-1,2,0)} 
The arguments from the proof of Lemma~\ref{L:(3;-1,2,2)}(ii) show that if 
$F$ is a rank 3 vector bundle on $\piii$ with $c_1(F) = -1$, $c_2(F) = 2$, 
$c_3(F) = 0$, $\tH^0(F(-1)) = 0$, $\tH^0(F^\vee(-1)) = 0$ and $\tH^0(F) \neq 
0$ then $F$ can be realized as an extension$\, :$ 
\[
0 \lra \sco_\piii \lra F \lra G \lra 0\, , 
\] 
where $G$ is a stable rank 2 vector bundle with $c_1(G) = -1$, $c_2(G) = 2$. 
\end{remark}

\begin{lemma}\label{L:(3;-1,2,-2)} 
Let $F$ be a stable rank $3$ vector bundle on $\piii$ with $c_1(F) = -1$, 
$c_2(F) = 2$, $c_3(F) = -2$. Then $F$ can be realized as a non-trivial 
extension$\, :$ 
\[
0 \lra \sco_\piii(-1) \lra F \lra G \lra 0\, ,  
\]
with $G$ a $2$-instanton. 
\end{lemma}

\begin{proof}
$F$ has spectrum $(0,0)$. One gets, from Riemann-Roch, that $\h^0(F(1)) \geq 
3$. A non-zero global section of $F(1)$ defines an exact sequence$\, :$ 
\[
0 \lra \scg \lra F^\vee \lra \sci_Z(1) \lra 0
\]
where $Z$ is a closed subscheme of $\piii$, of dimension $\leq 1$ and $\scg$ 
is a rank 2 reflexive sheaf. Using Remark~\ref{R:chern}(c), one gets$\, :$ 
\[
c_1(\scg) = 0\, ,\  c_2(\scg) = 2 - \text{deg}\, Z_{\text{CM}}\, ,\  
c_3(\scg) = 2\, \text{deg}\, Z_{\text{CM}} - 2\chi(\sco_{Z_{\text{CM}}})\, . 
\] 
Since $\tH^0(F^\vee(-1)) = 0$ it follows that $\tH^0(\scg(-1)) = 0$ hence 
$\scg$ is semistable. In particular, $c_2 (\scg) \geq 0$. 

\vskip2mm 

\noindent 
$\bullet$\quad If $c_2(\scg) = 2$ then $Z_{\text{CM}} = \emptyset$ and  
$c_3(\scg) = 0$ hence $\scg$ is a rank 2 vector bundle. This implies that 
$Z = Z_{\text{CM}} = \emptyset$ (see Remark~\ref{R:chern}(c)). Dualizing the 
above exact sequence and taking into account that $\scg^\vee \simeq \scg$ 
(because $c_1(\scg) = 0$) one gets an exact sequence$\, :$ 
\[
0 \lra \sco_\piii(-1) \lra F \lra \scg \lra 0\, , 
\]   
from which one deduces that $\tH^0(\scg) = 0$ hence $\scg$ is stable. Since 
it has Chern classes $c_1(\scg) = 0$ and $c_2(\scg) = 2$ it is a 2-instanton. 

\vskip2mm 

\noindent 
$\bullet$\quad If $c_2(\scg) = 1$ then $Z_{\text{CM}}$ is a line $L$. 
It follows that $c_3(\scg) = 0$ hence $\scg$ is a rank 2 vector bundle. One 
deduces that $Z = Z_{\text{CM}} = L$. Moreover, since $\scg$ has Chern classes 
$c_1(\scg) = 0$, $c_2(\scg) = 1$ it must be, actually, stable hence it is a 
nullcorrelation bundle $N$. Using the commutative diagram$\, :$ 
\[
\begin{CD}
0 @>>> \sco_\piii(-1) @>{v}>> 2\sco_\piii @>{u}>> 
\sci_L(1) @>>> 0\\
@. @V{\sigma}VV @VVV @\vert\\
0 @>>> N @>>> F^\vee @>>> \sci_L(1) @>>> 0  
\end{CD}
\]   
one gets an exact sequence$\, :$ 
\[
0 \lra \sco_\piii(-1) \xra{\left(\begin{smallmatrix} \sigma\\ v 
\end{smallmatrix}\right)} N \oplus 2\sco_\piii \lra F^\vee \lra 0\, .
\]
$\sigma$ is defined by a global section $s$ of $N(1)$. 
One has an exact sequence$\, :$ 
\[
0 \lra \sco_\piii(-1) \overset{\sigma}{\lra} N 
\xra{s \wedge -} \sci_X(1) \lra 0 
\]
where $X$ (= the zero scheme of $s$) is either the union of two disjoint lines 
or a double line on a nonsingular quadric surface. Since 
$(\sigma \, ,\, v)^{\text{t}}$ is a locally split monomorphism, one must have 
$X \cap L = \emptyset$. It follows that $(s \wedge -\, ,\, u) : N \oplus 
2\sco_\piii \ra \sco_\piii(1)$ induces an epimorphism $F^\vee \ra \sco_\piii(1)$. 
The kernel $G$ of this epimorphism must be, as above, a 2-instanton. 

\vskip2mm 

\noindent 
$\bullet$\quad If $c_2(\scg) = 0$ then the semistability of $\scg$ implies 
that $\scg \simeq 2\sco_\piii$. Then one must have $Z = Z_{\text{CM}}$. Moreover, 
$Z$ has degree 2. Since $\chi(\sci_Z) = \chi(F^\vee(-1)) = -1$, it follows that 
either $Z$ is the union of two disjoint lines or it is a double line on a 
nonsingular quadric surface. In both cases one has $\omega_Z \simeq 
\sco_Z(-2)$. A nowhere vanishing global section of $\omega_Z(2) \simeq 
\sco_Z$ defines an extension$\, :$ 
\[
0 \lra \sco_\piii(-1) \lra N \lra \sci_Z(1) \lra 0\, , 
\]        
where $N$ is a nullcorrelation bundle. Since $\text{Ext}^1(N , \sco_\piii) 
\simeq \tH^1(N^\vee) \simeq \tH^1(N) = 0$, one gets a commutative diagram$\, :$ 
\[
\begin{CD} 
0 @>>> \sco_\piii(-1) @>>> N @>>> \sci_Z(1) @>>> 0\\ 
@. @VVV @VVV @\vert\\ 
0 @>>> 2\sco_\piii @>>> F^\vee @>>> \sci_Z(1) @>>> 0  
\end{CD}
\] 
from which one deduces an exact sequence$\, :$ 
\[
0 \lra \sco_\piii(-1) \lra N \oplus 2\sco_\piii \lra F^\vee \lra 0\, , 
\]
and one concludes as in the previous case. 
\end{proof}

\begin{lemma}\label{L:(3;-1,2,-4)} 
Let $F$ be a stable rank $3$ vector bundle on $\piii$ with $c_1(F) = -1$, 
$c_2(F) = 2$, $c_3(F) = -4$. Then one has an exact sequence$\, :$ 
\[
0 \lra F \lra \sco_\piii(1) \oplus 3\sco_\piii \lra \sco_\piii(2) \lra 0\, . 
\]
\end{lemma} 

\begin{proof}
$F$ has spectrum $(1,0)$. It follows that $\tH^1(F(l)) = 0$ for $l \leq -3$, 
$\h^1(F(-2)) = 1$, $\h^1(F(-1)) = 3$ and $\tH^2(F(-2)) = 0$.  
We assert that the multiplication map $\tH^1(F(-2)) \otimes 
\tH^0(\sco_\piii(1)) \ra \tH^1(F(-1))$ is surjective. 

\emph{Indeed}, if it is not then there exists two linearly independent 
linear forms $h_0$ and $h_1$ annihilating $\tH^1(F(-2))$ inside $\tH^1_\ast(F)$. 
Let $L \subset \piii$ be the line of equations $h_0 = h_1 = 0$. Tensorizing by 
$F$ the exact sequence $0 \ra \sco_\piii(-2) \ra 2\sco_\piii(-1) \ra \sci_L 
\ra 0$ one deduces that $\tH^0(\sci_L \otimes F) \neq 0$ which 
\emph{contradicts} the fact that $\tH^0(F) = 0$. 

Consider, now, the universal extension$\, :$ 
\[
0 \lra F \lra Q \lra \sco_\piii(2) \lra 0\, . 
\]
$Q$ is a rank 4 vector bundle with $\tH^1(Q(-1)) = 0$, $\tH^2(Q(-2)) \simeq 
\tH^2(F(-2)) = 0$ and $\tH^3(Q(-3)) \simeq \tH^3(F(-3)) \simeq 
\tH^0(F^\vee(-1))^\vee = 0$. It follows that $Q$ is 0-regular. One has 
$\h^0(Q(-1)) = 1$ and $\h^0(Q) = \chi(F) + \chi(\sco_\piii(2)) = 7$. One 
deduces that the graded $S$-module $\tH^0_\ast(Q)$ has one minimal generator in 
degree $-1$ and three minimal generators in degree 0. The epimorphism 
$\sco_\piii(1) \oplus 3\sco_\piii \ra Q$ defined by these generators must be 
an isomorphism because $Q$ has rank 4. 
\end{proof}

\section{Rank 2 vector bundles with $c_1 = -1$, 
$c_2 = 4$}\label{A:(2;-1,4,0)} 

We prove, in this appendix, a number of complementary facts about rank 2 
vector bundles $G$ on $\piii$ with $c_1(G) = -1$, $c_2(G) = 4$, and 
$\tH^0(G(1)) = 0$.  
The stable rank 2 vector bundles on $\piii$ with $c_1 = 
-1$ and $c_2 = 4$ were studied by B\u{a}nic\u{a} and Manolache \cite{bm}. 
We use here a different approach, based on the method of Hartshorne \cite{ha},  
which uses a construction called \emph{reduction step} to reduce the study of 
stable rank 2 vector bundles on $\piii$ admitting an \emph{unstable plane} to 
the study of stable rank 2 reflexive sheaves with smaller second Chern class. 
This method was successfully applied by Chang \cite{ch} in the study of 
stable rank 2 reflexive sheaves with $c_1 = -1,\, 0$ and $c_2 \leq 3$. We only 
sketch this alternative approach because a better method for analysing these 
bundles will appear in a forthcoming paper by Ellia, Gruson and Skiti devoted 
to the study of stable rank 2 vector bundles on $\piii$ with $c_1 = -1$, 
$c_2 = 2m$, $m \geq 2$, and spectrum $(0 , \ldots , 0 , -1 , \ldots , -1)$.    

We begin by recalling the following general result$\, :$ 

\begin{lemma}\label{L:omegaxverty} 
Let $X$ be a locally Cohen-Macaulay projective scheme of pure dimension $1$ 
and $Y$, $Z$ closed subsets of $X$, of pure dimension $1$, such that $X = 
Y \cup Z$ as sets and $\dim(Y \cap Z) \leq 0$. Endow $Y$ (resp., $Z$) with 
the structure of closed subscheme of $X$ defined by the ideal sheaf 
${\fam0 Ker}(\sco_X \ra i_\ast \sco_{X \setminus Z})$ (resp., 
${\fam0 Ker}(\sco_X \ra j_\ast \sco_{X \setminus Y})$), $i$ and $j$ being the 
inclusion morphisms. Then$\, :$ 

(a) $Y$ and $Z$ are locally Cohen-Macaulay and $X = Y \cup Z$ as schemes. 

(b) If $X$ is locally Gorenstein at each point of $Y \cap Z$ then one has an 
exact sequence$\, :$ 
\[
0 \lra \omega_Y \lra \omega_X \vb Y \lra \omega_{Y \cap Z} \lra 0\, . 
\] 
\end{lemma} 

\begin{proof} 
(a) $Y$ and $Z$ are locally Cohen-Macaulay because their structure sheaves 
$\sco_Y$ and $\sco_Z$ embed into $i_\ast \sco_{X \setminus Z}$ and 
$j_\ast \sco_{X \setminus Y}$, respectively. 

$Y \cup Z$ is a closed subscheme of $X$. Since $X \setminus (Y \cap Z) = 
(Y \cup Z) \setminus (Y \cap Z)$ as schemes and since $X$ is locally 
Cohen-Macaulay it follows that $X = Y \cup Z$ as schemes. 

\vskip2mm 

(b) Applying $\sch om_{\sco_X}(-,\omega_X)$ to the exact sequence$\, :$ 
\[
0 \lra \sco_X \lra \sco_Y \oplus \sco_Z \lra \sco_{Y \cap Z} \lra 0 
\] 
one obtains an exact sequence$\, :$ 
\[
0 \lra \omega_Y \oplus \omega_Z \lra \omega_X \lra \omega_{Y \cap Z} \lra 0\, . 
\]
Restricting it to $Y$, one gets an exact sequence$\, :$ 
\[
\omega_Y \oplus (\omega_Z \vb Y) \lra \omega_X \vb Y \lra \omega_{Y \cap Z} 
\lra 0\, . 
\]
Since $\omega_Y \ra \omega_X \vb Y$ is an isomorphism over $Y \setminus Z$ and 
since $\omega_Y$ satisfies the condition $S_1$ of Serre (see, for example, 
Hartshorne \cite[Lemma~1.2]{hagd}) it follows that $\omega_Y \ra \omega_X \vb 
Y$ is a monomorphism. 

Since $\text{Supp}(\omega_Z \vb Y) \subseteq Y \cap Z$ and since $\omega_X \vb 
Y$ satisfies $S_1$ (here we use the hypothesis that $X$ is locally Gorenstein 
at the points of $Y \cap Z$) it follows that $\omega_Z \vb Y \ra \omega_X \vb 
Y$ is the zero morphism. 
\end{proof} 

\begin{remark}\label{R:divisorialpart} 
Let $Z$ be a closed subscheme, of dimension $\leq 1$, of a nonsingular surface 
$\Sigma$. Then $\sci_Z = \sci_D\sci_\Gamma$, where $D$ is an effective Cartier 
divisor on $\Sigma$ and $\Gamma$ is a closed subscheme of $\Sigma$, of 
dimension $\leq 0$. $D$ is called the \emph{divisorial part} of $Z$. 
\end{remark} 

\begin{prop}\label{P:doubleskewcubic} 
Let $G$ be a rank $2$ vector bundle on $\piii$ with $c_1(G) = -1$, 
$c_2(G) = 4$, and with ${\fam0 H}^0(G(1)) = 0$. If $G$ has no unstable plane 
then $G(2)$ has a global section whose zero scheme is a double structure on 
a twisted cubic curve $C \subset \piii$. 
\end{prop} 

\begin{proof} 
By \cite[Lemma~2]{bm}, the spectrum of $G$ must be $(0,0,-1,-1)$. This 
implies that $\tH^2(G(l)) = 0$ for $l \geq -1$, that $\tH^1(G(l)) = 0$ for 
$l \leq -2$ and that $\h^1(G(-1)) = 2$. Riemann-Roch implies, now, that 
$\h^1(G) = 5$ and $\h^1(G(1)) = 5$. The key point of the proof is the 
description of the multiplication map$\, :$ 
\[
\mu : \tH^1(G(-1)) \otimes_k \tH^0(\sco_\piii(1)) \lra \tH^1(G)\, . 
\] 
Since $G$ has no unstable plane, the bilinear map associated to $\mu$ is 
nondegenerate in the sense of the Bilinear map lemma \cite[Lemma~5.1]{ha}. 
The Bilinear map lemma implies that $\mu$ is surjective. Actually, one can 
describe $\mu$ concretely in convenient bases. Put $U := \tH^1(G(-1))$, 
$S_1 := \tH^0(\sco_\piii(1))$ and $W := \tH^1(G)$. Let $u_0,\, u_1$ be a 
$k$-basis of $U$ and $t_0,\, t_1$ the dual basis of $U^\vee$. On the projective 
space $\p (U)$ of 1-dimensional vector subspaces of $U$, the composite 
morphism, which we denote by $\phi$$\, :$ 
\[
\sco_{\p (U)}(-1) \otimes_k S_1 \lra \sco_{\p (U)} \otimes_k U \otimes_k S_1 
\xra{\text{id}\, \otimes \mu} \sco_{\p (U)} \otimes_k W 
\]
is a locally split monomorphism. Its cokernel must be isomorphic to 
$\sco_{\p (U)}(4)$ hence we have an exact sequence$\, :$ 
\[
0 \lra \sco_{\p (U)}(-1) \otimes_k S_1 \overset{\phi}{\lra} 
\sco_{\p (U)} \otimes_k W \overset{\pi}{\lra} \sco_{\p (U)}(4) \lra 0\, . 
\]
There exists a $k$-basis $w_0 , \ldots , w_4$ of $W$ such that $\pi(w_i) = 
t_0^{4-i}t_1^i$, $i = 0 , \ldots , 4$. Then there exists a basis $h_0 , \ldots , 
h_3$ of $S_1$ such that the matrix of $\phi$ with respect to those bases 
is$\, :$ 
\[
\begin{pmatrix} 
-t_1 & 0 & 0 & 0\\ 
t_0 & -t_1 & 0 & 0\\
0 & t_0 & -t_1 & 0\\ 
0 & 0 & t_0 & -t_1\\ 
0 & 0 & 0 & t_0 
\end{pmatrix}\, , 
\]
hence $\mu(u_i \otimes h_j) = (-1)^iw_{1-i+j}$, $i = 0,\, 1$, $j = 0, \ldots , 
4$. One deduces that the elements$\, :$ 
\[
u_0 \otimes h_j + u_1 \otimes h_{j+1}\, ,\  j = 0\, ,\, 1\, ,\, 2\, ,  
\]
form a $k$-basis of $\Ker \mu$. 

We describe, now, the \emph{Horrocks monad} of $G$ (see Barth and Hulek 
\cite{bh} for general information about monads). Since $\tH^2(G(-1)) = 0$ 
and $\tH^3(G(-2)) = 0$, the graded $S$-module $\tH^1_\ast(G)$ is generated in 
degrees $\leq 0$. Since $\mu$ is surjective, $\tH^1_\ast(G)$ is, actually, 
generated by $\tH^1(G(-1))$. Recall that $G^\vee \simeq G(1)$. Dualizing the 
extension$\, :$ 
\[
0 \lra G^\vee \lra K \lra 2\sco_\piii(2) \lra 0  
\]  
defined by the $k$-basis $u_0,\, u_1$ of $\tH^1(G^\vee (-2)) \simeq 
\tH^1(G(-1)) = U$, one gets an exact sequence$\, :$ 
\[
0 \lra 2\sco_\piii(-2) \lra K^\vee \lra G \lra 0\, . 
\]
It follows that $\tH^1_\ast(K^\vee) \izo \tH^1_\ast(G)$. Considering the 
extension$\, :$ 
\[
0 \lra K^\vee \lra B \lra 2\sco_\piii(1) \lra 0 
\]
defined by the $k$-basis $u_0,\, u_1$ of $\tH^1(K^\vee (-1)) \simeq 
\tH^1(G(-1)) = U$ one gets that $G$ is the middle cohomolgy of a monad$\, :$ 
\[
0 \lra 2\sco_\piii(-2) \overset{d^{-1}}{\lra} B \overset{d^0}{\lra} 
2\sco_\piii(1) \lra 0  
\]
where $B$ is a direct sum of line bundles. Since $B$ has rank 6, $c_1(B) = -3$, 
$\tH^0(B(-1)) = 0$ and $\h^0(B) = 3$ it follows that $B \simeq 3\sco_\piii 
\oplus 3\sco_\piii(-1)$. 

Finally, by the above description of $\Ker \mu$, the component $d_1^0 : 
3\sco_\piii \ra 2\sco_\piii(1)$ of $d^0 : B \ra 2\sco_\piii(1)$ is defined by the 
matrix$\, :$ 
\[
\begin{pmatrix} 
h_0 & h_1 & h_2\\ 
h_1 & h_2 & h_3
\end{pmatrix}\, . 
\]
The $2\times 2$ minors of this matrix define a twisted cubic curve $C \subset 
\piii$. Dualizing the Eagon-Northcott resolution$\, :$ 
\[
0 \lra 2\sco_\piii(-1) \xra{(d_1^0)^\vee} 3\sco_\piii \overset{\sigma}{\lra} 
\sci_C(2) \lra 0\, ,  
\]
one gets an exact sequence$\, :$ 
\[
0 \lra \sco_\piii(-2) \overset{\sigma^\vee}{\lra} 3\sco_\piii 
\overset{d_1^0}{\lra} 2\sco_\piii(1) \lra \omega_C(2) \lra 0\, . 
\]
The morphism $(\sigma^\vee , 0)^{\text{t}} : \sco_\piii(-2) \ra 3\sco_\piii \oplus 
3\sco_\piii(-1)$ induces a morphism $s : \sco_\piii(-2) \ra G$. One has $s \neq 
0$ because, otherwise, $(\sigma^\vee , 0)^{\text{t}}$ would factorize through 
$d^{-1} : 2\sco_\piii(-2) \ra 3\sco_\piii \oplus 3\sco_\piii(-1)$ which is not the 
case because $\sigma^\vee$ is not a locally split monomorphism. The image 
$\sci_Z(2)$ of $s^\vee : G^\vee \ra \sco_\piii(2)$ is contained in the image of 
$(\sigma , 0) : 3\sco_\piii \oplus 3\sco_\piii(1) \ra \sco_\piii(2)$ which is 
$\sci_C(2)$ hence $C \subset Z$. Since $\tH^0(G(1)) = 0$, $Z$ is a locally 
complete intersection curve in $\piii$, of degree 6, and with $\omega_Z \simeq 
\sco_Z(-1)$. Since, using an isomorphism $C \simeq \pj$, one has $\omega_C 
\simeq \sco_\pj(-2)$ and $\sco_C(-1) \simeq \sco_\pj(-3)$, 
Lemma~\ref{L:omegaxverty} implies that $Z$ is a double structure on $C$. 
\end{proof} 

Using Prop.~\ref{P:doubleskewcubic} one gets a new proof, based on vector 
bundles techniques, of the following result of Hartshorne and Hirschowitz 
\cite[Example~1.6.3]{hh}$\, :$ 

\begin{lemma}\label{L:h0g(2)neq0}  
If $G$ is a rank $2$ vector bundle on $\piii$ with $c_1(G) = -1$ and 
$c_2(G) = 4$ then ${\fam0 H}^0(G(2)) \neq 0$. 
\end{lemma} 

\begin{proof} 
We can assume that $\tH^0(G(1)) = 0$ and, then, by 
Prop.~\ref{P:doubleskewcubic}, that $G$ has an unstable plane $H_0$. 
Since $G$ has spectrum $(0,0,-1,-1)$, one has 
$\tH^1(G(-2)) = 0$ hence $\tH^0(G_{H_0}(-1)) = 0$. Since, by our assumption, 
$\tH^0(G_{H_0}) \neq 0$ it follows that one has an exact sequence$\, :$ 
\[
0 \lra \sco_{H_0} \lra G_{H_0} \lra \sci_{\Gamma , H_0}(-1) \lra 0 
\] 
where $\Gamma$ is a 0-dimensional subscheme of $H_0$, of length 4. 
Applying, now, a reduction step (see Hartshorne \cite[Prop.~9.1]{ha}) one 
gets an exact sequence$\, :$  
\[
0 \lra \scg^\prim (-1) \lra G \lra \sci_{\Gamma , H_0}(-1) \lra 0\, ,   
\]
where $\scg^\prim$ is a rank 2 reflexive sheaf with Chern classes 
$c_1^\prim = 0$, $c_2^\prim = 3$, $c_3^\prim = 4$. Since $\tH^0(G(1)) = 0$ it 
follows that $\tH^0(\scg^\prim) = 0$. \cite[Thm.~8.2(b)]{ha} implies, now, that 
$\tH^2(\scg^\prim(l)) = 0$ for $l \geq 0$. Since, by Riemann-Roch, 
$\chi(\scg^\prim(1)) = 1$ one deduces that $\h^0(\scg^\prim(1)) \geq 1$ hence 
$\tH^0(G(2)) \neq 0$. 
\end{proof}    

\begin{remark}[Double reduction step]\label{R:scgprimh0} 
Let $G$ be a rank 2 vector bundle on $\piii$ with $c_1(G) = -1$, $c_2(G) = 4$, 
such that $\tH^0(G(1)) = 0$. Assume that $G$ has an unstable plane $H_0 \subset 
\piii$, of equation $h_0 = 0$. We propose, here, a method for studying these 
bundles by a \emph{double reduction step}. 

As in the proof of Lemma~\ref{L:h0g(2)neq0}, one has exact sequences$\, :$ 
\begin{gather*} 
0 \lra \sco_{H_0} \lra G_{H_0} \lra \sci_{\Gamma , H_0}(-1) \lra 0\, ,\\ 
0 \lra \scg^\prim (-1) \overset{\phi}{\lra} G 
\overset{\e}{\lra} \sci_{\Gamma , H_0}(-1) \lra 0\, ,
\end{gather*} 
where $\Gamma$ is a 0-dimensional subscheme of $H_0$, of length 4, and 
$\scg^\prim$ is a stable rank 2 reflexive sheaf with $c_1(\scg^\prim) = 0$, 
$c_2(\scg^\prim) = 3$, $c_3(\scg^\prim) = 4$. Applying $\sch om_{\sco_\piii}(- , 
\sco_\piii(-1))$ to the second exact sequence and $\sch om_{\sco_\piii}(- , 
\sco_\piii)$ to the exact sequence$\, :$ 
\[
0 \lra \sci_{\Gamma , H_0} \lra \sco_{H_0} \lra \sco_\Gamma \lra 0\, , 
\] 
and taking into account that $\sce xt_{\sco_\piii}^i(\sco_{H_0} , \sco_\piii) = 0$ 
for $i \geq 2$ one gets that$\, :$ 
\[
\sce xt_{\sco_\piii}^1(\scg^\prim , \sco_\piii) \simeq 
\sce xt_{\sco_\piii}^2(\sci_{\Gamma , H_0} , \sco_\piii) \simeq 
\sce xt_{\sco_\piii}^3(\sco_\Gamma , \sco_\piii) \simeq \omega_\Gamma(4) 
\simeq \sco_\Gamma(4) \, .   
\]

\noindent 
{\bf Claim 1.}\quad $H_0$ \emph{is an unstable plane of order} $1$ \emph{for} 
$\scg^\prim$. 

\vskip2mm 

\noindent 
\emph{Indeed}, applying the Snake Lemma to the diagram$\, :$ 
\begin{equation*} 
\SelectTips{cm}{12}\xymatrix{ & & G(-1) \ar @{-->}[ld]_-{\psi(-1)} \ar[d]^{h_0} 
& & \\
0\ar[r] & \scg^\prim(-1) \ar[r]^-\phi \ar[d] & G \ar[r]^-\e \ar[d] &  
\sci_{\Gamma , H_0}(-1) \ar[r] \ar @{=}[d] & 0\\ 
0\ar[r] & \sco_{H_0} \ar[r] & G_{H_0} \ar[r] &  
\sci_{\Gamma , H_0}(-1) \ar[r] & 0}
\end{equation*}
one gets an exact sequence$\, :$ 
\[
0 \lra G \overset{\psi}{\lra} \scg^\prim \lra \sco_{H_0}(1) \lra 0\, .  
\] 
Since the morphism $\phi$ in the above diagram is a monomorphism, it follows 
that the diagram$\, :$ 
\[
\SelectTips{cm}{12}\xymatrix{\scg^\prim (-1) \ar[d]_-{h_0} \ar[r]^-{\phi} & 
G \ar[ld]^-\psi\\ 
\scg^\prim & } 
\]  
is commutative. Restricting this diagram to $H_0$, one deduces that 
$\psi_{H_0} \circ \phi_{H_0} = 0$. Taking into account the exact sequence$\, :$
\[
\scg^\prim_{H_0}(-1) \xra{\phi_{H_0}} G_{H_0} \lra 
\sci_{\Gamma , H_0}(-1) \lra 0\, , 
\]
it follows that the exact sequence$\, :$ 
\[
G_{H_0} \xra{\psi_{H_0}} \scg^\prim_{H_0} \lra \sco_{H_0}(1) \lra 0 
\] 
induces an exact sequence$\, :$ 
\[
\sci_{\Gamma , H_0}(-1) \lra \scg^\prim_{H_0} \lra \sco_{H_0}(1) \lra 0\, . 
\]
Since $\sci_{\Gamma , H_0}(-1)$ is a torsion free $\sco_{H_0}$-module of rank 1 
and $\scg^\prim_{H_0}$ is a torsion free $\sco_{H_0}$-module, any nonzero 
morphism $\sci_{\Gamma , H_0}(-1) \ra \scg^\prim_{H_0}$ must be a monomorphism. 
Consequently, one gets an exact sequence$\, :$ 
\[
0 \lra \sci_{\Gamma , H_0}(-1) \lra \scg^\prim_{H_0} \lra \sco_{H_0}(1) \lra 0\, . 
\] 
Consider, now, the pushforward diagram$\, :$ 
\[
\begin{CD} 
0 @>>> \sci_{\Gamma , H_0}(-1) @>>> \scg^\prim_{H_0} @>>> \sco_{H_0}(1) @>>> 0\\ 
@. @VVV @VVV @\vert\\ 
0 @>>> \sco_{H_0}(-1) @>>> F @>>> \sco_{H_0}(1) @>>> 0 
\end{CD} 
\]
One must have $F \simeq \sco_{H_0}(1) \oplus \sco_{H_0}(-1)$, whence an exact 
sequence$\, :$ 
\begin{equation}\label{E:scgprimh0}  
0 \lra \scg^\prim_{H_0} \lra \sco_{H_0}(1) \oplus \sco_{H_0}(-1) 
\overset{\pi}{\lra} \sco_\Gamma(-1) \lra 0\, , 
\end{equation}
where the component $\pi_2 \colon \sco_{H_0}(-1) \ra \sco_\Gamma(-1)$ of $\pi$ 
is the canonical epimorphism. The cokernel (resp., image) of the component 
$\pi_1 \colon \sco_{H_0}(1) \ra \sco_\Gamma(-1)$ of $\pi$ is isomorphic to 
$\sco_{\Gamma^\prim}(-1)$ (resp., $\sco_{\Gamma^\secund}(1)$), for some subscheme 
$\Gamma^\prim$ (resp., $\Gamma^\secund$) of $\Gamma$. One derives an exact 
sequence$\, :$ 
\begin{equation}\label{E:sprimsigmaprim} 
0 \lra \sci_{\Gamma^\secund , H_0}(1) \overset{s^\prim}{\lra} \scg^\prim_{H_0} 
\overset{\sigma^\prim}{\lra} \sci_{\Gamma^\prim , H_0}(-1) \lra 0\, , 
\end{equation}
which shows that $H_0$ is an unstable plane of order 1 for $\scg^\prim$. 

\vskip2mm 

One can perform, now, a \emph{second reduction step} and one gets an exact 
sequence$\, :$ 
\begin{equation}\label{E:scgsecund} 
0 \lra \scg^\secund \lra \scg^\prim \overset{\e^\prim}{\lra}  
\sci_{\Gamma^\prim , H_0}(-1) \lra 0\, , 
\end{equation}
where $\scg^\secund$ is a rank 2 reflexive sheaf with Chern classes 
$c_1(\scg^\secund) = -1$, $c_2(\scg^\secund) = 2$, $c_3(\scg^\secund) = 
2\, \text{length}\, \Gamma^\prim$. Since $\tH^0(G(1)) = 0$ it follows that 
$\tH^0(\scg^\prim) = 0$ hence $\tH^0(\scg^\secund) = 0$. This means that 
$\scg^\secund$ is stable. \cite[Thm.~8.2(d)]{ha} implies that $c_3(\scg^\secund) 
\in \{0,\, 2,\, 4\}$ hence $\text{length}\, \Gamma^\prim \in \{0,\, 1,\, 2\}$. 
Moreover, as at the beginning of the proof of Claim 1, one has an exact 
sequence$\, :$ 
\[
0 \lra \scg^\prim(-1) \lra \scg^\secund \lra \sci_{\Gamma^\secund , H_0}(1) 
\lra 0\, . 
\] 

\noindent 
{\bf Claim 2.}\quad \emph{There is an exact sequence}$\, :$ 
\[
0 \lra \sco_{\Gamma^\prim}(-1) \lra \sce xt_{\sco_\piii}^1(\scg^\secund , 
\sco_\piii(-1)) \lra \omega_{\Gamma^\prim}(4) \lra 0\, .  
\]

\noindent 
\emph{Indeed}, applying $\sch om_{\sco_\piii}(- , \sco_\piii(-1))$ to the exact 
sequence \eqref{E:scgsecund} one gets an exact sequence$\, :$ 
\begin{gather*}
\sce xt_{\sco_\piii}^1(\sci_{\Gamma^\prim , H_0}(-1) , \sco_\piii(-1)) 
\lra \sce xt_{\sco_\piii}^1(\scg^\prim , \sco_\piii(-1)) \lra 
\sce xt_{\sco_\piii}^1(\scg^\secund , \sco_\piii(-1)) \lra\\  
\lra \sce xt_{\sco_\piii}^2(\sci_{\Gamma^\prim , H_0}(-1) , \sco_\piii(-1)) 
\lra 0\, , 
\end{gather*}  
and $\sce xt_{\sco_\piii}^2(\sci_{\Gamma^\prim , H_0}(-1) , \sco_\piii(-1)) \simeq 
\omega_{\Gamma^\prim}(4)$. It remains only to explicitate the cokernel of 
$\sce xt_{\sco_\piii}^1(\e^\prim , \sco_\piii(-1))$. 
Applying $\sch om_{\sco_\piii}(\sci_{\Gamma^\prim , H_0}(-1) , -)$ and 
$\sch om_{\sco_\piii}(\scg^\prim , -)$ to the exact sequence $0 \ra \sco_\piii(-1) 
\overset{h_0}{\lra} \sco_\piii \ra \sco_{H_0} \ra 0$ one gets a commutative 
diagram$\, :$ 
\[
\SelectTips{cm}{12}\xymatrix @C=1.5pc {0\ar[r] & 0\ar[r]\ar[d] &  
\sch om_{\sco_\piii}(\sci_{\Gamma^\prim ,H_0}(-1), 
\sco_{H_0})\ar[r]^-\partial_-\sim\ar[d]_{\sch om(\e^\prim ,\, \sco_{H_0})} 
& \sce xt_{\sco_\piii}^1(\sci_{\Gamma^\prim ,H_0}(-1), 
\sco_\piii(-1))\ar[d]^{\sce xt^1(\e^\prim ,\, \sco_\piii(-1))}\ar[r] & 0\\  
0\ar[r] & (\scg^{\prim \vee})_{H_0}\ar[r] & 
\sch om_{\sco_\piii}(\scg^\prim ,\sco_{H_0})\ar[r]^-\partial
%\ar[d]_{s^{\prim \vee}} 
 &  \sce xt_{\sco_\piii}^1(\scg^\prim , \sco_\piii(-1))\ar[r] & 0} 
% & & \sch om_{\sco_\p}(\sci_{\Gamma^\secund , H}(t), \sco_H) & & }    
\] 
Notice that $\sch om(\e^\prim ,\, \sco_{H_0})$ can be identified with$\, :$ 
\[
\sigma^{\prim \vee} := \sch om_{\sco_{H_0}}(\sigma^\prim , \sco_{H_0}) \colon 
\sch om_{\sco_{H_0}}(\sci_{\Gamma^\prim ,H_0}(-1), 
\sco_{H_0}) \ra \sch om_{\sco_{H_0}}(\scg^\prim_{H_0} , \sco_{H_0}) =: 
(\scg^\prim_{H_0})^\vee 
\]   
and that one has an exact sequence$\, :$ 
\[
0 \lra \sch om_{\sco_{H_0}}(\sci_{\Gamma^\prim ,H_0}(-1), \sco_{H_0}) 
\overset{\sigma^{\prim \vee}}{\lra} (\scg^\prim_{H_0})^\vee 
\overset{s^{\prim \vee}}{\lra} 
\sch om_{\sco_{H_0}}(\sci_{\Gamma^\secund ,H_0}(1), \sco_{H_0}) \lra 0\, , 
\]
which is exact to the right because the exact sequence 
\eqref{E:sprimsigmaprim} was deduced from the exact sequence 
\eqref{E:scgprimh0}. 

Now, since the restriction of $\scg^\prim$ (resp., $\scg^\prim_{H_0}$) to 
$U := \piii \setminus \Gamma$ (resp., $U_0 := H_0 \setminus \Gamma$) is 
locally free of rank 2 with trivial determinant, one has a canonical 
isomorphism (resp., monomorphism) $\mu \colon \scg^\prim \izo 
\scg^{\prim \vee}$ (resp., $\mu_0 \colon \scg^\prim_{H_0} \ra 
(\scg^\prim_{H_0})^\vee$) such that the composite morphism$\, :$ 
\[
\scg^\prim_{H_0} \xra{\mu \vb H_0} (\scg^{\prim \vee})_{H_0} \lra 
(\scg^\prim_{H_0})^\vee 
\]   
equals $\mu_0$. Moreover, $s^\prim \vb U_0$ is defined by a global section of 
$\scg^\prim_{H_0}(-1) \vb U_0$ while $\sigma^\prim \vb U_0$ is defined by the 
exterior multiplication by that global section. One deduces easily that 
the composite map$\, :$ 
\[
\scg^\prim_{H_0} \overset{\mu_0}{\lra} (\scg^\prim_{H_0})^\vee 
\overset{s^{\prim \vee}}{\lra} 
\sch om_{\sco_{H_0}}(\sci_{\Gamma^\secund ,H_0}(1), \sco_{H_0}) 
\] 
can be identified with $\scg^\prim_{H_0} \overset{\sigma^\prim}{\lra} 
\sci_{\Gamma^\prim , H_0}(-1) \hookrightarrow \sco_{H_0}(-1)$ and this implies 
that the cokernel of $\sce xt_{\sco_\piii}^1(\e^\prim , \sco_\piii(-1))$ is 
isomorphic to $\sco_{\Gamma^\prim}(-1)$. Claim 2 is proven. 

\vskip2mm 

One can take, now, advantage of the fact that the stable rank 2 reflexive 
sheaves on $\piii$ with $c_1 = - 1$ and $c_2 = 2$ have been explicitly 
described in the literature. More precisely$\, :$ 

\vskip2mm 

\noindent 
$\bullet$\quad If $\Gamma^\prim = \emptyset$ then $c_3(\scg^\secund) = 0$ hence 
$\scg^\secund$ is a rank 2 vector bundle. In this case, by the results of 
Hartshorne and Sols \cite{hs} or Manolache \cite{ma}, one has an exact 
sequence$\, :$ 
\[ 
0 \lra \sco_\piii(-1) \lra \scg^\secund \lra \sci_X \lra 0\, ,  
\]
where $X$ is a double structure on a line $L \subset \piii$, defined by an 
exact sequence$\, :$ 
\[
0 \lra \sci_X \lra \sci_L \lra \sco_L(1) \lra 0\, . 
\]   

\noindent 
$\bullet$\quad If $\text{length}\, \Gamma^\prim = 1$ then $c_3(\scg^\secund) = 
2$. In this case, by Chang \cite[Lemma~2.4]{ch}, one has an exact 
sequence$\, :$ 
\[
0 \lra \sco_\piii(-1) \lra \scg^\secund \lra \sci_X \lra 0\, ,  
\]
where $X$ is either the union of two disjoint lines or a double structure on a 
line $L \subset \piii$ defined by an exact sequence$\, :$ 
\[
0 \lra \sci_X \lra \sci_L \lra \sco_L \lra 0\, . 
\] 
Taking into account the exact sequence from Claim 2, one sees  
that, actually, $X$ \emph{must be a double structure on a line}. 

\vskip2mm 

\noindent 
$\bullet$\quad If $\text{length}\, \Gamma^\prim = 2$ then $c_3(\scg^\secund) = 
4$. In this case, by the proof of \cite[Lemma~9.6]{ha}, one has an exact 
sequence$\, :$ 
\[
0 \lra \scg^\secund \lra 2\sco_\piii \lra \sci_{W, \Pi_0}(2) \lra 0\, ,  
\]     
where $W$ is a 0-dimensional complete intersection subscheme of type (2,2)  
of a plane $\Pi_0 \subset \piii$. 

\vskip2mm 

Using Prop.~\ref{P:doubleskewcubic} and the method of the double reduction 
step described above, one can show that if $G$ is a rank 2 vector bundle on 
$\piii$ with $c_1(G) = -1$, $c_2(G) = 4$ and $\tH^0(G(1)) = 0$ then $G$ is 
the cohomology sheaf of a monad of the form$\, :$ 
\[
0 \lra 2\sco_\piii(-2) \overset{\beta}{\lra} 3\sco_\piii \oplus 3\sco_\piii(-1) 
\overset{\alpha}{\lra} 2\sco_\piii(1) \lra 0\, ,  
\]  
with the property that the degeneracy locus of the component $\alpha_1 
\colon 3\sco_\piii \ra 2\sco_\piii(1)$ of $\alpha$ has codimension 2. 

\vskip2mm 

One can also show that $G(3)$ is globally generated if and only if $G$ has 
no jumping line of order $\geq 4$, i.e., there is no line $L \subset \piii$ 
such that $G_L \simeq \sco_L(a - 1) \oplus \sco_L(-a)$ with $a \geq 4$. 
Moreover, in this case, $\h^1(G(2)) \in \{1,\, 2\}$ and $\tH^1(G(3)) = 0$. 

We omit the details.   
\end{remark}

\section{Auxiliary results about 
instantons}\label{A:instantons} 

The definition and some basic properties of instantons are recalled in 
\cite[Remark~4.7]{acm1}.   

\begin{remark}\label{R:gs} 
We recall here the results of Gruson and Skiti \cite{gs} about the 
stratification of the moduli space of 3-instantons according to the number of 
their jumping lines of maximal order 3. Let $F^\prim$ be a 3-instanton. 

\vskip2mm

(i) If $F^\prim$ has no jumping line of order 3 then $\tH^0(F^\prim (1)) = 0$, 
$\h^1(F^\prim (1)) = 1$ (by Riemann-Roch), $\tH^1(F^\prim (l)) = 0$ for 
$l \geq 2$, and the multiplication map $\tH^0(F^\prim (2)) \otimes_k S_1 \ra 
\tH^0(F^\prim (3))$ is surjective (see the proof of \cite[Prop.~1.1.1]{gs}). 
In particular, $F^\prim (2)$ is globally generated. 

\vskip2mm

(ii) If $\tH^0(F^\prim (1)) = 0$ and $F^\prim$ has a jumping line $L$ of order 3 
then $L$ is the only jumping line of order 3 of $F^\prim$ and the cokernel of 
the evaluation morphism $\tH^0(F^\prim (2)) \otimes_k \sco_\piii \ra F^\prim (2)$ 
is isomorphic to $\sco_L(-1)$ (see the proof of \cite[Prop.~1.1.2]{gs}). 
Actually, Gruson and Skiti show that the kernel $\sck$ of the composite 
morphism $F^\prim \ra F^\prim_L \ra \sco_L(-3)$ is 2-regular. Their argument 
runs as follows$\, :$ $\tH^3(\sck(-1)) \simeq \tH^3(F^\prim (-1)) = 0$. 
Since $\tH^2(F^\prim) = 0$ and since, by Lemma~\ref{L:fpriml}(a) below, the map 
$\tH^1(F^\prim) \ra \tH^1(\sco_L(-3))$ is surjective, it follows that 
$\tH^2(\sck) = 0$. Finally, by the same lemma, the map $\tH^1(F^\prim (1)) \ra 
\tH^1(\sco_L(-2))$ is surjective, hence an isomorphism because 
$\h^1(F^\prim (1)) = 1$. One deduces that $\tH^1(\sck(1)) = 0$. 

\vskip2mm

(iii) If $\h^0(F^\prim (1)) = 1$ then $F^\prim$ can be realized as an 
extension$\, :$ 
\[
0 \lra \sco_\piii(-1) \lra F^\prim \lra \sci_Z(1) \lra 0 
\]    
where $Z$ is a curve of degree 4 which is a union of multiple structures of a 
special type on mutually disjoint lines (see \cite{nt}). Assume, for 
simplicity, that $Z$ is the union of four mutually disjoint lines $L_1, 
\ldots ,L_4$. $L_1 \cup L_2 \cup L_3$ is contained in a unique nonsingular 
quadric surface $Q$. Since $\h^0(F^\prim (1)) = 1$, $L_4$ is not contained in 
$Q$ hence either $L_4$ intersects $Q$ in two distinct points $P$ and $P^\prime$ 
or $L_4$ is tangent to $Q$ at a point $P$. In the former case, consider the 
lines $L$ and $L^\prime$ passing through $P$ and $P^\prime$, respectively, and 
belonging to the other ruling of $Q$ (than $L_1$, $L_2$, $L_3$) and put 
$X = L \cup L^\prime$. 
In the latter case, if $L$ is the line passing through $P$ and 
belonging to the other ruling of $Q$ then the 0-dimensional scheme $L_4 \cap 
Q$ is contained in the divisor $2L$ on $Q$. Put, in this case, $X = 2L$. One 
has, in both cases, $\sci_{X\cap Z,X} \simeq \sco_X(-4)$ whence an exact 
sequence$\, :$ 
\[
0 \lra \sco_X(3) \lra F^\prim_X \lra \sco_X(-3) \lra 0
\]   
which must split (for cohomological reasons). Consequently, $F^\prim$ has 
two jumping lines of order 3 which might coincide. Anyway, one shows, exactly 
as in case (ii), that the cokernel of the evaluation morphism 
$\tH^0(F^\prim (2)) \otimes_k \sco_\piii \ra F^\prim (2)$ is isomorphic to 
$\sco_X(-1)$.  

\vskip2mm

(iv) Finally, if $\h^0(F^\prim (1)) = 2$ then $F^\prim$ has infinitely many 
jumping lines of order 3. They form one ruling of a nonsingular quadric 
surface $Q \subset \piii$. Since one has an epimorphism $F^\prim \ra 
\sco_Q(0,-4) \ra 0$, $F^\prim (3)$ is not globally generated.    
\end{remark} 

\begin{lemma}\label{L:fpriml} 
Let $F^\prim$ be an instanton, $L \subset \piii$ a line, and $X \subset \piii$ 
a curve of degree $2$ which is either the union of two disjoint lines or its 
degeneration, a divisor of the form $2L$ on a smooth quadric surface $Q 
\subset \piii$, $L$ being a line. Then$\, :$ 

\emph{(a)} ${\fam0 H}^1(F^\prim (l)) \ra {\fam0 H}^1(F^\prim_L(l))$ is
surjective for $l \geq -1$$\, ;$ 

\emph{(b)} ${\fam0 H}^1(F^\prim (l)) \ra {\fam0 H}^1(F^\prim_X(l))$ is
surjective for $l \geq 0$. 
\end{lemma}

\begin{proof} 
(a) If $\scf$ is a coherent sheaf on $\piii$ with $\dim \text{Supp}\, \scf 
\leq 1$ then the multiplication map $\tH^1(\scf) \otimes_k S_1 \ra 
\tH^1(\scf(1))$ is surjective because $\tH^2(\scf \otimes \Omega_\piii(1)) = 0$. 
It thus suffices to show that $\tH^1(F^\prim (-1)) \ra \tH^1(F^\prim_L(-1))$ is 
surjective. Consider a plane $H \supset L$. Since $\tH^0(F^\prim (-1)) = 0$ and 
$\tH^1(F^\prim (-2)) = 0$ it follows that $\tH^0(F^\prim_H(-1)) = 0$. By Serre 
duality, $\tH^2(F^\prim_H(-2)) \simeq \tH^0(F_H^{\prim \vee}(-1))^\vee \simeq  
\tH^0(F_H^\prim (-1))^\vee = 0$. Now, using the exact sequences$\, :$ 
\begin{gather*}
\tH^1(F^\prim (-1)) \lra \tH^1(F^\prim_H(-1)) \lra \tH^2(F^\prim (-2)) = 0\\
\tH^1(F_H^\prim (-1)) \lra \tH^1(F^\prim_L(-1)) \lra \tH^2(F_H^\prim (-2)) = 0
\end{gather*} 
one gets that $\tH^1(F^\prim (-1)) \ra \tH^1(F^\prim_L(-1))$ is surjective. 

\vskip2mm 

(b) As in (a), it suffices to prove that $\tH^1(F^\prim) \ra \tH^1(F^\prim_X)$ is 
surjective. $X$ is contained, in both cases, in a smooth quadric surface $Q$ 
and, choosing a convenient isomorphism $Q \simeq \pj \times \pj$, one has an 
exact sequence$\, :$ 
\[
0 \lra \sco_Q(-2,0) \lra \sco_Q \lra \sco_X \lra 0\, . 
\]
Since $\tH^0(F^\prim) = 0$ and $\tH^1(F^\prim (-2)) = 0$, it follows that 
$\tH^0(F^\prim_Q) = 0$. By Serre duality, $\tH^2(F^\prim_Q(-2,0)) \simeq 
\tH^0(F^{\prim \vee}_Q(0,-2))^\vee \simeq \tH^0(F^\prim_Q(0,-2))^\vee = 0$. Now, 
using the exact sequences$\, :$ 
\begin{gather*}
\tH^1(F^\prim) \lra \tH^1(F^\prim_Q) \lra \tH^2(F^\prim (-2)) = 0\\
\tH^1(F^\prim_Q) \lra \tH^1(F^\prim_X) \lra \tH^2(F^\prim_Q (-2,0)) = 0 
\end{gather*}
one gets the desired surjectivity. 
\end{proof} 

We shall need the following two well known results (see, for example, 
\cite[Lemma~4.1]{ctt} and the recent paper of Ellia and Gruson \cite{eg}).  

\begin{lemma}\label{L:twojumpinglines} 
Let $G$ be a semistable rank $2$ vector bundle on $\pii$ with $c_1(G) = 0$ 
and $c_2(G) = n \geq 2$ and let $L_1$, $L_2$ be two distinct lines. One has 
$G_{L_i} \simeq \sco_{L_i}(a_i) \oplus \sco_{L_i}(-a_i)$ for some non-negative 
integer $a_i$, $i = 1,\, 2$. Then $a_1 + a_2 \leq n+1$. 
\qed 
\end{lemma}

\begin{lemma}\label{L:h0g(1)leq3} 
Let $G$ be a semistable rank $2$ vector bundle on $\pii$ with $c_1(G) = 0$ and 
$c_2(G) = n \geq 3$. Then ${\fam0 h}^0(G(1)) \leq 3$ unless $G$ has a jumping 
line of maximal order $n$. 
\qed  
\end{lemma}

\begin{remark}\label{R:h0g(1)leq2} 
One can show that if $G$ is a stable rank 2 vector bundle on $\pii$ with 
$c_1(G) = 0$ and $c_2(G) = n \geq 4$ then $\h^0(G(1)) \leq 2$ unless $G$ has a 
jumping line of maximal order $n-1$. 
\end{remark}

\begin{lemma}\label{L:h1fprim(2)=0} 
Let $F^\prim$ be a $4$-instanton with ${\fam0 H}^0(F^\prim (1)) = 0$, having no 
jumping line of maximal order $4$. Then ${\fam0 h}^1(F^\prim(2)) \leq 1$ and 
if ${\fam0 h}^1(F^\prim(2)) = 1$ then the multiplication map 
${\fam0 H}^1(F^\prim(1)) \otimes S_1 \ra {\fam0 H}^1(F^\prim(2))$ is a perfect 
pairing.  
\end{lemma}

\begin{proof} 
One has, by Riemann-Roch, $\h^1(F^\prim (1)) = 4$. 
Let $H \subset \piii$ be a plane of equation $h = 0$. Since $\tH^0(F^\prim) = 0$ 
and since $\tH^1(F^\prim (-2)) = 0$ it follows that $\tH^0(F^\prim_H(-1)) = 0$ 
hence $F_H^\prim$ is semistable. It follows, from Lemma~\ref{L:h0g(1)leq3}, 
that $\h^0(F^\prim_H(1)) \leq 3$ hence, by Riemann-Roch, $\h^1(F^\prim_H(1)) 
\leq 1$. This implies that $\tH^1(F^\prim_H(2)) = 0$. Using the exact 
sequence$\, :$ 
\[
\tH^1(F^\prim (1)) \overset{h}{\lra} \tH^1(F^\prim (2)) \lra 
\tH^1(F^\prim_H(2)) = 0 
\]
one gets 
that the mutiplication by $h : \tH^1(F^\prim (1)) \ra \tH^1(F^\prim (2))$ is 
surjective. Applying, now, the Bilinear Map Lemma \cite[Lemma~5.1]{ha} to the 
bilinear map $\tH^1(F^\prim (2))^\vee \otimes S_1 \ra \tH^1(F^\prim (1))^\vee$ 
deduced from the multiplication map $\tH^1(F^\prim (1)) \otimes S_1 \ra 
\tH^1(F^\prim (2))$ one deduces that $\h^1(F^\prim (2)) \leq 1$. If 
$\h^1(F^\prim (2)) = 1$ then the multiplication map $\tH^1(F^\prim (1)) \otimes 
S_1 \ra \tH^1(F^\prim (2))$ is a perfect pairing. 
\end{proof} 

\begin{lemma}\label{L:c1=0c2=4h1fprim(2)=0} 
Let $F^\prim$ be a stable rank $2$ vector bundle on $\piii$, with $c_1(F^\prim) 
= 0$ and $c_2(F^\prim) = 4$. If ${\fam0 H}^1(F^\prim(2)) = 0$ then $F^\prim$ is a 
$4$-instanton. 
\end{lemma} 

\begin{proof} 
We show that if $F^\prim$ is not a 4-instanton then $\tH^1(F^\prim(2)) \neq 0$. 
\emph{Indeed}, in this case $F^\prim$ must have spectrum $(1,0,0,-1)$. 
According to Chang \cite[Prop.~1.5]{ch2}, either $F^\prim$ has an unstable plane 
$H \subset \piii$ of order 1, or it is the cohomology of a selfdual 
monad$\, :$ 
\[
0 \lra \sco_\piii(-2) \lra 4\sco_\piii \lra \sco_\piii(2) \lra 0\, . 
\]
In the former case, one has an exact sequence$\, :$ 
\[
0 \lra \scf^\secund \lra F^\prim \lra \sci_{Z,H}(-1) \lra 0\, , 
\]
with $Z$ a 0-dimensional subscheme of $H$, of length 5, and $\scf^\secund$ a 
stable rank 2 reflexive sheaf with $c_1(\scf^\secund) = -1$, 
$c_2(\scf^\secund) = 3$, $c_3(\scf^\secund) = 5$. By \cite[Thm.~8.2(d)]{ha}, 
$\tH^2(\scf^\secund(2)) = 0$ hence $\h^1(F^\prim(2)) \geq \h^1(\sci_{Z,H}(1)) 
\geq 2$. 

In the latter case, using the (geometric) Koszul complex associated to the 
epimorphism $4\sco_\piii \ra \sco_\piii(2)$, one gets that $\h^1(F^\prim(2)) 
= 1$.  
\end{proof}

\section{Serre's method of extensions}\label{A:serre} 

We describe, here, a slightly more general variant of a method of 
Serre \cite{s} for constructing vector bundles from two codimensional 
subschemes. We include some arguments, for completness.  
We begin by recalling some facts about Serre duality, mainly 
because the duality isomorphisms from Hartshorne \cite[Chap.~III,~\S 7]{hag} 
are not explicit enough for our purposes. Firstly, a general result whose 
proof can be found in \cite[Lemma~3.3]{ct}.  

\begin{lemma}\label{L:homk(a)} 
Let $\mathcal A$ be an abelian category with enough injective objects, 
${\fam0 K}(\mathcal A)$ (resp., ${\fam0 D}(\mathcal A)$) the homotopy 
(resp., derived) category of complexes in $\mathcal A$, and $X^\bullet$ 
(resp., $Y^\bullet$) a complex in $\mathcal A$ which is bounded to the right 
(resp., left). Then the canonical map 
${\fam0 Hom}_{{\fam0 K}(\mathcal A)}(X^\bullet , Y^\bullet) \ra 
{\fam0 Hom}_{{\fam0 D}(\mathcal A)}(X^\bullet , Y^\bullet)$ is$\, :$ 

\emph{(i)} surjective if ${\fam0 Ext}_{\mathcal A}^{p-q}(X^p , Y^q) = 0$ for all 
$p > q$$\, ;$ 

\emph{(ii)} injective if ${\fam0 Ext}_{\mathcal A}^{p-q-1}(X^p , Y^q) = 0$ for all 
$p > q + 1$. 
\qed 
\end{lemma}  

\begin{lemma}\label{L:duality} 
Let $X$ be a projective scheme over an algebraically closed field $k$,  
of pure dimension $n$, $\sco_X(1)$ a very ample invertible sheaf on 
$X$, and $(\omega_X\, ,\, {\fam0 tr}_X \colon {\fam0 H}^n(\omega_X) \ra k)$ a 
dualizing sheaf on $X$. Assume that$\, :$ 
\begin{equation}\label{E:hiox(-j)} 
{\fam0 H}^i(\sco_X(-j)) = 0 \  for \  i < n \  and \  j >> 0\, . 
\end{equation}
Then, for any coherent $\sco_X$-module $\scf$, there exist isomorphisms$\, :$ 
\[
{\fam0 Ext}_{\sco_X}^p(\scf , \omega_X) \Izo {\fam0 H}^{n-p}(\scf)^\vee \, , \  
p \geq 0 \, . 
\]
\end{lemma} 

\begin{proof} 
Recall that, by the definition of a dualizing sheaf, the map$\, :$ 
\[
\text{Hom}_{\sco_X}(\scg , \omega_X) \lra \tH^n(\scg)^\vee \, , \  
f \mapsto \text{tr}_X \circ \tH^n(f)\, , 
\]
is bijective, for any coherent $\sco_X$-module $\scg$. Let $\text{K}(X)$ 
(resp., $\text{D}(X)$) be the homotopy (resp., derived) category of complexes 
of quasi-coherent $\sco_X$-modules. Let $\scf$ be a coherent $\sco_X$-module. 
It follows, from condition \eqref{E:hiox(-j)}, that $\scf$ admits a locally 
free resolution $\cdots \ra \scl^{-1} \ra \scl^0 \ra \scf \ra 0$, with 
$\scl^{-i} \simeq \sco_X(-m_i)^{N_i}$, such that $\tH^i(\scl^{-j}) = 0$, for 
$i < n$, $\forall \, j \geq 0$. Let $\scc^{-i}$ denote the cokernel of 
$\scl^{-i-1} \ra \scl^{-i}$. For $1 \leq p \leq n$, one gets isomorphisms$\, :$ 
\[
\tH^{n-p}(\scf) \Izo \tH^{n-p+1}(\scc^{-1}) \Izo \cdots \Izo 
\tH^{n-1}(\scc^{-p+1})\, . 
\] 
Using the commutative diagram$\, :$ 
\[
\begin{CD} 
\text{Hom}_{\sco_X}(\scl^{-p+1} , \omega_X) @>>> 
\text{Hom}_{\sco_X}(\scc^{-p} , \omega_X) @>>> 
\text{Hom}_{\text{K}(X)}(\scl^\bullet , \text{T}^p\omega_X) @>>> 0\\ 
@V{\wr}VV @V{\wr}VV\\ 
\tH^n(\scl^{-p+1})^\vee @>>> \tH^n(\scc^{-p})^\vee @>>> 
\tH^{n-1}(\scc^{-p+1})^\vee @>>> 0 
\end{CD}
\]
(where ``T'' denotes the translation functor for complexes) one gets an 
isomorphism$\, :$ 
\[
\text{Hom}_{\text{K}(X)}(\scl^\bullet , \text{T}^p\omega_X) \Izo 
\tH^{n-p}(\scf)^\vee \, . 
\]
If one chooses the resolution $\scl^\bullet$ such that, moreover, 
$\text{Ext}_{\sco_X}^i(\scl^{-j} , \omega_X) = 0$ for $i > 0$, $\forall \, j 
\geq 0$ (this is equivalent to $\tH^i(\omega_X(m_j)) = 0$ for $i > 0$, 
$\forall \, j \geq 0$) then Lemma~\ref{L:homk(a)} implies that$\, :$ 
\[
\text{Hom}_{\text{K}(X)}(\scl^\bullet , \text{T}^p\omega_X) \simeq 
\text{Hom}_{\text{D}(X)}(\scl^\bullet , \text{T}^p\omega_X) \simeq 
\text{Ext}_{\sco_X}^p(\scf , \omega_X)\, ,  
\]
hence $\text{Ext}_{\sco_X}^p(\scf , \omega_X) \izo \tH^{n-p}(\scf)^\vee$. 

The same argument shows that $\text{Ext}_{\sco_X}^p(\scf , \omega_X) = 0$ for 
$p > n$ because $\tH^0(\scc^{n-p}) = 0$.  
\end{proof}

\begin{lemma}\label{L:inducedduality} 
Under the hypotheses of Lemma~\ref{L:duality}, assume that, moreover, $X$ is 
nonsingular and connected. Let $Y$ be a locally Cohen-Macaulay closed 
subscheme of $X$, of pure codimension $p \geq 1$. Applying 
$\sch om_{\sco_X}(- , \omega_X)$ to a locally free resolution$\, :$ 
\[
0 \lra \sce^{-p} \lra \cdots \lra \sce^{-1} \lra \sco_X \lra \sco_Y \lra 0 
\] 
of $\sco_Y$, one gets an exact sequence$\, :$ 
\[
0 \lra \omega_X \lra \sce^{\prim -p+1} \lra \cdots \lra \sce^{\prim \, 0} \lra 
\omega_Y \lra 0\, , 
\]
where $\omega_Y := {\fam0 Coker}\, \sch om_{\sco_X}(d_{\sce}^{-p} , 
{\fam0 id}_{\omega_X}) \simeq \sce xt_{\sco_X}^p(\sco_Y , \omega_X)$ and 
$\sce^{\prim -i} = (\sce^{i-p})^{\vee} \otimes \omega_X$. Decomposing this  
sequence into short exact sequences one gets a map 
${\fam0 H}^{n-p}(\omega_Y) \ra {\fam0 H}^n(\omega_X)$. Then the pair$\, :$ 
\[
(\omega_Y\, ,\, {\fam0 H}^{n-p}(\omega_Y) \ra {\fam0 H}^n(\omega_X) 
\xra{{\fam0 tr}_X} k) 
\] 
is a dualizing sheaf on $Y$ and ${\fam0 H}^i(\sco_Y(-j)) = 0$ for $i < n-p$ 
and $j >> 0$. 
\end{lemma} 

\begin{proof} 
Let $\sce^{\prim \bullet}$ denote the complex $0 \ra \omega_X \ra 
\sce^{\prim -p+1} \ra \cdots \ra \sce^{\prim \, 0} \ra 0$ (so that $\sce^{\prim -p} = 
\omega_X$) and let $\sce^{\secund \bullet}$ be the complex 
$0 \ra \sce^{\prim -p+1} \ra \cdots \ra \sce^{\prim 0} \ra 0$. Let $\scf$ be a 
coherent $\sco_X$-module. Choose, as in the proof of Lemma~\ref{L:duality}, a 
locally free resolution $\scl^\bullet$ of $\scf$ such that $\tH^i(\scl^{-j}) = 0$ 
for $i < n$, $\forall \, j \geq 0$ and such that, moreover, 
$\text{Ext}_{\sco_X}^i(\scl^{-j} , \sce^{\prim -m}) = 0$ for $i > 0$ and $0 \leq m 
\leq p$, $\forall \, j \geq 0$. 
Now, Lemma~\ref{L:homk(a)} implies that$\, :$ 
\[
\text{Hom}_{\text{K}(X)}(\scl^\bullet , \sce^{\prim \bullet}) \Izo 
\text{Hom}_{\text{D}(X)}(\scl^\bullet , \sce^{\prim \bullet}) \Izo 
\text{Hom}_{\sco_X}(\scf , \omega_Y)\, . 
\]
Taking into account the proof of Lemma~\ref{L:duality}, it follows that, in 
order to show that the pair from the statement is a dualizing sheaf on $Y$, 
it suffices to show that if $\scf$ is annihilated by $\sci_Y$, i.e., if 
$\scf$ is an $\sco_Y$-module, then map$\, :$ 
\[
v \colon \text{Hom}_{\text{K}(X)}(\scl^\bullet , \sce^{\prim \bullet}) \lra 
\text{Hom}_{\text{K}(X)}(\scl^\bullet , \text{T}^p\omega_X) 
\]
is bijective. We use, for that purpose, the exact sequence$\, :$ 
\begin{gather*}
\text{Hom}_{\text{K}(X)}(\scl^\bullet , \sce^{\secund \bullet}) 
\overset{u}{\lra} 
\text{Hom}_{\text{K}(X)}(\scl^\bullet , \sce^{\prim \bullet}) \overset{v}{\lra} 
\text{Hom}_{\text{K}(X)}(\scl^\bullet , \text{T}^p\omega_X) \overset{w}{\lra}\\  
\lra \text{Hom}_{\text{K}(X)}(\scl^\bullet , \text{T}\sce^{\secund \bullet})\, . 
\end{gather*} 

\vskip2mm 

\noindent 
{\bf Claim 1.}\quad $\text{Hom}_{\text{K}(X)}(\scl^\bullet , 
\sce^{\secund \bullet}) = 0$. 

\vskip2mm 

\noindent 
\emph{Indeed}, it follows, from Lemma~\ref{L:homk(a)}, that$\, :$ 
\begin{equation}\label{E:homlesecund}   
\text{Hom}_{\text{K}(X)}(\scl^\bullet , \sce^{\secund \bullet}) \Izo 
\text{Hom}_{\text{D}(X)}(\scl^\bullet , \sce^{\secund \bullet}) \Izo 
\text{Hom}_{\text{D}(X)}(\scf , \sce^{\secund \bullet})\, .  
\end{equation}
Moreover, one has, for $i \geq 0$, an exact sequence$\, :$ 
\begin{equation}\label{E:homd(scf,-)} 
\text{Hom}_{\text{D}(X)}(\scf , \text{T}^i\sce^{\prim -i}) \lra 
\text{Hom}_{\text{D}(X)}(\scf , \sigma^{\leq -i}\sce^{\secund \bullet}) \lra 
\text{Hom}_{\text{D}(X)}(\scf , \sigma^{\leq -i-1}\sce^{\secund \bullet})\, ,   
\end{equation} 
where $\sigma^{\leq -i}\sce^{\secund \bullet}$ denotes the complex $0 \ra 
\sce^{\prim -p+1} \ra \cdots \ra \sce^{\prim -i} \ra 0$. 
It follows, from Lemma~\ref{L:duality}, that if $\sce$ is a locally free 
$\sco_X$-module then, for $i < p$$\, :$ 
\[
\text{Ext}_{\sco_X}^i(\scf , \sce^\vee \otimes \omega_X) \simeq 
\text{Ext}_{\sco_X}^i(\scf \otimes \sce , \omega_X) \simeq 
\tH^{n-i}(\scf \otimes \sce)^\vee = 0  
\]
because $\dim \text{Supp}\, \scf \leq n-p$. 
In particular, $\text{Hom}_{\text{D}(X)}(\scf , 
\text{T}^i\sce^{\prim -i}) = 0$, $i = 0, \ldots , p-1$, hence  
$\text{Hom}_{\text{D}(X)}(\scf , \sce^{\secund \bullet}) = 0$. 
Claim 1 is proven. 

\vskip2mm 

\noindent 
{\bf Claim 2.} $w$ \emph{is the zero morphism}. 

\vskip2mm 

\noindent 
\emph{Indeed}, let $\sce^{\, \prime \prime \prime \bullet}$ be the complex 
$0 \ra \sce^{\prim -p+2} \ra \cdots \ra \sce^{\prim \, 0} \ra 0$. One has an exact 
sequence$\, :$ 
\[
\text{Hom}_{\text{K}(X)}(\scl^\bullet , 
\text{T}\sce^{\prime \prime \prime \bullet}) \lra 
\text{Hom}_{\text{K}(X)}(\scl^\bullet , \text{T}\sce^{\secund \bullet}) \lra 
\text{Hom}_{\text{K}(X)}(\scl^\bullet , \text{T}^p\sce^{\prim -p+1})\, , 
\] 
and, as in the proof of Claim 1, $\text{Hom}_{\text{K}(X)}(\scl^\bullet , 
\text{T}\sce^{\prime \prime \prime \bullet}) = 0$. It follows that the 
morphism $\text{Hom}_{\text{K}(X)}(\scl^\bullet , \text{T}\sce^{\secund \bullet}) \ra 
\text{Hom}_{\text{K}(X)}(\scl^\bullet , \text{T}^p\sce^{\prim -p+1})$ is injective. 

Now, the composite map$\, :$ 
\[
\text{Hom}_{\text{K}(X)}(\scl^\bullet , \text{T}^p\omega_X) \overset{w}{\lra} 
\text{Hom}_{\text{K}(X)}(\scl^\bullet , \text{T}\sce^{\secund \bullet}) \lra 
\text{Hom}_{\text{K}(X)}(\scl^\bullet , \text{T}^p\sce^{\prim -p+1})
\] 
can be identified with the morphism $\text{Ext}_{\sco_X}^p(\scf , \omega_X) 
\ra \text{Ext}_{\sco_X}^p(\scf , \sce^{\prim -p+1})$ hence with the morphism 
$\text{Ext}_{\sco_X}^p(\scf , \omega_X) 
\ra \text{Ext}_{\sco_X}^p(\scf \otimes \sce^{-1}, \omega_X)$ which is 0 because 
the morphism $\scf \otimes \sce^{-1} \ra \scf$ is 0. One deduces that $w = 0$. 

\vskip2mm 

The claims 1 and 2 and the exact sequence before Claim 1 show that $v$ is an 
isomorphism hence the pair from the statement is a dualizing sheaf for $Y$. 
The proof of the fact that $\tH^i(\sco_Y(-j)) = 0$ for $i < n-p$ and 
$j >> 0$ is easy because any locally free $\sco_X$-module $\sce$ embeds into 
one of the form $\sco_X(m)^N$ hence $\tH^i(\sce(-j)) = 0$ for $i < n$ and 
$j >> 0$. 
\end{proof} 

The following result is the key technical point in our proof of Serre's 
method of extensions. 

\begin{lemma}\label{L:keylift} 
Under the hypotheses of Lemma~\ref{L:inducedduality}, let $\scf$ be a 
coherent $\sco_X$-module such that ${\fam0 H}^{n-i}(\scf \otimes \sce^{i-p}) = 
0$, $i = 1, \ldots , p-1$. Then any morphism $f \colon \scf \ra \omega_Y$ with 
the property that ${\fam0 tr}_Y \circ {\fam0 H}^{n-p}(f) = 0$ lifts to a 
morphism ${\widetilde f} \colon \scf \ra \sce^{\prim \, 0}$. 
\end{lemma}

\begin{proof} 
We use the notation from the proof of Lemma~\ref{L:inducedduality}. $f$ can 
be extended to a morphism of complexes $\phi \colon \scl^\bullet \ra 
\sce^{\prim \bullet}$. Let $\scc^{-i}$ denote the cokernel of $\scl^{-i-1} \ra 
\scl^{-i}$ and let $\phi^{\prime -p} \colon \scc^{-p} \ra \omega_X$ denote 
the morphism induced by $\phi^{-p} \colon \scl^{-p} \ra \omega_X$.  
The description of $\text{tr}_Y : \tH^{n-p}(\omega_Y) \ra k$ from 
the conclusion of Lemma~\ref{L:inducedduality} shows that the composite 
morphism$\, :$  
\[
\tH^{n-p}(\scf) \Izo \tH^n(\scc^{-p}) \xra{\tH^n(\phi^{\prime -p})} 
\tH^n(\omega_X) \xra{\text{tr}_X} k
\] 
is 0. Taking into account the isomorphism $\text{Hom}_{\text{K}(X)}(\scl^\bullet , 
\text{T}^p\omega_X) \izo \tH^{n-p}(\scf)^\vee$ from the proof of 
Lemma~\ref{L:duality}, it follows that the image of $\phi$ into 
$\text{Hom}_{\text{K}(X)}(\scl^\bullet , \text{T}^p\omega_X)$ is 0.  

Recall, now, from the proof of Lemma~\ref{L:inducedduality}, the exact 
sequence$\, :$ 
\[
\text{Hom}_{\text{K}(X)}(\scl^\bullet , \sce^{\secund \bullet}) 
\overset{u}{\lra} 
\text{Hom}_{\text{K}(X)}(\scl^\bullet , \sce^{\prim \bullet}) \overset{v}{\lra} 
\text{Hom}_{\text{K}(X)}(\scl^\bullet , \text{T}^p\omega_X)\, , 
\]
and the isomorphisms \eqref{E:homlesecund}. By Lemma~\ref{L:duality}, 
the hypothesis $\tH^{n-i}(\scf \otimes \sce^{i-p}) = 0$ is equivalent to 
$\text{Hom}_{\text{D}(X)}(\scf , \text{T}^i\sce^{\prim -i}) = 0$, $i = 1, \ldots , 
p-1$. One deduces, using the exact sequences \eqref{E:homd(scf,-)} in 
the proof of Lemma~\ref{L:inducedduality}, that one has an 
epimorphism$\, :$ 
\[
\text{Hom}_{\sco_X}(\scf , \sce^{\prim \, 0}) \lra 
\text{Hom}_{\text{D}(X)}(\scf , \sce^{\secund \bullet})\, . 
\]  
The commutative diagram$\, :$ 
\[
\begin{CD} 
\text{Hom}_{\text{K}(X)}(\scl^\bullet , \sce^{\secund \bullet}) 
@>u>> 
\text{Hom}_{\text{K}(X)}(\scl^\bullet , \sce^{\prim \bullet}) @>v>>  
\text{Hom}_{\text{K}(X)}(\scl^\bullet , \text{T}^p\omega_X)\\ 
@A{\text{epi}}AA @V{\wr}VV\\ 
\text{Hom}_{\sco_X}(\scf , \sce^{\prim \, 0}) @>>> 
\text{Hom}_{\sco_X}(\scf , \omega_Y)
\end{CD}
\]
shows, now, that $f$ can be lifted to an ${\widetilde f} \in 
\text{Hom}_{\sco_X}(\scf , \sce^{\prim \, 0})$. 
\end{proof} 

\begin{thm}\label{T:serreext} 
Let $X$ be a nonsingular connected projective variety, of dimension 
$n \geq 2$, $(\omega_X\, ,\, {\fam0 tr}_X \colon {\fam0 H}^n(\omega_X) \ra k)$ 
a dualizing sheaf on $X$, $Y$ a locally Cohen-Macaulay closed subscheme of 
$X$, of pure codimension $2$, and 
$(\omega_Y \, ,\, {\fam0 tr}_Y : {\fam0 H}^{n-2}(\omega_Y) \ra k)$ 
a dualizing sheaf on $Y$. Let $F$ be a locally free $\sco_X$-module. 
Then there exists an isomorphism $\omega_Y \simeq \sce xt_{\sco_X}^2(\sco_Y , 
\omega_X)$ such that, for any epimorphism 
$\delta \colon F^\vee \otimes \omega_X \ra \omega_Y$ with ${\fam0 tr}_Y \circ 
{\fam0 H}^{n-2}(\delta) = 0$, there exists an extension$\, :$ 
\[
0 \lra F \lra E \lra \sci_Y \lra 0, 
\]
with $E$ locally free and such that, when applying 
$\sch om_{\sco_X}(- , \omega_X)$ to it, the connecting morphism$\, :$ 
\[
\partial \colon F^\vee \otimes \omega_X \lra 
\sce xt_{\sco_X}^1(\sci_Y , \omega_X) \simeq 
\sce xt_{\sco_X}^2(\sco_Y , \omega_X)
\] 
is identified with $\delta$. Notice that the hypothesis ${\fam0 tr}_Y \circ 
{\fam0 H}^{n-2}(\delta) = 0$ is automatically satisfied if ${\fam0 H}^2(F) 
= 0$. 
\end{thm} 

\begin{proof} 
Choose a very ample invertible sheaf $\sco_X(1)$ on $X$. There exists an 
integer $m$ such that $\tH^1(F(l)) = 0$, $\forall \, l \geq m$. Choose a 
locally free resolution$\, :$ 
\[
0 \lra \sce^{-2} \xra{d^{-2}} \sce^{-1} \xra{d^{-1}} \sco_X \lra \sco_Y \lra 0 
\] 
of $\sco_Y$, with $\sce^{-1} = \sco_X(-m_1)^{N_1}$, $m_1 \geq m$, and put$\, :$ 
\[
\omega_Y^\circ := \Cok \sch om_{\sco_X}(d^{-2} , \text{id}_{\omega_X}) 
\simeq \sce xt_{\sco_X}^2(\sco_Y , \omega_X)\, . 
\]
Let $\text{tr}_Y^\circ \colon \tH^{n-2}(\omega_Y^\circ ) \ra k$ be the trace 
map defined in Lemma~\ref{L:inducedduality}. By the uniqueness of the 
dualizing sheaf (see \cite[III,~Prop.~7.2]{hag}), there exists a unique 
isomorphism $\omega_Y \simeq \omega_Y^\circ$ identifying $\text{tr}_Y$ with 
$\text{tr}_Y^\circ$. 

Consider, now, an epimorphism $\delta$ as in the statement and let 
$\delta^\circ \colon F^\vee \otimes \omega_X \ra \omega_Y^\circ$ be the 
epimorphism obtained by composing $\delta$ with $\omega_Y \izo \omega_Y^\circ$. 
Since, by Serre duality on $X$, $\tH^{n-1}(F^\vee \otimes \omega_X \otimes 
\sce^{-1}) = 0$, Lemma~\ref{L:keylift} implies that $\delta^\circ$ lifts to a 
morphism ${\widetilde \delta} \colon F^\vee \otimes \omega_X \ra 
\sce^{-2 \vee} \otimes \omega_X$. One has ${\widetilde \delta} = 
\eta \otimes \text{id}_{\omega_X}$, with $\eta \colon F^\vee \ra \sce^{-2 \vee}$. 

Consider the extension defined by the push-forward diagram$\, :$ 
\[
\begin{CD}
0 @>>> \sce^{-2} @>{d^{-2}}>> \sce^{-1} @>{d^{-1}}>> \sci_Y @>>> 0\\
@. @V{\eta^\vee}VV @VVV @\vert\\ 
0 @>>> F @>>> E @>>> \sci_Y @>>> 0
\end{CD}
\]   
Applying $\sch om_{\sco_X}(- , \omega_X)$ to this diagram one deduces that$\, :$ 
\[
\partial \colon F^\vee \otimes \omega_X \lra \sce xt_{\sco_X}^1(\sci_Y , \omega_X) 
\simeq \omega_Y^\circ 
\]
coincides with $\delta^\circ$. Moreover, $E$ is locally free because one has 
an exact sequence$\, :$ 
\[
0 \lra \sce^{-2} \xra{\left(\begin{smallmatrix} d^{-2}\\ -\eta^\vee 
\end{smallmatrix}\right)} \sce^{-1} \oplus F \lra E \lra 0\, , 
\]
and $(d^{-2 \vee}\, ,\, -\eta) \colon \sce^{-1 \vee} \oplus F^\vee \ra 
\sce^{-2 \vee}$ is an epimorphism. (We have used tacitly the fact that 
$\omega_X$ is invertible which follows, however, easily by embedding $X$ 
into a projective space and then using Lemma~\ref{L:inducedduality}.) 
\end{proof} 

\begin{remark}\label{R:uniqueness} 
Let $X$, $Y$ and $F$ be as in Thm.~\ref{T:serreext}. Assume, moreover, that 
${\fam0 H}^1(F) = 0$. Consider two extensions$\, :$ 
\[
0 \lra F \lra E_i \lra \sci_Y \lra 0\, ,\  i = 0,\, 1\, , 
\]
with $E_0$ and $E_1$ locally free. Applying $\sch om_{\sco_X}(-,\omega_X)$ to 
these extensions, one gets exact sequences$\, :$ 
\[
0 \lra \omega_X \lra E_i^\vee \otimes \omega_X \lra F^\vee \otimes \omega_X 
\overset{\delta_i}{\lra} \omega_Y \lra 0\, ,\  i = 0,\, 1\, .  
\]
If $\delta_0 = \delta_1$ then $E_0 \simeq E_1$. 

\vskip2mm

\noindent 
\emph{Indeed},  
consider a locally free resolution of $\sco_Y$ as at the beginning of the 
proof of Thm.~\ref{T:serreext}. Since $\text{Ext}_{\sco_X}^1(\sce^{-1} , F) = 0$, 
one gets, for $i = 0,\, 1$, a commutative diagram$\, :$ 
\[
\begin{CD}
0 @>>> \sce^{-2} @>{d^{-2}}>> \sce^{-1} @>{d^{-1}}>> \sci_Y @>>> 0\\ 
@. @V{\eta_i^\vee}VV @VVV @\vert\\ 
0 @>>> F @>>> E_i @>>> \sci_Y @>>> 0
\end{CD}
\]
for a certain morphism $\eta_i \colon F^\vee \ra \sce^{-2 \vee}$. Since  
$\omega_Y \simeq \sce xt_{\sco_X}^1(\sci_Y , \omega_X) \simeq 
\Cok d^{-2 \vee}\otimes \text{id}_{\omega_X}$, one gets, applying 
$\sch om_{\sco_X}( - , \omega_X)$ to the top exact sequence of the above 
diagram, an exact sequence$\, :$ 
\[
0 \lra \omega_X \lra \sce^{-1 \vee} \otimes \omega_X \lra \sce^{-2 \vee} 
\otimes \omega_X \lra \omega_Y \lra 0\, . 
\] 
Then $\delta_i$ lifts to the morphism ${\widetilde \delta}_i := \eta_i \otimes 
\text{id}_{\omega_X} \colon F^\vee \otimes \omega_X \ra \sce^{-2 \vee} \otimes 
\omega_X$. Since $\delta_0 = \delta_1$, ${\widetilde \delta}_1 - 
{\widetilde \delta}_0$ factorizes through the kernel $\sck$ of the 
epimorphism $\sce^{-2 \vee} \otimes \omega_X \ra \omega_Y$. One also has an 
exact sequence$\, :$ 
\[
0 \lra \omega_X \lra \sce^{-1 \vee} \otimes \omega_X \lra \sck \lra 0\, .    
\] 
Since $\text{Ext}_{\sco_X}^1(F^\vee \otimes \omega_X , \omega_X) \simeq 
\tH^1(F) = 0$, one deduces that there exists a morphism $\phi \colon F^\vee 
\ra \sce^{-1 \vee}$ such that $\eta_1 - \eta_0 = \phi \circ d^{-2 \vee}$. 
Using, now, the exact sequences$\, :$ 
\[
0 \lra E_i^\vee \lra \sce^{-1 \vee} \oplus F^\vee \xra{(d^{-2 \vee}\, ,\, -\eta_i)} 
\sce^{-2 \vee} \lra 0\, ,\  i = 0,\, 1\, ,  
\]
one deduces that $E_0 \simeq E_1$. 
\end{remark}

\section{Miscellaneous auxiliary results}\label{A:miscellaneous} 

\begin{lemma}\label{L:tvbqraoq(2,4)}  
Let $Q \simeq \pj \times \pj$ be a nonsingular quadric surface in $\piii$. 
Then there exist epimorphisms $\eta \colon {\fam0 T}_\piii(-1) \vb Q \ra 
\sco_Q(1,3)$ such that the map ${\fam0 H}^0(\eta(1,1)) 
\colon {\fam0 H}^0({\fam0 T}_\piii \vb Q) \ra {\fam0 H}^0(\sco_Q(2,4))$ is 
injective (hence bijective).   
\end{lemma} 

\begin{proof}
The kernel $\scg$ of any epimorphism $\eta \colon \text{T}_\piii(-1) \vb Q \ra 
\sco_Q(1,3)$ is a rank 2 vector bundle on $Q$ with $\text{det}\, \scg \simeq 
\sco_Q(0,-2)$. It follows that $\scg^\vee \simeq \scg(0,2)$ (on $Q$). 
One deduces an exact sequence$\, :$ 
\[
0 \lra \sco_Q(0,-4) \overset{\rho}{\lra} (\Omega_\piii \vb Q)(2,0) 
\lra \scg(1,1) \lra 0\, ,   
\] 
where $\rho = \eta^\vee(1,-1)$.  
One has $\tH^0((\Omega_\piii \vb Q)(2,0)) = 0$ (because $\Omega_\piii \vb Q$ 
embeds into $4\sco_Q(-1,-1)$) hence $\tH^0(\eta(1,1))$ is injective iff 
$\tH^0(\scg(1,1)) = 0$ iff $\tH^1(\rho)$ is injective. 

Now, since $(\Omega_\piii \vb Q)(2,4)$ is globally generated (because 
$\Omega_\piii(2)$ is) a general morphism $\rho \colon \sco_Q(0,-4) \ra 
(\Omega_\piii \vb Q)(2,0)$ is a locally split monomorphism. It thus remains to 
show that there exist morphisms $\rho \colon \sco_Q(0,-4) \ra 
(\Omega_\piii \vb Q)(2,0)$ such that $\tH^1(\rho)$ is injective. 

Consider, for that purpose, the exact sequence$\, :$ 
\[
0 \lra \sco_Q(-2,-2) \lra \Omega_\piii \vb Q \lra \sco_Q(-2,0) \oplus 
\sco_Q(0,-2) \lra 0\, . 
\]
Tensorize it by $\sco_Q(2,0)$. Since $\text{Ext}_{\sco_Q}^1(\sco_Q(0,-4), 
\sco_Q(0,-2)) = 0$, the morphism$\, :$ 
\[
\text{Hom}(\sco_Q(0,-4), (\Omega_\piii \vb Q)(2,0)) \lra 
\text{Hom}(\sco_Q(0,-4), \sco_Q \oplus \sco_Q(2,-2))  
\]
is surjective. Moreover, $\tH^1((\Omega_\piii \vb Q)(2,0)) \izo 
\tH^1(\sco_Q(2,-2))$. It thus suffices to show that, for a general element 
$f \in \tH^0(\sco_Q(2,2))$, the multiplication by $f \cdot - \colon 
\tH^1(\sco_Q(0,-4)) \ra \tH^1(\sco_Q(2,-2))$ is injective. 

In order to prove the existence of such an $f$, let us recall 
that$\, :$ 
\begin{gather*}
\tH^0(\sco_Q(2,2)) \simeq \tH^0(\sco_\pj(2)) \otimes \tH^0(\sco_\pj(2))\  ,\  
\tH^1(\sco_Q(0,-4)) \simeq \tH^0(\sco_\pj) \otimes \tH^1(\sco_\pj(-4))\  ,\\  
\tH^1(\sco_\pj(2,-2)) \simeq \tH^0(\sco_\pj(2)) \otimes \tH^1(\sco_\pj(-2))\, . 
\end{gather*}
Write $f$ as $u_0^2 \otimes f_0 + u_0u_1 \otimes f_1 + u_1^2 \otimes f_2$, 
with $f_0,\, f_1,\, f_2 \in k[t_0,t_1]_2$. If $\xi \in \tH^1(\sco_\pj(-4))$ 
then$\, :$ 
\[
f \cdot (1\otimes \xi) = u_0^2 \otimes (f_0\xi) + u_0u_1 \otimes (f_1\xi) + 
u_1^2 \otimes (f_2\xi)\, . 
\]
If $f_0,\, f_1,\, f_2$ is a k-basis of $k[t_0,t_1]_2$ then $f_i\xi = 0$, $i = 
0,\, 1,\, 2$, implies $\xi = 0$, hence, in this case, multiplication by 
$f \cdot - \colon \tH^1(\sco_Q(0,-4)) \ra \tH^1(\sco_Q(2,-2))$ is injective. 
\end{proof}  

\begin{lemma}\label{L:2oraoy(1)} 
Let $Y$ be the union of three mutually disjoint lines in $\piii$. If $\sck$ is 
the kernel of a general epimorphism $\delta \colon 2\sco_\piii \ra \sco_Y(1)$ 
then $\sck(2)$ is globally generated.  
\end{lemma}

\begin{proof} 
Let $Q \subset \piii$ be the unique quadric surface containing $Y$ and let 
$K$ be the kernel of $\delta_Q \colon 2\sco_Q \ra \sco_Y(1)$. $K$ is a rank 2 
vector bundle on $Q$. Since one has an exact sequence$\, :$ 
\[
0 \lra 2\sco_\piii(-2) \lra \sck \lra K \lra 0\, , 
\]
$\sck(2)$ is globally generated iff $K(2)$ is globally generated. It thus 
suffices to show that if $K$ is the kernel of a general epimorphism 
$\e \colon 2\sco_Q \ra \sco_Y(1)$ then $K(2)$ is globally generated. 

In order to prove the last assertion, fix an isomorphism $Q \simeq \pj \times 
\pj$. Assume that $L_0$, $L_1$, $L_2$ belong to the linear system $\vert\, 
\sco_Q(1,0)\, \vert$. Tensorizing by $\sco_Q(2,1)$ the exact sequence$\, :$ 
\[
0 \lra \sco_Q(-3,0) \lra \sco_Q \lra \sco_Y \lra 0\, , 
\]  
one gets an exact sequence $0 \ra \sco_Q(-1,1) \ra \sco_Q(2,1) \ra \sco_Y(1) 
\ra 0$ from which one deduces that $\tH^0(\sco_Q(2,1)) \izo 
\tH^0(\sco_Y(1))$. Choose $\lambda,\, \lambda^\prime \in \tH^0(\sco_Q(2,1))$ 
such the intersection of the zero divisors of $\lambda$ and $\lambda^\prime$ 
consists of four simple points $P_1, \ldots , P_4$, none of them situated on 
$Y$, and such that none of the lines $\overline{P_iP_j}$, $1 \leq i < j \leq 
4$, belongs to the linear system $\vert \, \sco_Q(0,1)\, \vert$. [To see that 
this is possible, choose, firstly, $\lambda$ such that its zero divisor is a 
nonsingular twited cubic curve $C \subset Q$. Fix an isomorphism $\pj \izo C$. 
One has $\sco_Q(2,1) \vb C \simeq \sco_\pj(4)$. Using the exact sequence 
$0 \ra \sco_Q \ra \sco_Q(2,1) \ra \sco_\pj(4) \ra 0$, one sees that the map 
$\tH^0(\sco_Q(2,1)) \ra \tH^0(\sco_\pj(4))$ is surjective. Choose $P_1, \ldots 
, P_4$ on $C$ satisfying the above conditions. Finally, choose $\lambda^\prime 
\in \tH^0(\sco_Q(2,1))$ such that the zero divisor of $\lambda^\prime \vb C$ 
is $P_1 + \cdots + P_4$.]   

Let $\e \colon 2\sco_Q \ra \sco_Y(1)$ be the epimorphism defined by $\lambda 
\vb Y$ and $\lambda^\prim \vb Y$ and let $K$ be its kernel. 

\vskip2mm 
 
\noindent 
{\bf Claim.}\quad $\tH^0(K(1,2)) = 0$. 

\vskip2mm 

\noindent 
\emph{Indeed}, applying the Snake Lemma to the commutative diagram$\, :$ 
\[
\begin{CD} 
0 @>>> \sco_Q(-2,-1) @>{(-\lambda^\prime , \lambda)^{\text{t}}}>> 2\sco_Q 
@>{(\lambda , \lambda^\prime)}>> \sci_{\Gamma , Q}(2,1) @>>> 0\\ 
@. @VVV @VV{\e}V @VVV\\  
0 @>>> 0 @>>> \sco_Y(1) @= \sco_Y(1) @>>> 0
\end{CD}
\]  
where $\Gamma := \{P_1, \ldots , P_4\}$, one gets an exact sequence$\, :$ 
\[
0 \lra \sco_Q(-2,-1) \lra K \lra \sci_{\Gamma , Q}(-1,1) \lra 0\, .  
\]
Tensorizing this exact sequence by $\sco_Q(1,2)$ and using the fact that 
$\tH^0(\sci_{\Gamma , Q}(0,3)) = 0$ (because none of the lines 
$\overline{P_iP_j}$, $1 \leq i < j \leq 4$, belongs to the linear system 
$\vert \, \sco_Q(0,1)\, \vert$) one gets the claim.  

\vskip2mm 

Using the exact sequence $0 \ra K(1,2) \ra 2\sco_Q(1,2) \ra \sco_Y(3) \ra 0$, 
one deduces, from the above claim, 
that $\tH^1(K(1,2)) = 0$. Let $L \subset Q$ be a line belonging to the linear 
system $\vert \, \sco_Q(1,0)\, \vert$. Then $K_L \simeq \sco_L(1) \oplus 
\sco_L(-1)$ if $L = L_i$, $i = 1,\, 2,\, 3$, and $K_L \simeq 2\sco_L$ 
otherwise. Tensorizing by $K(2)$ ($= K(2,2)$) the exact sequence$\, :$  
\[
0 \lra \sco_Q(-1,0) \lra \sco_Q \lra \sco_L \lra 0\, , 
\]
one gets that $\tH^0(K(2)) \izo \tH^0(K_L(2))$. Since $K_L(2)$ is globally 
generated, for every line $L$ belonging to $\vert \, \sco_Q(1,0)\, \vert$, 
it follows that $K(2)$ is globally generated. 
\end{proof}

\begin{lemma}\label{L:2oomega(1)raol(2)} 
Consider a morphism $\delta \colon 2\sco_\piii \oplus \Omega_\piii(1) \ra 
\sco_L(2)$, where $L \subset \piii$ is a line. If the component 
$\delta^\prim \colon 2\sco_\piii \ra \sco_L(2)$ of $\delta$ is an epimorphism and 
if ${\fam0 H}^0(\delta_L) \colon {\fam0 H}^0(2\sco_L \oplus 
(\Omega_\piii(1) \vb L)) \ra {\fam0 H}^0(\sco_L(2))$ is surjective then 
${\fam0 Ker}\, \delta(1)$ is globally generated. 
\end{lemma} 

\begin{proof} 
One has $\Omega_\piii(1) \vb L \simeq 2\sco_L \oplus \sco_L(-1)$. Since 
$\tH^0(\delta_L)$ is surjective and since $\tH^0(\delta_L^\prim)$ must be 
injective, there exists a nonzero global section $\sigma$ of 
$\Omega_\piii(1) \vb L$ and constants $c_1,\, c_2 \in k$ such that the kernel 
of $\tH^0(\delta_L)$ is generated by $(c_1,\, c_2,\, \sigma)$. Recalling the 
exact sequence $0 \ra \Omega_\piii(1) \ra 4\sco_\piii \ra \sco_\piii(1) \ra 0$, 
there exists a plane $H \supset L$ and a nonzero global section 
$\widetilde \sigma$ of $\Omega_\piii(1) \vb H$ such that 
${\widetilde \sigma} \vb L = \sigma$. 

Now, since $\Omega_\piii(1) \vb H \simeq \sco_H \oplus \Omega_H(1)$ and 
$\tH^0(\Omega_H(1)) = 0$, there exists an automorphism $\phi$ of 
$2\sco_H \oplus (\Omega_\piii(1) \vb H)$ mapping $(c_1,\, c_2,\, 
{\widetilde \sigma})$ to $(0,\, 0,\, {\widetilde \sigma})$. Since 
$\text{Hom}_{\sco_\piii}(\Omega_\piii(1) , 
\sco_\piii) \izo \text{Hom}_{\sco_H}((\Omega_\piii(1) \vb H) , \sco_H)$, $\phi$ 
lifts to an automorphism $\psi$ of $2\sco_\piii \oplus \Omega_\piii(1)$.   
Then $\delta \circ \psi^{-1}$ factorizes as$\, :$ 
\[
2\sco_\piii \oplus \Omega_\piii(1) \overset{\gamma}{\lra} 
2\sco_H \oplus \Omega_H(1) \overset{\eta}{\lra} \sco_L(2)\, , 
\]
where $\gamma$ is the canonical epimorphism and $\eta$ is such that 
its component $\eta^\prim \colon 2\sco_H \ra \sco_L(2)$ is an epimorphism and 
$\tH^0(\eta_L) \colon \tH^0(2\sco_L \oplus (\Omega_H(1) \vb L)) \ra 
\tH^0(\sco_L(2))$ is bijective. Since the kernel of the canonical epimorphism 
$\Omega_\piii(1) \ra \Omega_H(1)$ is isomorphic to $3\sco_\piii(-1)$, one gets 
an exact sequence$\, :$ 
\[
0 \lra 5\sco_\piii(-1) \lra \Ker (\delta \circ \psi^{-1}) \lra \Ker \eta 
\lra 0\, . 
\]
It follows that, in order to show that $\Ker \delta(1)$ is globally generated, 
it suffices to show that $\Ker \eta(1)$ is globally generated. 

Let $K$ be the kernel of $\eta$ and $K^\prim$ the kernel of $\eta^\prim$. One 
has an exact sequence $0 \ra K^\prim \ra K \ra \Omega_H(1) \ra 0$.  
Since $\Ker \eta_L^\prim \simeq \sco_L(-2)$, one also has an exact 
sequence$\, :$ 
\[
0 \lra 2\sco_H(-1) \lra K^\prim \lra \sco_L(-2) \lra 0\, , 
\] 
hence the cokernel of the evaluation morphism $\tH^0(K^\prim(1)) \otimes 
\sco_H \ra K^\prim(1)$ is isomorphic to $\sco_L(-1)$ and $\tH^1(K^\prim(1)) = 0$. 
Applying the Snake Lemma to the commutative diagram$\, :$ 
\[
\begin{CD} 
%@. @. @. \sco_H(-1)\\ 
%@. @. @. @VVV\\ 
0 @>>> \tH^0(K^\prim(1)) \otimes \sco_H @>>> \tH^0(K(1)) \otimes \sco_H 
@>>> 3\sco_H @>>> 0\\ 
@. @V{\text{ev}}VV @V{\text{ev}}VV @V{\text{ev}}VV\\ 
0 @>>> K^\prim(1) @>>> K(1) @>>> \Omega_H(2) @>>> 0  
%@. @VVV\\ 
%@. \sco_L(-1)  
\end{CD}
\] 
one deduces that $K(1)$ is globally generated if and only if the connecting 
morphism $\partial \colon \sco_H(-1) \ra \sco_L(-1)$ is an epimorphism.  
This happens if and only if $\partial \neq 0$. If $\partial = 0$ then the 
cokernel of the evaluation morphism of $K(1)$ is isomorphic to $\sco_L(-1)$. 
But \emph{this is not possible} because $K_L \simeq \sco_L(1) \oplus 
3\sco_L(-1)$, as one sees using the exact sequence$\, :$ 
\[
K_L \lra 2\sco_L \oplus (\sco_L \oplus \sco_L(-1)) \xra{\eta \vb L} \sco_L(2) 
\lra 0 
\]
and the fact that $c_1(K) = -2$. 
\end{proof} 

\begin{lemma}\label{L:lineinp5} 
Let $V$ be a $4$-dimensional $k$-vector space and $W$ a $2$-dimensional 
vector subspace of $\bigwedge^2 V$. Let $\Delta$ be the image of 
the wedge product $V \times V \ra \bigwedge^2 V$. $\Delta$ is the 
affine cone over the Pl\"{u}cker hyperquadric $\mathbb{G} \subset 
\p(\bigwedge^2 V) \simeq \pv$. 

\emph{(a)} If $\p(W)$ intersects $\mathbb{G}$ in two distinct points then 
there exists a basis $v_0 , \ldots , v_3$ of $V$ such that $v_0 \wedge v_1$ 
and $v_2 \wedge v_3$ is a basis of $W$. 

\emph{(b)} If $\p(W)$ is tangent to $\mathbb{G}$ then there exists a basis 
$v_0 , \ldots , v_3$ of $V$ such that $v_0 \wedge v_1$ and $v_0 \wedge v_3 + 
v_1 \wedge v_2$ is a basis of $W$. 

\emph{(c)} If $\p(W) \subset \mathbb{G}$ then there exists a basis 
$v_0 , \ldots , v_3$ of $V$ such that $v_0 \wedge v_1$ and $v_0 \wedge v_2$ 
is a basis of $W$.  
\end{lemma}

\begin{proof} 
(a) The two points $P$ and $P^\prim$ in which $\p(W)$ intersects $\mathbb{G}$ 
correspond to two distinct vector subspaces $U$ and $U^\prim$ of $V$, of 
dimension 2. Let $V^\prim := U + U^\prim$. One must have $V^\prim = V$ because, 
otherwise, $\p(W) \subset \p(\bigwedge^2 V^\prim) \subset 
\mathbb{G}$. One chooses a basis $v_0,\, v_1$ of $U$ and a basis $v_2,\, v_3$ 
of $U^\prim$. 

(b) The point $P$ at which $\p(W)$ is tangent to $\mathbb{G}$ corresponds to 
a 2-dimensional vector subspace $U$ of $V$. The (embedded) tangent space to 
$\mathbb{G}$ at $P$ is $\text{T}_P\mathbb{G} = \p(U \wedge V)$ hence 
$W \subset U \wedge V$. Choose a basis $v_0,\, v_1$ of $U$. $W$ has a basis 
of the form $v_0 \wedge v_1$, $v_0 \wedge v_3 + v_1 \wedge v_2$, where 
$v_2$ and $v_3$ are two other vectors in $V$. Let $V^\prim$ be the subspace 
of $V$ spanned by $v_0 , \ldots , v_3$. One must have $V^\prim = V$ because, 
otherwise, $\p(W) \subset \p(\bigwedge^2 V^\prim) \subset 
\mathbb{G}$. 

(c) Choose two distinct points $P$ and $P^\prim$ of $\p(W)$. Using the notation 
from the proof of (a), one must have $V^\prim \neq V$. Since $W \subset 
\bigwedge^2 V^\prim$, the assertion follows easily.    
\end{proof} 

%\vskip7mm 

\begin{lemma}\label{L:t(-2)ra2o} 
Let $S_1 := {\fam0 H}^0(\sco_\piii(1))$ and let $W$, $\Delta$ and $\mathbb{G}$ 
be as in Lemma~\ref{L:lineinp5} for $V = S_1$. The map $L \mapsto 
[\bigwedge^2 {\fam0 H}^0(\sci_L(1))]$ identifies the lines in 
$\piii$ with the points of $\mathbb{G}$. Recall the exact sequence 
$0 \ra \Omega_\piii(1) \overset{u}{\lra} \sco_\piii \otimes_k S_1 
\overset{{\fam0 ev}}{\lra} \sco_\piii(1) \ra 0$ and the Koszul complex$\, :$ 
\[
\cdots \lra \sco_\piii(-1) \otimes_k {\textstyle \bigwedge}^2 S_1 
\overset{d_2}{\lra} \sco_\piii \otimes_k S_1 \overset{d_1}{\lra} \sco_\piii(1) 
\lra 0\, , 
\]
with $d_1$ the evaluation morphism. Let us denote by $\rho$ the composite 
morphism$\, :$ 
\[
\sco_\piii \otimes_k W \lra \sco_\piii \otimes_k 
{\textstyle \bigwedge}^2 S_1 \xra{d_2(1)} \Omega_\piii(2)\, . 
\]

\emph{(a)} If $\p(W)$ intersects $\mathbb{G}$ in two distinct points,  
corresponding to two disjoint lines $L$ and $L^\prime$ in $\piii$, then 
${\fam0 Coker}\, \rho^\vee \simeq \sco_{L\cup L^\prime}$.  

\emph{(b)} If $\p(W)$ is tangent to $\mathbb{G}$ at a point corresponding to 
a line $L \subset \piii$ then ${\fam0 Coker}\, \rho^\vee \simeq \sco_X$, where 
$X$ is the divisor $2L$ on a nonsingular quadric surface $Q \supset L$. 

\emph{(c)} If $\p(W) \subset \mathbb{G}$ then ${\fam0 Coker}\, \rho^\vee 
\simeq \sci_{\{x\},H}(1)$, for some plane $H \subset \piii$ and some point 
$x \in H$. 
\end{lemma} 

%\newpage 

\begin{proof} 
(a) According to Lemma~\ref{L:lineinp5}(a), there exists a basis $h_0, \ldots 
, h_3$ of $S_1$ such that $W$ is spanned by $h_0 \wedge h_1$ and $h_2 \wedge 
h_3$. $L$ (resp., $L^\prime$) is the line of equations $h_0 = h_1 = 0$ (resp., 
$h_2 = h_3 = 0$). Let $\alpha \colon S_1 \ra S_1$ be the linear automorphism 
defined by $\alpha(h_i) = -h_i$, $i = 0,\, 1$, $\alpha(h_i) = h_i$, 
$i = 2,\, 3$, and let $\theta$ be the composite morphism$\, :$ 
\[
\Omega_\piii(1) \overset{u}{\lra} \sco_\piii\otimes_k S_1 \overset{\alpha}{\lra} 
\sco_\piii\otimes_k S_1 \overset{\text{ev}}{\lra} \sco_\piii(1)\, . 
\] 
Recall that $d_2(1)$ maps $h \wedge h^\prime \in \bigwedge^2 S_1$ to 
$h \otimes h^\prime - h^\prime \otimes h \in \tH^0(\sco_\piii(1)) \otimes_k S_1$. 
Computing $\theta(1)(h_i \wedge h_j)$ one sees easily that $\text{Im}\, \theta 
= \sci_{L \cup L^\prime}(1)$ and that one has an exact sequence$\, :$ 
\[
0 \lra \sco_\piii \otimes_k W \overset{\rho}{\lra} \Omega_\piii(2) 
\xra{\theta(1)} \sci_{L \cup L^\prime}(2) \lra 0\, . 
\]
Dualizing this exact sequence and taking into account that$\, :$ 
\[
\sce xt_{\sco_\piii}^1(\sci_{L \cup L^\prime}(2)\, ,\, \sco_\piii) \simeq 
\omega_{L \cup L^\prime}(2) \simeq \sco_{L \cup L^\prime}\, , 
\]
one gets the result from the statement. 

\vskip3mm 

(b) According to Lemma~\ref{L:lineinp5}(b), there exists a basis $h_0, \ldots 
, h_3$ of $S_1$ such that $W$ is spanned by $h_0 \wedge h_1$ and $h_0 \wedge 
h_3 - h_1 \wedge h_2$. $L$ is the line of equations $h_0 = h_1 = 0$. 
Let $\beta \colon S_1 \ra S_1$ be the linear endomorphism 
defined by $\beta(h_i) = 0$, $i = 0,\, 1$, $\beta(h_i) = h_{i-2}$, 
$i = 2,\, 3$, and let $\tau$ be the composite morphism$\, :$ 
\[
\Omega_\piii(1) \overset{u}{\lra} \sco_\piii\otimes_k S_1 \overset{\beta}{\lra} 
\sco_\piii\otimes_k S_1 \overset{\text{ev}}{\lra} \sco_\piii(1)\, . 
\]   
Computing $\tau(1)(h_i \wedge h_j)$ one sees easily that $\text{Im}\, \tau 
= \sci_X(1)$, where $X$ is the divisor $2L$ on the quadric surface $Q$ of 
equation $h_0h_3 - h_1h_2 = 0$ and that one has an exact sequence$\, :$ 
\[
0 \lra \sco_\piii \otimes_k W \overset{\rho}{\lra} \Omega_\piii(2) 
\xra{\tau(1)} \sci_X(2) \lra 0\, . 
\]
One concludes as in (a). 

\vskip3mm 

(c) According to Lemma~\ref{L:lineinp5}(c), there exists a basis $h_0, \ldots 
, h_3$ of $S_1$ such that $W$ is spanned by $h_0 \wedge h_1$ and $h_0 \wedge 
h_2$. Let $x \in \piii$ be the point of equations $h_0 = h_1 = h_2 = 0$ and 
$H$ the plane of equation $h_0 = 0$. Let $\lambda : S_1 \ra k$ be the linear 
function defined by $\lambda(h_i) = 0$, $i = 0,\, 1,\, 2$, $\lambda(h_3) = 1$,  
and let $\pi$ be the composite morphism$\, :$ 
\[
\Omega_\piii(1) \overset{u}{\lra} \sco_\piii \otimes_k S_1 
\overset{\lambda}{\lra} \sco_\piii \, . 
\]  
Computing $\pi(1)(h_i \wedge h_j)$ one sees easily that $\text{Im}\, \pi 
= \sci_{\{x\}}$ and that one has an exact sequence$\, :$ 
\[
0 \ra \sco_\piii(-1) \otimes_k {\textstyle \bigwedge}^3 V^\prim 
\xra{d_3(1)} \sco_\piii \otimes_k {\textstyle \bigwedge}^2 V^\prim 
\xra{d_2(1)} \Omega_\piii(2) \xra{\pi(1)} \sci_{\{x\}}(1) \ra 0\, , 
\]
where $V^\prim = kh_0 + kh_1 + kh_2$. Let $\mu \colon \bigwedge^2 V^\prim \ra k$ 
be the linear function vanishing on $W$ and such that 
$\mu(h_1 \wedge h_2) = 1$. Dualizing the short exact sequence of (horizontal) 
complexes$\, :$ 
\[
\begin{CD} 
0 @>>> 0 @>>> \sco_\piii \otimes W @>{\rho}>> \Omega_\piii(2) @>>> 0\\  
@VVV @VVV @VVV @\vert @VVV\\ 
0 @>>> \sco_\piii(-1) \otimes {\textstyle \bigwedge}^3 V^\prim 
@>{d_3(1)}>> \sco_\piii \otimes {\textstyle \bigwedge}^2 V^\prim 
@>{d_2(1)}>> \Omega_\piii(2) @>>> 0\\ 
@VVV @V{\wr}VV @V{\mu}VV @VVV @VVV\\ 
0 @>>> \sco_\piii(-1) @>{h_0}>> \sco_\piii @>>> 0 @>>> 0   
\end{CD}
\] 
and writting a convenient part of the long exact sequence of cohomology 
sheaves one gets an exact sequence$\, :$ 
\[
0 \lra \sco_\piii(-1) \xra{\pi^\vee(-1)} \text{T}_\piii(-2) 
\overset{\rho^\vee}{\lra} \sco_\piii \otimes_k W^\vee \lra \sci_{\{x\},H}(1) 
\lra 0\, . 
\qedhere 
\] 
\end{proof} 

\begin{lemma}\label{L:hyperplaneinp5} 
Let $V$ be a $4$-dimensional $k$-vector space, $V^\prim$ a $3$-dimensional 
vector subspace of $V$, and $W^\prim$ a $5$-dimensional vector subspace of 
$\bigwedge^2 V$. If $\bigwedge^2 V^\prim \subset W^\prim$ 
then there exists a $2$-dimensional vector subspace $U$ of $V^\prim$ such that 
$W^\prim = U \wedge V$. 
\end{lemma} 

\begin{proof} 
Let $v_0 , \ldots , v_3$ be a $k$-basis of $V$ such that $v_0,\, v_1,\, v_2$ is 
a basis of $V^\prim$. Then $\bigwedge^2 V = 
\bigwedge^2 V^\prim \oplus (V^\prim \wedge v_3)$. Since 
$W^\prim \cap (V^\prim \wedge v_3)$ has dimension 2 (it has dimension at least 
2 and cannot have dimension 3 because $W^\prim$ is smaller than 
$\bigwedge^2 V$) it follows that $W^\prim = 
\bigwedge^2 V^\prim \oplus (W^\prim \cap (V^\prim \wedge v_3))$. 
There exists a 2-dimensional vector subspace $U$ of $V^\prim$ such that 
$W^\prim \cap (V^\prim \wedge v_3) = U \wedge v_3$. Since 
$\bigwedge^2 V^\prim = U \wedge V^\prim$ it follows that 
$W^\prim = U \wedge V$. 
\end{proof}

\begin{lemma}\label{L:5oraomega(2)} 
With the notation from Lemma~\ref{L:t(-2)ra2o} let $W^\prim$ be a 
$5$-dimensional vector subspace of $\bigwedge^2 S_1$ and let 
$\rho^\prim$ denote the composite morphism$\, :$ 
\[
\sco_\piii \otimes_k W^\prim \lra \sco_\piii \otimes_k 
{\textstyle \bigwedge}^2 S_1 \xra{d_2(1)} \Omega_\piii(2)\, . 
\] 

\emph{(a)} If $\p(W^\prim)$ intersects $\mathbb{G}$ transversely then 
$\rho^\prim$ is an epimorphism. 

\emph{(b)} If $\p(W^\prim)$ is tangent to $\mathbb{G}$ at a point $P$ 
corresponding to a line $L \subset \piii$ then ${\fam0 Coker}\, \rho^\prim 
\simeq \sco_L$. 
\end{lemma}

\begin{proof} 
(a) Let $x$ be a point in $\piii$. If $V^\prim = \tH^0(\sci_{\{x\}}(1))$ then 
the kernel of the surjective map$\, :$ 
\[
{\textstyle \bigwedge}^2 S_1 \Izo \tH^0(\Omega_\piii(2)) \lra 
\Omega_\piii(2)(x) 
\]
is $\bigwedge^2 V^\prim$. Since $\p(W^\prim )$ is not tangent to 
$\mathbb{G}$ at any of its points, Lemma~\ref{L:hyperplaneinp5} implies 
that $W^\prim$ does not contain $\bigwedge^2 V^\prim$. It follows 
that the composite map$\, :$ 
\[
W^\prim \xra{\tH^0(\rho^\prim )} \tH^0(\Omega_\piii(2)) \lra \Omega_\piii(2)(x) 
\] 
is surjective. 

(b) Let $U = \tH^0(\sci_L(1)) \subset S_1$. One has $W^\prim = U \wedge S_1$. 
Extend a basis $h_0,\, h_1$ of $U$ to a basis $h_0, \ldots ,h_3$ of $S_1$. 
Let $\mu^\prim \colon \bigwedge^2 S_1 \ra k$ be the linear function 
vanishing on $W^\prim$ and such that $\mu^\prim (h_2 \wedge h_3) = 1$. The image 
of the composite morphism$\, :$ 
\[
\sco_\piii(-1) \otimes {\textstyle \bigwedge}^3 S_1 \xra{d_3(1)} 
\sco_\piii \otimes {\textstyle \bigwedge}^2 S_1 
\overset{\mu^\prim}{\lra} \sco_\piii 
\]  
is $\sci_L$. One can conclude, now, as in the last part of the proof of 
Lemma~\ref{L:t(-2)ra2o}(c) (without dualizing the short exact sequence of 
complexes). 
\end{proof} 

\begin{lemma}\label{L:igamma(2)} 
Let $\Gamma$ be a subscheme of $\piii$ consisting of five simple points 
$P_1 , \ldots , P_5$ such that no four of them are contained in a plane. Then 
$\sci_\Gamma(2)$ is globally generated. 
\end{lemma} 

\begin{proof}
One has $\h^0(\sci_\Gamma(2)) = 5$ because the points of $\Gamma$ impose 
independent conditions on quadratic forms. 
It is easy to see that the linear system $\vert \, \tH^0(\sci_\Gamma(2)) \, 
\vert$ has no base point outside the union of the planes containing three of 
the points of $\Gamma$. Let $H$ be the plane containing $P_1,\, P_2,\, P_3$ 
and let $h = 0$ be an equation of $H$. The kernel of the restriction map 
$\tH^0(\sci_\Gamma(2)) \ra \tH^0(\sci_{H \cap \Gamma , H}(2))$ consists of the 
elements of the form $hh^\prime$, where $h^\prime$ is a linear form vanishing at 
$P_4$ and $P_5$. One deduces that the above restriction map is surjective 
hence the base points of the linear system $\vert \, \tH^0(\sci_\Gamma(2)) \, 
\vert$ contained in $H$ are $P_1$, $P_2$ and $P_3$. Consequently, the base 
points of $\vert \, \tH^0(\sci_\Gamma(2)) \, \vert$ are $P_1, \ldots , P_5$. 

Finally, considering the elements of $\tH^0(\sci_\Gamma(2))$ of the form 
$hh^\prime$ where either $h$ vanishes at $P_1$ and $P_2$ and $h^\prime$ vanishes 
at $P_3$, $P_4$ and $P_5$ or $h$ vanishes at $P_1$ and $P_3$ and $h^\prime$ 
vanishes at $P_2$, $P_4$ and $P_5$, one sees that the map 
$\tH^0(\sci_\Gamma(2)) \ra (\sci_{\{P_1\}}/\sci^2_{\{P_1\}})(2)$ is surjective. 

There are, actually, general results of Saint-Donat \cite[Lemme~3]{s-d} and 
Green and Lazarsfeld \cite[Thm.~1]{gl} which imply that the homogeneous ideal 
$I(\Gamma) \subset S$ is generated by $I(\Gamma)_2$.    
\end{proof} 

\begin{lemma}\label{L:3oqraix(1,3)} 
Let $Q \subset \piii$ be a nonsingular quadric surface and $x$ a point of 
$Q$. Fix an isomorphism $Q \simeq \pj \times \pj$. If $\scg$ is the kernel 
of a general epimorphism $3\sco_Q \ra \sci_{\{x\},Q}(1,3)$ then $\scg(2,2)$ is 
globally generated and ${\fam0 H}^1(\scg(2,2)) = 0$.  
\end{lemma} 

\begin{proof} 
Consider, for the moment, an arbitrary epimorphism $\phi \colon 3\sco_Q \ra 
\sci_{\{x\},Q}(1,3)$ and let $\scg$ be its kernel. One has $\text{det}\, \scg 
\simeq \sco_Q(-1,-3)$. Since $\scg$ is locally free of rank 2 it follows that 
$\scg^\vee \simeq \scg(1,3)$. Dualizing the exact sequence$\, :$ 
\[
0 \lra \scg \lra 3\sco_Q \overset{\phi}{\lra} \sci_{\{x\},Q}(1,3) \lra 0\, , 
\] 
one gets, consequently, an exact sequence$\, :$ 
\[
0 \lra \sco_Q(-1,-3) \lra 3\sco_Q \lra \scg(1,3) \lra \sco_{\{x\}} \lra 0\, . 
\]
One deduces, easily, that $\tH^1(\scg(1,2)) = 0$ and $\tH^1(\scg(2,2)) = 0$. 
If $L \subset Q$ is a line belonging to the linear system 
$\vert \, \sco_Q(1,0) \, \vert$ then, using the exact sequence$\, :$ 
\[
0 \lra \scg(1,2) \lra \scg(2,2) \lra \scg_L(2) \lra 0\, , 
\]
one gets that the map $\tH^0(\scg(2,2)) \ra \tH^0(\scg_L(2))$ is surjective. 
Consequently, $\scg(2,2)$ is globally generated if and only if $\scg_L(2)$ is 
globally generated, for every line $L \subset Q$ belonging to the linear
system $\vert \, \sco_Q(1,0) \, \vert$. 

For such a line $L$, $\scg_L$ is a subsheaf of $3\sco_L$. Since 
$\text{det}\, \scg_L \simeq \sco_L(-3)$ it follows that either $\scg_L \simeq 
\sco_L(-1) \oplus \sco_L(-2)$ or $\scg_L \simeq \sco_L \oplus \sco_L(-3)$. 
One deduces that $\scg(2,2)$ is globally generated if and only if 
$\tH^0(\scg_L) = 0$, for every line $L \subset Q$ belonging to the linear 
system $\vert \, \sco_Q(1,0) \, \vert$. 

We shall construct, now, an epimorphism $\phi \colon 3\sco_Q \ra 
\sci_{\{x\},Q}(1,3)$ such that its kernel $\scg$ satisfies the preceding 
condition. Choose $f_0 \in \tH^0(\sco_Q(1,3))$ such that its zero divisor is a 
nonsingular quartic rational curve $C \subset Q$ containing $x$. Let 
$L_0 \subset Q$ be the line from the linear system $\vert \, \sco_Q(1,0) \, 
\vert$ passing through $x$.   
If $p_1,\, p_2 \colon C \ra \pj$ are the canonical projections then $p_2$ is 
an isomorphism and $p_1^\ast\sco_\pj(1) \simeq p_2^\ast\sco_\pj(3)$ hence $p_1$ 
is defined by a base point free 2-dimensional vector subspace $\Lambda_1 
\subset \tH^0(p_2^\ast\sco_\pj(3))$. Choose $0 \neq \lambda \in 
\tH^0(p_2^\ast\sco_\pj(1))$ vanishing at $x$ and a base point free 2-dimensional 
vector subspace $\Lambda \subset \tH^0(p_2^\ast\sco_\pj(5))$. We work under the 
following$\, :$ 

\vskip2mm 

\noindent 
\emph{Assumption.}\quad None of the zero divisors of the elements of 
$\Lambda \setminus \{0\}$ contains a fiber of $p_1$ or the divisor 
$(L_0 \cap C) - x$ as a subscheme. 

\vskip2mm 

\noindent 
Denoting by $q_0$ a nonzero element of $\tH^0(p_2^\ast\sco_\pj(2))$ whose 
zero divisor is $(L_0 \cap C) - x$, the above assumption is equivalent to the 
fact that, on one hand, $\Lambda$ intersects the image of the bilinear 
multiplication map $\Lambda_1 \times \tH^0(p_2^\ast\sco_\pj(2)) \ra 
\tH^0(p_2^\ast\sco_\pj(5))$ only in $0$ and, on the 
other hand, that it intersects the vector subspace 
$q_0\tH^0(p_2^\ast\sco_\pj(3))$ of $\tH^0(p_2^\ast\sco_\pj(5))$ only in $0$. It 
follows that a general $\Lambda$ satisfies the assumption. 

\vskip2mm 

Choose, now, a $k$-basis $g_1,\, g_2$ of $\Lambda$.  
One has $\sco_Q(1,3) \vb C \simeq p_2^\ast\sco_\pj(6)$ and, using the exact 
sequence$\, :$ 
\[
0 \lra \sco_Q \overset{f_0}{\lra} \sco_Q(1,3) \lra \sco_Q(1,3) \vb C 
\lra 0\, , 
\]   
one sees that the map $\tH^0(\sco_Q(1,3)) \ra \tH^0(\sco_Q(1,3) \vb C)$ is 
surjective. Lift the element $\lambda g_i$ of $\tH^0(p_2^\ast\sco_\pj(6))$ to  
$f_i \in \tH^0(\sco_Q(1,3))$, $i = 1,\, 2$. $f_0,\, f_1,\, f_2$ define an 
epimorphism $\phi \colon 3\sco_Q \ra \sci_{\{x\},Q}(1,3)$. Let $\scg$ be its 
kernel.

Let us denote by $\psi$ the composite morphism $3\sco_Q \overset{\phi}{\lra} 
\sci_{\{x\},Q}(1,3) \hookrightarrow \sco_Q(1,3)$. If $L \subset Q$ is a line 
from the linear system $\vert \, \sco_Q(1,0) \, \vert$ then our assumption 
implies that $\tH^0(\psi_L) \colon \tH^0(3\sco_L) \ra \tH^0(\sco_L(3))$ is 
injective [\emph{indeed}, if $c_1f_1 + c_2f_2$ is a nontrivial linear 
combination of $f_1$ and $f_2$ then $(c_1f_1 + c_2f_2) \vb L$ does not vanish 
on the divisor $L \cap C$ because, if it does, then either $c_1g_1 +c_2g_2$ 
vanishes on $L \cap C$, if $L \neq L_0$, or it vanishes on $(L_0 \cap C) - x$, 
if $L = L_0$. Since the zero divisor of $f_0 \vb L$ is $L \cap C$, one derives  
that $f_0 \vb L,\, f_1 \vb L,\, f_2 \vb L$ are linearly independent]. 
Denoting by $\phi_L$ the morphism$\, :$ 
\[
\phi \otimes_{\sco_Q} \text{id}_{\sco_L} \colon 3\sco_L \lra 
\sci_{\{x\},Q}(1,3) \otimes_{\sco_Q} \sco_L  
\]  
one deduces that $\tH^0(\phi_L)$ is injective, hence $\tH^0(\scg_L) = 0$. 
\end{proof}

\begin{lemma}\label{L:2oraogamma} 
Let $\Gamma$ be a subscheme of $\p^n$, $n \geq 2$, consisting of $n + 1$ 
simple points $P_0, \ldots , P_n$ not contained in a hyperplane. Consider a 
morphism $\e \colon 2\sco_{\p^n} \ra \sco_\Gamma(1)$ with the property that the 
composite map$\, :$ 
\[
{\fam0 H}^0(2\sco_{\p^n}) \xra{{\fam0 H}^0(\e)} {\fam0 H}^0(\sco_\Gamma(1)) \lra 
{\fam0 H}^0(\sco_{\{P_i , P_j\}}(1)) 
\]
is bijective, for $0 \leq i < j \leq n$. Then one has an exact sequence$\, :$ 
\[
{\textstyle \binom{n+1}{3}}\sco_{\p^n}(-3) \overset{d_2}{\lra} 
(n+1)\sco_{\p^n}(-1) 
\overset{d_1}{\lra} 2\sco_{\p^n} \overset{\e}{\lra} \sco_\Gamma(1) \lra 0\, . 
\]
\end{lemma}

\begin{proof} 
We begin with a general observation$\, :$ let $\scf$ be a coherent sheaf on 
$\p^n$ such that $\tH^0(\scf(-1)) = 0$, $\scf$ is 1-regular, and the 
multiplication map $\tH^0(\scf) \otimes S_1 \ra \tH^0(\scf(1))$ is bijective. 
Then the graded $S$-module $\tH^0_\ast(\scf)$ has a graded resolution of the 
form$\, :$ 
\[
0 \lra \beta_nS(-n-1) \lra \cdots \lra \beta_1S(-2) \lra 
\tH^0(\scf) \otimes_k S \lra \tH^0_\ast(\scf) \lra 0\, . 
\]

\noindent 
\emph{Indeed}, the graded $S$-module $\tH^0_\ast(\scf)$ is generated in degrees 
$\leq 1$ (because $\scf$ is 1-regular) hence it is, actually, generated by 
$\tH^0(\scf)$. The evaluation morphism $\tH^0(\scf) \otimes_k \sco_{\p^n} \ra 
\scf$ of $\scf$ is, consequently, an epimorphism and its kernel $\scg$ is 
2-regular and $\tH^0(\scg(1)) = 0$. The above assertion follows immediately. 
Notice, also, that$\, :$ 
\[
\beta_1 = \dim_k\Ker(\tH^0(\scf) \otimes S_2 \ra \tH^0(\scf(2)))\, . 
\]

Now, since $\tH^0(\sco_{\p^n}(1)) \izo \tH^0(\sco_\Gamma(1))$, there exists two 
linear forms $h,\, h^\prime \in \tH^0(\sco_{\p^n}(1))$ such that $\e$ is defined 
by $h \vb \Gamma$ and $h^\prime \vb \Gamma$. If $L$ is the linear subspace of 
$\p^n$ of equations $h = h^\prime = 0$ then the hypothesis of the lemma is 
equivalent to the fact that $L$ does not intersect any of the lines 
$\overline{P_iP_j}$, $0 \leq i < j \leq n$. In particular, $L$ has codimension 
2 in $\p^n$. 

Since $\e$ can be written as the composite map$\, :$ 
\[
2\sco_{\p^n} \xra{(h\, ,\, h^\prime)} \sci_L(1) \hookrightarrow \sco_{\p^n}(1) 
\twoheadrightarrow \sco_\Gamma(1)\, , 
\] 
it follows that $\e$ is an epimorphism and that if $\sck$ is its kernel then 
one has exact sequences$\, :$ 
\begin{gather*} 
0 \lra \sco_{\p^n}(-1) \lra \sck \lra \sci_{L \cup \Gamma}(1) \lra 0\, ,\\ 
0 \lra \sci_{L \cup \Gamma} \lra \sci_L \lra \sco_\Gamma \lra 0\, . 
\end{gather*}
The map $\tH^0(\sci_L(2)) \ra \tH^0(\sco_\Gamma(2))$ is clearly surjective, 
hence $\sci_{L \cup \Gamma}$ is 3-regular. Moreover, 
$\tH^0(\sci_{L \cup \Gamma}(1)) = 0$. One deduces that $\sck$ is 2-regular and 
$\tH^0(\sck) = 0$.   

\vskip2mm 

\noindent 
{\bf Claim.}\quad \emph{The multiplication map} $\mu \colon 
\tH^0(\sci_{L \cup \Gamma}(2)) \otimes S_1 \ra \tH^0(\sci_{L \cup \Gamma}(3))$ 
\emph{is bijective}. 

\vskip2mm 

\noindent 
\emph{Indeed}, applying the Snake Lemma to the diagram$\, :$ 
\[
\begin{CD} 
0 @>>> \tH^0(\sci_{L \cup \Gamma}(2)) \otimes S_1 @>>> \tH^0(\sci_L(2)) 
\otimes S_1 @>>> \tH^0(\sco_\Gamma(2)) \otimes S_1 @>>> 0\\ 
@. @VV{\mu}V @VV{\mu_L}V @VV{\mu_\Gamma}V\\ 
0 @>>> \tH^0(\sci_{L \cup \Gamma}(3)) @>>> \tH^0(\sci_L(3)) @>>> 
\tH^0(\sco_\Gamma(3)) @>>> 0 
\end{CD}
\]
one sees that it suffices to show that the map $\Ker \mu_L \ra 
\Ker \mu_\Gamma$ is surjective (because, by dimensional reasons, it will be 
bijective). Let $h_i = 0$ (resp., $h_i^\prime = 0$) be an equation of the 
hyperplane containing $\Gamma \setminus \{P_i\}$ (resp., $L \cup \{P_i\}$), and 
let $e_i$ be the element of $\tH^0(\sco_\Gamma)$ defined by $e_i(P_j) = 
\delta_{ij}$, $j = 0, \ldots , n$, $i = 0, \ldots , n$. Since 
$\mu_\Gamma(e_i \otimes h_j) = \delta_{ij}e_i$ it follows that $\Ker \mu_\Gamma$ 
is generated by the elements $e_i \otimes h_j$ with $i \neq j$. But, for 
$i \neq j$, 
\[
\Ker \mu_L \ni h_ih_j^\prime \otimes h_j - h_jh_j^\prime \otimes h_i \mapsto 
(h_ih_j^\prime)(P_i) \cdot e_i \otimes h_j \in \Ker \mu_\Gamma  
\]   
hence the map $\Ker \mu_L \ra \Ker \mu_\Gamma$ is surjective. 

\vskip2mm 

One deduces, from the claim, that the multiplication map $\tH^0(\sck(1)) 
\otimes S_1 \ra \tH^0(\sck(2))$ is bijective. The conclusion of the lemma 
follows, now, by applying to $\scf : = \sck(1)$ the observation from the 
beginning of the proof. One can, actually, do something more, namely one can 
explicitate the matrices defining the differentials of the exact sequence 
from the conclusion of the lemma.  

\vskip2mm 

\noindent 
\emph{Indeed}, $h_i^\prime$ can be expressed as a linear combination 
$h_i^\prime = a_ih + a_i^\prime h^\prime$, $a_i,\, a_i^\prime \in k$, $i = 0, 
\ldots ,n$. Since $h_ih_i^\prime$ vanishes on $\Gamma$, $i = 0, \ldots , n$, 
one deduces that the differential $d_1$ is defined by the matrix$\, :$ 
\[
\begin{pmatrix} 
a_0h_0 & a_1h_1 & \ldots & a_nh_n\\ 
a_0^\prime h_0 & a_1^\prime h_1 & \ldots & a_n^\prime h_n
\end{pmatrix}\, . 
\]  
Notice that the hypothesis of the lemma implies that all the $2 \times 2$ 
minors of this matrix are non-zero. Notice, also, that if one chooses 
$h = h_0^\prime$ and $h^\prime = h_1^\prime$ then $a_0 = 1$, $a_1 = 0$, $a_0^\prime 
= 0$, and $a_1^\prime = 1$. 

Let $\gamma_0, \ldots , \gamma_n$ be the columns of the above matrix. 
For $0 \leq i < j < l \leq n$, let $(x_{ijl} , y_{ijl} , z_{ijl}) \in k^3$ be a 
non-trivial solution of the system of linear equations$\, :$ 
\[
\begin{cases} 
a_ix + a_jy + a_lz = 0\, ,\\
a_i^\prime x + a_j^\prime y + a_l^\prime z = 0\, .
\end{cases}
\] 
One must have $x_{ijl} \neq 0$, $y_{ijl} \neq 0$, $z_{ijl} \neq 0$ (because of 
the non-vanishing of all the $2\times 2$ minors of the above matrix). One 
deduces that$\, :$ 
\[
x_{ijl}h_jh_l \cdot \gamma_i + y_{ijl}h_ih_l \cdot \gamma_j + 
z_{ijl}h_ih_j \cdot \gamma_l = 0\, ,\  0 \leq i < j < l \leq n\, , 
\]
is a complete system of quadratic relations between the columns of the above 
matrix. One deduces that the transpose of the column of the matrix of $d_2$ 
indexed by a triple $(i,j,l)$ with $0 \leq i < j < l \leq n$ is$\, :$ 
\[
(0\, , \ldots , 0\, , x_{ijl}h_jh_l\, , 0\, , \ldots , 0\, , y_{ijl}h_ih_l\, , 
0\, , \ldots , 0\, , z_{ijl}h_ih_j\, , 0\, , \ldots , 0)\, . 
\qedhere 
\]  
\end{proof} 

\begin{lemma}\label{L:vahlenquintic} 
Let $C \subset \piii$ be a nonsingular rational quintic curve, not contained 
in a quadric surface, and let $L \subset \piii$ be the unique $4$-secant of 
$C$. Then one has an exact sequence$\, :$ 
\[
{\fam0 H}^0(\sci_C(3)) \otimes_k \sco_\piii \overset{{\fam0 ev}}{\lra} 
\sci_C(3) \lra \sci_{L \cap C\, ,\, L}(3) \lra 0\, .  
\] 
Moreover, ${\fam0 H}^1(\sci_C(3)) = 0$.  
\end{lemma} 

\begin{proof} 
Consider the exact sequence$\, :$ 
\[
0 \lra \sco_{C \cup L} \lra \sco_C \times \sco_L \lra \sco_{C \cap L} \lra 0\, . 
\]
The map $\tH^0(\sco_C(l)) \ra \tH^0(\sco_{C \cap L}(l))$ is surjective for 
$l \geq 1$. One deduces that$\, :$ 
\[
\h^0(\sco_{C \cup L}(l)) = \h^0(\sco_C(l)) + \h^0(\sco_L(l)) - 4\  \text{and}\  
\tH^1(\sco_{C \cup L}(l)) = 0\, ,\  \forall l \geq 1\, . 
\]
It follows, in particular, that $\tH^1(\sci_{C \cup L}(2)) = 0$ and 
$\tH^2(\sci_{C \cup L}(1)) = 0$ hence $\sci_{C \cup L}$ is 3-regular. Since 
$\h^0(\sci_{C \cup L}(3)) = 4$ and $\h^0(\sci_{C \cup L}(4)) = 13$, one gets an 
exact sequence$\, :$ 
\[
0 \lra 3\sco_\piii(-1) \lra 4\sco_\piii \lra \sci_{C \cup L}(3) \lra 0\, . 
\] 
Since $\tH^0(\sci_{C \cup L}(3)) = \tH^0(\sci_C(3))$ (any cubic surface 
containing $C$ must also contain its 4-secant $L$), one deduces that the 
cokernel of the evaluation morphism of $\sci_C(3)$ is$\, :$ 
\[
\sci_C(3)/\sci_{C \cup L}(3) \simeq \sci_{C \cap L}(3)/\sci_L(3) =: 
\sci_{C \cap L\, ,\, L}(3)\, . 
\]
Moreover, since $4 = \h^0(\sci_{C \cup L}(3)) = \h^0(\sci_C(3))$ it follows that 
$\tH^1(\sci_C(3)) = 0$. 
\end{proof}

\begin{lemma}\label{L:3lines+4points} 
Let $Y$ be the union of three mutually disjoint lines $L_1,\, L_2,\, L_3$ in 
$\piii$ and let $P_0 , \ldots , P_3$ be four points in $\piii$ such that 
$\Gamma := \{P_0 , \ldots , P_3\}$ is not contained in a plane. We assume that 
none of the four points belongs to the quadric surface $Q \subset \piii$ 
containing $Y$ and that ${\fam0 H}^0(\sci_{(Y \setminus L_l) \cup \Gamma}(2)) 
= 0$, $l = 1,\, 2,\, 3$. Then the homogeneous ideal of $Y \cup \Gamma$ is 
generated by cubic forms.  
\end{lemma} 

\begin{proof} 
The hypothesis $\tH^0(\sci_{(Y \setminus L_l) \cup \Gamma}(2)) = 0$ implies that 
$\tH^0(\sci_{Y \setminus L_l}(2)) \izo \tH^0(\sco_\Gamma(2))$, $l = 1,\, 2,\, 3$. 
One deduces easily that the map $\tH^0(\sci_Y(3)) \ra \tH^0(\sco_\Gamma(3))$ 
is surjective hence $\tH^1(\sci_{Y \cup \Gamma}(3)) = 0$ and 
$\h^0(\sci_{Y \cup \Gamma}(3)) = 4$. It follows that $\sci_{Y \cup \Gamma}$ is 
4-regular. It remains to show that the multiplication map  
$
\mu \colon S_1 \otimes \tH^0(\sci_{Y \cup \Gamma}(3)) \ra 
\tH^0(\sci_{Y \cup \Gamma}(4))  
$ 
is surjective. Using the commutative diagram$\, :$ 
\[
\begin{CD} 
0 @>>> S_1 \otimes \tH^0(\sci_{Y \cup \Gamma}(3)) @>>> 
S_1 \otimes \tH^0(\sci_Y(3)) @>>> S_1 \otimes \tH^0(\sco_\Gamma(3)) @>>> 0\\ 
@. @VV{\mu}V @VV{\mu_Y}V @VV{\mu_\Gamma}V\\ 
0 @>>> \tH^0(\sci_{Y \cup \Gamma}(4)) @>>> \tH^0(\sci_Y(4)) @>>> 
\tH^0(\sco_\Gamma(4)) @>>> 0 
\end{CD}
\] 
one sees that it suffices to show that the map $\Ker \mu_Y \ra \Ker \mu_\Gamma$ 
induced by this diagram is surjective. 

For $0 \leq i \leq 3$, let $h_i = 0$ be an equation of the plane containing 
$\Gamma \setminus \{P_i\}$ and let $e_i$ be the element of $\tH^0(\sco_\Gamma)$ 
defined by $e_i(P_j) = \delta_{ij}$, $j = 0, \ldots , 3$. Since 
$\mu_\Gamma(h_j \otimes e_i) = h_j(P_i)e_i$ it follows that $\Ker \mu_\Gamma$ has 
a $k$-basis consisting of the elements $h_j \otimes e_i$, with $0 \leq i \leq 
3$, $0 \leq j \leq 3$ and $i \neq j$. 

Take an $i \in \{0, \ldots , 3\}$. For $1 \leq l \leq 3$, let $h_{il} = 0$ 
be an equation of the plane containing $L_l \cup \{P_i\}$. We assert that 
$h_{i1},\, h_{i2},\, h_{i3}$ are linearly independent. \emph{Indeed}, if they 
are linearly dependent then they vanish on a line $L \subset \piii$ 
containing $P_i$. $L$ is, then, a 3-secant of $Y = L_1 \cup L_2 \cup L_3$ 
hence $L$ is contained in the quadric surface $Q \subset \piii$ containing 
$Y$ and this \emph{contradicts} the fact that $P_i \notin Q$. It, thus, 
remains that $h_{i1},\, h_{i2},\, h_{i3}$ are linearly independent. 

For $1 \leq l \leq 3$, let $q_{il} = 0$ be an equation of the unique quadric 
surface containing $(Y \setminus L_l) \cup (\Gamma \setminus \{P_i\})$ (recall 
that $\tH^0(\sci_{Y \setminus L_l}(2)) \izo \tH^0(\sco_\Gamma(2))$). Choose, also, 
$h_{il}^\prime \in S_1$ vanishing on $L_l$ but not at $P_i$. Then$\, :$ 
\[
h_{il} \otimes h_{il}^\prime q_{il} - h_{il}^\prime \otimes h_{il}q_{il} 
\]  
belongs to $\Ker \mu_Y$ and its image into $\Ker \mu_\Gamma$ is 
$h_{il} \otimes (h_{il}^\prime q_{il})(P_i)e_i$. Since $h_{il}^\prime$ and $q_{il}$ 
do not vanish at $P_i$, one deduces that $h_{il} \otimes e_i$ belongs to the 
image of $\Ker \mu_Y \ra \Ker \mu_\Gamma$. 

Finally, if $j \in \{0, \ldots , 3\} \setminus \{i\}$ then $h_j(P_i) = 0$. 
Since $h_{i1},\, h_{i2},\, h_{i3}$ vanish at $P_i$ and are linearly independent, 
it follows that $h_j$ is a linear combination of $h_{i1},\, h_{i2},\, h_{i3}$ 
hence $h_j \otimes e_i$ belongs to the image of $\Ker \mu_Y \ra 
\Ker \mu_\Gamma$. Since $i \in \{0, \ldots , 3\}$ and $j \in \{0, \ldots , 3\} 
\setminus \{i\}$ were arbitrary, the map $\Ker \mu_Y \ra \Ker \mu_\Gamma$ is 
surjective. 
\end{proof} 

\begin{lemma}\label{L:d=6g=2} 
If $C$ is a nonsingular (connected) curve of degree $6$ and genus $2$ in 
$\piii$ then $\sci_C(2)$ is the cohomology sheaf of a monad of the form$\, :$ 
\[
0 \lra 2\sco_\piii(-2) \lra 3\sco_\piii \oplus \sco_\piii(-1) \lra 
\sco_\piii(1) 
\lra 0\, . 
\]
\end{lemma} 

\begin{proof} 
$C$ is not contained in a surface of degree 2 (see 
\cite[Chap.~IV,~Remark~6.4.1]{hag}). Using Riemann-Roch on $C$ one deduces 
that $\h^1(\sci_C(1)) = 1$, $\h^1(\sci_C(2)) = 1$ and that $\h^0(\sci_C(3)) 
\geq 3$. It follows that $C$ is directly linked, by a complete intersection 
of type $(3,3)$, to a curve $Y$ of degree 3, locally Cohen-Macaulay and 
locally complete intersection except at finitely many points. One has 
$\h^1(\sci_Y) = 1$ and $\h^1(\sci_Y(1)) = 1$ hence $\h^0(\sco_Y) = 2$ and 
$\h^0(\sco_Y(1)) = 5$.  

If $Y$ is not connected then one of its connected components must be a line 
and the union of the other connected components is a curve $X$ of degree 
2 with $\h^0(\sco_X) = 1$ hence $X$ is a complete intersection of type 
$(1,2)$. One deduces that, in this case, $\sci_Y$ has a resolution of the 
form$\, :$ 
\begin{equation}\label{E:resiy} 
0 \ra \sco_\piii(-5) \lra \sco_\piii(-3) \oplus 3\sco_\piii(-4) \lra 
2\sco_\piii(-2) \oplus 2\sco_\piii(-3) \lra \sci_Y \ra 0 
\end{equation} 
(see, for example, \cite[Lemma~B.1]{acm2}). Notice that the elements of 
$\tH^0(\sci_Y(2))$ are multiples of the equation of the plane containing $X$.  

If $Y$ is connected then it cannot be reduced (because $\h^0(\sco_Y) = 2$) 
hence it is either the union of a line $L$ and of a double structure $X$ 
on another line $L_1$ intersecting $L$ or it is a quasiprimitive triple 
structure one a line $L$.  

In the former case one has $\sci_X/\sci_{L_1} \simeq \sco_{L_1}(l)$ 
for some $l \geq -1$. The proof of \cite[Prop.~B.10]{acm2} shows that one has 
an exact sequence of the form$\, :$ 
\[
0 \lra \sci_{L \cup L_1^{(1)}} \lra \sci_{L \cup X} \lra \sco_L(-l^\prim) 
\lra 0\, , 
\]  
where $l^\prim$ is either $l+2$ or $l+3$, and $L_1^{(1)}$ is the first 
infinitesimal neighbourhood of $L_1$ in $\piii$. Moreover, by 
\cite[Lemma~B.9]{acm2}, $L \cup L_1^{(1)}$ is arithmetically Cohen-Macaulay 
in $\piii$, $\h^2(\sci_{L \cup L_1^{(1)}}) = 1$ and 
$\h^2(\sci_{L \cup L_1^{(1)}}(1)) 
= 0$. Since $\h^1(\sci_Y(t)) = 1$, $t = 0,\, 1$, it follows that 
$l^\prim = 3$. 
Using \cite[Lemma~B.9]{acm2} again, one deduces that $\sci_Y$ has a resolution 
of the form \eqref{E:resiy} above. Notice, also, that $\tH^0(\sci_Y(2)) = 
\tH^0(\sci_{L \cup L_1^{(1)}}(2))$ hence the elements of $\tH^0(\sci_Y(2))$ 
are multiples of the equation of the plane containing $L \cup L_1$. 

In the latter case, $Y$ contains, as a subscheme, a double structure $X$ on 
the line $L$ such that $\sci_L/\sci_X \simeq \sco_L(l)$ for some $l \geq -1$ 
and $\sci_X/\sci_Y \simeq \sco_L(2l+m)$ for some $m \geq 0$ (see, for example, 
\cite[\S A.5]{acm2}). Since $\h^0(\sco_Y) = 2$ it follows that $l = -1$ and 
$m = 2$. Since $l = -1$, $X$ is the divisor $2L$ one some plane $H \supset L$. 
Moreover, one has an exact sequence$\, :$ 
\[
0 \lra \sci_L\sci_X \lra \sci_Y \lra \sco_L(-3) \lra 0 
\]    
(see the exact sequence (A.33) at the end of \cite[\S A.5]{acm2}). One 
deduces that $\sci_Y$ has a resolution of the form \eqref{E:resiy} above and 
that the elements of $\tH^0(\sci_Y(2)) = \tH^0((\sci_L\sci_X)(2))$ are 
multiples of the equation of $H$. 

Let, now, $f = g = 0$ be the equations of the complete intersection linking 
$C$ to $Y$. Since the linear system $\vert \,  kf + kg \, \vert$ consists 
of irreducible cubic surfaces (because all of them contain $C$) it follows 
that $S_1I(Y)_2 \cap (kf + kg) = (0)$. Applying, now, Ferrand's result about 
resolutions under liaison (see \cite[Prop.~2.5]{ps}) to the resolution 
\eqref{E:resiy} of $\sci_Y$ one deduces that $\sci_C(2)$ is the cohomology of 
a monad of the form from the statement.    
\end{proof} 

The next result is related to Construction 6.4 in the proof of 
Prop.~\ref{P:h2e(-3)=0h1e(-3)neq0c2=12}. 

\begin{lemma}\label{L:d=6g=0} 
Let $C$ be a nonsingular (connected) rational curve of degree $6$ in $\piii$, 
not contained in a quadric surface. If ${\fam0 H}^1(\sci_C(3)) \neq 0$ then 
$C$ admits a $5$-secant. 
\end{lemma}

\begin{proof} 
The condition $\tH^1(\sci_C(3)) \neq 0$ is equivalent to $\h^0(\sci_C(3)) \geq 
2$. Then $C$ is directly linked, by a complete intersection of type $(3,3)$, 
to a curve $Y$ of degree 3, locally Cohen-Macaulay and locally complete 
intersection except at finitely many points. As it is well known from 
liaison theory, $\tH^1_\ast(\sci_Y)(6) \simeq \tH^1_\ast(\sci_C)^\vee(4)$ 
hence 
$\h^1(\sci_Y) = \h^1(\sci_C(2)) = 3$. It follows that $\h^0(\sco_Y) = 4$. In 
particular, $Y$ is not reduced.  

Now, we make the following observation$\, :$ let $Y^\prim$ be a locally 
complete intersection curve of degree 2, which is a subscheme of $Y$. If 
$C^\prim$ is the residual curve of $Y^\prim$ in the above complete 
intersection 
then one must have $C^\prim = C \cup L$, for some line $L$. Again, from 
general facts about liaison, $\chi(\sco_{C^\prim}) - \chi(\sco_{Y^\prim}) 
= -5$ hence the length of $C \cap L$ is $7 - \chi(\sco_{Y^\prim})$. Since 
$C$ admits no 6-secant one deduces that $Y$ has no subscheme which is a 
complete intersection of type $(1,2)$. 

It follows, from the above observation, that either $Y = X \cup L^\prime$ 
where $X$ is a double structure on a line $L$ such that $\sci_L/\sci_X \simeq 
\sco_L(l)$ for some $l \geq 0$ and $L^\prime$ is another line, not 
intersecting 
$L$, or $Y$ is a triple structure on a line $L$, containing a double 
structure $X$ on $L$ such that $\sci_L/\sci_X \simeq \sco_L(l)$ for some 
$l \geq 0$ and $\sci_X/\sci_Y \simeq \sco_L(2l + m)$, for some $m \geq 0$ 
(see, for example, \cite[\S A.4,~\S A.5,~\S B.1]{acm2}). Since 
$\h^0(\sco_Y) = 4$, one deduces that $l = 1$ in the former case and that 
$l = 0$, $m = 1$ in the latter one. Anyway, in both cases $Y$ admits  
a closed subscheme of $Y^\prim$ which is a locally complete intersection curve 
of degree 2, with $\chi(\sco_{Y^\prim}) = 2$ ($Y^\prim = L \cup L^\prime$ in 
the 
former case and $Y^\prim = X$ in the latter one) hence, by the above 
observation, $L$ is a 5-secant of $C$.      
\end{proof}

\end{document}